\documentclass{article}
\usepackage[utf8]{inputenc}

\usepackage[in]{fullpage}
\usepackage{bbm}
\usepackage{amsmath}
\usepackage{amsfonts}
\usepackage{amssymb}
\usepackage{amsthm}
\usepackage{enumerate}
\usepackage{color}
\usepackage{mathrsfs}
\usepackage{stmaryrd}
\usepackage{hyperref}

\SetSymbolFont{stmry}{bold}{U}{stmry}{m}{n}

\newtheorem{thm}{Theorem}
\newtheorem{cor}[thm]{Corollary}
\newtheorem{lem}[thm]{Lemma}
\newtheorem{prop}[thm]{Proposition}

\theoremstyle{definition}
\newtheorem{defn}[thm]{Definition}
\newtheorem*{defn*}{Definition}
\newtheorem{assump}[thm]{Assumption}

\theoremstyle{remark}
\newtheorem{rem}[thm]{Remark}
\newtheorem{nrem}[thm]{Notational Remark}

\newcommand{\deq}{\mathrel{\mathop:}=}
\newcommand{\e}[1]{\mathrm{e}^{#1}}
\newcommand{\R} {\mathbb{R}}
\newcommand{\C} {\mathbb{C}}
\newcommand{\N} {\mathbb{N}}

\newcommand{\E} {\mathbb{E}}

\newcommand{\braN} {\llbracket 1,N\rrbracket}
\newcommand{\adj}{^{*}} 
\newcommand{\tp}{^{\intercal}}

\newcommand{\angi}{\langle i\rangle}
\newcommand{\mr}[1]{\mathring{#1}}

\DeclareMathOperator{\diag}{diag}
\DeclareMathOperator{\tr}{tr}
\DeclareMathOperator{\Tr}{Tr}

\DeclareMathOperator{\supp}{supp}

\DeclareMathOperator{\re}{\mathrm{Re}}
\DeclareMathOperator{\im}{\mathrm{Im}}

\newcommand{\caD}{{\mathcal D}}

\newcommand{\caG}{{\mathcal G}}
\newcommand{\caH}{{\mathcal H}}

\newcommand{\caK}{{\mathcal K}}
\newcommand{\caL}{{\mathcal L}}

\newcommand{\caN}{{\mathcal N}}
\newcommand{\caO}{{\mathcal O}}

\newcommand{\caQ}{{\mathcal Q}}

\newcommand{\caS}{{\mathcal S}}
\newcommand{\caT}{{\mathcal T}}
\newcommand{\caU}{{\mathcal U}}

\newcommand{\caW}{{\mathcal W}}

\newcommand{\caY}{{\mathcal Y}}
\newcommand{\caZ}{{\mathcal Z}}

\newcommand{\bbC}{{\mathbb C}}

\newcommand{\bbE}{{\mathbb E}}

\newcommand{\bbN}{{\mathbb N}}

\newcommand{\bbP}{{\mathbb P}}

\newcommand{\bbR}{{\mathbb R}}
\newcommand{\bbS}{{\mathbb S}}
\newcommand{\bbT}{{\mathbb T}}

\newcommand{\bfu}{{\mathbf u}}
\newcommand{\bfv}{{\mathbf v}}
\newcommand{\bfw}{{\mathbf w}}
\newcommand{\bfx}{{\mathbf x}}
\newcommand{\bfy}{{\mathbf y}}

\newcommand{\fra}{{\mathfrak a}}
\newcommand{\frb}{{\mathfrak b}}

\newcommand{\frd}{{\mathfrak d}}

\newcommand{\frj}{{\mathfrak j}}

\newcommand{\frl}{{\mathfrak l}}
\newcommand{\frm}{{\mathfrak m}}
\newcommand{\frn}{{\mathfrak n}}

\newcommand{\frD}{{\mathfrak D}}

\newcommand{\frL}{{\mathfrak L}}

\newcommand{\frX}{{\mathfrak X}}

\newcommand{\frZ}{{\mathfrak Z}}

\newcommand{\bsd}{{\boldsymbol d}}
\newcommand{\bse}{{\boldsymbol e}}

\newcommand{\bsg}{{\boldsymbol g}}
\newcommand{\bsh}{{\boldsymbol h}}

\newcommand{\bsr}{{\boldsymbol r}}

\newcommand{\bsv}{{\boldsymbol v}}
\newcommand{\bsw}{{\boldsymbol w}}
\newcommand{\bsx}{{\boldsymbol x}}

\newcommand{\rme}{\mathrm{e}}

\newcommand{\rmk}{\mathrm{k}}

\newcommand{\rmm}{\mathrm{m}}
\newcommand{\rmn}{\mathrm{n}}

\newcommand{\rmx}{\mathrm{x}}

\newcommand{\anga}{\langle a\rangle}

\newcommand{\wt}{\widetilde}
\newcommand{\ol}{\overline}
\newcommand{\wh}{\widehat}
\newcommand{\beq}{ \begin{equation} }
\newcommand{\eeq}{ \end{equation} }
\newcommand{\beqs}{\begin{equation*}}
\newcommand{\eeqs}{\end{equation*}}

\newcommand{\lone}{\mathbbm{1}} 

\newcommand{\dd}{\mathrm{d}}
\newcommand{\ii}{\mathrm{i}}

\renewcommand{\P}{\mathbb{P}}
\newcommand{\SC}{\mathrm{sc}}

\newcommand{\AND}{\quad\text{and}\quad}

\newcommand{\llbra}{\llbracket}
\newcommand{\rrbra}{\rrbracket}

\newcommand\norm[1]{\Vert#1\Vert}

\newcommand\expct[1]{\mathbb{E}[#1]}
\newcommand\Expct[1]{\mathbb{E}\left[#1\right]}

\newcommand\prob[1]{\mathbf{P}\left[#1\right]}

\newcommand\Absv[1]{\left\vert#1\right\vert}
\newcommand\absv[1]{\vert#1\vert}

\numberwithin{equation}{section} 
\numberwithin{thm}{section}

\title{Tracy-Widom limit for free sum of random matrices}

\date{\today}

\author{Hong Chang Ji \thanks{Institute of Science and Technology Austria} \thanks{Email:hongchang.ji@ist.ac.at} \and Jaewhi Park \thanks{Department of Mathematical Sciences, KAIST} \thanks{Email:jw-park@kaist.ac.kr}
}

\begin{document}

\maketitle

\begin{abstract}
	We consider fluctuations of the largest eigenvalues of the random matrix model $A+UBU\adj$ where $A$ and $B$ are $N \times N$ deterministic Hermitian or symmetric matrices and $U$ is a Haar-distributed unitary or orthogonal matrix. We prove that the largest eigenvalue weakly converges to the Tracy-Widom distribution, under mild assumptions on $A$ and $B$ to guarantee that the density of states of the model decays as square root around the upper edge. Our proof is based on the comparison of the Green function along the Dyson Brownian motion starting from the matrix $A + UBU^*$ and ending at time $N^{-1/3+\chi}$. As a byproduct of our proof, we also prove an optimal local law for the Dyson Brownian motion up to the constant time scale.
\end{abstract}

{\footnotesize\textit{AMS Subject Classification (2020)}: 60B20, 46L54

}
\section{Introduction}

In this paper, we consider the sum of two Hermitian random matrices,
\beq\label{eq:themodel}
H=A+UBU\adj\,,
\eeq
where $A$ and $B$ are $N\times N$ deterministic Hermitian matrices and $U$ is Haar distributed on the unitary group $\caU(N)$ of degree $N$. We prove that the law of the largest eigenvalue of the random matrix $H$ in \eqref{eq:themodel} converges to the GUE Tracy-Widom distribution $F_{2}$ with the scale $N^{-2/3}$, under mild assumptions on $A$ and $B$. Our result extends to the case when $U$ is Haar distributed on the orthogonal group $\caO(N)$ of degree $N$, in which case the limit is the GOE Tracy-Widom distrbution $F_{1}$.

The model in \eqref{eq:themodel} is one of the most fundamental examples that show the connection between free probability and Hermitian random matrices. The eigenvectors of $A$ and $UBU\adj$ are in general position, and thus the two matrices are asymptotically free as the matrix size grows to infinity, as proved by Voiculescu in his influential work \cite{Voiculescu1991}. The empirical spectral distribution (ESD) of the ensemble \eqref{eq:themodel} converges to the free additive convolution $\mu_{\alpha}\boxplus\mu_{\beta}$ where $\mu_{\alpha}$ and $\mu_{\beta}$ are limiting eigenvalue distributions of $A$ and $B$, respectively. The convergence of ESD also holds on local scales \cite{Bao-Erdos-Schnelli2016,Bao-Erdos-Schnelli2017Adv,Bao-Erdos-Schnelli2017CMP,Bao-Erdos-Schnelli2020JAM,Bao-Erdos-Schnelli2020,Kargin2012PTRF,Kargin2013}.

The two major assumptions on $A$ and $B$ are (i) that the averaged eigenvalue distributions $\mu_{A}$ and $\mu_{B}$ respectively of $A$ and $B$ converge sufficiently fast to their limits $\mu_{\alpha}$ and $\mu_{\beta}$ and (ii) that the densities of $\mu_{\alpha}$ and $\mu_{\beta}$ have power-law decay around the upper edge with exponents between $-1$ and $1$. The second condition ensures that the free convolution $\mu_{\alpha}\boxplus\mu_{\beta}$ has a  regular upper edge, that is, its density has square-root decay at the upper edge. The first condition guarantees that the $N$-dependent convolution $\mu_{A}\boxplus\mu_{B}$ inherits the same property.

The most important aspect of our result is that the law of the largest eigenvalue of $H$ not only is independent of $A$ and $B$ but also coincides with that of a GUE. For Gaussian unitary and orthogonal ensembles (GUE and GOE, respectively) Tracy and Widom identified the distribution in \cite{Tracy-Widom1994,Tracy-Widom1996}; let $\lambda_{1}$ be the largest eigenvalue of an  $(N\times N)$ GUE or GOE. Then
\beq\label{eq:Tracy-Widom}
\lim_{N\to\infty}\prob{N^{-2/3}(\lambda_{1}-2)<x}=F_{\beta}(x),
\eeq
where $F_{\beta}(x)$ is a distribution function determined by the parameter $\beta=1$ or $2$ corresponding respectively to the orthogonal or unitary ensembles. For many random matrix models, the limiting distribution of the largest eigenvalue matches with that of GOE or GUE after proper normalization when the matrix is real symmetric or complex Hermitian, respectively, which is referred to as \emph{edge universality.} Edge universality for Wigner matrices was proved in \cite{Erdos-Knowles-Yau-Yin2012,Lee-Yin2014,Soshnikov1999,Tao-Vu2010}, and it was extended to deformed Wigner matrices in \cite{Landon-Yau2017,Lee-Schnelli2015,Lee-Schnelli-Stetler-Yau2016}.

Our proof roughly follows the strategy of \cite{Lee-Schnelli2015} where edge universality for deformed Wigner matrices, $H^{\mathrm{dWig}}\deq A+W^{\mathrm{Wig}}$ for a Wigner matrix $W^{\mathrm{Wig}}$, was proved. To be specific, we consider a continuous flow $H_{t}$ of random matrices starting from $H_{0}=H$ and for this flow the proof can be divided into three components:
\begin{itemize}
	\item[(i)] Optimal, uniform (over $t$), entrywise local law for $H_{t}$ around the edge,
	\item[(ii)] edge universality at the endpoint of the flow $H_{t_{0}}$ for a suitable $t_{0}>0$, and
	\item[(iii)] comparison of the distribution of $\lambda_{1}(H_{t})$ as $t$ increases from $0$ to $t_{0}$.
\end{itemize}
Our choice of $H_{t}$ is the Dyson Brownian motion (DBM), or Dyson matrix flow,
\beq\label{eq:DBM}
H_{t}\deq H+\sqrt{t}W, \quad t\geq 0,
\eeq
where $W$ is a GUE/GOE (depending on the symmetry class of $H$) independent of $H$. The main reason for choosing \eqref{eq:DBM} is to directly apply the result of Landon and Yau \cite{Landon-Yau2017}, where it was proved that the DBM whose initial matrix has a regular edge reaches its equilibrium at the edge after time $t_{0}=N^{-1/3+\chi}$. In other words, the edge statistics of $H_{t_{0}}$ matches that of GUE/GOE, which directly establishes (ii). Here the results of \cite{Bao-Erdos-Schnelli2020} guarantee that the initial matrix $H$ has a regular upper edge. The shortness of time scale $t_{0}$ greatly simplifies the last step, compared to \cite{Lee-Schnelli2015} where logarithmic time scale was used with Ornstein--Uhlenbeck version of \eqref{eq:DBM}.

The bulk of the proof is devoted to (i) and (iii), local law for $H_{t}$ and comparison over the flow. To prove the local law, we adapt the argument in \cite{Bao-Erdos-Schnelli2020} with some modification to the subordinate system. To summarize, we prove that the resolvent of $H_{t}$ is subordinate to that of $A$, that is, $G_{H_{t}}(z)\approx G_{A}(\omega_{A,t}(z))$ for a complex analytic self-map $\omega_{A,t}$ of the upper half-plane where $G_{H_{t}}$ and $G_{A}$ are the resolvents of $H_{t}$ and $A$, respectively. Here we emphasize that the function $\omega_{A,t}$ captures the dependency on time $t$, and $G_{A}$ is independent of $t$.

For (iii), we approximate the (unnormalized) eigenvalue density of $H_{t}$ with the Green function $\E\im\Tr G_{H_{t}}(z)$, for spectral parameters $z$ with $\im z$ on a smaller scale than the typical size $N^{-2/3}$ of fluctuations of $\lambda_{1}$. Then we prove a \emph{Green function comparison theorem} stating that $\E\im\Tr G_{H_{t}}(z)$ has small enough time derivative, which establishes (iii). 

The proof of Green function comparison mainly concerns `decoupling' the index $a\in\{1,\cdots,N\}$ from each diagonal entry $\E (G_{H_{t}}^{2})_{aa}$, that is, to express them as a function of entries of $A$ and tracial quantities; see Proposition \ref{prop:decoup}. These expansions reveal additional cancellation within the time derivative of $\E\im \Tr G_{H_{t}}$, which is hard to observe otherwise. Even in the simplest case of deformed GUE, $A+W$, the cancellation is highly nontrivial; see Section \ref{sec:dGUE} for an exposition. Such an argument was first used in \cite{Lee-Schnelli2015} for deformed Wigner matrices, where non Gaussian $W^{\mathrm{Wig}}$ posed central difficulty.

In traditional applications of the three step strategy, the structure of $H_{0}$ remained intact along the flow $H_{t}$; for example when $H_{0}$ is a Wigner matrix, $H_{t}$ remains a Wigner matrix for all $t$ (see \cite{Erdos-Knowles-Yau-Yin2012,Lee-Schnelli2015,Lee-Schnelli2016} for examples). In the same vein, we may rewrite \eqref{eq:DBM} so that it has the same form as $H_{0}$, namely,
\beq
H_{t}=A+(UBU\adj+\sqrt{t}W)=A+U_{t}B_{t}U_{t}\adj.
\eeq
Here $B_{t}$ is a diagonal matrix consisting of the eigenvalues of $UBU\adj+\sqrt{t}W$, and $U_{t}$ is a Haar-distributed unitary/orthogonal matrix by the rotational invariance of $W$. Then we might attempt to analyze two flows of matrices $B_{t}$ and $U_{t}$, yet we avoid this approach for two reasons. First, for general $B$ the behavior of the spectrum of $B_{t}$ is much harder to analyze compared to the whole matrix $H_{t}$. More specifically, while the results of \cite{Bao-Erdos-Schnelli2020JAM} guarantees that $B_{t}$ has a regular edge for $t\sim1$, the neighborhood on which the square root decay holds true diminishes when $t\ll1$. In contrast, since $H_{0}$ has regular edge, the same holds true for $H_{t}$ no matter how small $t$ is. Second, even though the unitary matrix $U_{t}$ is Haar distributed for each fixed $t$, studying it as a stochastic process over $t$ is a difficult task. Indeed these two problems were handled in \cite{Che-Landon2019}, by assuming stronger conditions on $B$ and introducing a diffusion process $\wt{U}_{t}$ on the unitary/orthogonal group so that $A+\wt{U}_{t}B\wt{U}_{t}\adj$ has the same local statistics in the bulk as $H_{t}$ at $t\gg N^{-1}$. For our purposes it suffices to consider the Dyson matrix flow \eqref{eq:DBM} as the sum of three matrices, rather than a perturbation of the sum of two matrices. We explain how we handle the sum of three matrices in the next paragraph.

The major novelty of the proof of local laws for $H_{t}$ is that we introduce a time dependent, deterministic system of equations (see \eqref{eq:Phi_def}) that characterizes the subordination function $\omega_{A,t}$ above. We take this system as the deterministic equivalent of $G_{H_{t}}$, which allows us to consider the Brownian motion $\sqrt{t}W$ as one of the leading term, but not a perturbation. In fact, the new system \eqref{eq:Phi_def} is consistent with the one used in \cite{Bao-Erdos-Schnelli2020} in the sense that simply introducing a variant of $F$-transform therein, the negative reciprocal of Stieltjes transform, can fully reflect the time dependence. This consistency allows  us to prove the local law using the exact same strategy as in \cite{Bao-Erdos-Schnelli2020} with some extra bounds for terms originating from the GUE. Due to the same reason, our local laws hold for all finite time scales, which could be of separate interest. Finally, we remark that while the results in \cite{Landon-Yau2017} can prove an averaged local law for $H_{t}$, a direct adaptation of their method cannot prove entrywise local laws, even for $t\sim 1$. To be more specific, an entry of $G_{H_{t}}$ necessarily involves entries of $A$ (recall $G_{H_{t}}\approx G_{A}(\omega_{A,t}))$, whereas the averaged local law in \cite{Landon-Yau2017} is (by design) written solely in terms of the initial ESD. Taking $H_{0}$ to be the free sum $A+UBU\adj$, it already carries enough randomness via $U$ so that individual entries of $A$ are not visible through the ESD of $H_{0}$.

The proof of Green function comparison for $H_{t}$ requires new ideas compared to \cite{Lee-Schnelli2015} since the randomness in our matrix model originates from Haar unitary matrices. Firstly, entries of $H$ are correlated unlike those of $H^{\mathrm{dWig}}$. To aid this, we use partial randomness decomposition as in \cite{Bao-Erdos-Schnelli2020} to express $U$ in terms of the independent pair of an $N$-dimensional Gaussian vector and a Haar unitary matrix of degree $N-1$. 

Secondly, when expanding diagonal entries $\E (G_{H_{t}}^{2})_{aa}$, matrices of the form $\E G_{H_{t}}UBU\adj G_{H_{t}}$ emerge. Applying the same expansion to these quantities results in more factors of $U$ thus cannot lead to an accessible form. Such a problem did not appear in \cite{Lee-Schnelli2015}; see Section \ref{sec:frX_prelim} for details. To solve this we use the symmetry of our model, namely to consider $U\adj G_{H_{t}}U$ as the resolvent of
\beq\label{eq:themodel_sym}
\caH_{t}\deq U\adj AU+B+\sqrt{t}U\adj WU.
\eeq
Due to the invariance of GUE, we see that the matrix \eqref{eq:themodel_sym} in fact has the same form as \eqref{eq:DBM}. Thus we establish a system of linear equations involving weighted traces of $G_{H_{t}}^{2}$ and $G_{\caH_{t}}^{2}$, from which the Green function comparison follows. To the best of our knowledge, such calculations involving system of equations did not appear in previous proofs of edge universality.

\subsection{Related works}
The convergence of the ESD of the model in \eqref{eq:themodel} was first considered by Voiculescu \cite{Voiculescu1991} and extended to a local scale by Kargin \cite{Kargin2012PTRF,Kargin2013}. The properties of the free additive convolution in two deterministic measures such as stability, behavior and its qualitative description was studied by Bao, Erd\H{o}s, and Schnelli \cite{Bao-Erdos-Schnelli2016,Bao-Erdos-Schnelli2020JAM,Bao-Erdos-Schnelli2020}. They also established optimal local laws for $H$ and the convergence of the Green function of eigenvalue distribution when the parameter is close to the spectrum \cite{Bao-Erdos-Schnelli2017Adv,Bao-Erdos-Schnelli2017CMP,Bao-Erdos-Schnelli2020}. As a result, they proved in \cite{Bao-Erdos-Schnelli2020} that the typical eigenvalue spacing is of size $N^{-1}$ and $N^{-2/3}$ around the bulk and edge, respectively. The bulk universality for \eqref{eq:themodel} was established by Che and Landon \cite{Che-Landon2019}, that is, the local eigenvalue statistics of $H$ in the bulk on the scale $N^{-1}$ coincides with that of a GUE or GOE.

For the sum of i.i.d. unitarily invariant matrices whose number of summands exceeds certain threshold, the edge universality was obtained by Ahn in \cite[Theorem 1.1]{Ahn2020} using multivariate Bessel generating functions. As a byproduct, it was also proved in \cite[Theorem 1.3]{Ahn2020} that the Tracy--Widom limit holds for \eqref{eq:themodel} when $\mu_{\alpha}$ and $\mu_{\beta}$ are exactly beta distributions $\mathrm{Beta}(a,b)$ with $a\in [-1/2,\infty)$ and $b\in(-1,1/2]$. It should be noted that \cite{Ahn2020} did not cover the orthogonal case, and the proofs therein rely on Harish-Chandra-Itzykson-Zuber integral (see \cite[Lemma 2.5]{Ahn2020}). Since the corresponding integral for Haar orthogonal matrices does not lead to a determinantal form, it is not clear whether the same method as in \cite{Ahn2020} applies to the orthogonal case. On the other hand \cite[Theorem 1.1]{Ahn2020} allows for more general summands, for example whose spectral density has power-law decay with exponent larger than $1$: In this case the number of summands necessarily exceeds two, for otherwise the edge may not have square-root decay.

The linear eigenvalue statistics of \eqref{eq:themodel} were studied in \cite{Bao-Schnelli-Xu2020,Pastur-Vasilchuk2007}. The Gaussian fluctuation of linear eigenvalue statistics on the global scale was obtained by Pastur and Vasilchuk, \cite{Pastur-Vasilchuk2007}. On the mesoscopic scale, which concerns relatively few eigenvalues around a fixed energy level, Bao, Schnelli, and Xu proved in \cite{Bao-Schnelli-Xu2020} a central limit theorem for linear eigenvalue statistics when the energy is in the regular bulk. Also \cite{Bao-Schnelli-Xu2020} covered mesoscopic averages at regular edges up to the scale $N^{-2/5+o(1)}$, while the optimal scale would be $N^{-2/3+o(1)}$, slightly above the gap scaling. We expect that methods in the present paper can shed light on the extension of their result to the full mesoscopic scale at the edge.

\subsection{Organization}\label{sec:orga}
The rest of this paper is organized as follows. In Section 2, we present the model and assumptions on it, and rigorously state our main result. In Section 3, we collect preliminary results on free probability, in particular analytic subordination, and recall partial randomness decomposition. Section 4 is devoted to the proof of the main result, where we state steps (i) -- (iii) above as propositions. In Sections 5 and 6 we prove the decoupling lemma for an expected diagonal entry $\E(G^{2})_{aa}$ and use it to prove the Green function comparison theorem. The expansion of $\E(G^{2})_{aa}$ in Section 6 generates several sub-leading, not decoupled terms, and we decouple them in Section 7. Sections 8 collects proofs of various probabilistic estimates that are used throughout the paper. Finally, in Section 9 we show how to modify the proof in order to extend the result to orthogonal $U$.

Appendix A is devoted to detailed analysis of the limiting eigenvalue density of $H_{t}$. In Appendices B and C mainly concern the proof of local laws for the Dyson matrix flow. Lastly in Appendix D, we present formulas for the derivatives with respect to the Gaussian vector from partial randomness decomposition.

\begin{nrem}\label{nrem}
	We denote by $\C_{+}$ the complex open upper half-plane. The alphabet $N$ always denotes the size of our matrix $H$ in \eqref{eq:themodel}. All quantities, especially matrices and their entries, depend on $N$ unless otherwise specified. We denote by $c$ and $C$ (small and large, respectively) positive constants that do not depend on $N$, whose value may vary by line. We use the standard big-$O$ notations; for sequences $X\equiv X_{N}\in\C$ and $Y\equiv Y_{N}>0$, we write $X=O(Y)$ or $X\lesssim Y$ if $\absv{X}\leq CY$ for all $N$. We write $\tr$ and $\Tr$ for the normalized and usual traces; for $A\in\C^{d\times d}$,
	\beqs
	\tr A\deq \frac{1}{d}\Tr A.
	\eeqs
	For each $i\in\N$, we denote by $\bse_{i}$ the $i$-th standard coordinate vector whose dimension may vary by lines. For $\bbT \subset \llbra 1,N\rrbra$, we use the shorthand notations
	\begin{align*}
		\sum_{i}\deq \sum_{i=1}^{N},
		&&
		\sum_{i}^{(\bbT)}\deq \sum_{\substack{i=1 \\ i\notin \bbT}}^{N},
	\end{align*}
	and we abbreviate $(i)=(\{i\})$. 
\end{nrem}
\begin{nrem}
	For two sequences $X\deq X^{(N)}$ and $Y\deq Y^{(N)}$ of random variables with $Y\ge0$, we say that $Y$ stochastically dominates $X$ if, for all (small) $\epsilon>0$ and (large) $D>0$,
	\beqs
	\prob{\absv{X}>N^{\epsilon}Y}\leq N^{-D},
	\eeqs
	for sufficiently large $N\geq N_{0}(\epsilon,D)$. In this case we write $X\prec Y$ or $X=O_{\prec}(Y)$. 
\end{nrem}

\section{Definitions and main result}
\begin{defn}\label{def:model}
	Let $\{a_{i}:1\leq i\leq N\}$ and $\{b_{i}:1\leq i\leq N\}$ be nondecreasing sequences of real numbers and define $(N\times N)$ diagonal matrices $A$ and $B$ by
	\begin{align*}
		&A\deq \diag(\fra_{1},\cdots,\fra_{N}),&
		&B\deq \diag(\frb_{1},\cdots,\frb_{N}).
	\end{align*}
	We denote the empirical spectral distributions of $A$ and $B$ by $\mu_{A}$ and $\mu_{B}$, respectively;
	\begin{align*}
		&\mu_{A}\deq\frac{1}{N}\sum_{i=1}^{N}\delta_{\fra_{i}},&
		&\mu_{B}\deq\frac{1}{N}\sum_{i=1}^{N}\delta_{\frb_{i}}.
	\end{align*}
	Let $U$ be an $(N\times N)$ random unitary or orthogonal matrix drawn from the Haar measure respectively on the unitary group $\caU(N)$ or the orthogonal group $\caO(N)$. We define
	\begin{align*}
		&\wt{A}\deq U\adj AU, &
		&\wt{B}\deq UBU\adj, &
		&H_{0}\deq A+ \wt{B},&
		&\caH_{0}\deq B+\wt{A},
	\end{align*}
	so that $\caH_{0}=U\adj H_{0}U$.
\end{defn}

\begin{defn}\label{defn:LSD}
	Let $\mu_{\alpha}$ and $\mu_{\beta}$ be probability measures on $\R$ satisfying the following:
	\begin{itemize}
		\item[(i)] Both of $\mu_{\alpha}$ and $\mu_{\beta}$ are compactly supported and absolutely continuous. We denote their densities respectively by $f_{\alpha}$ and $f_{\beta}$, and write
		\begin{align*}
			&E_{\alpha}^{-}=\inf\supp\mu_{\alpha},&
			&E_{\alpha}^{+}=\sup\supp\mu_{\alpha},&
			&E_{\beta}^{-}=\inf\supp\mu_{\beta}, &
			&E_{\beta}^{+}=\sup\supp\mu_{\beta}.
		\end{align*}
		
		\item[(ii)] There exist $t_{\alpha}^{+},t_{\beta}^{+}\in(-1,1)$ and positive constants $\tau_{\alpha},\tau_{\beta}$, and $C$ such that
		\begin{align}\label{eq:Jac_edge_a}
			&C^{-1}(E_{\alpha}^{+}-x)^{t^{+}_{\alpha}}
			\leq f_{\alpha}(x)
			\leq C(E_{\alpha}^{+}-x)^{t^{+}_{\alpha}},& 
			&\forall x\in[E_{\alpha}^{+}-\tau_{\alpha},E_{\alpha}^{+}]\\\label{eq:Jac_edge_b}
			&C^{-1}(E_{\beta}^{+}-x)^{t^{+}_{\beta}}
			\leq f_{\beta}(x)
			\leq C(E_{\beta}^{+}-x)^{t^{+}_{\beta}}_{+},&
			&\forall x\in[E_{\beta}^{+}-\tau_{\beta},E_{\beta}^{+}].
		\end{align}
	\end{itemize}
\end{defn}

\begin{assump}\label{assump:ABconv}
	We assume the following:
	\begin{itemize}
		\item[(i)]	For each $\epsilon>0$, we have
		\beqs
		\bsd\deq d_{L}(\mu_{A},\mu_{\alpha})+d_{L}(\mu_{B},\mu_{\beta})\leq N^{-1+\epsilon}
		\eeqs
		for sufficiently large $N$, where $d_{L}$ stands for the L\'{e}vy distance.
		
		\item[(ii)] For each $\delta>0$, we have
		\begin{align*}
			&\sup\supp\mu_{A}\leq E_{\alpha}^{+}+\delta, &
			&\sup\supp\mu_{B}\leq E_{\beta}^{+}+\delta,
		\end{align*}
		for sufficiently large $N$.
		\item[(iii)] There exists $C>0$ such that 
		\begin{align*}
			&\inf\supp\mu_{A}\geq -C, &
			&\inf\supp\mu_{B}\geq -C.
		\end{align*}
		\item[(iv)] We have $\Tr A=0=\Tr B$.
	\end{itemize}
\end{assump}

Under Assumption \ref{assump:ABconv}, it is well known that the empirical eigenvalue distribution of $H_{0}$ converges to the free additive convolution $\mu_{\alpha}\boxplus\mu_{\beta}$. For our choices of $\mu_{\alpha}$ and $\mu_{\beta}$ in Definition \ref{defn:LSD}, the measure $\mu_{\alpha}\boxplus\mu_{\beta}$ has a \emph{regular} upper edge in the sense that $\mu_{\alpha}\boxplus\mu_{\beta}$ resembles the semi-circle distribution around the edge:
\begin{lem}[Theorem 2.2 of \cite{Bao-Erdos-Schnelli2020JAM}]\label{lem:fc_dens}
	The free convolution $\mu_{\alpha}\boxplus\mu_{\beta}$ is compactly supported and absolutely continuous with continuous, bounded density $\rho_{0}$. Furthermore, there exist $E_{+,0}\in\R$ and $\gamma_{0}>0$ such that $E_{+,0}=\sup\{x\in\R:\rho_{0}(x)>0\}$ and
	\beqs
	\lim_{x\nearrow E_{+,0}}\frac{\rho_{0}(x)}{\sqrt{E_{+,0}-x}}=\frac{\gamma_{0}^{3/2}}{\pi}.
	\eeqs
\end{lem}
Now we are ready to introduce the main result of this paper, edge universality for $H_{0}$:
\begin{thm}\label{thm:main result}
	Let $A$ and $B$ be deterministic diagonal matrices satisfying Assumption \ref{assump:ABconv}, and let $\lambda_{1}$ be the largest eigenvalue of the matrix $H_{0}$ defined in Definition \ref{def:model}. Then we have
	\beqs
	\lim_{N\to\infty}\P\left(\gamma_{0}^{-1}N^{2/3}(\lambda_{1}-E_{+,0})\leq s\right)=\begin{cases} F_{2}(s)& \text{if $U$ is unitary},\\
		F_{1}(s) & \text{if $U$ is orthogonal}.
	\end{cases}
	\eeqs
	where $F_{1}$ and $F_{2}$ are respectively the GUE and GOE Tracy-Widom distribution functions.
\end{thm}
\begin{rem}\label{rem:joint}
	As in \cite[Theorem 6.4]{Erdos-Knowles-Yau-Yin2012}, Theorem \ref{thm:main result} easily extends to the joint statistics of the top $k$ eigenvalues of $H_{0}$. More precisely, the final result would be
	\begin{multline*}
		\P\left(\gamma_{0}^{-1}N^{2/3}(\lambda_{1}-E_{+,0})\leq s_{1},\cdots,\gamma_{0}^{-1}N^{2/3}(\lambda_{k}-E_{+,0})\leq s_{k} \right)	\\
		-\P\left(N^{2/3}(\mu_{1}-2)\leq s_{1},\cdots,N^{2/3}(\mu_{k}-2)\leq s_{k}\right)\to 0
	\end{multline*}
	as $N\to\infty$, where $\mu_{1}\geq\cdots\geq\mu_{N}$ are the eigenvalues of an $(N\times N)$ GUE.
\end{rem}

\begin{rem}
	For definiteness, we assume that $U$ is a Haar unitary matrix until Section \ref{sec:der_tech}. At the end of the paper, in Section \ref{sec:ort}, we show how to modify the proof for orthogonal $U$.
\end{rem}

We conclude this section with remarks on Theorem \ref{thm:main result} and assumptions we imposed.
\begin{rem}
	Having power-law type decay around the edge as in Definition \ref{defn:LSD} is a typical property among limiting spectral distributions of random matrices. Prime examples are semi-circle (see Definition \ref{defn:m}) and Mar\v{c}henko-Pastur distributions, and the arcsine distribution is an example with exponent $t^{+}=-1/2$ which is the limit of the ESD of $U+U\adj$.
	
	We emphasize that $t^{+}_{\alpha}<1$ is not a technical assumption. When either $t^{+}_{\alpha}$ or $t^{+}_{\beta}>1$, the density of the free convolution $\mu_{\alpha}\boxplus\mu_{\beta}$ may not have square root decay around the edge as in Lemma \ref{lem:fc_dens}. Indeed, it was proved in \cite{Lee-Schnelli2013} that when $\mu_{\alpha}$ is the semi-circle law and $t^{+}_{\beta}>1$, the density of $\mu_{\alpha}\boxplus\mu_{\beta}$ can decay as $t^{+}_{\beta}$ depending on the ratio of variances of $\mu_{\alpha}$ and $\mu_{\beta}$. In this case, the typical eigenvalue spacing of $H$ around the edge would be $N^{-1/(t^{+}_{\beta}+1)}$, which is incompatible with our result where the largest eigenvalue is scaled by $N^{2/3}$.
\end{rem}
\begin{rem}\label{rem:assu}
	Here we explain the role of Assumption \ref{assump:ABconv}. Firstly, assumption (i) guarantees that the $N$-dependent measure $\mu_{A}\boxplus\mu_{B}$ inherits the stability of $\mu_{\alpha}\boxplus\mu_{\beta}$, on the optical scale of $O(N^{-1+o(1)})$. We might be able to combine Definition \ref{defn:LSD} (ii) with Assumption \ref{assump:ABconv} (i) to make a statement on $\mu_{A}$ and $\mu_{B}$ that would ensure the same stability of $\mu_{A}\boxplus\mu_{B}$, uniformly over $N$. While we believe that this is possible following \cite{Ahn2020}, that is, imposing conditions on the inverses of their Stieltjes transforms, we choose the current assumptions to make direct connection with \cite{Bao-Erdos-Schnelli2020} and to avoid using inverse Stieltjes transforms.
	
	The second assumption (ii) should be understood in connection with so-called BBP transition \cite{Baik-BenArous-Peche2005}, meaning that spiked eigenvalues of $A$ and $B$ may result in those of $H_{0}$. The corresponding phenomenon was proved in \cite{Belinschi-Bercovici-Capitaine-Fevrier2017}, so that when $\max_{i\in\braN}a_{i}$ exceeds certain threshold ($\omega_{\alpha,0}(E_{+,0})$, to be specific), the largest eigenvalue $\lambda_{1}(H)$ of $H$ converges to a point strictly larger than $E_{+,0}$. In this case, we expect that the fluctuation of $\lambda_{1}(H)$ would be a Gaussian with magnitude $O(N^{-1/2})$ as in \cite{Baik-BenArous-Peche2005}, which is in a different regime from our result. We will pursue this line of study in a future work. While we can generalize our result by allowing few eigenvalues of $A$ to detach from $\supp\mu_{\alpha}$ but within the threshold above, we refrain ourselves for simplicity. The third assumption ensures that $A$ and $B$ are both norm-bounded.
	
	Finally, the last assumption (iv) is a mere shift, which can be dropped by simply translating $E_{+}$ by $\tr A+\tr B$. Note also that assumptions (i) and (iv) together imply that $\mu_{\alpha}$ and $\mu_{\beta}$ are of mean zero, that is,
	\beq\label{eq:mean_zero}
	\int_{\R}x\dd\mu_{\alpha}(x)=0=\int_{\R}x\dd\mu_{\beta}(x).
	\eeq
\end{rem}
\begin{rem}
	The most important and interesting examples would arise when $A$ and $B$ are random. In particular, if $A$ is a random matrix for which the optimal local laws (see Theorem B.1 for example) are known, $d_{L}(\mu_{A},\mu_{\alpha})\ll N^{-1+\epsilon}$ holds with high probability (see \cite[Section 5.1]{Ajanki-Erdos-Kruger2017} for a proof). Thus our theorem applies to random $A$ and $B$ as long as the optimal local laws are proved and they are unitarily invariant. There are many instances satisfying these criteria, including invariant ensembles \cite{Bourgade-Mody-Pain2022} that also cover GUE and Wishart ensemble. Other interesting examples concern self-adjoint polynomials of a Haar unitary matrix, such as $\re U^{m}$ \cite{Meckes-Meckes2013}.
	
	Also, due to \cite{Bao-Erdos-Schnelli2020}, the sum $H$ itself can serve as a summand if the conditions on $A$ and $B$ hold in the whole spectrum (see Corollary 2.8 of \cite{Bao-Erdos-Schnelli2020}). More specifically, our result applies to the sum of any finite number of summands, that is, sums of $U_{j}A_{j}U_{j}$'s where $U_{j}$'s are independent Haar unitary matrices and $A_{j}$'s satisfy the assumptions of \cite[Corollary 2.8]{Bao-Erdos-Schnelli2020}.
\end{rem}
\section{Preliminaries}\label{sec:prelim}

As mentioned in the introduction, our proof mainly involves the Dyson Brownian motion starting from $H_{0}$, whose ESD is approximated by the free additive convolutions of three measures, $\mu_{\alpha}\boxplus\mu_{\beta}\boxplus\mu_{\SC}^{(t)}$ or $\mu_{A}\boxplus\mu_{B}\boxplus\mu_{\SC}^{(t)}$. This section collects some complex analytic preliminary results on these free additive convolutions, including analytic subordination phenomenon and their properties.
\begin{defn}\label{defn:m}
	For $t\geq0$ and a probability measure $\mu$ on $\R$, we define functions $m_{\mu},F_{\mu,t}$ on $\C_{+}$ by
	\begin{align}\label{eq:mF_def}
		&m_{\mu}(z)=\int_{\R}\frac{1}{x-z}\dd\mu(x), &
		&F_{\mu,t}(z)=-\frac{1}{m_{\mu}(z)}+tm_{\mu}(z),&
		&z\in\C_{+}.
	\end{align}
	Also we denote the semicircle distribution on $[-2\sqrt{t},2\sqrt{t}]$ by $\mu_{\SC}^{(t)}$, that is,
	\beqs
	\dd\mu_{\SC}^{(t)}(x)= \frac{1}{2\pi t}\sqrt{(4t-x^{2})_{+}}\dd x.
	\eeqs
\end{defn}

One of the most efficient tools for studying free convolutions, or specifically its regularity, is the analytic subordination. We present the corresponding result for free additive convolution of two measures and a dilated semi-circle distribution:
\begin{prop}\label{prop:subor}
	Given $t>0$ and two Borel probability measures $\mu_{1}$ and $\mu_{2}$ on $\R$, there exist unique analytic functions $\omega_{1,t},\omega_{2,t}:\C_{+}\to\C_{+}$ that satisfy the following:
	\begin{itemize}
		\item[(i)] We have  $\im\omega_{1,t}(z),\im\omega_{2,t}(z)\geq\im z$ for all $z\in\C_{+}$ and 
		\beqs
		\lim_{\eta\nearrow\infty}\frac{\omega_{1,t}(\ii\eta)}{\ii\eta}=\lim_{\eta\nearrow\infty}\frac{\omega_{2,t}(\ii\eta)}{\ii\eta}=1:
		\eeqs
		
		\item[(ii)] For all $z\in\C_{+}$,
		\beq\label{eq:subor}
		F_{\mu_{1}\boxplus\mu_{2}\boxplus\mu_{\SC}^{(t)},t}(z)=F_{\mu_{1},t}(\omega_{1,t}(z))=F_{\mu_{2},t}(\omega_{2,t}(z))=\omega_{1,t}(z)+\omega_{2,t}(z)-z.
		\eeq
	\end{itemize}
\end{prop}
\begin{proof}
	The result can be proved with a straightforward modification of Theorem 4.1 of \cite{Belinschi-Bercovici2007}, which covers the case of $t=0$. Or we can apply the same theorem twice, firstly to the free convolution $\mu_{1}\boxplus\mu_{2}$ and then to $(\mu_{1}\boxplus\mu_{2})\boxplus\mu_{\SC}^{(t)}$. The second proof also reveals the relation
	\begin{align}\label{eq:omega(z+tm)}
		&\omega_{1,t}(z)=\omega_{1,0}(z+tm_{\rho}(z)), &
		&\omega_{2,t}(z)=\omega_{2,0}(z+tm_{\rho}(z)),
	\end{align}
	where we abbreviated $\rho=\mu_{1}\boxplus\mu_{2}\boxplus\mu_{\SC}^{(t)}$.
\end{proof}
Note that a direct consequence of \eqref{eq:subor} is
\beq\label{eq:subor_m}
m_{\mu_{1}\boxplus\mu_{2}\boxplus\mu_{\SC}^{(t)}}(z)=m_{\mu_{1}}(\omega_{1,t}(z))=m_{\mu_{2}}(\omega_{2,t}(z)),\qquad z\in\C_{+}.
\eeq

For simplicity, we use the following abbreviations;
\begin{align*}
	&m_{\mu_{\alpha}}\equiv m_{\alpha},&
	&m_{\mu_{\beta}}\equiv m_{\beta},&
	&m_{\mu_{A}}\equiv m_{A}, &
	&m_{\mu_{B}}\equiv m_{B}, \\
	&F_{\mu_{\alpha},t}\equiv F_{\alpha,t},&
	&F_{\mu_{\beta},t}\equiv F_{\beta,t},&
	&F_{\mu_{A},t}\equiv F_{A,t}, &
	&F_{\mu_{B},t}\equiv F_{B,t},\\
	&\mu_{\alpha}\boxplus\mu_{\beta}\boxplus\mu_{\SC}^{(t)}\equiv \mu_{t},&&&
	&\mu_{A}\boxplus\mu_{B}\boxplus\mu_{\SC}^{(t)}\equiv \wh{\mu}_{t}. &&
\end{align*}
We further denote the subordination functions corresponding to the pairs $(\mu_{\alpha},\mu_{\beta})$ and $(\mu_{A},\mu_{B})$ respectively by $(\omega_{\alpha,t},\omega_{\beta,t})$ and $(\omega_{A,t},\omega_{B,t})$, so that
\beqs
F_{\alpha,t}(\omega_{\alpha,t}(z))=F_{\beta,t}(\omega_{\beta,t}(z))=F_{\mu_{t}}(z),	\qquad
F_{A,t}(\omega_{A,t}(z))=F_{B,t}(\omega_{B,t}(z))=F_{\wh{\mu}_{t}}(z).
\eeqs
For later uses, we define a $t$-dependent function $\Phi_{\alpha\beta}\equiv\Phi_{\alpha\beta,t}=(\Phi_{\alpha,t},\Phi_{\beta,t}):\C_{+}^{3}\to\C^{2}$ as
\begin{align}\label{eq:Phi_def}
	&\Phi_{\alpha}(\omega_{1},\omega_{2},z)\equiv\Phi_{\alpha,t}(\omega_{1},\omega_{2},z)\deq F_{\alpha,t}(\omega_{1})-\omega_{1}-\omega_{2}+z,\\
	&\Phi_{\beta}(\omega_{1},\omega_{2},z)\equiv\Phi_{\beta,t}(\omega_{1},\omega_{2},z)\deq F_{\beta,t}(\omega_{2})-\omega_{1}-\omega_{2}+z,
\end{align}
and $\Phi_{AB}$ is defined similarly. Note that \eqref{eq:subor} is equivalent to $\Phi_{\alpha\beta}(\omega_{\alpha,t}(z),\omega_{\beta,t}(z),z)\equiv 0$.
Finally, we omit the subscript $t$, for example $\omega_{A,t}\equiv\omega_{A}$, when there is no confusion.

The boundary behaviors of the subordination functions are studied by Belinschi in the series of papers \cite{Belinschi2006,Belinschi2008,Belinschi2014}. In particular it is proved that if $\mu_{1}$ and $\mu_{2}$ are compactly supported measures such that $\mu_{1}(a)+\mu_{2}(b)<1$ for all $a,b\in\R$, then the corresponding subordination functions $\omega_{1,0}$ and $\omega_{2,0}$ extend continuously to $\C_{+}\cup\R$, possibly attaining value $\infty$. We can easily see that all measures considered in the present paper satisfy the assumption. In particular, $\omega_{\alpha,t},\omega_{\beta,t},\omega_{A,t}$ and $\omega_{B,t}$ extend continuously to $\C_{+}\cup\R$ for each fixed $t\geq0$.

We denote the upper edges of $\mu_{t}$ and $\wh{\mu}_{t}$ respectively by $E_{+,t}$ and $\wh{E}_{+,t}$, and consider the spectral domain
\begin{align}\label{eq:D_def}
	\caD_{\tau}(\eta_{1},\eta_{2})\equiv\caD_{\tau,t}(\eta_{1},\eta_{2})\deq\{z=E+\ii\eta\in\C_{+}:E\in[E_{+,t}-\tau,\tau^{-1}],\,\eta\in(\eta_{1},\eta_{2}]\},
\end{align}
for $\tau,\eta_{1},\eta_{2},t\geq0$. Furthermore we denote $\kappa\equiv\kappa_{t}(z)=\absv{z-E_{+,t}}$ for $z\in\C_{+}$.

In the following lemma, we present our results on the behavior of $\omega_{\alpha,t}(z)$ and $\omega_{\beta,t}(z)$. Its proof is deferred to Appendix \ref{sec:stab}.
\begin{lem}\label{lem:stab}
	Let $\omega_{\alpha,t},\omega_{\beta,t}$ be the subordination functions corresponding to the pair $(\mu_{\alpha},\mu_{\beta})$. Then the following hold true:
	\begin{itemize}
		\item[(i)] The functions $(t,z)\mapsto \omega_{\alpha,t}(z), \omega_{\beta,t}(z),$ and $m_{\mu_{t}}(z)$ are continuous on $[0,\infty)\times (\C_{+}\cup\R)$, with finite values.
		
		\item[(ii)] There exists a (small) constant $\tau>0$ such that for all fixed $\eta_{M}>0$ we have
		\beqs
		\sup_{t\in[0,1]}\sup_{z\in\caD_{\tau}(0,\eta_{M})}\absv{\omega_{\alpha,t}(z)}+\absv{\omega_{\beta,t}(z)}\leq C.
		\eeqs
		
		\item[(iii)] There exist constants $\tau>0$ and $k_{0}$ such that
		\beq\label{eq:stab}\begin{aligned}
			&\inf\{\absv{\omega_{\alpha,t}(z)-x}:t\in[0,1],\,z\in\caD_{\tau}(0,\infty),\,x\in\supp\mu_{\alpha}\}\geq k_{0},\\
			&\inf\{\absv{\omega_{\beta,t}(z)-x}:t\in[0,1],\,z\in\caD_{\tau}(0,\infty),\,x\in\supp\mu_{\beta}\}\geq k_{0}.
		\end{aligned}\eeq
		Furthermore, $\omega_{\alpha,t}(E_{+,t})>E_{\alpha}^{+}$ and $\omega_{\beta,t}(E_{+,t})>E_{\beta}^{+}$.
		
		\item[(iv)] For each $t\geq 0$, the edge $E_{+,t}$ satisfies the following equation.
		\beq\label{eq:S_at_E}
		(F_{\alpha}'(\omega_{\alpha}(E_{+,t}))-1)(F_{\beta}'(\omega_{\beta}(E_{+,t}))-1)-1=0.
		\eeq
		
		\item[(v)] For each $t\geq0$, the measure $\mu_{t}$ has a continuous density $\rho_{t}$ around $E_{+,t}$ and there exists a constant $\gamma_{t}>0$ such that the following holds;
		\beq\label{eq:gamma_def}
		\lim_{E\nearrow E_{+,t}}\frac{\rho_{t}(x)}{\sqrt{E_{+,t}-x}}=\frac{\gamma_{t}^{3/2}}{\pi}.
		\eeq
		Furthermore, $\gamma_{t}\sim 1$ and $\frac{\dd}{\dd t}\gamma_{t}\lesssim1$ for $t\in[0,1]$.
	\end{itemize}
\end{lem}
The identity \eqref{eq:S_at_E} in fact characterizes the edge; see Lemma \ref{lem:edgechar} for details. Viewing $\mu_{t}$ as the free convolution of $(\mu_{\alpha}\boxplus\mu_{\beta})$ and $\mu_{\SC}^{(t)}$, the edge $E_{+,t}$  admits another characterization as follows; see \cite[Lemma 2.3]{Landon-Yau2017} for a proof.
\beq\label{eq:Egamma}
\begin{aligned}
	\frac{1}{t}&=\int_{\R}\frac{1}{\absv{x-E_{+,t}-tm_{\mu_{t}}(E_{+,t})}^{2}}\dd(\mu_{\alpha}\boxplus\mu_{\beta})(x),\\
	\gamma_{t}&=\left(-t^{3}\int_{\R}\frac{1}{(x-E_{+,t}-tm_{\mu_{t}}(E_{+,t}))^{3}}\dd(\mu_{\alpha}\boxplus\mu_{\beta})(x)\right)^{-1/3}.
\end{aligned}
\eeq
Notice the integrating measure $\mu_{\alpha}\boxplus\mu_{\beta}$ in \eqref{eq:Egamma}. Both identities in \eqref{eq:Egamma} originate from the well-known \emph{Pastur equation} \cite{Pastur1972}, which is a special case of \eqref{eq:subor_m} with $\mu_{2}=\delta_{0}$;
\beq\label{eq:Pastur}
m_{\mu_{1}\boxplus\mu_{\SC}^{(t)}}(z)=\int_{\R}\frac{1}{x-z-tm_{\mu_{1}\boxplus\mu_{\SC}^{(t)}}(z)}\dd \mu_{1}(x),\qquad z\in\C_{+}.
\eeq
Finally, note that the dilation $\gamma_{t}^{-1}\rho_{t}(\gamma_{t}^{-1}\cdot)$ has the exactly the same decay $\pi^{-1}\sqrt{x}$ as the semi-circle distribution at the upper edge. In what follows we denote this rescaled edge by
\beq\label{eq:L_def}
L_{+,t}\deq \gamma_{t}E_{+,t}.
\eeq

We next present the partial randomness decomposition of a Haar unitary matrix and related notations, first introduced in \cite{Diaconis-Shahshahani1987} (see also \cite[Section 8]{Mezzadri2007}). They will be extensively used in the remaining sections.
\begin{lem}[{\cite[Lemma 4.1]{Diaconis-Shahshahani1987}\footnote{\cite[Lemma 4.1]{Diaconis-Shahshahani1987} applies to any Polish topological group. See page 27--28 therein for an application to $\caU(N)$.}}]\label{lem:prd}
	Let $U$ be the $(N\times N)$ Haar unitary random matrix in Definition \ref{def:model}. For each $i\in\llbra1,N\rrbra$, let $\bsv_{i}\deq U\bse_{i}$ be the $i$-th column vector of $U$, $\theta_{i}\in[0,2\pi)$ be the argument of $\bse_{i}\adj\bsv_{i}$, and
	\beq\label{eq:prd}
	U^{\angi}\deq -\e{-\ii\theta_{i}}R_{i}U, \quad\text{where}\quad \bsr_{i}\deq \sqrt{2}\frac{\bse_{i}+\e{-\ii\theta_{i}}\bsv_{i}}{\norm{\bse_{i}+\e{-\ii\theta_{i}}\bsv_{i}}_{2}},\quad R_{i}\deq I-\bsr_{i}\bsr_{i}\adj.
	\eeq
	Then we have
	\beq\label{eq:Uangi}
	U^{\angi}=\begin{pmatrix}
		1 & O\\ O & U^{i}
	\end{pmatrix},\qquad U^{i}\in\C^{(N-1)\times (N-1)},
	\eeq
	and $\bsv_{i}$ and $U^{i}$ are independent and uniformly distributed on $\bbS_{\C}^{N}\deq \{\bsv\in\C^{N}:\norm{\bsv}=1\}$ and $\caU(N-1)=\{V\in\C^{(N-1)\times (N-1)}:VV\adj=I\}$, respectively.
\end{lem}
Note that \eqref{eq:Uangi} holds for any unitary $U$ by a direct computation, hence only the properties of distributions of $\bsv_{i}$ and $U^{i}$ require Haar-distributed $U$. Note also that the matrix $R_{i}$ is a Householder reflection:
\beqs
R_{i}=R_{i}\adj\AND R_{i}^{2}=(I-\bsr_{i}\bsr_{i}\adj)(I-\bsr_{i}\bsr_{i}\adj)=I-2\bsr_{i}\bsr_{i}\adj+\norm{\bsr_{i}}_{2}^{2}\bsr_{i}\bsr_{i}\adj=I.
\eeqs
Using the matrix $U^{\angi}$ above, we further define 
\beqs
\wt{B}^{\angi}\deq U^{\angi}B(U^{\angi})\adj.
\eeqs
Since $\bsv_{i}$ is uniformly distributed on the unit sphere $\bbS^{N-1}_{\C}$, we can define a Gaussian vector $\wt{\bsg}_{i}\sim\caN_{\C}(0,N^{-1}I_{N})$ such that
\beqs
\bsv_{i}= \frac{\wt{\bsg}_{i}}{\norm{\wt{\bsg}_{i}}_{2}}.
\eeqs
Using the Gaussian vector $\wt{\bsg}_{i}$, we define
\beq\label{eq:def_g}\begin{aligned}
	\bsg_{i}\deq& \e{-\ii\theta_{i}}\wt{\bsg}_{i},\quad	&
	\bsh_{i}\deq& \frac{\bsg_{i}}{\norm{\bsg_{i}}}=\e{-\ii\theta_{i}}\bsv_{i},	\quad &	
	\ell_{i}\deq& \frac{\sqrt{2}}{\norm{\bse_{i}+\bsh_{i}}},	&	\\
	\mr{\bsg}_{i}\deq&\bsg_{i}-g_{ii}\bse_{i},	\quad &
	\mr{\bsh}_{i}\deq& \bsh_{i}-h_{ii}\bse_{i}.
\end{aligned}\eeq
Now for the vector $\bsh_{i}$ we have 
\beq\label{eq:Re=-h}
R_{i}\bse_{i}=-\bsh_{i} \AND R_{i}\bsh_{i}=-\bse_{i},
\eeq
so that
\beq\label{eq:Rh=-e}
\begin{aligned}
	\bsh_{i}\adj \wt{B}^{\angi}R_{i}=\bsh_{i}\adj U^{\angi}B(U^{\angi})\adj R_{i}	=\bsh_{i}\adj R_{i}UBU\adj =-\bse_{i}\adj \wt{B},	\\
	\bse_{i}\adj\wt{B}^{\angi}R_{i}=(-\e{\ii\theta_{i}})\bse_{i}\adj B U\adj =-b_{i}\bsh_{i}\adj=\bse_{i}\adj R_{i}\wt{B}=	-\bsh_{i}\adj \wt{B}.
\end{aligned}
\eeq

\begin{rem}
	The same decomposition applies to the orthogonal case. Namely, for an $(N\times N)$ Haar orthogonal matrix $U$, Lemma \ref{lem:prd} remains valid except that $\bsv_{i}$ and $U^{i}$ are uniformly distributed respectively on $\{\bsv\in\R^{N}:\norm{\bsv}=1\}$ and the orthogonal group of order $(N-1)$. Note also that in this case we have $\e{\ii\theta_{i}}=\mathrm{sign}(\bse_{i}\adj\bsv_{i})$.
\end{rem}

We conclude this section with two elementary identities that are used for computations.
\begin{itemize}
	\item (Stein's lemma)\quad For a $C^{1}$ function $F:\C\to\C$ and a standard complex Gaussian random variable $X$, that is, $\re X$ and $\im X$ are i.i.d. with law $\caN(0,1/2)$, we have
	\begin{align*}
		&\expct{\ol{X}F(X)}=\expct{\frac{\partial}{\partial X}F(X)},&
		&\expct{XF(X)}=\expct{\frac{\partial}{\partial \ol{X}}F(X)}
	\end{align*}
	whenever $\expct{\absv{XF(X)}},\expct{\absv{F_{x}(X)}},\expct{\absv{F_{y}(X)}}<\infty$, where $\frac{\partial}{\partial X}$ and $\frac{\partial}{\partial\ol{X}}$ denote the holomorphic and anti-holomorphic derivatives, i.e.
	\begin{align*}
		&\frac{\partial}{\partial X}=\frac{1}{2}\left(\frac{\partial}{\partial\re X}-\ii\frac{\partial}{\partial \im X}\right), &
		&\frac{\partial}{\partial \ol{X}}=\frac{1}{2}\left(\frac{\partial}{\partial \re X}+\ii\frac{\partial}{\partial \im X}\right).
	\end{align*}
	
	\item For the resolvent $G\equiv G(z)=(H-z)^{-1}$ of an $(N\times N)$ complex Hermitian matrix $H$ and $a,b\in\llbra 1,N\rrbra$, we have
	\begin{align*}
		&\frac{\partial}{\partial H_{ab}}G=-G\bse_{a}\bse_{b}\adj G, &
		&\frac{\partial}{\partial \ol{H_{ab}}}G=\frac{\partial}{\partial H_{ba}}G=-G\bse_{b}\bse_{a}\adj G,
	\end{align*}
	where we consider $G$ as a $\C^{N\times N}$-valued function of $N(N-1)/2$ complex variables $\{H_{ab}:a>b\}$ and $N$ real variables $\{H_{aa}\}$.
\end{itemize}

\section{Proof of Theorem \ref{thm:main result}}\label{sec:outline of proof of gfc}

The main idea of our proof is to apply Green function comparison to the Dyson Brownian motion (DBM) $H_{t}$ starting from $H_{0}$, whose precise definition is as follows:
\begin{defn}\label{def:rescaled model}
	We define the $(N\times N)$ random matrix $H_{t}$ as
	\begin{align}\label{eq:ht_def}
		&H_{t}=H_{0}+\sqrt{t}W,&
		&t\in[0,t_{0}], &
		&t_{0}\deq N^{-1/3+\chi},
	\end{align}
	where $\chi>0$ is a sufficiently small constant, $W$ is a GUE independent of $H_{0}$, and $\gamma_{t}$ is defined in \eqref{eq:gamma_def}. For each $t\in[0,t_{0}]$ we denote the eigenvalues of $H_{t}$ by
	\beqs
	\lambda_{1,t}\geq\cdots\geq \lambda_{N,t}.
	\eeqs
	For $z\in\C_{+}$, we define the resolvent and Green functions of $\gamma_{t} H_{t}$ as
	\begin{align}\label{eq:gt_def}
		&G_{t}(z)\deq (\gamma_{t} H_{t}-z)^{-1},& &m_{t}(z)\deq \tr G_{t}.
	\end{align}
	We introduce the symmetric counterparts of \eqref{eq:ht_def} and \eqref{eq:gt_def} as follows:
	\begin{align*}
		&\caH_{t}=\caH_{0}+\sqrt{t}U\adj WU, &
		&\caG_{t}=(\gamma_{t} \caH_{t}-z)^{-1}.
	\end{align*}
\end{defn}
Note that the limiting eigenvalue density of $\gamma_{t} H_{t}$ is exactly the dilation $\gamma_{t}^{-1}\rho_{t}(\gamma_{t}^{-1}\cdot)$ of $\mu_{t}$, so that it has the upper edge $L_{+,t}=\gamma_{t}E_{+,t}$ and decays as $\frac{1}{\pi}\sqrt{L_{+,t}-x}$ around $L_{+,t}$. For simplicity, we often omit the subscript $t$ to denote $H\equiv H_{t}$, $G\equiv G_{t}$, et cetera. 

\subsection{Proof of Theorem \ref{thm:main result}}\label{sec:prf_main}
As mentioned in the introduction, we prove Theorem \ref{thm:main result} by combining two results; edge universality at time $t=t_{0}$, and comparison along $t\in[0,t_{0}]$. In the next proposition, we prove the first result.
\begin{prop}\label{prop:t0 tracy widom} Let $F:\bbR^{k}\rightarrow \bbR$ be a smooth function such that $\norm{F}_{\infty}\leq C$ and $\norm{\nabla F}_{\infty}\leq C$ for some $C>0$. Then there exists a constant $c>0$ such that
	\begin{align}\label{eq:t0 tracy widom}
		\big\vert \E[F(\gamma_{t_{0}} N^{2/3}(\lambda_{1,t_{0}}&-E_{+,t_{0}}),\cdots,\gamma_{t_{0}} N^{2/3}(\lambda_{k,t_{0}}-E_{+,t_{0}}))]\nonumber\\
		&-\E[F(N^{2/3}(\mu_{1}-2),\cdots, N^{2/3}(\mu_{k}-2))]\big\vert\leq N^{-c},
	\end{align}
	where $\mu_{1}\geq\cdots\geq \mu_{N}$ are the eigenvalues of a GUE.
\end{prop}

In order to compare the largest eigenvalues of $H_{0}$ and $H_{t}$, we employ Green function comparison whose precise statement is as follows.
\begin{prop}[Green function comparison]\label{prop:gfc}
	Let $F:\R\to\R$ be a smooth function satisfying
	\beq\label{eq:Fcond}
	\sup_{x\in\R}\absv{F^{(\ell)}(x)}(\absv{x}+1)^{-C}\leq C,\quad \ell=1,2,3,4
	\eeq
	for a constant $C>0$. Then there exist constants $C'>0$ and $\epsilon_{0}>0$ so that the following holds: For any $\epsilon\in(0,\epsilon_{0})$, there exists $N_{0}\in\N$ such that for any $N\geq N_{0}$ and $E_{1},E_{2}\in\R$ with $\absv{E_{1}},\absv{E_{2}}\leq N^{-2/3+\epsilon}$ we have
	\begin{multline}\label{eq:gfc_def}
		\bigg\vert\expct{F\left(\int_{E_{1}}^{E_{2}}\im \Tr G_{0}(L_{+,0}+E+\ii\eta_{0})\dd E\right)	\\
			-F\left(\int_{E_{1}}^{E_{2}}\im \Tr G_{t_0}(L_{+,t_{0}}+E +\ii\eta_{0})\dd E\right)}\bigg\vert\leq N^{-1/6+C'\epsilon+\chi},
	\end{multline}
	where we abbreviated $\eta_{0}=N^{-2/3-\epsilon}$.
\end{prop}
Note that the positive constant $C'$ in Proposition \ref{prop:gfc} is uniform over $\epsilon\in(0,\epsilon_{0})$, but the threshold $N_{0}$ for $N$ may depend on $\epsilon$.

We prove Propositions \ref{prop:t0 tracy widom} and \ref{prop:gfc} in the next two subsections. Along their proofs and also the derivation of Theorem \ref{thm:main result} from them, we need the following local laws for $H_{t}$ near the edge, which holds uniformly over all bounded $t$:
\begin{prop}\label{prop:ll} Under the settings in Proposition \ref{prop:gfc}, the followings hold uniformly over $t\in[0,t_{0}]$ and $z\in \{E+L_{+,t}+\ii\eta_{0}: E\in[E_{1},E_{2}]\}$:
	\begin{align}
		\absv{\gamma_{t}m_{t}(z)-m_{\mu_{t}}(E_{+,t})} &\prec N^{-1/3+\epsilon};\label{eq:ll_aver}\\
		\max_{a,b\in\llbra 1,N\rrbra}\Absv{\gamma_{t}G_{ab}-\frac{\delta_{ab}}{\fra_{a}-\omega_{\alpha,t}(E_{+,t})}}
		+\absv{(U\adj G)_{ab}} &\prec N^{-1/3+\epsilon};\label{eq:ll_etr}\\
		\max_{a,b\in\llbra 1,N\rrbra}\Absv{\gamma_{t}(\wt{B}G)_{ab}-\frac{\delta_{ab}(\omega_{\beta,t}(E_{+,t})+m_{\mu_{t}}(E_{+,t})^{-1})}{\fra_{a}-\omega_{\alpha,t}(E_{+,t})}}&\prec N^{-1/3+\epsilon},\label{eq:ll_etr_BG}
	\end{align}
	where we abbreviated $G\equiv G_{t}(z)$ and $\caG\equiv \caG_{t}(z)$. The same estimates remain true if we interchange roles of $(A,B)$, $(U,U\adj)$, $(W,\caW)$, $(G,\caG)$, pairwise.
\end{prop}
In fact Proposition \ref{prop:ll} is a specialization of more general local law, Theorem \ref{thm:ll}, that allows for $\im z\gg N^{-1+}$ beyond $\im z=\eta_{0}$ as well as $t\in[0,1]$. As a standard corollary of the general local law, we have the following rigidity of eigenvalues of $H_{t}$. Let $\gamma_{j,t}$ be the $j$-th $N$-quantile of $\mu_{t}$, that is, the smallest real number such that
\beqs
\mu_{t}((-\infty,\gamma_{j,t}])=\mu_{\alpha}\boxplus\mu_{\beta}\boxplus\mu_{\text{sc}}^{(t)}((-\infty,\gamma_{j,t}])=\frac{N-j+1}{N}.
\eeqs
\begin{lem}[Rigidity around the edge]\label{lem:rigidity} There exists a (small) constant $c>0$ such that 
	\begin{align}\label{eq:rigidity}
		\max_{1\leq i\leq cN}|\lambda_{i,t}-\gamma_{i,t}|\prec i^{-1/3}N^{-2/3}
	\end{align}
	hold uniformly over $t\in[0,t_{0}]$.
\end{lem}
The proof of Proposition \ref{prop:ll} is presented in Appendix \ref{sec:rigidity}. We omit the proof of Lemma \ref{lem:rigidity} since it only requires minor modifications to that of \cite[Theorem 2.6]{Bao-Erdos-Schnelli2020}. Armed with Propositions \ref{prop:gfc}--\ref{prop:ll} and Lemma \ref{lem:rigidity}, we now prove Theorem \ref{thm:main result}.
\begin{proof}[Proof of Theorem \ref{thm:main result}]
	We follow the proof of Theorem 2.10 in \cite{Lee-Schnelli2018} with some modifications. Namely, we can simply plugin the inputs above into their counterparts in \cite{Lee-Schnelli2018}. Below we briefly explain the role of each component.
	
	First of all, we claim that \cite[Proposition 7.1]{Lee-Schnelli2018} remains true if we replace the sparse Wigner matrix $H$ therein by $H_{0}$ and $H_{t_{0}}$ in \eqref{eq:ht_def}. In other words, the cumulative distribution functions of $\lambda_{1,0}$ and $\lambda_{1,t_{0}}$ are well approximated by respectively the first and second terms on the left side of \eqref{eq:gfc_def} for a suitably chosen $F$. One can easily find that the proof of \cite[Prposition 7.1]{Lee-Schnelli2018} requires only three inputs, namely averaged local law, eigenvalue rigidity, and square-root decay. Simply replacing these inputs respectively by Proposition \ref{prop:ll}, Lemma \ref{lem:rigidity}, and Lemma \ref{lem:stab} proves the analogue.
	
	Secondly, Proposition \ref{prop:gfc} plays exactly the same role as \cite[Proposition 7.2]{Lee-Schnelli2018}, so that combining with the analogue of \cite[Proposition 7.1]{Lee-Schnelli2018} above proves that the cumulative distribution functions of $\lambda_{1,0}$ and of $\lambda_{1,t_{0}}$ have the same limit. Since Proposition \ref{prop:t0 tracy widom} shows that the distribution of $\lambda_{1,t_{0}}$ weakly converges to the Tracy-Widom distribution $F_{2}$, Theorem \ref{thm:main result} follows.
\end{proof}
\subsection{Proof of Proposition \ref{prop:t0 tracy widom}}\label{sec:prf_t0}
In this section we prove Proposition \ref{prop:t0 tracy widom}. By Lemma~\ref{lem:stab} and Theorem~\ref{thm:ll}, we find that the ESD $\mu_{H_{0}}$ of $H_{0}$ satisfies the assumptions of \cite[Theorem 2.2]{Landon-Yau2017} with high probability. Specifically, taking $\eta_{0}'=N^{-2/3+\chi/4}$, the diagonalization of $H_{0}$ is $\eta_{0}'$-regular with high probability. Therefore \cite[Theorem 2.2]{Landon-Yau2017} implies a \emph{random} version of Proposition \ref{prop:t0 tracy widom}, where random quantities $\wh{\gamma}_{t_{0}}$ and $\wh{E}_{+,t_{0}}$ replace $\gamma_{t_{0}}$ and $E_{+,t_{0}}$ in \eqref{eq:t0 tracy widom}, respectively; these quantities are defined as follows. Firstly, $\wh{E}_{+,t_{0}}$ is the upper edge of $\wh{\mu}_{t_{0}}=\mu_{H_{0}}\boxplus\mu_{\SC}^{(t_{0})}$. Secondly, the scale $\wh{\gamma}_{t_{0}}$ is defined as 
\beqs
\wh{\gamma}_{t_{0}}\deq \left(-t_{0}^{3}\int_{\R}\frac{1}{(x-\wh{E}_{+,t_{0}}-t_{0}m_{\wh{\mu}_{t_{0}}}(\wh{E}_{+,t_{0}}))^{3}}\dd \mu_{H_{0}}(x)\right)^{-1/3}.
\eeqs

Therefore, in order to prove Proposition \ref{prop:t0 tracy widom}, it suffices to show that $\absv{\wh{E}_{+,t_{0}}-E_{+,t_{0}}}$ is much smaller than the scale $N^{-2/3}$ of the fluctuations of $\lambda_{1}$, and that $\absv{\wh{\gamma}_{t_{0}}-\gamma_{t_{0}}}=o(1)$. We prove these two statements in the next lemma.
\begin{lem}\label{lem:center_diff}
	Let $D>0$. There exist a constant $C>0$ and an event $\Xi$ with $\P[\Xi^{c}]\leq N^{-D}$ such that
	\begin{align}
		&\lone_{\Xi}\absv{\wh{E}_{+,t_{0}}-E_{+,t_{0}}}\leq  CN^{-2/3-\chi/2},\label{eq:center_diff}\\\label{eq:scale_diff}
		&\lone_{\Xi}\absv{\gamma_{t_{0}}-\wh{\gamma}_{t_{0}}}\leq CN^{-5\chi/2}.
	\end{align}
\end{lem}
\begin{proof}
We first collect preliminary facts on the measure $\wh{\mu}_{t_{0}}$. As pointed out in \cite[Equation (7.9)]{Landon-Yau2017}, $\wh{E}_{+}$ is characterized as $\wh{E}_{+}=\wh{\xi}_{+}-t_{0}m_{\mu_{H_{0}}}(\wh{\xi}_{+})$ where $\wh{\xi}_{+}>\lambda_{1,0}$ is the rightmost solution of 
\beq\label{eq:whxi_char}
\int_{\R}\frac{1}{(x-\wh{\xi})^{2}}\dd\mu_{H_{0}}=\frac{1}{t_{0}}.
\eeq
Recall from \eqref{eq:Egamma} that $\xi_{+}\deq E_{+,t_{0}}+t_{0}m_{\mu_{t_{0}}}(E_{+,t_{0}})$ satisfy the same equation with $\mu_{H_{0}}$ replaced by $\mu_{0}$. We also have $E_{+,t_{0}}=\xi_{+}-t_{0}m_{\mu_{0}}(\xi_{+})$ due to \eqref{eq:Pastur}. Furthermore, by Lemma \ref{lem:stab} and \cite[Equation (7.10)]{Landon-Yau2017}, we have with high probability that
\beq\label{eq:xi-E_1}
\xi_{+}-E_{+,0}\sim t_{0}^{2}\sim\wh{\xi}_{+}-\lambda_{1,0}.
\eeq
Also recall from the eigenvalue rigidity for $t=0$ in Lemma \ref{lem:rigidity} that 
\begin{align}\label{eq:rigid_1}
	\absv{\lambda_{1,0}-E_{+,0}}\prec N^{-2/3}.
\end{align}
Combining \eqref{eq:xi-E_1} and \eqref{eq:rigid_1}, we have a constant $C>1$ and an event $\Xi_{0}$ with $\P[\Xi_{0}^{c}]\leq N^{-D-2}$ such that on $\Xi_{0}$ the following holds;
\begin{align}
	&C^{-1}t_{0}^{2}\leq \min(\xi_{+}-E_{+,0},\wh{\xi}_{+}-\lambda_{1,0})\leq \max(\xi_{+}-E_{+,0},\wh{\xi}_{+}-\lambda_{1,0})\leq Ct_{0}^{2},	\label{eq:xi-E}\\
	&\absv{\lambda_{1,0}-E_{+,0}}\leq CN^{-2/3+\chi/100}\leq (2C)^{-1}t_{0}^{2}	\label{eq:rigid}.
\end{align}

We next construct an event $\Xi$ with $\P[\Xi^{c}]\leq N^{D-1}$ such that
\beq\label{eq:xi_diff}
\lone_{\Xi}\absv{\xi_{+}-\wh{\xi}_{+}}\leq C N^{-2/3-\chi/2}.
\eeq
By \cite[Theorem 2.5]{Bao-Erdos-Schnelli2020} (see also Theorem \ref{thm:ll}), for any fixed $\sigma,\epsilon'>0$ we have an event $\Xi_{1}(\sigma,\epsilon')$ with $\P[\Xi_{1}(\sigma,\epsilon')^{c}]\leq N^{-D}$ on which we have
\beq\label{eq:ll_out}
\absv{m_{\mu_{H_{0}}}(z)-m_{\mu_{0}}(z)}\leq\frac{N^{\epsilon'}}{N\absv{z-E_{+,0}}},\,\forall z\in\caD_{\tau}(0,1)\cap\{\re z\geq E_{+,0}+N^{-2/3+\sigma}\}.
\eeq
Let $s=N^{-2/3-\chi/2}$. By \eqref{eq:xi-E}, we may take a constant $r>0$ so that on the event $\Xi_{0}$
\beqs
2rt_{0}^{2}<\min(\xi_{+}-E_{+,0},\wh{\xi}_{+}-\lambda_{1,0}).
\eeqs
Taking the Cauchy integral along a circle of radius $rt_{0}^{2}$ about $\xi_{+}+s$, we find that, on the event $\Xi_{1}(\chi/200,\epsilon')\cap\Xi_{0}$,
\beq\label{eq:xi_diff_0}\begin{aligned}
	&\Absv{\int_{\R}\frac{1}{(x-\xi_{+}-s)^{2}}\dd \mu_{H_{0}}(x)-\int_{\R}\frac{1}{(x-\xi_{+}-s)^{2}}\dd \mu_{0}(x)}	\\
	=&\frac{1}{2\pi}\Absv{\oint_{\absv{z-\xi_{+}-s}=rt_{0}^{2}}\frac{m_{H_{0}}(z)-m_{\mu_{0}}(z)}{(z-\xi_{+}-s)^{2}}\dd z}\leq \frac{N^{-1+\epsilon'}}{2\pi (rt_{0}^{2})^{2}}=\frac{1}{2\pi r^{2}}N^{1/3-4\chi+\epsilon'},
\end{aligned}\eeq
where we used $\xi_{+}+s-rt^{2}>E_{+,0}+rt^{2}$ to ensure that $\absv{z-E_{+}}\geq rt^{2}$ on the circle. On the other hand	we have constants $c,c'>0$ such that
\beq\begin{aligned}\label{eq:xi_diff_1}
	&\int_{\R}\frac{1}{(x-\xi_{+})^{2}}\dd\mu_{0}(x)-\int_{\R}\frac{1}{(x-\xi_{+}-s)^{2}}\dd\mu_{0}(x)=s\int_{\R}\frac{2\xi_{+}+s-2x}{(x-\xi_{+})^{2}(x-\xi_{+}-s)^{2}}\dd\mu_{0}(x)	\\
	\geq& c's\int_{\R}\frac{1}{(E_{+,0}+rt_{0}^{2}-x)^{3}}\dd\mu_{0}(x)\geq cst_{0}^{-3}=cN^{1/3-7\chi/2},
\end{aligned}\eeq
where we used 
\beqs
(\xi_{+}+s-x)\sim(\xi_{+}+s/2-x)\sim(\xi_{+}-x) \qquad\text{uniformly over }x\in\supp\mu_{0}
\eeqs
in the first inequality and the elementary asymptotics
\beq
\int_{0}^{1}\frac{\sqrt{x}}{(x+s)^{3}}\sim s^{-3/2},\qquad \text{as }s\to 0
\eeq
in the second. Combining \eqref{eq:xi_diff_0} and \eqref{eq:xi_diff_1}, on the event $\Xi_{1}(\chi/200,\epsilon')\cap\Xi_{0}$ we have
\beqs
\int_{\R}\frac{1}{(x-\xi_{+}-s)^{2}}\dd\mu_{H_{0}}(x)\leq \int_{\R}\frac{1}{(x-\xi_{+})^{2}}\dd\mu_{0}(x)-cN^{1/3-7\chi/2}+\frac{N^{1/3-4\chi+\epsilon'}}{2\pi r^{2}}.
\eeqs
Choosing $\epsilon'<\chi/2$ and using \eqref{eq:Egamma}, on the event $\Xi=\Xi_{1}(\chi/200,\epsilon')\cap\Xi_{0}$ we have
\beqs
\int_{\R}\frac{1}{(x-\xi_{+}-s)^{2}}\dd\mu_{H_{0}}(x)<\frac{1}{t_{0}}.
\eeqs
By a similar argument, on the event $\Xi$ we also have
\beqs
\int_{\R}\frac{1}{(x-\xi_{+}+s)^{2}} \dd \mu_{H_{0}}(x)>\frac{1}{t_{0}}.
\eeqs
On the event $\Xi$ we have $\xi_{+}-s>\lambda_{1,0}$ due to \eqref{eq:xi-E} and \eqref{eq:rigid}, so that the map
\beq\label{eq:tau-whtau}
y\mapsto \int_{\R}\frac{1}{(x-\xi_{+}+y)^{2}}\dd \mu_{H_{0}}(x), \quad y\in(-s,s)
\eeq
is monotone increasing. Therefore there exists $y_{0}\in(-s,s)$ such that $\xi_{+}+y_{0}=\wh{\xi}_{+}$ satisfies \eqref{eq:whxi_char}. Recalling that $s=N^{-2/3-\chi/2}$, we have proved \eqref{eq:xi_diff}.

Now we can prove \eqref{eq:center_diff} by recovering $(E_{+},\wh{E}_{+})$ from $(\xi_{+},\wh{\xi}_{+})$. On the event $\Xi$, we have
\beqs
\begin{split}
	E_{+}-\wh{E}_{+}&=\xi_{+}-\wh{\xi}_{+}-t_{0}\int_{\R}\frac{1}{x-\xi_{+}}\dd \mu_{0}(x)+t_{0}\int_{\R}\frac{1}{x-\wh{\xi}_{+}}\dd \mu_{H_{0}}(x)\\
	&=(\xi_{+}-\wh{\xi}_{+})\left(1-t_{0}\int_{\R}\frac{1}{(x-\xi_{+})(x-\wh{\xi}_{+})}\dd\mu_{0}(x)\right)+O(N^{-2/3-\chi/2})\\
	&=(\xi_{+}-\wh{\xi}_{+})\left(t_{0}\int_{\R}\frac{\xi_{+}-\wh{\xi}_{+}}{(x-\xi_{+})^{2}(x-\wh{\xi}_{+})}\dd\mu_{0}(x)\right)+O(N^{-2/3-\chi/2})\\
	&=(\xi_{+}-\wh{\xi}_{+})^{2}O(t_{0}^{-2})+O(N^{-2/3-\chi/2})
	=O(N^{-2/3-\chi/2}),
\end{split}
\eeqs
where we used \eqref{eq:ll_out} and $\Xi\subset\Xi_{1}(\chi/200,\epsilon')$ to get
\beq
t\int_{\R}\frac{1}{x-\wh{\xi}_{+}}\dd\mu_{H_{0}}(x)-t\int_{\R}\frac{1}{x-\wh{\xi}_{+}}\dd\mu_{0}(x)\leq Ct\cdot\frac{N^{\epsilon'}}{Nt^{2}}\leq CN^{-2/3-\chi/2}
\eeq
in the second equality, \eqref{eq:Egamma} in the third, and \eqref{eq:xi_diff_1} in the fourth.
This concludes the proof of \eqref{eq:center_diff}.

Next, we turn to the proof of \eqref{eq:scale_diff}. Since $\gamma_{t_{0}}\sim 1$ from \eqref{eq:gamma_def}, it suffices to estimate
\beq\label{eq:gamma_diff}
\begin{split}
	&\absv{\gamma_{t}^{-3}-\wh{\gamma}_{t}^{-3}}=t^{3} \Absv{\int_{\R}\frac{1}{(x-\xi_{+})^{3}}\dd \mu_{0}(x)-\int_{\R} \frac{1}{(x-\wh{\xi}_{+})^{3}} \dd \mu_{H_{0}}(x)}	\\
	&\leq t^{3}\Absv{\int_{\R}\frac{1}{(x-\xi_{+})^{3}}-\frac{1}{(x-\wh{\xi}_{+})^{3}}\dd \mu_{0}(x)}
	+t^{3}\Absv{\int_{\R}\frac{1}{(x-\wh{\xi}_{+})^{3}}\dd (\mu_{0}-\mu_{H_{0}})(x)},
\end{split}
\eeq
where we used $m_{\mu_{H_{0}}}(\wh{\xi}_{+})=m_{\wh{\mu}_{t_{0}}}(\wh{E}_{+,t_{0}})$ due to \eqref{eq:Pastur}.
The second term on the right-hand side of \eqref{eq:gamma_diff} can be estimated as in \eqref{eq:ll_out}, that is, taking the Cauchy integral:
\beqs
\lone_{\Xi}t^{3}\left| \int\frac{1}{(x-\wh{\xi}_{+})^{3}}\dd \mu_{0}(x)-\int\frac{1}{(x-\wh{\xi}_{+})^{3}} \dd \mu_{H_{0}}(x)\right|\leq Ct^{3}\frac{N^{\epsilon}}{N(rt^{2})^{3}}=CN^{-3\chi+\epsilon}.
\eeqs
The first term on the right-hand side of \eqref{eq:gamma_diff} can be estimated as
\beqs
\lone_{\Xi}t^{3}\Absv{\int \frac{1}{(x-\xi_{+})^{3}}\dd \mu_{0}(x) -\int \frac{1}{(x-\wh{\xi}_{+})^{3}}\dd \mu_{0}(x)}
\leq O(t^{-2}|\xi_{+}-\wh{\xi}_{+}|)=O(N^{-5\chi/2}),
\eeqs
where we used explicit calculations as in \eqref{eq:xi_diff_1} in the first inequality and $\absv{\xi_{+}-\wh{\xi}_{+}}\lesssim N^{-2/3-\chi/2}$ in the second. This completes the proof of Lemma \ref{lem:center_diff}.
\end{proof}

\subsection{Proof of Proposition \ref{prop:gfc}}\label{sec:gfc_prf}
Define
\beqs
Y\equiv Y_{t}\deq N\int_{E_{1}}^{E_{2}} \im m_{t}(E+L_{+,t}+\ii\eta_0) \dd E.
\eeqs
Recalling that $t_{0}=N^{-1/3+\chi}$, it suffices to show that 
\beq\label{eq:dt_gfc_goal}
\frac{\dd}{\dd t}\expct{F(Y)}\leq N^{1/6+C'\epsilon},
\eeq
in order to prove \eqref{eq:gfc_def}. Computing the derivative explicitly, we obtain
\begin{align}\label{eq:EFY_der}
	&\frac{\dd\expct{F(Y)}}{\dd t}=\Expct{F'(Y) \frac{\dd Y}{\dd t}}=\Expct{F'(Y)\im \int_{E_{1}}^{E_{2}}\frac{\dd \Tr G(L_{+}+E+\ii\eta_{0})}{\dd t}\dd E}\\
	=&\bbE\Bigg[F'(Y)\im \int_{E_{1}}^{E_{2}}\bigg(\dot{L}_{+}\Tr G^{2} -\dot{\gamma}\Tr GHG-\frac{\gamma}{2\sqrt{t}}\Tr GWG\bigg)\dd E\Bigg]\nonumber,
\end{align}
where we abbreviated $G\equiv G(L_{+}+\ii\eta_{0}+E)$ and $\dot{L}_{+},\dot{\gamma}$ denote time derivatives of $L_{+},\gamma$, respectively. Since $W$ is a GUE, we can apply Stein's lemma to the last term on the right-hand side of \eqref{eq:EFY_der} to get
\beq\label{eq:Stein_W}\begin{aligned}
	\expct{F'(Y)\Tr WG^{2}}
	=&-2\gamma\sqrt{t}\expct{F'(Y)\tr G\Tr G^{2}}	\\
	&+\frac{\sqrt{t}}{N}\sum_{a,b}\expct{F''(Y)(G^{2})_{ba}\frac{\partial Y}{\partial H_{ba}}}.
\end{aligned}\eeq
Also, we can calculate the derivative $\partial Y/(\partial H_{cb})$ explicitly as
\beq\label{eq:Yder}\begin{aligned}
	\frac{\partial Y}{\partial H_{ba}}
	=-\frac{\gamma}{2\ii}\int_{E_{1}}^{E_{2}}((\wt{G}^{2})_{ab}-((\wt{G}\adj)^{2})_{ab})\dd\wt{E}
	=-\gamma\int_{E_{1}}^{E_{2}}(\im [\wt{G}^{2}])_{ab}\dd\wt{E},
\end{aligned}\eeq
where we abbreviated $\wt{G}\deq G(\wt{E}+L_{+}+\ii\eta_{0})$; the same notation applies to the rest of this paper. In summary, we have
\beq\begin{aligned}\label{eq:dt_gfc}
	\frac{\dd\expct{F(Y)}}{\dd t}=&\Expct{F'(Y)\im\int_{E_{1}}^{E_{2}}\left(\dot{L}_{+}\Tr G^{2} -\dot{\gamma}\Tr GHG+\gamma^{2}(\tr G)\Tr G^{2}\right)\dd E}\\
	&+\frac{1}{N}\frac{\gamma^{2}}{2}\Expct{F''(Y)\int_{E_{1}}^{E_{2}}\int_{E_{1}}^{E_{2}}\Tr \im[\wt{G}^{2}]\im[G^{2}]\dd\wt{E}\dd E}.
\end{aligned}\eeq

Next, we further simplify the first line of \eqref{eq:dt_gfc}. Recall from \eqref{eq:Pastur} that
\beqs
E_{+,t}=E_{+,t}+tm_{\mu_{t}}(E_{+,t})-t\int_{\R}\frac{1}{x-E_{+,t}-tm_{\mu_{t}}(E_{+,t})}\dd\mu_{0}(x),
\eeqs
where we recall $\mu_{0}=\mu_{t}\vert_{t=0}=\mu_{\alpha}\boxplus\mu_{\beta}$. Taking the time derivative of both sides, we obtain
\begin{align*}
	\dot{E}_{+,t}=-m_{\mu_{t}}(E_{+,t})+\left(1-t\int_{\R}\frac{1}{(x-E_{+,t}-tm_{\mu_{t}}(E_{+,t}))^{2}}\dd\mu_{0}(x)\right)\cdot\frac{\partial}{\partial t}\left[E_{+,t}+tm_{\mu_{t}}(E_{+,t})\right].
\end{align*}
By \eqref{eq:Egamma} the second term vanishes so that $\dot{E}_{+,t}=-m_{\mu_{t}}(E_{+,t})$, from which we get
\beq\label{eq:dot_L}
\dot{L}_{+}=\dot{\gamma}E_{+}+\gamma\dot{E}_{+}
=\dot{\gamma}\frac{L_+}{\gamma}-\gamma m_{\mu_{t}}(E_{+,t}).
\eeq
On the other hand, by the identity $\gamma HG=zG+I$ we have
\beq\label{eq:resolv_id}
\Tr GHG=\frac{z}{\gamma}\Tr G^{2}+\frac{1}{\gamma}\Tr G.
\eeq
Plugging in \eqref{eq:dot_L} and \eqref{eq:resolv_id} to the integrand in the first line of \eqref{eq:dt_gfc} yields
\beq\label{eq:dt_gfc_1}\begin{aligned}
	&\dot{L}_{+}\Tr G^{2}-\dot{\gamma}\Tr GHG+\gamma^{2}\tr G\Tr G^{2}	\\
	=&\frac{\dot{\gamma}}{\gamma}(L_{+}-z)\Tr G^{2}-\frac{\dot{\gamma}}{\gamma}\Tr G
	+\gamma(\gamma\tr G-m_{\mu_{t}}(E_{+,t}))\Tr G^{2}.
\end{aligned}\eeq

Note that \eqref{eq:dt_gfc_1} is a non-asymptotic, exact identity, and we have used only \eqref{eq:Egamma} and \eqref{eq:Pastur} along the proof. We now use the asymptotic inputs Lemma \ref{lem:stab} and Proposition \ref{prop:ll} to prove that only the last term of \eqref{eq:dt_gfc_1} is relevant, that is, the contributions of the first two terms to $\partial_{t}\E[F(Y)]$ are $O_{\prec}(N^{1/6+C'\epsilon'})$. To this end, we first roughly estimate the size of $F'(Y)$. By $\gamma_{t}\sim 1$ from Lemma \ref{lem:stab} (v) and \eqref{eq:ll_aver} we have
\beq\label{eq:imm}
\im m_{t}=\im{\tr G_{t}}=\frac{\im{m_{\mu_{t}}(E_{+})}}{\gamma_{t}}+O_{\prec}(N^{-1/3+\epsilon})=O_{\prec}(N^{-1/3+\epsilon})
\eeq
uniformly over $z\in L_{+}+[E_{1},E_{2}]+\ii\eta_{0}$. This further implies
\beqs
0\leq Y=N\int_{E_{1}}^{E_{2}}\im{m_{t}(E+L_{+}+\ii\eta_{0})}\dd E\prec N^{2/3+\epsilon}(E_{2}-E_{1})=N^{2\epsilon},
\eeqs
which, together with \eqref{eq:Fcond}, gives 
\beq\label{eq:FY}
\absv{F^{(\ell)}(Y)}\prec N^{2C\epsilon},\qquad \ell=1,\cdots,4.
\eeq

For the first term on the right-hand side of \eqref{eq:dt_gfc_1}, note that Ward identity (i.e. $\im G/\im z=\absv{G}^{2}$)  implies
\beq\label{eq:G2_rough}
\Absv{\Tr G^{2}}\leq \Tr \absv{G}^{2}=\frac{\im\Tr G}{\eta_{0}}\prec N^{4/3+2\epsilon},
\eeq
where we used \eqref{eq:ll_etr} and \eqref{eq:imm}. Recalling $\gamma\sim 1$ and $\dot{\gamma}\lesssim 1$ from Lemma \ref{lem:stab} (v), we have
\beq\label{eq:dt_gfc_2}
\absv{F'(Y)}\int_{E_{1}}^{E_{2}}\Absv{\frac{\dot{\gamma}}{\gamma}(L_{+}-z)\Tr G^{2}}\dd E
\prec (\absv{E_{1}}+\absv{E_{2}})^{2}N^{3/4+2(1+C)\epsilon}\leq N^{2(2+C)\epsilon}.
\eeq
On the other hand for the second term of \eqref{eq:dt_gfc_1}, we use \eqref{eq:imm} again to get
\beq\label{eq:dt_gfc_3}
\frac{\dot{\gamma}}{\gamma}\Absv{F'(Y)}\int_{E_{1}}^{E_{2}}\im\Tr G\dd E\prec N^{2C\epsilon}\cdot (E_{2}-E_{1})\cdot N^{2/3+\epsilon}=N^{2(1+C)\epsilon}.
\eeq
Plugging in \eqref{eq:dt_gfc_2} and \eqref{eq:dt_gfc_3} to \eqref{eq:dt_gfc_1} and then to \eqref{eq:dt_gfc}, we finally obtain
\beq\label{eq:dt_gfc_result}
\frac{\dd \E[F(Y)]}{\dd t}=
\frac{\gamma}{2}\im\int_{E_{1}}^{E_{2}}\frX\dd E+O(N^{2(3+C)\epsilon}),
\eeq
where we defined
\beq\begin{aligned}\label{eq:frX_def}
	&\frX\equiv\frX(E)	\\
	\deq&\E\bigg[2F'(Y)(\gamma\tr G-m_{\mu_{t}}(E_{+,t})\Tr G^{2})	
	+\frac{\gamma}{N}F''(Y)\int_{E_{1}}^{E_{2}}\Tr(\im[\wt{G}^{2}]G^{2})\dd\wt{E}\bigg].
\end{aligned}\eeq

Applying the rough estimates from Proposition \ref{prop:ll} to $\frX$, we find the first term of \eqref{eq:frX_def} is estimated as
\beqs
F'(Y)(\gamma\tr G-m_{\mu_{t}}(E_{+,t}))\Tr G^{2}\prec N^{1+(2C+3)\epsilon},
\eeqs
and the second term admits the rough upper bound
\beq\label{eq:frX_rough}\begin{aligned}
	&\Absv{\frac{1}{N}\int_{E_{1}}^{E_{2}}\Tr (\im[\wt{G}^{2}]G^{2})\dd\wt{E}}\leq \frac{1}{N}\int_{E_{1}}^{E_{2}}\sum_{a,b,c,d}\frac{\absv{\wt{G}_{ab}\wt{G}_{bc}}+\absv{\wt{G}_{cb}\wt{G}_{ba}}}{2}{G_{cd}G_{da}}\dd\wt{E}	\\
	\prec &N^{3}(E_{2}-E_{1})N^{-4/3+4\epsilon}=N^{1+5\epsilon},
\end{aligned}\eeq
so that $\frX=O(N^{1+C\epsilon})$. Comparing \eqref{eq:dt_gfc_result} with \eqref{eq:dt_gfc_goal}, we need to improve the rough estimate for $\im\frX$ by a factor of $N^{-1/6}$. We present the required estimate in the next proposition, whose proof is postponed to the next section;
\begin{prop}\label{prop:main}
	Let $\epsilon>0$, take $E_{1},E_{2}$ and $\eta_{0}$ as in Proposition \ref{prop:gfc}, and define $\frX$ by \eqref{eq:frX_def} for each $E\in[E_{1},E_{2}]$. Then there exist constants $\epsilon_{1},C''>0$ such that for any $\epsilon\in(0,\epsilon_{1})$ we have
	\beq\label{eq:frX_result}
	\sup_{E\in[E_{1},E_{2}]}\im \frX(E)\leq N^{5/6+C''\epsilon}
	\eeq
	uniformly over $E\in[E_{1},E_{2}]$ for all $N\geq N_{1}(\epsilon)$.
\end{prop}
Proposition \ref{prop:main} is the main technical achievement of this paper, and all of Sections \ref{sec:frX}--\ref{sec:der_tech} are devoted to its proof. Assuming validity of Proposition \ref{prop:main}, we immediately have
\begin{align}
	\frac{\dd\expct{F(Y)}}{\dd t}=&\gamma\int_{E_{1}}^{E_{2}}\im \frX\dd E+O(N^{2(3+C)\epsilon})	\nonumber\\
	\leq& (E_{2}-E_{1})N^{5/6+C''\epsilon}+O(N^{2(3+C)\epsilon})=O(N^{1/6+(C''+1)\epsilon})\leq N^{1/6+(C''+2)\epsilon}
\end{align}
for all sufficiently large $N$. This establishes \eqref{eq:dt_gfc_goal}, hence concludes the proof of Proposition \ref{prop:gfc}.

\begin{rem}\label{rem:model_dep}
	The arguments in this section apply to a wide variety of matrix models. More precisely, except for local laws, the only truly model-dependent component of the proof of Theorem \ref{thm:main result} is Proposition \ref{prop:main}. Consider a general initial matrix $H_{0}$ that has a regular edge as in Lemma \ref{lem:fc_dens} and satisfies an optimal local law at the edge as in Proposition \ref{prop:ll} for $t=0$. Then a minor modification \cite{Landon-Yau2017} proves almost\footnote{Finer regularity of the limiting density $\rho_{t}$ may involve a model-dependent proof, for example the continuity of $m_{\mu_{t}}$ and $\dot{\gamma}_{t}\sim 1$ in Lemma \ref{lem:stab} (i) and (v). See \cite[Appendix A]{Adhikari-Huang2020} for an instance, where it is assumed that $m_{\mu_{0}}(z)=A(z)+\sqrt{B(z)}$ for some analytic functions $A,B$ around the edge.} all of required inputs, that is, edge regularity and the optimal local law carry over to the DBM $H+\sqrt{t}W$. For such an $H_{0}$, following the exact same arguments as in this section shows that it suffices to prove the analogue of Proposition \ref{prop:main} in order to prove edge universality.
\end{rem}

\section{Proof of Proposition \ref{prop:main}}\label{sec:frX}

\subsection{Special case: deformed GUE}\label{sec:dGUE}
The proof of Proposition \ref{prop:main} mainly concerns finding non-trivial cancellation within $\im\frX$. In order to facilitate the (much technical) computations for the free sum, here we first consider the special case of deformed GUE. To be precise, we make the following simplifications; we assume (i) $\wt{B}$ is a GUE (so that $\mu_{\beta}=\mu_{\SC}^{(1)}$), (ii) $t=0$, (iii) $\gamma=1$, (iv) and $F(Y)\equiv Y$. It should be noted that the content of this section was essentially covered in \cite{Lee-Schnelli2015}, and is included in here for purely pedagogical purposes. The only difference of this section and \cite{Lee-Schnelli2015} is that we use Stein's lemma whereas \cite{Lee-Schnelli2015} used Schur's complement when expanding the resolvent $G$.

With these choices, $H_{0}$ becomes deformed GUE for which Proposition \ref{prop:ll} was proved in \cite{Lee-Schnelli2013}. Also the assumptions $\gamma=1$ and $F(Y)\equiv Y$ imply
\beq
L_{+}=E_{+},\qquad G\equiv G(E_{+}+E+\ii\eta_{0}),\qquad \frX=2\Expct{(\tr G-m_{\mu})\Tr G^{2}}.
\eeq
For simplicity, we abbreviate $m_{\mu}\equiv m_{\mu}(E_{+})$ and $\omega_{\alpha}\equiv \omega_{\alpha}(E_{+})$.
The most important point is that, in order to see the cancellation, we have to expand the diagonal entry $\E(G^{2})_{aa}$ but not $\frX$ itself. The goal of our expansion is to ``decouple'' the index $a$ from $\E(G^{2})_{aa}$: Specifically, we prove
\beq\label{eq:dGUE_decouped}
\E(G^{2})_{aa}=\frac{\E\tr G^{2}}{(\fra_{a}-\omega_{\alpha})^{2}}+
\frac{\frX}{(\fra_{a}-\omega_{\alpha})^{3}}+\frac{1}{(\fra_{a}-\omega_{\alpha})^{2}}+O(N^{-1/3+C\epsilon}).
\eeq
Notice that, on the right-hand side of \eqref{eq:dGUE_decouped}, the index $a$ only appear as deterministic factors, and all the remaining quantities exclusively involve traces of $G$ and $G^{2}$. We refer to such an expression as \emph{decoupled} (of the index $a$).

Indeed, we only need few more lines to derive the required cancellation from \eqref{eq:dGUE_decouped}. Taking the sum over $a$ and the imaginary part of \eqref{eq:dGUE_decouped}, we obtain
\beq\label{eq:dGUE_optical}
\left(1-\frac{1}{N}\sum_{a}\frac{1}{(\fra_{a}-\omega_{\alpha})^{2}}\right)\im\E\Tr G^{2}=\frac{1}{N}\sum_{a}\frac{1}{(\fra_{a}-\omega_{\alpha})^{3}}\im\frX+O(N^{2/3+C\epsilon}).
\eeq
On the other hand, from Assumption~\ref{assump:ABconv}, \eqref{eq:stab}, and \eqref{eq:Egamma} we have
\beq\label{eq:dGUE_edge}
1-\frac{1}{N}\sum_{a}\frac{1}{(\fra_{a}-\omega_{\alpha})^{2}}=1-\int_{\R}\frac{1}{(x-\omega_{a})^{2}}\dd\mu_{\alpha}(x)+O(\bsd)=O(N^{-1+\epsilon}).
\eeq
Similarly Assumption \ref{assump:ABconv} and \eqref{eq:stab} also implies
\beq\label{eq:dGUE_scale}
-\frac{1}{N}\sum_{a}\frac{1}{(\fra_{a}-\omega_{\alpha})^{3}}>c
\eeq
for some constant $c>0$. Therefore we finally get
\beq\label{eq:dGUE_frX}
\absv{\im\frX}\lesssim \bsd\absv{\E\Tr G^{2}}=O(N^{2/3+C\epsilon}),
\eeq
where the first inequality is due to \eqref{eq:dGUE_optical}--\eqref{eq:dGUE_scale} and the second follows from Assumption~\ref{assump:ABconv} and \eqref{eq:G2_rough}.

Now we move on to the proof of \eqref{eq:dGUE_decouped}. As we took $\wt{B}$ to be a GUE, Stein's lemma gives
\beq\label{eq:Stein_first}
\E[\wt{B}_{ab}G_{cd}]=\frac{1}{N}\Expct{\bse_{c}\adj\frac{\partial G}{\partial\wt{B}_{ba}}\bse_{d}}=\frac{1}{N}\E[G_{cb}G_{ad}],\qquad \forall a,b,c,d\in\{1,\cdots,N\}.
\eeq
Using $zG+I=HG$ and \eqref{eq:Stein_first}, we find that
\beq\label{eq:dGUE_expa_1}\begin{aligned}
	z\E(G^{2})_{aa}=&\fra_{a}\E (G^{2})_{aa}+\sum_{b,c}\E\wt{B}_{ab}G_{bc}G_{ca}-\E G_{aa}	\\
	=&\fra_{a}\E(G^{2})_{aa}-\expct{(\tr G)(G^{2})_{aa}}-\expct{G_{aa}\tr G^{2}}-\E G_{aa}.
\end{aligned}\eeq
Then the local law \eqref{eq:ll_etr} and \eqref{eq:G2_rough} imply
\beq\label{eq:dGUE_0}\begin{aligned}
	z\E(G^{2})_{aa}=&(\fra_{a}-m_{\mu})\E (G^{2})_{aa}-\E[G_{aa}\tr G^{2}]\\
	&-\expct{(\tr G-m_{\mu})(G^{2})_{aa}}-\frac{1}{\fra_{a}-\omega_{\alpha}}+O(N^{-1/3+C\epsilon}).
\end{aligned}\eeq

Since $\mu_{\beta}=\mu_{\SC}^{(1)}$, \eqref{eq:Pastur} implies $z+m_{\mu}(z)=\omega_{\alpha}(z)$. Recalling that $z\in E_{+}+[E_{1},E_{2}]+\ii\eta_{0}$, this further gives
\beq\label{eq:dGUE_01}
z+m_{\mu}=E_{+}+m_{\mu}(E_{+})+O(N^{-2/3+\epsilon})=\omega_{\alpha}(E_{+})+O(N^{-2/3+\epsilon}).
\eeq
Moving the first term on the right-hand side of \eqref{eq:dGUE_0} to the left and using \eqref{eq:dGUE_01}, we have
\beq\label{eq:dGUE_1}\begin{aligned}
	(\fra_{a}-\omega_{\alpha})\E(G^{2})_{aa}=&\expct{G_{aa}\tr G^{2}}	\\
	&+\expct{(\tr G-m_{\mu})(G^{2})_{aa}}+\frac{1}{\fra_{a}-\omega_{\alpha}}
	+O(N^{-1/3+C\epsilon}),
\end{aligned}\eeq
where we also used the same estimates as in \eqref{eq:G2_rough} to get $\E\absv{(G^{2})_{aa}}=O(N^{1/3+3\epsilon})$. Recall $\absv{\fra_{a}-\omega_{\alpha}}\sim 1$ from \eqref{eq:stab}, so that we may divide both sides of \eqref{eq:dGUE_1} by $(\fra_{a}-\omega_{\alpha})$. Comparing with our goal \eqref{eq:dGUE_decouped}, it suffices to prove the following:
\begin{align}
	\expct{G_{aa}\tr G^{2}}&=\frac{\E\tr G^{2}}{\fra_{a}-\omega_{\alpha}}+\frac{\frX}{2(\fra_{a}-\omega_{\alpha})^{2}}+O(N^{-1/3+C\epsilon}), \label{eq:dGUE_decoup_1}	\\
	\expct{(\tr G-m_{\mu})(G^{2})_{aa}}&=\frac{\frX}{2(\fra_{a}-\omega_{\alpha})^{2}}+O(N^{-1/3+C\epsilon}).\label{eq:dGUE_decoup_2}
\end{align}
These are yet another form of decoupling; the index $a$ on the right-hand side appears only as deterministic factors. 

Next, we prove \eqref{eq:dGUE_decoup_1}. We start from the same expansion as in \eqref{eq:dGUE_0};
\beq\label{eq:dGUE_expa_2}\begin{aligned}
	&z\E[G_{aa}\tr G^{2}] =\fra_{a}\E[G_{aa}\tr G^{2}]+\E[(\wt{B}G)_{aa}\tr G^{2}]-\E\tr G^{2}	\\
	=&\fra_{a}\E[G_{aa}\tr G^{2}]-\E[(\tr G)G_{aa}\tr G^{2}]-\E\tr G^{2}
	-\frac{2}{N^{2}}\E[(G^{4})_{aa}]	\\
	=&\fra_{a}\E[G_{aa}\tr G^{2}]-\E[(\tr G)G_{aa}\tr G^{2}]-\E\tr G^{2}+O(N^{-1/3+C\epsilon}),
\end{aligned}\eeq
where in the last line we used \eqref{eq:ll_etr} to get
\beq
\absv{(G^{4})_{aa}}\leq\sum_{b,c,d}\absv{G_{ab}G_{bc}G_{cd}G_{da}}\prec N^{3}\cdot N^{-4/3+4\epsilon}.
\eeq
Then, by the exact same rearrangement as in \eqref{eq:dGUE_1} we get
\begin{align}
	(\fra_{a}-\omega_{\alpha})\E[G_{aa}\tr G^{2}]
	=&\E[(\tr G-m_{\mu})G_{aa}\tr G^{2}]+\E\tr G^{2}+ O(N^{-1/3+5\epsilon})	\nonumber\\
	=&\frac{1}{\fra_{a}-\omega_{\alpha}}\E[(\tr G-m_{\mu})\tr G^{2}]+\E\tr G^{2}+O(N^{-1/3+5\epsilon}),
\end{align}
and dividing both sides by $(\fra_{a}-\omega_{\alpha})$ proves \eqref{eq:dGUE_decoup_1}.

To prove \eqref{eq:dGUE_decoup_2}, we repeat the same procedure except that we apply the identity $zG+I=HG$ to the factor $(G^{2})_{aa}$;
\beq\begin{aligned}\label{eq:dGUE_expa_3}
	&z\E[(\tr G-m_{\mu})(G^{2})_{aa}]	\\
	=&\fra_{a}\E[(\tr G-m_{\mu})(G^{2})_{aa}]+\E[(\tr G-m_{\mu})(\wt{B}G^{2})_{aa}]-\E[G_{aa}(\tr G-m_{\mu})]\\
	=&\fra_{a}\E[(\tr G-m_{\mu})(G^{2})_{aa}]-\E[(\tr G-m_{\mu})(\tr G)(G^{2})_{aa}]	\\
	&-\Expct{(\tr G-m_{\mu})G_{aa}\tr G^{2}}	
	-\Expct{\frac{(G^{4})_{aa}}{N^{2}}}-\E[G_{aa}(\tr G-m_{\mu})]	\\
	=&\fra_{a}\E[(\tr G-m_{\mu})(G^{2})_{aa}]-m_{\mu}\E[(\tr G-m_{\mu})(G^{2})_{aa}]	\\
	&-\frac{1}{\fra_{a}-\omega_{a}}\Expct{(\tr G-m_{\mu})\tr G^{2}}
	+O(N^{-1/3+5\epsilon}).	
\end{aligned}\eeq
After rearranging, we obtain
\beq\begin{aligned}
	\E[(\tr G-m_{\mu})(G^{2})_{aa}]
	=&\frac{1}{(\fra_{a}-\omega_{\alpha})^{2}}\E[(\tr G-m_{\mu})\tr G^{2}]+O(N^{-1/3+5\epsilon}).
\end{aligned}\eeq
This completes the proof of Proposition \ref{prop:main} in the simplified deformed GUE case.

\begin{rem}
	One can easily notice that the right-hand side of \eqref{eq:frX_result} and \eqref{eq:dGUE_frX} are different; the former is $O(N^{5/6+C\epsilon})$ whereas the latter is $O(N^{2/3+C\epsilon})$. The extra factor of $N^{1/6}$ in the free sum case is due to the same factors in Lemmas \ref{lem:BG2_conc} -- \ref{lem:Yder}. For more details, we refer to Remark \ref{rem:extra} after the proof of Lemma \ref{lem:BG2_conc}.
\end{rem}

\subsection{Preliminaries for the proof of Proposition \ref{prop:main}}\label{sec:frX_prelim}
In this section we introduce new notations that are used throughout the proof of Proposition \ref{prop:main} for the general free sum $H$. We first introduce the following abbreviations;
\beq\label{eq:abbrv}\begin{aligned}
	m_{\mu}&\equiv m_{\mu_{t}}(E_{+,t}), &
	\omega_{\alpha}&\equiv \omega_{\alpha,t}(E_{+,t}), &
	\omega_{\beta}&\equiv \omega_{\beta,t}(E_{+,t}).
\end{aligned}\eeq
We always take the spectral parameter of $G,\caG$ to be $z=L_{+}+E+\ii\eta_{0}$ with $E\in[E_{1},E_{2}]$ and often omit the dependence on $z$ to write $G\equiv G(z),\caG\equiv\caG(z)$. We remark that all of $\omega_{\alpha},\omega_{\beta}$, and $m_{\mu}$ are deterministic and $N$-independent, hence \eqref{eq:abbrv} should not be confused with $G\equiv G(z)$. 

We aim at applying similar arguments as in the previous section to the free sum. However, there is an additional, fundamental difficulty compared to the deformed GUE case, namely that the derivative of $G$ with respect to $U$ has an additional factor of $B$. Indeed, if not for the constraint $UU\adj=I$ (corresponding to complex Ginibre $U$) we would have
\beqs
\frac{\partial G}{\partial U_{ab}}=G\bse_{a}\bse_{b}\adj BU\adj G.
\eeqs
Recall from \eqref{eq:Stein_first} that in the deformed GUE case the corresponding derivative of $G$ did not involve any other matrix than $G$. Obviously, in practice, we need to take the constraint $UU\adj=I$ into account hence resort to partial randomness decomposition; see \eqref{eq:dgbd} for instance.

Consequently, we will shortly see that the expansion of $(G^{2})_{aa}$ involves traces and entries of $G\wt{B}^{k}G$ for $k\in\{0,1,2\}$. In this regard, for each $a\in\llbra 1,N\rrbra$ we define three dimensional complex vectors $\bfx_{a}\equiv(\rmx_{ai})_{i=1,2,3}$ and $\bsx_{a}\equiv(x_{ai})_{i=1,2,3}$ by
\beq\label{eq:def_x}\begin{aligned}
	& \rmx_{ai}\deq\expct{F'(Y)(\caG\wt{A}^{i-1}\caG)_{aa}},& 
	& x_{ai}\deq\expct{F'(Y)(G\wt{B}^{i-1}G)_{aa}},	&&i\in\{1,2,3\}.
\end{aligned}\eeq
The components of $\bfx_{a}$ and $\bsx_{a}$ generalize $\E(G^{2})_{aa}$ in \eqref{eq:dGUE_0}. Likewise, the sub-leading order terms corresponding to \eqref{eq:dGUE_decoup_2} involve factors of $\tr \wt{B}^{k}G$ in place of $\tr G$. We denote related quantities as follows; define random diagonal matrices $D,\frD$ and random numbers $d_{k},\frd_{k}$ for $k=1,2,3$ by
\beq\begin{aligned}
	&	D\deq\diag\left(\gamma G_{aa}-\frac{1}{\fra_{a}-\omega_{\alpha}}\right)_{a\in\llbra 1,N\rrbra}, &
	& 	\frD\deq\diag\left(\gamma \mathcal{G}_{bb}-\frac{1}{\frb_{a}-\omega_{\beta}}\right)_{a\in\llbra 1,N\rrbra},\\
	&	d_{k}\deq\tr(A^{k-1}D),&
	&	\frd_{k}\deq\tr(B^{k-1}\frD).	\label{eq:d_def}
\end{aligned}\eeq
For each $a\in\llbra{1, N}\rrbra$, we define $Z_{a},\caZ_{a},\wt{Z}_{a},\wt{\caZ}_{a}\in\C^{3\times 3}$ to be the following generalizations of the left-hand side of \eqref{eq:dGUE_decoup_2}; their $(k,\ell)$-th entries are given by
\beq\label{eq:def_Z}\begin{aligned}
	&Z_{akl}\deq\expct{F'(Y)\frd_{k}(G\wt{B}^{l-1}G)_{aa}},& 
	&\caZ_{akl}\deq\expct{F'(Y)d_{k}(\caG\wt{A}^{l-1}\caG)_{aa}},\\
	&\wt{Z}_{akl}\deq\expct{F'(Y) \frd_{k} (\caG\wt{A}^{l-1}\caG)_{aa}},&
	&\wt{\caZ}_{akl}\deq\expct{F'(Y)d_{k}(G\wt{B}^{l-1}G)_{aa}}.
\end{aligned}\eeq
For vectors in \eqref{eq:def_x} and matrices in \eqref{eq:def_Z}, the same notations without the subscript $a$ stand for sums over $a$; for example $\bfx:=\sum_{a}\bfx_{a}$ and $Z\deq \sum_{a}Z_{a}$.

We next present rough estimates for above quantities due to the local law, Proposition \ref{prop:ll}. Firstly for $d_{k}$ and $\frd_{k}$, we often use that
\beq\label{eq:trace identity}
d_{k}=\gamma\tr A^{k-1}G-\int_{\R}\frac{x^{k-1}}{x-\omega_{\alpha}}\dd\mu_{\alpha}(x)+O(\bsd),
\eeq
which follows from Assumption \ref{assump:ABconv} and \eqref{eq:stab}. By \eqref{eq:subor_m}, the integral on the right-hand side of \eqref{eq:trace identity} can be simplified as
\beq\label{eq:trace}
\int_{\R}\frac{x^{k-1}}{x-\omega_{\alpha}}\dd\mu_{\alpha}(x)=\begin{cases}
	m_{\mu}, & k=1,\\
	\omega_{\alpha}m_{\mu}+1, & k=2,\\
	\omega_{\alpha}(\omega_{\alpha}m_{\mu}+1), & k=3,
\end{cases}
\eeq
where for $k=3$ we also used that $\mu_{\alpha}$ is of mean zero (recall \eqref{eq:mean_zero}).

Secondly for those in \eqref{eq:def_x}, we use Cauchy-Schwarz inequality and Ward identity to write
\beq\label{eq:x_rough_1}\begin{aligned}
	(G\wt{B}^{i-1}G)_{aa}\leq\norm{G\adj \bse_{a}}^{2}+\norm{\wt{B}^{i-1}G\bse_{a}}^{2}	
	\leq (1+\norm{\wt{B}}^{i-1})\frac{\im G_{aa}}{\eta_{0}}\prec N^{1/3+2\epsilon},
\end{aligned}\eeq
where the third inequality follows from \eqref{eq:ll_etr}. The same inequalities apply to $(\caG\wt{A}^{i-1}\caG)_{aa}$, and combining with \eqref{eq:FY} yields
\beq\label{eq:x_rough}
\rmx_{ai},x_{ai}=O(N^{1/3+(2C+3)\epsilon}).
\eeq
Thirdly for \eqref{eq:d_def}, Proposition \ref{prop:ll} implies that $d_{k}$, $\frd_{k}$, and all entries of $D,\frD$ are $O_{\prec}(N^{-1/3+\epsilon})$. Lastly, combining \eqref{eq:x_rough_1} and $d_{k},\frd_{k}=O_{\prec}(N^{-1/3+2\epsilon})$ we get
\beq\label{eq:Z_rough}
Z_{akl},\caZ_{akl},\wt{Z}_{akl},\wt{\caZ}_{akl}=O(N^{(2C+4\epsilon)}).
\eeq

We conclude this section by defining deterministic real vectors
\beq\label{eq:uv def}
\begin{split}
	&\bfu_{\alpha}\equiv(u_{\alpha i})_{i=1,2,3}:=\begin{pmatrix}
		\int\frac{1}{(x-\omega_{\alpha})^{2}}\dd \mu_{\alpha}(x), &
		\int\frac{x}{(x-\omega_{\alpha})^{2}}\dd \mu_{\alpha}(x), &
		\int\frac{x^{2}}{(x-\omega_{\alpha})^{2}}\dd\mu_{\alpha}(x)
	\end{pmatrix}^{\intercal},\\
	&\bfv_{\alpha}\equiv(v_{\alpha i})_{i=1,2,3}:=\begin{pmatrix}
		(\omega_{\alpha}+m_{\mu}^{-1})^{2}, &
		-2(\omega_{\alpha}+m_{\mu}^{-1}), &
		1
	\end{pmatrix}^{\intercal},
\end{split}
\eeq
and $\bfu_{\beta}, \bfv_{\beta}$ are defined by the same equation with roles of $\mu_{\alpha}$ and $\mu_{\beta}$ interchanged. From Lemma \ref{lem:stab}, all components of $\bfu_{\alpha},\bfu_{\beta},\bfv_{\alpha},\bfv_{\beta}$ are bounded. Furthermore, the first components $u_{\alpha 1}$ and $u_{\beta 1}$ of $\bfu_{\alpha}$ and $\bfu_{\beta}$ are positive and bounded from below.

\subsection{Proof of Proposition \ref{prop:main}}\label{sec:frX_conc}
In this section, we prove Proposition \ref{prop:main}. The proof consists of three steps; 
\begin{itemize}
	\item[(i)] proving an analogue of \eqref{eq:dGUE_decouped} for the free sum;
	\item[(ii)] deducing $\caK\im\frX=O(N^{5/6+C\epsilon})$ for some deterministic, real, $N$-dependent factor $\caK$;
	\item[(iii)] proving $\absv{\caK}$ is bounded from below.
\end{itemize} 

The following proposition handles the first step, whose proof is postponed to Section \ref{sec:decouple}.
\begin{prop}\label{prop:decoup} Under the settings of Proposition \ref{prop:gfc}, there exists a constant $C>0$ such that the following holds true uniformly over $z\in\{E+L_{+}+\ii\eta_{0}:E\in[E_{1},E_{2}]\}$.
	\beq\label{eq:xa1}
	\begin{split}
		x_{a1}&=
		\frac{1}{N(\fra_{a}-\omega_{\alpha})^{2}}\bfv_{\beta}\tp\bsx
		-\frac{(\bfv_{\beta}\tp\bfu_{\beta})^{2}}{u_{\beta 1}^{2}}\left(-\frac{1}{N(\fra_{a}-\omega_{\alpha})^{3}}+\frac{m_{\mu}}{N(\fra_{a}-\omega_{\alpha})^{2}}\right)\frX\\
		&+\frac{t}{N}\frac{1}{(\fra_{a}-\omega_{\alpha})^{2}}\rmx_{1}
		+\frac{\E[F'(Y)]}{\gamma^{2}(\fra_{a}-\omega_{\alpha})^{2}}
		+O(N^{-1/6+C\epsilon}).
	\end{split}\eeq
\end{prop}
Notice that \eqref{eq:xa1} decouples the index $a$ from $x_{a1}$, in the sense described below \eqref{eq:dGUE_decouped}. 

We next move on to the second step, that is, finding a cancellation from \eqref{eq:xa1}. We take the sum of \eqref{eq:xa1} over $a$ with weights $\fra_{a}^{k-1}$ for $k=1,2,3$, so that
\beq\label{eq:bfx}\begin{aligned}
	\bfx&=\bfu_{\alpha}\bfv_{\beta}\tp\bsx
	+t\bfu_{\alpha}x_{1}	\\
	&-\frac{(\bfv_{\beta}\tp\bfu_{\beta})^{2}}{u_{\beta 1}^{2}}\left( -\bfw_{\alpha}+m_{\mu}\bfu_{\alpha} \right)\frX
	+\gamma^{-2}N\bfu_{\alpha}\E[F'(Y)]
	+O(N^{5/6+C\epsilon}),
\end{aligned}\eeq
where we used $\rmx_{1}=x_{1}$ and
\beqs
\rmx_{k}=\E[F'(Y)\Tr \caG\wt{A}^{k-1}\caG]=\E[F'(Y)\Tr \wt{A}^{k-1}\caG^{2}]=\E[F'(Y)\Tr A^{k-1}G^{2}],
\eeqs
and defined $\bfw_{\alpha}\equiv(w_{\alpha 1},w_{\alpha 2},w_{\alpha 3})$ by
\beqs
w_{\alpha k}\deq \int_{\R}\frac{x^{k-1}}{(x-\omega_{\alpha})^{3}}\dd\mu_{\alpha}(x).
\eeqs
By symmetry, we may interchange the roles of $(A,B)$, $(\alpha,\beta)$, $(U,U\adj)$, and $(W,\caW)$ in \eqref{eq:bfx} to get
\beq\label{eq:bsx}\begin{aligned}
	\bsx=&\bfu_{\beta}\bfv_{\alpha}\tp\bfx
	+t\bfu_{\beta}\rmx_{1}	\\
	&-\frac{(\bfv_{\alpha}\tp\bfu_{\alpha})^{2}}{u_{\alpha 1}^{2}}\left( -\bfw_{\beta}+m_{\mu}\bfu_{\beta} \right)\frX
	+\gamma^{-2}N\bfu_{\beta}\E[F'(Y)]
	+O(N^{5/6+C\epsilon}),
\end{aligned}\eeq
where $\bfw_{\beta}$ is defined in a similar way. Here we used that $Y$ and $\frX$ are invariant under the interchange; this fact follows from 
\beqs\begin{aligned}
	&\Tr G=\Tr \caG, \quad&& \Tr \im[\wt{G}^{2}]G^{2}=\Tr \im[\wt{\caG}^{2}]\caG^{2},&& \wt{\caG}\equiv\caG(\wt{E}+L_{+}+\ii\eta_{0}).
\end{aligned}\eeqs
Combining \eqref{eq:bfx} and \eqref{eq:bsx}, we obtain 
\beq\label{eq:x_eq}
\begin{split}
	&\left(1-\bfv_{\alpha}\tp \bfu_{\alpha}\bfv_{\beta}\tp \bfu_{\beta}\right)\bfv_{\alpha}\tp\bfx
	=-\bfv_{\alpha}\tp \bfu_{\alpha}\frac{(\bfv_{\alpha}\tp \bfu_{\alpha})^{2}}{u_{\alpha 1}^{2}}\bfv_{\beta}\tp(-\bfw_{\beta}+m_{\mu}\bfu_{\beta}) \frX
	-\frac{(\bfv_{\beta}\tp \bfu_{\beta})^{2}}{u_{\beta 1}^{2}}\bfv_{\alpha}\tp (-\bfw_{\alpha}+m_{\mu}\bfu_{\alpha})\frX\\
	&+t\bfv_{\alpha}\tp\bfu_{\alpha}\bfv_{\beta}\tp\bfu_{\beta}\rmx_{1}
	+t\bfv_{\alpha}\tp\bfu_{\alpha}x_{1}
	+\gamma^{-2}N\bfv_{\alpha}\tp \bfu_{\alpha}(\bfv_{\beta}\tp\bfu_{\beta}+1)\E[F'(Y)]
	+O(N^{5/6+C\epsilon}).
\end{split}
\eeq

We next show that all terms in \eqref{eq:x_eq} involving $\bsx,\bfx$ cancel out. Collecting only the leading orders from \eqref{eq:bfx} and \eqref{eq:bsx}, we have
\begin{align*}
	&\bfx=\bfu_{\alpha}\bfv_{\beta}\tp\bsx
	+O(N^{1+C\epsilon+\chi}), &
	&\bsx=\bfu_{\beta}\bfv_{\alpha}\tp\bfx
	+O(N^{1+C\epsilon+\chi}),
\end{align*}
where we used \eqref{eq:frX_rough}, \eqref{eq:x_rough}, and \eqref{eq:Z_rough}. In particular, taking the first coordinates of both equalities, we obtain
\begin{align}\label{eq:x_linear}
	&\bfv_{\alpha}\tp\bfu_{\alpha}\rmx_{1}=u_{\alpha 1}\bfv_{\alpha}\tp\bfx +O(N^{1+C\epsilon+\chi}),&
	&x_{1}=u_{\beta1}\bfv_{\alpha}\tp\bfx+O(N^{1+C\epsilon+\chi}).
\end{align}
This in turn implies
\beq\begin{aligned}\label{eq:tx=tux}
	&\left(1-\bfv_{\alpha}\tp\bfu_{\alpha}\bfv_{\beta}\tp\bfu_{\beta}\right) \bfv_{\alpha}\tp\bfx
	-t\bfv_{\alpha}\tp\bfu_{\alpha}\bfv_{\beta}\tp\bfu_{\beta}x_{1}-t\bfv_{\alpha}\tp\bfu_{\alpha}x_{1}\\
	=&\left(1-\bfv_{\alpha}\tp \bfu_{\alpha}\bfv_{\beta}\tp \bfu_{\beta}-tu_{\alpha1}\bfv_{\beta}\tp\bfu_{\beta}-tu_{\beta1}\bfv_{\alpha}\tp\bfu_{\alpha}\right)\bfv_{\alpha}\tp\bfx+O(N^{2/3+C\epsilon+2\chi}),
\end{aligned}\eeq
where we applied the first and second equalities of \eqref{eq:x_linear} respectively to the second and third terms of \eqref{eq:tx=tux}.
Then we may rewrite \eqref{eq:x_eq} as
\beq\label{eq:x_eq_1}\begin{aligned}
	&\left(1-\bfv_{\alpha}\tp \bfu_{\alpha}\bfv_{\beta}\tp \bfu_{\beta}-tu_{\alpha1}\bfv_{\beta}\tp\bfu_{\beta}-tu_{\beta1}\bfv_{\alpha}\tp\bfu_{\alpha}\right)\bfv_{\alpha}\tp\bfx	\\
	=&-\bfv_{\alpha}\tp \bfu_{\alpha}\frac{(\bfv_{\alpha}\tp \bfu_{\alpha})^{2}}{u_{\alpha 1}^{2}}\bfv_{\beta}\tp(-\bfw_{\beta}+m_{\mu}\bfu_{\beta}) \frX
	-\frac{(\bfv_{\beta}\tp \bfu_{\beta})^{2}}{u_{\beta 1}^{2}}\bfv_{\alpha}\tp (-\bfw_{\alpha}+m_{\mu}\bfu_{\alpha})\frX\\
	&+\gamma^{-2}N\bfv_{\alpha}\tp \bfu_{\alpha}\bfv_{\beta}\tp\bfu_{\beta}
	+\gamma^{-2}N\bfv_{\alpha}\tp \bfu_{\alpha}
	+O(N^{5/6+C\epsilon}).
\end{aligned}\eeq
On the other hand, note that
\beq\label{eq:vu=F'}
(\bfv_{\alpha}+t\bse_{1})\tp\bfu_{\alpha}=\int_{\R}\frac{(\omega_{\alpha}+m_{\mu}^{-1}-x)^{2}+t}{(x-\omega_{\alpha})^{2}}\dd\mu_{\alpha}(x)=F_{\alpha,t}'(\omega_{\alpha})-1.
\eeq
The same identity holds true with $\alpha$ replaced by $\beta$. Therefore \eqref{eq:S_at_E} implies
\beq\label{eq:edge}
1-\bfv_{\alpha}\tp \bfu_{\alpha}\bfv_{\beta}\tp \bfu_{\beta}-tu_{\alpha1}\bfv_{\beta}\tp\bfu_{\beta}-tu_{\beta1}\bfv_{\alpha}\tp\bfu_{\alpha}=O(t^{2})=O(N^{-2/3+2\chi}).
\eeq
Plugging \eqref{eq:edge} into \eqref{eq:x_eq_1} and taking the imaginary part proves
\beq\label{eq:frX}
O(N^{5/6+C\epsilon})
=\left(\bfv_{\alpha}\tp \bfu_{\alpha}\frac{(\bfv_{\alpha}\tp \bfu_{\alpha})^{2}}{u_{\alpha 1}^{2}}\bfv_{\beta}\tp(-\bfw_{\beta}+m_{\mu}\bfu_{\beta}) 
+\frac{(\bfv_{\beta}\tp \bfu_{\beta})^{2}}{u_{\beta 1}^{2}}\bfv_{\alpha}\tp (-\bfw_{\alpha}+m_{\mu}\bfu_{\alpha})\right)\im\frX,
\eeq
where we used that $\E[F'(Y)]$ is real. This completes the second step.

Finally, to conclude the third step, it only remains to show that the deterministic factor in \eqref{eq:frX} is bounded from below. Recall from \eqref{eq:edge} that
\beq\label{eq:edge_rough}
(\bfv_{\alpha}\tp\bfu_{\alpha})(\bfv_{\beta}\tp\bfu_{\beta})=1+O(t),\qquad \bfv_{\alpha}\tp\bfu_{\alpha}\sim 1\sim\bfv_{\beta}\tp\bfu_{\beta},
\eeq
where the second asymptotics is due to the fact that
\beqs
\bfv_{\alpha}\tp\bfu_{\alpha}=\int_{\R}\frac{(\omega_{\alpha}+m_{\mu}^{-1}-x)^{2}}{(x-\omega_{\alpha})^{2}}\dd\mu_{\alpha}
\eeqs
is strictly positive and bounded. Recalling also that $u_{\alpha1},u_{\beta 1}\geq c$, it suffices to prove for a constant $c>0$ that
\beq\label{eq:K_lb}\begin{aligned}
	\bfv_{\alpha}\tp\bfw_{\alpha}-m_{\mu}\bfv_{\alpha}\tp\bfu_{\alpha}\leq -c,\\
	\bfv_{\beta}\tp\bfw_{\beta}-m_{\mu}\bfv_{\beta}\bfu_{\beta}\leq -c.
\end{aligned}\eeq
We remark that \eqref{eq:K_lb} is not a triviality in contrast to \eqref{eq:dGUE_optical}; since $\omega_{\alpha}=\omega_{\alpha,t}(E_{+,t})>E_{\alpha}^{+}$ from Lemma \ref{lem:stab} (iii), for the first term in \eqref{eq:K_lb} we have
\beqs
\bfv_{\alpha}\tp\bfw_{\alpha}=\int_{\R}\frac{(\omega_{\alpha}+m_{\mu}^{-1}+x)^{2}}{(x-\omega_{\alpha})^{3}}\leq0,
\eeqs
whereas the second term $-m_{\mu}\bfv_{\alpha}\tp\bfu_{\alpha}$ is positive since
\beqs
m_{\mu}=m_{\alpha}(\omega_{\alpha})=\int_{\R}\frac{1}{x-\omega_{\alpha}}\dd\mu_{\alpha}(x)\leq 0,\qquad \bfv_{\alpha}\tp\bfu_{\alpha}\geq 0.
\eeqs
Note that
\beq\label{eq:vw_expa}\begin{aligned}
	&\bfv_{\alpha}\tp\bsw_{\alpha}-m_{\mu}\bfv_{\alpha}\tp\bfu_{\alpha}=\int_{\R}\left(\frac{(\omega_{\alpha}+m_{\mu}^{-1}-x)^{2}}{(x-\omega_{\alpha})^{3}}-m_{\mu}\frac{(\omega_{\alpha}+m_{\mu}^{-1}-x)^{2}}{(x-\omega_{\alpha})^{2}}\right)\dd\mu_{\alpha}(x)\\
	=&\frac{1}{2}m_{\mu}^{-2}m_{\alpha}''(\omega_{\alpha})-m_{\mu}^{-1}m_{\alpha}'(\omega_{\alpha}) =\frac{1}{2}F_{\alpha,0}''(\omega_{\alpha})+m_{\mu}^{-1}m_{\alpha}'(\omega_{\alpha})(m_{\mu}^{-2}m_{\alpha}'(\omega_{\alpha})-1)	\\
	=&\frac{1}{2}F_{\alpha,0}''(\omega_{\alpha})+m_{\mu}F_{\alpha,0}'(\omega_{\alpha})(F_{\alpha,0}'(\omega_{\alpha})-1),
\end{aligned}\eeq
where the third equality is due to
\beqs
F_{\alpha,0}''(z)=\frac{\dd}{\dd z}\frac{m_{\alpha}'(z)}{m_{\alpha}(z)^{2}}=\frac{m_{\alpha}''(z)}{m_{\alpha}(z)^{2}}-2\frac{m_{\alpha}'(z)^{2}}{m_{\alpha}(z)^{3}}.
\eeqs
Since $z\mapsto F_{\alpha,0}(z)-z$ is a Pick function vanishing at infinity (see Lemma \ref{lem:Pick_F} for details), it is the Stieltjes transform of a measure $\wh{\mu}_{\alpha}$ on $\R$; 
\beq\label{eq:Pick_F}
F_{\alpha,0}(z)-z=\int_{\R}\frac{1}{x-z}\dd\wh{\mu}_{\alpha,0}(x).
\eeq
In Lemma \ref{lem:Pick_F} we further prove that $\wh{\mu}_{\alpha,0}$ is indeed a non-trivial, finite measure supported in $\supp\mu_{\alpha}$;
\beqs
\supp\wh{\mu}_{\alpha,0}\subset[E_{\alpha}^{-},E_{\alpha}^{+}],\qquad \wh{\mu}_{\alpha,0}(\R)=\int_{\R}x^{2}\dd\mu_{\alpha}(x)-\left(\int_{\R}x\dd\mu_{\alpha}(x)\right)^{2}.
\eeqs
In particular $\wh{\mu}_{\alpha,0}=0$ if and only if $\mu_{\alpha}$ is a point mass. Taking the first derivative of \eqref{eq:Pick_F}, for a constant $c>0$ we have
\beq
F_{\alpha}'(\omega_{\alpha})-1=\int_{\R}\frac{1}{(x-\omega_{\alpha})^{2}}\dd\wh{\mu}_{\alpha,0}(x)\geq\frac{ \wh{\mu}_{\alpha,0}(\R)}{(\omega_{\alpha,t}-E_{\alpha}^{-})^{2}}\geq c.
\eeq
Similarly we have $m_{\mu}=m_{\alpha}(\omega_{\alpha})<-c$ and $F_{\alpha}''(\omega_{\alpha})<-c$, so that by \eqref{eq:vw_expa} we obtain
\beq
\bfv_{\alpha}\tp\bsw_{\alpha}-m_{\mu}\bfv_{\alpha}\tp\bfu_{\alpha}<-c.
\eeq
This completes the proof of the first inequality \eqref{eq:K_lb}, and the second follows analogously. Plugging \eqref{eq:K_lb} back to \eqref{eq:frX} proves
\beq
\im\frX=O(N^{5/6+C\epsilon}),
\eeq 
concluding the proof of Proposition \ref{prop:main}.

%

\section{Proof of Proposition \ref{prop:decoup}}\label{sec:decouple}

\subsection{Proof of Proposition \ref{prop:decoup}}
As mentioned above \eqref{eq:xa1} is an analogue of \eqref{eq:dGUE_decouped}, and their proofs  roughly follow the same outline. Namely, we prove expansions resembling \eqref{eq:dGUE_expa_1}, \eqref{eq:dGUE_decoup_1}, and \eqref{eq:dGUE_decoup_2}, and simply combine them to conclude \eqref{eq:xa1}. In this section, we first collect the results of such expansions in Lemmas \ref{lem:BG2_conc} -- \ref{lem:Yder}, and then use them to conclude Proposition \ref{prop:decoup}. Proofs of Lemmas \ref{lem:BG2_conc} -- \ref{lem:Yder} are postponed to later sections.

Recall that in \eqref{eq:dGUE_expa_1} we expanded $\E(\wt{B}G^{2})_{aa}$ using Stein's lemma. Analogously, the proof of \eqref{eq:xa1} naturally involves expanding the same quantity. We present the resulting expansion in the following lemma, whose proof is postponed to the next subsection.
\begin{lem}\label{lem:BG2_conc}
	Under the conditions of Proposition \ref{prop:decoup}, we have
	\begin{align}\label{eq:BG2_1}
		&\E[F'(Y)(\wt{B}G^{2})_{aa}]
		=(m_{\mu}^{-1}+\omega_{\beta})x_{a1}	\\
		&+\gamma\Expct{F'(Y)\left(-\Tr\wt{B}^{2}G^{2}+(\omega_{\beta}+m_{\mu}^{-1})\Tr \wt{B}G^{2}\right)\frac{G_{aa}}{N}}	\label{eq:BG2_2}\\
		&+\gamma\Expct{F'(Y)\left(\Tr \wt{B}G^{2}-(\omega_{\beta}+m_{\mu}^{-1})\Tr G^{2}\right)\frac{(\wt{B}G)_{aa}}{N}}		\label{eq:BG2_3}\\
		&-\gamma\Expct{F'(Y)\Tr \left((\wt{B}-\frb_{a})(\wt{B}-\omega_{\beta}-m_{\mu}^{-1})G^{2}\right)\frac{\bsh_{a}\adj G\bse_{a}}{N}}	\label{eq:BG2_4}\\
		&+\expct{F'(Y)\left(-\frd_{3}+(\omega_{\beta}+m_{\mu}^{-1})\frd_{2}\right)(G^{2})_{aa}}	\label{eq:BG2_5}\\
		&+\expct{F'(Y)\left(-(\omega_{\beta}+m_{\mu}^{-1})\frd_{1}+\frd_{2}\right)(\wt{B}G^{2})_{aa}}	\label{eq:BG2_6}\\
		&+\expct{F'(Y)\left((\frb_{a}\frd_{2}-\frd_{3})-(m_{\mu}^{-1}+\omega_{\beta})(\frb_{a}\frd_{1}-\frd_{2})\right)\bsh_{a}\adj G^{2}\bse_{a}}	\label{eq:BG2_7}\\
		&+\frac{1}{N}\sum_{c}^{(a)}\Expct{F''(Y)\frac{\partial Y}{\partial g_{ac}}\left(-(\wt{B}G^{2})_{ca}+(\wh{\omega}_{\beta}+m_{\mu}^{-1})(G^{2})_{ca}\right)}	\label{eq:BG2_8}\\
		&+O(N^{-1/6+C\epsilon}),\nonumber
	\end{align}
	where $g_{ac}$ is the $c$-th component of the vector $\bsg_{a}$ defined in \eqref{eq:def_g}
\end{lem}

Note that each term in \eqref{eq:BG2_2} -- \eqref{eq:BG2_8} has either one of the following forms:
\begin{align}\label{eq:DaG2}
	\gamma\Expct{F'(Y) \Tr \wt{B}^{k}G^{2} \frac{(KG)_{aa}}{N}},\\
	\Expct{F'(Y)\frd_{k}(KG^{2})_{aa}},\label{eq:DaG3}
\end{align}
where $k\in\{0,1,2\}$ and $K$ can be $I$, $\wt{B}$, or $R_{i}$, due to \eqref{eq:Re=-h}. 
Precisely, those in \eqref{eq:BG2_2} -- \eqref{eq:BG2_4} corresponds to \eqref{eq:DaG2}, and \eqref{eq:BG2_5} -- \eqref{eq:BG2_7} to \eqref{eq:DaG3}.
One can immediately see the resemblance respectively between  \eqref{eq:DaG2}, \eqref{eq:DaG3} and the left-hand sides of \eqref{eq:dGUE_decoup_1}, \eqref{eq:dGUE_decoup_2}. Indeed, we decouple the index $a$ from \eqref{eq:DaG2} and \eqref{eq:DaG3} in the next two lemmas, which are analogues of \eqref{eq:dGUE_decoup_1} and \eqref{eq:dGUE_decoup_2}, respectively. We postpone their proofs to Sections~\ref{sec:GBG_pf}~and~\ref{sec:dBGG_pf}.
\begin{lem}\label{lem:GBG}
	Under the conditions of Proposition \ref{prop:decoup}, the following holds true uniformly over $a\in\braN$ and $k=1,2,3$:
	\begin{align}
		\gamma\expct{F'(Y)&G_{aa}\Tr \wt{B}^{k-1} G^{2}}=\frac{1}{\fra_{a}-\omega_{\alpha}}x_{k}
		+\frac{1}{(\fra_{a}-\omega_{\alpha})^{2}}\bfv_{\beta}^{\intercal}Z\bse_{k}	
		+O(N^{5/6+C\epsilon}),\label{eq:GBGdG_conc}\\
		\gamma\expct{F'(Y)&(\wt{B}G)_{aa}\Tr \wt{B}^{k-1} G^{2}}=\frac{(\omega_{\beta}+m_{\mu}^{-1})}{\fra_{a}-\omega_{\alpha}}x_{k}	\nonumber\\
		&+\left(\frac{(\omega_{\beta}+m_{\mu}^{-1})}{(\fra_{a}-\omega_{\alpha})^{2}}
		-\frac{1}{\fra_{a}-\omega_{\alpha}}\right)\bfv_{\beta}^{\intercal}Z\bse_{k}
		+O(N^{5/6+C\epsilon}),	\label{eq:GBGdBG_conc}\\
		\gamma\expct{F'(Y)&\bsh_{a}\adj G\bse_{a}\Tr \wt{B}^{k-1}G^{2}}	\nonumber\\
		&=\frac{\omega_{\beta}(\omega_{\beta}+m_{\mu}^{-1})Z_{1k}-(2\omega_{\beta}+m_{\mu}^{-1})Z_{2k}+Z_{3k}}{\gamma(\fra_{a}-\omega_{\alpha})(\frb_{a}-\omega_{\beta})}
		+O(N^{5/6+C\epsilon}).	\label{eq:GBGhG_conc}
	\end{align}
\end{lem}
\begin{lem}\label{lem:dBGG}
	Under the conditions of Proposition \ref{prop:decoup}, the following holds true uniformly over $a\in\braN$ and $k=1,2,3$:
	\begin{align}
		\E[F'(Y) \frd_{k}&(G^{2})_{aa}]	
		=\frac{1}{N(\fra_{a}-\omega_{\alpha})^2}\bse_{k}\tp Z\bfv_{\beta}+O(N^{-1/6+C\epsilon}),	\label{eq:dGG_conc}\\
		\E[F'(Y)\frd_{k}&(\wt{B}G^{2})_{aa}]
		=\left(\frac{(\omega_{\beta}+m_{\mu}^{-1})}{N(\fra_{a}-\omega_{\alpha})^{2}}-\frac{1}{N(\fra_{a}-\omega_{\alpha})}\right)\bse_{k}^{\intercal}Z\bfv_{\beta}+O(N^{-1/6+C\epsilon}),\label{eq:dBGG_conc}\\
		\E[F'(Y)\frd_{k}&\bsh_{a}\adj G^{2}\bse_{a}] \nonumber \\
		&=\frac{\omega_{\beta}(\omega_{\beta}+m_{\mu}^{-1})Z_{k1}-(2\omega_{\beta}+m_{\mu}^{-1})Z_{k2}+Z_{k3}}{N(\fra_{a}-\omega_{\alpha})(\frb_{a}-\omega_{\beta})}
		+O(N^{-1/6+C\epsilon}).\label{eq:dhGG_conc}
	\end{align}
\end{lem}
While Lemmas \ref{lem:GBG} and \ref{lem:dBGG} do decouple $a$ from \eqref{eq:DaG2} and \eqref{eq:DaG3}, there is a discrepancy compared to the deformed GUE case. Namely, the estimates in Lemmas \ref{lem:GBG} and \ref{lem:dBGG} involve $Z_{ij}$ for all choices of $i$ and $j$, whereas \eqref{eq:dGUE_decoup_1} and \eqref{eq:dGUE_decoup_2} only have $\frX=2Z_{11}$. In the next lemma, we show that $Z_{ij}$ are in fact all comparable to each another. We prove the lemma in Section \ref{sec:veq_pf}.
\begin{lem}\label{lem:veq1}
	Under the conditions in Proposition \ref{prop:decoup}, the following hold true uniformly over $E\in[E_{1},E_{2}]$.
	\beq\label{eq:Z_veq}
	Z=\frac{\bfu_{\beta}\bfu_{\beta}\tp}{u_{\beta 1}^{2}}Z_{11}+O(N^{5/6+C\epsilon}).
	\eeq
\end{lem}

Applying Lemma~\ref{lem:veq1} to $Z$-dependent factors in Lemmas \ref{lem:GBG} and \ref{lem:dBGG} yields that
\beq\label{eq:veq_1}
\bfv_{\beta}\tp Z\bse_{k}=\bfv_{\beta}\tp \bfu_{\beta} \frac{u_{\beta k}}{u_{\beta 1}^{2}}Z_{11}+O(N^{5/6+C\epsilon})
=\bse_{k}\tp Z\bfv_{\beta}.
\eeq
In fact, Lemma \ref{lem:veq1} implies that the linear combinations of $Z$ appearing in \eqref{eq:GBGhG_conc} and \eqref{eq:dhGG_conc} are negligible:
\beq\begin{aligned}\label{eq:veq_2}
	&\omega_{\beta}(\omega_{\beta}+m_{\mu}^{-1})Z_{k1}-(2\omega_{\beta}+m_{\mu}^{-1})Z_{k2}+Z_{k3}	\\
	=&\frac{u_{\beta k}}{u_{\beta 1}^{2}}(\omega_{\beta}(\omega_{\beta}+m_{\mu}^{-1})u_{\beta1}-(2\omega_{\beta}+m_{\mu}^{-1})u_{\beta2}+u_{\beta3})Z_{11}+O(N^{5/6+C\epsilon})\\
	=&\omega_{\beta}(\omega_{\beta}+m_{\mu}^{-1})Z_{1k}-(2\omega_{\beta}+m_{\mu}^{-1})Z_{2k}+Z_{3k} +O(N^{5/6+C\epsilon}),
\end{aligned}\eeq
and the second line of \eqref{eq:veq_2} is $O(N^{5/6+C\epsilon})$ since
\beq\begin{aligned}\label{eq:veq_2_cancel}
	&\omega_{\beta}(\omega_{\beta}+m_{\mu}^{-1})u_{\beta1}-(2\omega_{\beta}+m_{\mu}^{-1})u_{\beta2}+u_{\beta3}	\\
	=&\int_{\R}\frac{(x-\omega_{\beta})(x-\omega_{\beta}-m_{\mu}^{-1})}{(x-\omega_{\beta})^{2}}\dd\mu_{\beta}(x)	
	=1-m_{\mu}^{-1}\int_{\R}\frac{1}{x-\omega_{\beta}}\dd\mu_{\beta}(x)=0,
\end{aligned}\eeq
where the last equality is due to \eqref{eq:subor_m}. Plugging in \eqref{eq:veq_2} and \eqref{eq:veq_2_cancel} into \eqref{eq:GBGhG_conc} and \eqref{eq:dhGG_conc}, we have
\beq\label{eq:h_cancel}
	\E[F'(Y)\bsh_{a}\adj G\bse_{a}\Tr \wt{B}^{k-1}G^{2}]=O(N^{5/6+C\epsilon}),\qquad \E[F'(Y)\frd_{k}\bsh_{a}\adj G^{2}\bse_{a}]=O(N^{-1/6+C\epsilon}).
\eeq
Note that \eqref{eq:veq_2_cancel} also implies
\beq\label{eq:veq_2_cancel_2}
\bfv_{\beta}\tp\bfu_{\beta}-m_{\mu}^{-1}\left((\omega_{\beta}+m_{\mu}^{-1})u_{\beta1}+u_{\beta 2}\right)=0.
\eeq

Finally, we deal with the last term of \eqref{eq:BG2_1} in the following lemma, whose proof is presented in Section \ref{sec:Yder_pf}.
\begin{lem}\label{lem:Yder}
	Under the conditions of Proposition \ref{prop:decoup}, for $k=1,2$, we have
	\begin{align}
		\sum_{c}^{(a)}\expct{F''(Y)\frac{\partial Y}{\partial g_{ac}}(\wt{B}^{k-1}G^{2})_{ca}}\nonumber
		=&\gamma\frac{\bfv_{\beta}\tp\bfu_{\beta}}{u_{\beta 1}^{2}}\left(-\frac{(\omega_{\beta}+m_{\mu}^{-1})u_{\beta k}-u_{\beta(k+1)}}{N(\fra_{a}-\omega_{\alpha})^{2}}+\frac{u_{\beta k}}{N(\fra_{a}-\omega_{\alpha})}\right) \\ 
		&\times \int_{E_{1}}^{E_{2}}\expct{F''(Y)\Tr G^{2}\im (\wt{G}^{2})}\dd \wt{E}+O(N^{5/6+C\epsilon}),\label{eq:Y_der_conc}
	\end{align}
	where we recall $\wt{G}\equiv G(L_{+}+\wt{E}+\ii\eta_{0})$.
\end{lem}

We have collected all necessary ingredients thus move on to the proof of Proposition \ref{prop:decoup}.
\begin{proof}[Proof of Proposition~\ref{prop:decoup}]
	We first use the identity $zG+I=\gamma HG$ to write
	\beq\label{eq:x_expa}\begin{aligned}
		&\frac{z}{\gamma}\expct{F'(Y)(G^{2})_{aa}}+\frac{1}{\gamma}\expct{F'(Y)G_{aa}}
		=\expct{F'(Y)\bse_{a}\adj(A+\wt{B}+\sqrt{t}W)G^{2}\bse_{a}}\\
		=&\fra_{a}\expct{F'(Y)(G^{2})_{aa}}+\expct{F'(Y)(\wt{B}G^{2})_{aa}}-\gamma t\expct{F'(Y)G_{aa}\tr G^{2}}	\\
		&-\gamma t\expct{F'(Y)(G^{2})_{aa}\tr G} +\sqrt{t}\frac{1}{N}\sum_{b}\Expct{F''(Y)(G^{2})_{ba}\frac{\partial Y}{\partial \ol{W}_{ab}}},
	\end{aligned}\eeq
	where we applied Stein's lemma to $W$ in the second equality. We next simplify the third to fifth terms on the right-hand side of \eqref{eq:x_expa}. For the last term, recall from \eqref{eq:Yder} that
	\beq\label{eq:Yder_1}
	\frac{\partial Y}{\partial\ol{W}_{ab}}=\frac{\partial Y}{\partial W_{ba}}=-\gamma\sqrt{t}\int_{E_{1}}^{E_{2}}(\im[\wt{G}^{2}])_{ab}\dd\wt{E},
	\eeq
	where $\wt{G}=G(L_{+}+\wt{E}+\ii\eta_{0})$. Thus, as in \eqref{eq:frX_rough}, we use Proposition \ref{prop:ll} to find that
	\beq\label{eq:x_expa_rem1}
	\Absv{\sqrt{t}\frac{1}{N}\sum_{b}\Expct{F''(Y)(G^{2})_{ba}\frac{\partial Y}{\partial \ol{W}_{ab}}}}=O(tN^{C\epsilon}).
	\eeq
	Similarly, for the third and fourth terms of \eqref{eq:x_expa} we use Proposition \ref{prop:ll} and \eqref{eq:x_rough_1} to obtain
	\beq\label{eq:x_expa_rem2}\begin{aligned}
		\gamma t\expct{F'(Y)G_{aa}\tr G^{2}}-t\frac{1}{a-\omega_{\alpha}}\expct{F'(Y)\tr G^{2}}
		&=t\expct{O_{\prec}(N^{C\epsilon})}=tO(N^{C\epsilon}),\\
		\gamma t\expct{F'(Y)(G^{2})_{aa}\tr G}-tm_{\mu}\expct{F'(Y)(G^{2})_{aa}}&=tO(N^{C\epsilon}).
	\end{aligned}\eeq
	Plugging in \eqref{eq:x_expa_rem1} and \eqref{eq:x_expa_rem2} to \eqref{eq:x_expa} and then using $t=N^{-1/3+\chi}\leq N^{-1/6}$, we have
	\beq\label{eq:x_expa_result}\begin{aligned}
		\left(\frac{z}{\gamma}-\fra_{a}+tm_{\mu}\right)x_{a1}	=\expct{F'(Y)(\wt{B}G^{2})_{aa}}-t\frac{1}{N(\fra_{a}-\omega_{\alpha})}\rmx_{1}+O(N^{-1/6+C\epsilon}).
	\end{aligned}\eeq
	
	We next claim that it suffices to prove the following:
	\begin{align}
		\E[F'(Y)(\wt{B}G^{2})_{aa}]
		&=(m_{\mu}^{-1}+\omega_{\beta})x_{a1}
		-\frac{1}{N(\fra_{a}-\omega_{\alpha})}\bfv_{\beta}\tp\bsx\nonumber\\
		&-\frac{(\bfv_{\beta}\tp\bfu_{\beta})^{2}}{u_{\beta 1}^{2}}\left(\frac{1}{N(\fra_{a}-\omega_{\alpha})^{2}}-\frac{m_{\mu}}{N(\fra_{a}-\omega_{\alpha})}\right)\frX+O(N^{-1/6+C\epsilon}).\label{eq:Bab2}
	\end{align}
	Indeed, given \eqref{eq:Bab2}, we may simply substitute the first term on the right-hand side of \eqref{eq:x_expa_result} by \eqref{eq:Bab2} to obtain
	\beq\begin{aligned}
		&\left(E_{+}-\fra_{a}+tm_{\mu}-\frac{1}{m_{\mu}}-\omega_{\beta}\right)x_{a1}	\\
		=&\left(\frac{z}{\gamma}-\fra_{a}+tm_{\mu}-\frac{1}{m_{\mu}}-\omega_{\beta}\right)x_{a1}+O(N^{-1/3+C\epsilon})	\label{eq:Bab2_conseq}\\
		=&-\frac{1}{N(\fra_{a}-\omega_{\alpha})}\bfv_{\beta}\tp\bsx-t\frac{1}{\fra_{a}-\omega_{\alpha}}\rmx_{1}\\
		&-\frac{(\bfv_{\beta}\tp\bfu_{\beta})^{2}}{u_{\beta1}^{2}}\left(\frac{1}{N(\fra_{a}-\omega_{\alpha})^{2}}-\frac{m_{\mu}}{N(\fra_{a}-\omega_{\alpha})}\right)\frX+O(N^{-1/6+C\epsilon}),
	\end{aligned}\eeq
	where we used $\absv{E_{+}-z/\gamma}=O(N^{-2/3+C\epsilon})$ and \eqref{eq:x_rough} in the first equality. Then we notice that the deterministic factor on the leftmost side of \eqref{eq:Bab2_conseq} can be simplified as
	\begin{align}\label{eq:B2_factor}
		E_{+}-\fra_{a}+tm_{\mu}-\frac{1}{m_{\mu}}-\omega_{\beta}
		=&F_{\mu_{t},t}(E_{+,t})-\omega_{\beta}(E_{+,t})+E_{+,t}-\fra_{a}	\nonumber\\
		=&\omega_{\alpha}(E_{+,t})-\fra_{a},
	\end{align}
	where the second equality is due to \eqref{eq:subor}. By \eqref{eq:stab} we may divide \eqref{eq:Bab2_conseq} by $(\omega_{\alpha}-\fra_{a})$, leading to exactly the desired conclusion:
	\beqs\begin{aligned}
		x_{a1}=&\frac{1}{N(\fra_{a}-\omega_{\alpha})^{2}}\bfv_{\beta}\tp\bsx
		+t\frac{1}{(\fra_{a}-\omega_{\alpha})^{2}}\rmx_{1}\\
		&+\frac{(\bfv_{\beta}\tp\bfu_{\beta})^{2}}{u_{\beta1}^{2}}\left(\frac{1}{N(\fra_{a}-\omega_{\alpha})^{3}}-\frac{m_{\mu}}{N(\fra_{a}-\omega_{\alpha})^{2}}\right)\frX+O(N^{-1/6+C\epsilon}).
	\end{aligned}\eeqs
	
	Now we prove \eqref{eq:Bab2} by substituting Lemmas \ref{lem:GBG} -- \ref{lem:Yder} into each of \eqref{eq:BG2_2} -- \eqref{eq:BG2_8}. We explain the detail line by line.
	\subsubsection*{Second line \eqref{eq:BG2_2}}
	We apply \eqref{eq:GBGdG_conc} with $k=2,3$ and then \eqref{eq:veq_1} to obtain
	\beq\begin{aligned}\label{eq:BG2_2_conc}
		\text{\eqref{eq:BG2_2}}=&\frac{1}{N(\fra_{a}-\omega_{\alpha})}\left((\omega_{\beta}+m_{\mu}^{-1})x_{2}-x_{3}\right)	\\+&\frac{\bfv_{\beta}\tp\bfu_{\beta}}{N(\fra_{a}-\omega_{\alpha})^{2}u_{\beta1}^{2}}\left((\omega_{\beta}+m_{\mu}^{-1})u_{\beta2}-u_{\beta 3}\right)Z_{11}+O(N^{-1/6+C\epsilon}).
	\end{aligned}\eeq
	Notice that \eqref{eq:BG2_2_conc} has an overall additional factor of $N^{-1}$ compared to \eqref{eq:GBGdG_conc} so that the error in \eqref{eq:BG2_2_conc} is $O(N^{-1}\cdot N^{5/6+C\epsilon})$. This scaling also applies to \eqref{eq:BG2_3} and \eqref{eq:BG2_4}.
	
	\subsubsection*{Third line \eqref{eq:BG2_3}}
	We apply \eqref{eq:GBGdBG_conc} with $k=1,2$ and \eqref{eq:veq_1} to get
	\beq\begin{aligned}\label{eq:BG2_3_conc}
		&\text{\eqref{eq:BG2_3}}
		=\frac{\omega_{\beta}+m_{\mu}^{-1}}{N(\fra_{a}-\omega_{\alpha})}\left(x_{2}-(\omega_{\beta}+m_{\mu}^{-1})x_{1}\right)	\\
		+&\frac{\bfv_{\beta}\tp\bfu_{\beta}}{Nu_{\beta1}^{2}}\left(\frac{\omega_{\beta}+m_{\mu}^{-1}}{(\fra_{a}-\omega_{\alpha})^{2}}-\frac{1}{\fra_{a}-\omega_{\alpha}}\right)(u_{\beta2}-(\omega_{\beta}+m_{\mu}^{-1})u_{\beta1})Z_{11}+O(N^{-1/6+C\epsilon}).
	\end{aligned}\eeq
	\subsubsection*{Fourth line \eqref{eq:BG2_4}}
	Applying \eqref{eq:GBGhG_conc} and \eqref{eq:h_cancel} we find that
	\beq\label{eq:BG2_4_conc}\begin{aligned}
		\text{\eqref{eq:BG2_4}}=O(N^{-1/6+C\epsilon}).
	\end{aligned}\eeq
	Here we briefly pause to simplify the combined contribution of \eqref{eq:BG2_2} -- \eqref{eq:BG2_4}. Summing \eqref{eq:BG2_2_conc}, \eqref{eq:BG2_3_conc}, and \eqref{eq:BG2_4_conc} gives
	\begin{align}
		\text{\eqref{eq:BG2_2}+\eqref{eq:BG2_3}+\eqref{eq:BG2_4}}=&-\frac{1}{N(\fra_{a}-\omega_{\alpha})}\left((\omega_{\beta}+m_{\mu}^{-1})^{2}x_{1}-2(\omega_{\beta}+m_{\mu}^{-1})x_{2}+x_{3}\right)	\nonumber\\
		&-\frac{\bfv_{\beta}\tp\bfu_{\beta}}{N(\fra_{a}-\omega_{\alpha})^{2}u_{\beta1}^{2}}\left((\omega_{\beta}+m_{\mu}^{-1})^{2}u_{\beta1}-2(\omega_{\beta}+m_{\mu}^{-1})u_{\beta2}+u_{\beta3}\right)Z_{11}	\nonumber\\
		&+\frac{\bfv_{\beta}\tp\bfu_{\beta}}{N(\fra_{a}-\omega_{\alpha})u_{\beta1}^{2}}\left((\omega_{\beta}+m_{\mu}^{-1})u_{\beta 1}-u_{\beta2}\right)Z_{11}+O(N^{-1/6+C\epsilon})	\label{eq:BG2_234_conc}\\
		=&-\frac{\bfv_{\beta}\tp\bsx}{N(\fra_{a}-\omega_{\alpha})}-\frac{(\bfv_{\beta}\tp\bfu_{\beta})^{2}}{Nu_{\beta1}^{2}}\left(\frac{1}{(\fra_{a}-\omega_{\alpha})^{2}}-\frac{m_{\mu}}{\fra_{a}-\omega_{\alpha}}\right)Z_{11}+O(N^{-1/6+C\epsilon})\nonumber,
	\end{align}
	where we used \eqref{eq:veq_2_cancel_2} in the last line.
	\subsubsection*{Fifth line \eqref{eq:BG2_5}}
	Here we apply \eqref{eq:dGG_conc} with $k=2,3$ and \eqref{eq:veq_1} to obtain
	\beq\label{eq:BG2_5_conc}
	\text{\eqref{eq:BG2_5}}
	=\frac{\bfv_{\beta}\tp\bfu_{\beta}}{N(\fra_{a}-\omega_{\alpha})^{2}}((\omega_{\beta}+m_{\mu}^{-1})u_{\beta2}-u_{\beta3})Z_{11}+O(N^{-1/6+C\epsilon}).\eeq
	\subsubsection*{Sixth line \eqref{eq:BG2_6}}
	For \eqref{eq:BG2_6}, we use \eqref{eq:dBGG_conc} with $k=1,2$ and \eqref{eq:veq_1} so that
	\beq\label{eq:BG2_6_conc}
	\text{\eqref{eq:BG2_6}}
	=\frac{\bfv_{\beta}\tp\bfu_{\beta}}{N}\left(\frac{\omega_{\beta}+m_{\mu}^{-1}}{(\fra_{a}-\omega_{\alpha})^{2}}-\frac{1}{\fra_{a}-\omega_{a}}\right)(u_{\beta2}-(\omega_{\beta}+m_{\mu}^{-1})u_{\beta1})
	\eeq
	\subsubsection*{Seventh line \eqref{eq:BG2_7}}
	Combining \eqref{eq:dhGG_conc} and \eqref{eq:h_cancel} we have
	\beq
	\text{\eqref{eq:BG2_7}}=O(N^{-1/6+C\epsilon}).
	\eeq
	Notice that the right-hand sides of \eqref{eq:BG2_5_conc} and \eqref{eq:BG2_6_conc} exactly match the second lines of \eqref{eq:BG2_2_conc} and \eqref{eq:BG2_3_conc}, respectively. Thus, by \eqref{eq:BG2_234_conc}, the combined contribution of \eqref{eq:BG2_5} -- \eqref{eq:BG2_7} equals to
	\beq\label{eq:BG2_567_conc}
	\text{\eqref{eq:BG2_5}+\eqref{eq:BG2_6}+\eqref{eq:BG2_7}}=-\frac{(\bfv_{\beta}\tp\bfu_{\beta})^{2}}{Nu_{\beta1}^{2}}\left(\frac{1}{(\fra_{a}-\omega_{\alpha})^{2}}-\frac{m_{\mu}}{\fra_{a}-\omega_{\alpha}}\right)Z_{11}+O(N^{-1/6+C\epsilon}).
	\eeq
	\subsubsection*{Eighth line \eqref{eq:BG2_8}}
	For \eqref{eq:BG2_8} we apply \eqref{eq:Y_der_conc} with $k=1,2$ and use the same algebraic manipulation as in \eqref{eq:BG2_234_conc} to obtain
	\begin{multline}\label{eq:BG2_8_conc}
		\text{\eqref{eq:BG2_8}}=-\frac{\gamma(\bfv_{\beta}\tp\bfu_{\beta})^{2}}{Nu_{\beta1}^{2}}\left(\frac{1}{(\fra_{a}-\omega_{\alpha})^{3}}-\frac{m_{\mu}}{(\fra_{a}-\omega_{\alpha})^{2}}\right)	\\
		\times\int_{E_{1}}^{E_{2}}\expct{F''(Y)\Tr G^{2}\im (\wt{G}^{2})}\dd \wt{E}+O(N^{-1/6+C\epsilon}).
	\end{multline}
	As in \eqref{eq:BG2_2_conc}, notice the additional $N^{-1}$ factors in \eqref{eq:BG2_8} and \eqref{eq:BG2_8_conc} compared to \eqref{eq:Y_der_conc}.
	
	Summing \eqref{eq:BG2_234_conc}, \eqref{eq:BG2_567_conc}, and \eqref{eq:BG2_8_conc} we get
	\beqs\begin{aligned}
		&\text{\eqref{eq:BG2_234_conc}+\eqref{eq:BG2_567_conc}+\eqref{eq:BG2_8_conc}}
		=-\frac{(\bfv_{\beta}\tp\bfu_{\beta})^{2}}{Nu_{\beta1}^{2}}\left(\frac{1}{(\fra_{a}-\omega_{\alpha})^{2}}-\frac{m_{\mu}}{(\fra_{a}-\omega_{\alpha})^{2}}\right)\frX +O(N^{-1/6+C\epsilon}),
	\end{aligned}\eeqs
	where we used definitions of $\frX$ and $Z_{11}$. Finally, \eqref{eq:Bab2} follows immediately as
	\beq\begin{aligned}\label{eq:Bab2_conc}
		&\E[F'(Y)(\wt{B}G^{2})_{aa}]=(\omega_{\beta}+m_{\mu}^{-1})x_{a1}+\text{\eqref{eq:BG2_234_conc}+\eqref{eq:BG2_567_conc}+\eqref{eq:BG2_8_conc}}+O(N^{-1/6+C\epsilon})	\\
		=&(\omega_{\beta}+m_{\mu}^{-1})x_{a1}
		-\frac{(\bfv_{\beta}\tp\bfu_{\beta})^{2}}{Nu_{\beta1}^{2}}\left(\frac{1}{(\fra_{a}-\omega_{\alpha})^{2}}-\frac{m_{\mu}}{(\fra_{a}-\omega_{\alpha})^{2}}\right)\frX+O(N^{-1/6+C\epsilon}).
	\end{aligned}\eeq
	This concludes the proof of Proposition \ref{prop:decoup}.
\end{proof}

\subsection{Proof of Lemma \ref{lem:BG2_conc}}\label{sec:B2_prf}
In this section we prove the first expansion, Lemma \ref{lem:BG2_conc}. As mentioned in the introduction, along the proof we often apply Stein's lemma with respect to the Gaussian vector $\bsg_{a}$ from the partial randomness decomposition. Hence derivatives of $G$ with respect to components of $\bsg_{a}$ naturally appear, whose precise form is given in the following lemma; see Appendix \ref{sec:der} for its proof.
\begin{lem}\label{lem:dgbd}
	For $a\neq d\in\llbra 1, N\rrbra$, we have
	\beq\begin{aligned}\label{eq:dgbd}
		\frac{\partial G}{\partial g_{ad}}=&-\gamma\frac{\ell_{a}^{2}}{\norm{\bsg_{a}}}G[\bse_{d}(\bse_{a}+\bsh_{a})\adj,\wt{B}]G
		+\gamma\frac{\ell_{a}^{2}}{2\norm{\bsg_{a}}}\ol{h}_{ad}G[(\bse_{a}+2\bsh_{a})\bse_{a}\adj,\wt{B}]G\\
		-&\gamma\frac{\ell_{a}^{4}}{2\norm{\bsg_{a}}}h_{aa}\ol{h}_{ad}G[\bse_{a}\bse_{a}\adj+\bse_{a}\bsh_{a}\adj+\bsh_{a}\bse_{a}\adj,\wt{B}]G,
	\end{aligned}\eeq
	where $[P,Q]\deq PQ-QP$ for matrices $P$ and $Q$ of the same size. Similarly, we have
	\beq\label{eq:dgbd_bar}\begin{aligned}
		\frac{\partial G}{\overline{\partial}g_{ad}}=&-\gamma \frac{\ell_{a}^2}{\norm{\bsg_{a}}} G [\wt{B},(\bse_{a}+\bsh_{a})\bse_{d}\adj]G
		+\gamma \frac{\ell_{a}^{2}}{2\norm{\bsg_{a}}}h_{ad}G[\wt{B},\bse_{a}(\bse_{a}+2\bsh_{a})\adj]G\\
		-&\gamma \frac{\ell_{a}^{4}}{2\norm{\bsg_{b}}}h_{aa}h_{ad}G[\wt{B},\bse_{a}\bse_{a}\adj+\bse_{a}\bsh_{a}\adj+\bsh_{a}\bse_{a}\adj]G.
	\end{aligned}\eeq 
\end{lem}

Note that \eqref{eq:dgbd} is an exact identity that remains true for all $N$. However, as $N\to\infty$ we have $h_{aa},\ol{h}_{ad}=O_{\prec}(N^{-1/2})$, so that the second and third terms of \eqref{eq:dgbd} are typically much smaller than the first. Similarly we have 
\beq
\norm{\bsg_{a}}^{2}=1+O_{\prec}(N^{-1/2}),\qquad \ell_{a}^{2}=\frac{2}{\norm{\bse_{a}+\bsh_{a}}^{2}}=\frac{1}{1+h_{aa}}=1+O_{\prec}(N^{-1/2}).
\eeq 
These two facts lead to the following heuristic asymptotics:
\beq\label{eq:dg_heu}
\frac{\partial G}{\partial g_{ac}}=-\gamma G[\bse_{c}(\bse_{a}+\bsh_{a})\adj,\wt{B}]G+\text{(remainders)}.
\eeq
While \eqref{eq:dg_heu} hardly a rigorous statement, it helps to see how the leading terms arise in our estimates; see Lemmas \ref{lem:dgbd_1}, \ref{lem:dgdiag}, and \ref{lem:dg_cr}, and also \eqref{eq:Yder_BG}, \eqref{eq:Yder_hG}. In practice, the `remainders' are treated differently depending on the precise form of the estimate.

The following asymptotic lemma is the first application of Lemma \ref{lem:dgbd}, which will serve as a major input for all proofs of Lemmas \ref{lem:BG2_conc} -- \ref{lem:Yder}. We prove the lemma in Section \ref{sec:der_tech}.
\begin{lem}\label{lem:dgbd_1}
	The following holds uniformly over $a,b\in\llbra{1,N\rrbra}$ and $E\in[E_{1},E_{2}]$:
	\beq\begin{aligned}
		&\frac{1}{N}\sum_{c}^{(a)}\frac{\partial}{\partial g_{ac}}\left[\frac{1}{\norm{\bsg_{a}}}\bse_{c}\adj \wt{B}^{\anga}R_{a}G\bse_{b}G_{ba}\right]	\\
		=& -\gamma \tr\wt{B}G (\bse_{a}+\bsh_{a})\adj \wt{B}G\bse_{b}G_{ba}+\gamma\tr\wt{B}^{2}G(\bse_{a}+\bsh_{a})\adj G\bse_{b}G_{ba}	\\
		&-\gamma\frac{(G\wt{B}G)_{bb}}{N}(\bse_{a}+\bsh_{a})\adj\wt{B}G\bse_{a}
		+\gamma\frac{(G\wt{B}^{2}G)_{bb}}{N}(\bse_{a}+\bsh_{a})\adj G\bse_{a}	\\
		&+O_{\prec}(N^{-1/2}(N^{-2/3+2\epsilon}+\delta_{ab})). \label{eq:dgBGG}	\\
	\end{aligned}\eeq
	Similarly, we have
	\beq\begin{aligned}
		&\frac{1}{N}\sum_{c}^{(a)}\frac{\partial}{\partial g_{ac}}\left[\frac{1}{\norm{\bsg_{a}}}G_{cb}G_{ba}\right]	\\
		=&-\gamma\tr G(\bse_{a}+\bsh_{a})\adj \wt{B}G\bse_{b}G_{ba}
		+\gamma\tr\wt{B}G(\bse_{a}+\bsh_{a})\adj G\bse_{b}G_{ba}	\\
		&-\gamma\frac{(G^{2})_{bb}}{N}(\bse_{a}+\bsh_{a})\adj \wt{B}G\bse_{a} 
		+\gamma\frac{(G\wt{B}G)_{bb}}{N}(\bse_{a}+\bsh_{a})\adj G\bse_{a}\\
		&+O_{\prec}(N^{-1/2}(N^{-2/3+2\epsilon}+\delta_{ab})).\label{eq:dgGG}
	\end{aligned}\eeq
\end{lem}

\begin{proof}[Proof of Lemma \ref{lem:BG2_conc}]
	The proof consists of three steps; 
	\begin{itemize}
		\item[(i)] Using \eqref{eq:dgBGG} to express $\E[F'(Y)(\wt{B}G^{2})_{aa}]$ as a linear combination of $\E[F'(Y)\bsh_{a}\adj G^{2}\bse_{a}]$ and terms appearing in \eqref{eq:BG2_1} -- \eqref{eq:BG2_8}:
		\item[(ii)] Using \eqref{eq:dgGG} to express $\E[F'(Y)\bsh_{a}\adj G^{2}\bse_{a}]$ as a linear combination of $\E[F'(Y)(\wt{B}G^{2})_{aa}]$ and terms appearing in \eqref{eq:BG2_1} -- \eqref{eq:BG2_8}:
		\item[(iii)] Solving the system of linear equations from (i) and (ii) for $\E[F'(Y)(\wt{B}G^{2})_{aa}]$ and $\E[F'(Y)\bsh_{a}\adj G^{2}\bse_{a}]$.
	\end{itemize}
	We will use similar strategies for other proofs; see Section \ref{sec:dBGG_pf} for example. In what follows, we repeatedly use the following direct consequence of Proposition \ref{prop:ll}:
	\beq\label{eq:BG2_rough}
	(\wt{B}G^{2})_{aa}=\sum_{b}(\wt{B}G)_{ab}G_{ba}\prec N^{1/3+2\epsilon},\quad \bsh_{a}\adj G^{2}\bse_{a}=\sum_{b}\bsh_{a}\adj G\bse_{b}G_{ba}\prec N^{1/3+2\epsilon}.
	\eeq
	
	\subsubsection*{Step (i): Expansion of $\E[F'(Y)(\wt{B}G^{2})_{aa}]$}
	We first use $\bse_{a}\adj\wt{B}=-\bsh_{a}\adj \wt{B}^{\anga}R_{a}$ to write
	\beq\label{eq:B2}\begin{aligned}
		&\E[F'(Y)(\wt{B}G^{2})_{aa}]
		=-\expct{F'(Y)\bsh_{a}\adj \wt{B}^{\anga}R_{a}G^{2}\bse_{a}}	\\
		=&-\sum_{c}^{(a)}\Expct{F'(Y)\frac{\ol{g}_{ac}}{\norm{\bsg_{a}}}\bse_{c}\adj\wt{B}^{\anga}R_{a}G^{2}\bse_{a}}+O(N^{-1/6+C\epsilon})	\\
		=&-\frac{1}{N}\sum_{c}^{(a)}\Expct{\frac{\partial}{\partial g_{ac}}\left(F'(Y)\norm{\bsg_{a}}^{-1}\bse_{c}\adj\wt{B}^{\anga}R_{a}G^{2}\bse_{a}\right)}+O(N^{-1/6+C\epsilon}),
	\end{aligned}\eeq
	where we dropped the summand for $c=a$ in the second equality using $F'(Y)\prec N^{C\epsilon}$ and
	\beqs
	h_{aa}\bse_{a}\adj \wt{B}^{\anga}R_{a}G^{2}\bse_{a}=h_{aa}\frb_{a}\bsh_{a}\adj G^{2}\bse_{a}\prec N^{-1/6+2\epsilon},
	\eeqs
	which is due to $h_{aa}\prec N^{-1/2}$, $\norm{B}\lesssim 1$, and \eqref{eq:BG2_rough}.
	
	By Leibniz rule, the derivative in the rightmost side of \eqref{eq:B2} can be divided into that of $F'(Y)$ and the rest. For the former, we write
	\beq\label{eq:B2_Yder}\begin{aligned}
		&-\frac{1}{N}\sum_{c}^{(a)}\Expct{\frac{\partial F'(Y)}{\partial g_{ac}}\norm{\bsg_{a}}^{-1}\bse_{c}\adj\wt{B}^{(a)}R_{a}G^{2}\bse_{a}}	\\
		=&-\frac{1}{N}\sum_{c}^{(a)}\Expct{F''(Y)\frac{\partial Y}{\partial g_{ac}}\bse_{c}\adj\left(\wt{B}+\left(\norm{\bsg}_{a}^{-1}\wt{B}^{\anga}R_{a}-\wt{B}\right)\right)G^{2}\bse_{a}}.
	\end{aligned}\eeq
	We next prove that the matrix $(\norm{\bsg_{a}}^{-1}\wt{B}^{\anga}R_{a}-\wt{B})$ has negligible contribution to \eqref{eq:B2_Yder}. Since $\wt{B}^{\anga}=R_{a}\wt{B}R_{a}$ and $R_{a}=I-\bsr_{a}\bsr_{a}\adj$, we have
	\beq\label{eq:BR_1}\begin{aligned}
		&\norm{\bsg_{a}}^{-1}\bse_{c}\adj(\wt{B}^{\anga}R_{a}-\wt{B})G^{2}\bse_{a}
		=\ell_{a}^{2}\norm{\bsg_{a}}^{-1}\bse_{c}\adj\bsr_{a} \bsr_{a}\adj\wt{B}G^{2}\bse_{a}	\\
		=&\ell_{a}^{2}\norm{\bsg_{a}}^{-1}h_{ac}((\wt{B}G^{2})_{aa}+\frb_{a}\bsh_{a}\adj G^{2}\bse_{a})\prec N^{-1/6+2\epsilon},
	\end{aligned}\eeq
	where we used $\ell_{a}^{2}=1+O_{\prec}(N^{-1/2})$, $\norm{\bsg_{a}}=1+O_{\prec}(N^{-1/2})$, $h_{ac}\prec N^{-1/2}$, $\absv{\frb_{a}}\leq \norm{B}\lesssim 1$, and \eqref{eq:BG2_rough} in the last inequality. Similarly we have
	\beq
	(\norm{\bsg_{a}}^{-1}-1)(\wt{B}G^{2})_{ca}\prec N^{-1/6+2\epsilon}.
	\eeq
	Hence for all $a\neq c$ we have
	\beq\label{eq:BG2_8_rough}
	\bse_{c}\adj (\norm{\bsg_{a}}^{-1}\wt{B}^{\anga}R_{a}-\wt{B})G^{2}\bse_{a}\prec N^{-1/6+2\epsilon}.
	\eeq
	Finally, we roughly estimate the size of $\frac{\partial Y}{\partial g_{ac}}$ in the following lemma.
	\begin{lem}\label{lem:Yder_rough}
		The following holds uniformly over $a\neq c\in\llbra 1,N\rrbra$:
		\beq\label{eq:Yder_rough}
		\frac{\partial Y}{\partial g_{ac}}\prec N^{-1/3+C\epsilon}.
		\eeq
	\end{lem}
	We postpone the proof of Lemma \ref{lem:Yder_rough} to Section \ref{sec:der_tech}. By \eqref{eq:BG2_8_rough} and Lemma \ref{lem:Yder_rough} we get
	\beq\label{eq:B2_Yder_err}
	\frac{1}{N}\sum_{c}^{(a)}\frac{\partial Y}{\partial g_{ac}}\bse_{c}\adj(\norm{\bsg_{a}}^{-1}\wt{B}^{\anga}R_{a}-\wt{B})G^{2}\bse_{a}\prec N^{-1/2+C\epsilon}.
	\eeq
	Plugging in \eqref{eq:B2_Yder_err} to \eqref{eq:B2_Yder} and then to \eqref{eq:B2}, we have
	\beq\label{eq:B2_1}\begin{aligned}
		&\E[F'(Y)(\wt{B}G^{2})_{aa}]
		=-\frac{1}{N}\sum_{c}^{(a)}\Expct{F'(Y)\frac{\partial}{\partial g_{ac}}\left(\norm{\bsg_{a}}^{-1}\bse_{c}\adj\wt{B}^{\anga}R_{a}G^{2}\bse_{a}\right)}\\
		&-\frac{1}{N}\sum_{c}^{(a)}\Expct{F''(Y)\frac{\partial Y}{\partial g_{ac}}(\wt{B}G^{2})_{ca}}+O(N^{-1/2+C\epsilon}).
	\end{aligned}\eeq
	
	Notice that the the derivative of the first term in \eqref{eq:B2_1} is exactly the sum of \eqref{eq:dgBGG} over $b$. Therefore we conclude
	\beq\label{eq:BGG_conc}\begin{aligned}
		\E[&F'(Y)(\wt{B}G^{2})_{aa}]	\\
		=&-\omega_{\beta}(1+\omega_{\beta}m_{\mu})x_{a1}+(1+\omega_{\beta}m_{\mu})\E[F'(Y)(\wt{B}G^{2})_{aa}]	\\
		&+(\frb_{a}-\omega_{\beta})(1+\omega_{\beta}m_{\mu})\expct{F'(Y)\bsh_{a}\adj G^{2}\bse_{a}}\\
		&-\expct{F'(Y)\frd_{3}(G^{2})_{aa}}
		+\expct{F'(Y)\frd_{2}(\wt{B}G^{2})_{aa}}
		+\expct{F'(Y)(\frb_{a}\frd_{2}-\frd_{3})\bsh_{a}\adj G^{2}\bse_{a}}\\
		&+\gamma\expct{F'(Y)\tr(G \wt{B}G)(\wt{B}G)_{aa}}	+\gamma\expct{\frb_{a}F'(Y)\tr(G\wt{B}G)\bsh_a\adj G\bse_{a}}\\
		&-\gamma\expct{F'(Y)\tr(G\wt{B}^{2}G)G_{aa}}-\gamma\expct{F'(Y)\tr(G\wt{B}^{2}G)\bsh_a\adj G\bse_{a}}	\\
		&-\frac{1}{N}\sum_{c}^{(a)}\Expct{F''(Y)\frac{\partial Y}{\partial g_{ac}}(\wt{B}G^{2})_{ca}}
		+O(N^{-1/6+C\epsilon}).
	\end{aligned}
	\eeq
	Here we used $\bsh_{a}\adj\wt{B}=\frb_{a}\bsh_{a}\adj$, and applied \eqref{eq:trace identity} to the tracial prefactors in \eqref{eq:dgBGG}; for example, the first and fourth terms on the right-hand side of \eqref{eq:BGG_conc} comes from
	\beq\label{eq:tr_id_appl}
	F'(Y)\left(\gamma\tr( \wt{B}^{2}G)-\omega_{\beta}(\omega_{\beta}m_{\mu}+1)-\frd_{3}\right)(G^{2})_{aa}\prec\bsd N^{1/3+C\epsilon}\prec N^{-2/3+C\epsilon}.
	\eeq
	
	\subsubsection*{Step (ii): Expansion of $\E[F'(Y)\bsh_{a}\adj G^{2}\bse_{a}]$}
	Now we expand $\expct{F'(Y)\bsh_{a}\adj G^{2}\bse_{a}}$ using the same method. Specifically, we write
	\beqs
	\expct{F'(Y)\bsh_{a}\adj G^{2}\bse_{a}}
	=\frac{1}{N}\sum_{c}^{(a)}\Expct{\frac{\partial}{\partial g_{ac}}\left(\norm{\bsg_{a}}^{-1}F'(Y)(G^{2})_{ca}\right)}+O(N^{-1/6+C\epsilon}),
	\eeqs
	where we estimated the summand for $c=a$ using
	\beq
	F'(Y)h_{aa}(G^{2})_{aa}\prec N^{-1/2+C\epsilon}\cdot N^{1/3+2\epsilon}=N^{-1/6+2\epsilon}.
	\eeq
	
	The rest of the expansion is completely analogous to (in fact easier than) the proof of \eqref{eq:BGG_conc}, except that we use \eqref{eq:dgGG} instead of \eqref{eq:dgBGG}. We omit further details and record the resulting expansion as 
	\beq\label{eq:hGG_conc}
	\begin{split}
		&\expct{F'(Y)\bsh_{a}\adj G^{2}\bse_{a}}
		=-m_{\mu}\E[F'(Y)(\wt{B}G^{2})_{aa}]
		+(1+\omega_{\beta}m_{\mu})x_{a1}	\\
		&+(1+\omega_{\beta}m_{\mu}-m_{\mu}\frb_{a})\expct{F'(Y)\bsh_{a}\adj G^{2}\bse_{a}}\\
		&+\expct{F'(Y)\frd_{2}(G^{2})_{aa}}
		-\expct{F'(Y)\frd_{1}(\wt{B}G^{2})_{aa}}
		+\expct{F'(Y)(\frd_{2}-\frb_{a}\frd_{1})\bsh_{a}\adj G^{2}\bse_{a}}\\
		&-\gamma\Expct{F'(Y)\left(\tr(G^{2})((\wt{B}G)_{aa}+\frb_{a}\bsh_{a}\adj G\bse_{a})-\tr(G\wt{B}G)(G_{aa}+\bsh_{a}\adj G\bse_{a})\right)}\\
		&+\frac{1}{N}\sum_{c}^{(a)}\expct{F''(Y)\frac{\partial Y}{\partial g_{ac}}(G^{2})_{ca}}+O(N^{-1/6+C\epsilon}).
	\end{split}
	\eeq
	\subsubsection*{Step (iii): Conclusion}
	After some algebra, we find that
	\beq
	\text{\eqref{eq:BGG_conc}}+(\omega_{\beta}+m_{\mu}^{-1})\cdot \text{\eqref{eq:hGG_conc}}
	\eeq
	exactly matches \eqref{eq:BG2_1} -- \eqref{eq:BG2_8}. This completes the proof of Lemma \ref{lem:BG2_conc}, modulo those of Lemmas \ref{lem:dgbd_1} and \ref{lem:Yder_rough}.
\end{proof}
\begin{rem}\label{rem:extra}
	Upon closely inspecting the proof, we can see that the main errors of size $N^{-1/6+C\epsilon}$ in \eqref{eq:BG2_1} come from the following estimates:
	\begin{itemize}
		\item[(i)] In \eqref{eq:B2}, we used $h_{aa}\prec N^{-1/2}$:
		\item[(ii)] In \eqref{eq:BG2_8_rough} we used $h_{ac}\prec N^{-1/2}$ and $\norm{\bsg_{a}}^{2}-1\prec N^{-1/2}$:
		\item[(iii)] In the proof of \eqref{eq:dgBGG}, we used $h_{aa},\norm{\bsg_{a}}^{2}-1\prec N^{-1/2}$ (see \eqref{eq:Delta1_goal} for details):
		\item[(iv)] Each of (i) -- (iii) has a counterpart along the expansion of $\E[F'(Y)\bsh_{a}\adj G^{2}\bse_{a}]$.
	\end{itemize}
	We believe that it is technically possible to expand these quantities further, but we do not need such precise estimates due to the smallness of the time scale $t_{0}=N^{-1/3+\chi}$. For example, we know that $\sqrt{N}g_{aa}$ is a $\chi$-distributed random variable for which explicit formulas for moments are available, so that we can apply cumulant expansion to quantities involving $h_{aa}$. Similarly we may write $\norm{\bsg_{a}}^{2}=\sum_{c}\absv{g_{ac}}^{2}$ and apply Stein's lemma for each summand $\ol{g}_{ac}g_{ac}$. Analogous comments apply to Lemma \ref{lem:GBG}.
\end{rem}

\section{Decoupling lemmas for remainders}\label{sec:decoup_rem}
We now prove Lemmas \ref{lem:GBG}--\ref{lem:Yder}. As in the previous section, we separately state the most technical part of the proof in the next two lemmas, that serve similar roles as Lemmas \ref{lem:dgbd_1} and \ref{lem:Yder_rough}. Their proofs are deferred to Section \ref{sec:der_tech}.
\begin{lem}\label{lem:dgdiag}
	The following holds uniformly over $a\in\llbra 1,N\rrbra$:
	\beq\label{eq:dgdiag_BG}\begin{aligned}
		&\frac{1}{N}\sum_{c}^{(a)}\frac{\partial}{\partial g_{ac}}\left[\frac{1}{\norm{\bsg_{a}}}\bse_{c}\adj \wt{B}^{\anga}R_{a}G\bse_{a}\right]	\\ 
		=&-(\omega_{\beta}m_{\mu}+1)(\wt{B}G)_{aa}+\omega_{\beta}(\omega_{\beta}m_{\mu}+1)G_{aa}	
		+\frac{-(\omega_{\beta}+m_{\mu}^{-1})\frd_{2}+\frd_{3}}{\gamma(\fra_{a}-\omega_{\alpha})}	\\
		&+(\omega_{\beta}m_{\mu}+1)(\omega_{\beta}-\frb_{a})\bsh_{a}\adj G\bse_{a}+O_{\prec}(N^{-1/2}),
	\end{aligned}\eeq
	and similarly
	\beq\label{eq:dgdiag_hG}\begin{aligned}
		&\frac{1}{N}\sum_{c}^{(a)}\frac{\partial}{\partial g_{ac}}\left[\frac{1}{\norm{\bsg_{a}}}G_{ca}\right]	\\
		=&-m_{\mu}(\wt{B}G)_{aa}+(\omega_{\beta}m_{\mu}+1)G_{aa}+\frac{-(\omega_{\beta}+m_{\mu}^{-1})\frd_{1}+\frd_{2}}{\gamma(\fra_{a}-\omega_{\alpha})}\\
		&+m_{\mu}(\omega_{\beta}+m_{\mu}^{-1}-\frb_{a})\bsh_{a}\adj G\bse_{a}+O_{\prec}(N^{-1/2}).
	\end{aligned}\eeq
\end{lem}
\begin{lem}\label{lem:dg_cr}
	The following holds uniformly over $a\in\llbra 1,N\rrbra$ and $k\in\{1,2,3\}$:
	\begin{align}\label{eq:dg_cr_BG}
		\frac{1}{N}\sum_{c}^{(a)}\bse_{c}\adj\wt{B}^{\anga}R_{a}G\bse_{a}\frac{\partial}{\partial g_{ac}}\Tr G\wt{B}^{k-1}G\prec N^{2/3+C\epsilon},\\
		\frac{1}{N}\sum_{c}^{(a)}G_{ca}\frac{\partial}{\partial g_{ac}}\Tr G\wt{B}^{k-1}G\prec N^{2/3+C\epsilon}\label{eq:dg_cr_hG}.
	\end{align}
	Similarly we have
	\begin{align}
		\frac{1}{N}\sum_{c}^{(a)}\frac{\partial(\tr \wt{B}^{k-1}G)}{\partial g_{ac}}\bse_{c}\adj\wt{B}^{\anga}R_{a}G^{2}\bse_{a}\prec N^{-1/3+C\epsilon}, \label{eq:dg_cr_BGG}\\
		\frac{1}{N}\sum_{c}^{(a)}\frac{\partial (\tr \wt{B}^{k-1}G)}{\partial g_{ac}}\bse_{c}\adj G^{2}\bse_{a}\prec N^{-1/3+C\epsilon}.\label{eq:dg_cr_hGG}
	\end{align}
	Furthermore, \eqref{eq:dg_cr_BGG} and \eqref{eq:dg_cr_hGG} remain true if we replace $\tr\wt{B}^{k-1}G$ by $\tr A^{k-1}G$.
\end{lem}

\subsection{Proof of Lemma \ref{lem:GBG}}\label{sec:GBG_pf}
The proof is parallel to that of Proposition \ref{prop:decoup} in the following sense. For each fixed $k=0,1,2,$ it consists of four steps:
\begin{itemize}
	\item[(0)] Using Stein's lemma, with respect to $W$ in $(WG)_{aa}$, to expand $\E[F'(Y)G_{aa}\Tr G\wt{B}^{k}G]$:
	\item[(i)] Using \eqref{eq:dgdiag_BG} to expand $\E[F'(Y)(\wt{B}G)_{aa}\Tr G\wt{B}^{k}G]$:
	\item[(ii)] Using \eqref{eq:dgdiag_hG} to expand $\E[F'(Y)\bsh_{a}\adj G\bse_{a}\Tr G\wt{B}^{k}G]$:
	\item[(iii)] Solving the system of three equations from Steps (0) -- (ii).
\end{itemize}
To compare with the proof of Proposition \ref{prop:decoup}, Step (0) corresponds to \eqref{eq:Bab2_conseq} and Steps (i) -- (iii) are parallel to those in Section \ref{sec:B2_prf}.
\begin{proof}[Proof of Lemma \ref{lem:GBG}]\
	\subsubsection*{Step (0): Expansion of $\E[F'(Y)G_{aa}\Tr G\wt{B}^{k}G]$}
	
	We use the identity $\gamma HG=zG+I$ to the first resolvent factor $G_{aa}$ to get
	\begin{align}
		z\expct{F'(Y)G_{aa}\Tr G\wt{B}^{k}G}	
		=&\fra_{a}\gamma\expct{F'(Y)G_{aa}\Tr G\wt{B}^{k}G}
		+\gamma\expct{F'(Y)(\wt{B}G)_{aa}\Tr G\wt{B}^{k}G}	\nonumber\\
		&+\sqrt{t}\gamma\expct{F'(Y)(WG)_{aa}\Tr G\wt{B}^{k}G}-\expct{F'(Y)\Tr G\wt{B}^{k}G}.\label{eq:YBG3}
	\end{align}
	Then we apply Stein's lemma to the third term so that
	\beq\begin{aligned}\label{eq:YBG4}
		\sqrt{t}&\E[F'(Y)(WG)_{aa}\Tr G\wt{B}^{k}G]=-\gamma t\E[F'(Y)G_{aa}\tr G\Tr G\wt{B}^{k}G]
		\\
		&-\frac{\gamma t}{N}\expct{F'(Y)((G\wt{B}^{k}G^{3})_{aa}+(G^{2}\wt{B}^{k}G^{2})_{aa})}	\\
		&+\frac{\gamma t}{N}\Expct{F''(Y)\int_{E_{1}}^{E_{2}}(\im[\wt{G}^{2}]G)_{aa}\dd\wt{E} \Tr G\wt{B}^{k}G}	\\
		=&-tm_{\mu}\E[F'(Y)G_{aa}\Tr G\wt{B}^{k}G]-\frac{tZ_{k1}}{\gamma(\fra_{a}-\omega_{\alpha})}+O(tN^{2/3+C\epsilon}),
	\end{aligned}\eeq
	where we applied \eqref{eq:x_rough_1} together with the following direct consequences of Proposition \ref{prop:ll}:
	\beq\begin{aligned}\label{eq:YBG_1}
		\absv{(G\wt{B}^{k}G^{3})_{aa}}
		\leq& \sum_{b}\absv{(G\wt{B}^{k}G)_{ab}}\absv{(G^{2})_{ba}}	\\
		\leq& \norm{G\bse_{a}}^{2}\norm{B}^{k}\sum_{b}\norm{G\bse_{b}}^{2}\lesssim \frac{\im G_{aa}\im\Tr G}{\eta^{2}}\prec N^{5/3+4\epsilon},\\
		\absv{(G^{2}\wt{B}^{k}G^{2})_{aa}}\leq &\sum_{b,c}\absv{G_{ab}}\absv{(G\wt{B}^{k}G)_{bc}}\absv{G_{ca}}	\\
		\leq& \norm{B}^{k}\sum_{b,c}\absv{G_{ab}}\frac{\sqrt{\im G_{bb}\im G_{cc}}}{\eta}\absv{G_{ca}}\prec N^{5/3+4\epsilon}.
	\end{aligned}\eeq
	Substituting \eqref{eq:YBG4} into \eqref{eq:YBG3} gives
	\beq\label{eq:YG}\begin{aligned}
		&\left(\frac{z}{\gamma}-\fra_{a}+tm_{\mu}\right)\expct{F'(Y)G_{aa}\Tr G\wt{B}^{k}G}	\\
		=&\E[F'(Y)(\wt{B}G)_{aa}\Tr G\wt{B}^{k}G]-\E[F'(Y)\Tr G\wt{B}^{k}G]+O(t N^{1+C\epsilon}),
	\end{aligned}\eeq 
	where we absorbed the second term on the right-most side of \eqref{eq:YBG4} into the error. Recalling $t\leq t_{0}=N^{-1/3+\chi}\ll N^{-1/6}$, we may replaced the error in \eqref{eq:YG} with $O(N^{5/6+C\epsilon})$. This completes Step (0).
	\subsubsection*{Step (i): Expansion of $\E[F'(Y)(\wt{B}G)_{aa}\Tr G\wt{B}^{k}G]$}
	Here we extract the Gaussian vector $\bsg_{a}\adj$ from $\bse_{a}\adj\wt{B}$ as in \eqref{eq:B2} to get
	\beq\label{eq:YBG0}
	\begin{split}
		&\expct{F'(Y)(\wt{B}G)_{aa}\Tr G\wt{B}^{k} G}	\\
		&=-\frac{1}{N}\sum_{c}^{(a)}\Expct{\frac{\partial}{\partial g_{ac}}\left[F'(Y)\frac{1}{\norm{\bsg_{a}}}\bse_{c}\adj\wt{B}^{\anga}R_{a}G\bse_{a}\Tr G\wt{B}^{k}G\right]} +O(N^{5/6+C\epsilon}),
	\end{split}\eeq
	where we estimated the summand for $c=a$ using $h_{aa}\prec N^{-1/2}$ as in \eqref{eq:B2}. By Leibniz rule we divide the right-hand side of \eqref{eq:YBG0} into three parts, according to which of the following three quantities are differentiated with respect to $g_{ac}$;
	\beq\label{eq:YBG33}
	F'(Y),\qquad \frac{1}{\norm{\bsg_{a}}}\bse_{c}\adj\wt{B}^{\anga}R_{a}G\bse_{a},\qquad \Tr G\wt{B}^{k}G.
	\eeq
	The part involving derivative of $F'(Y)$ is estimated as
	\beq
	\frac{1}{N}\sum_{c}^{(a)}\Expct{F''(Y)\frac{\partial Y}{\partial g_{ac}}\frac{1}{\norm{\bsg_{a}}} \bse_{c}\adj\wt{B}^{\anga}R_{a}G\bse_{a}\Tr G\wt{B}^{k}G}\lesssim N^{2/3+C\epsilon},
	\eeq
	where we used \eqref{eq:x_rough_1}, Lemma \ref{lem:Yder_rough}, $\norm{\bsg_{a}}^{-1}\prec 1$, and (see also \eqref{eq:BR_1})
	\beq
	\bse_{c}\adj \wt{B}^{\anga}R_{a}G\bse_{a}=(\wt{B}G)_{ca}-\ell_{a}^{2}h_{ac}\bsr_{a}\adj\wt{B}G\bse_{a}\prec \delta_{ac}+N^{-1/3+\epsilon}.
	\eeq
	The contributions of (derivatives of) the remaining two quantities in \eqref{eq:YBG33} are exactly those in \eqref{eq:dgdiag_BG} and \eqref{eq:dg_cr_BG}, hence we can simply plug in the results. Altogether we conclude
	\beq\begin{split}\label{eq:YBG}
		&\expct{F'(Y)(\wt{B}G)_{aa}\Tr G\wt{B}^{k}G}	\\
		=&(\omega_{\beta}m_{\mu}+1)\expct{F'(Y)(\wt{B}G)_{aa}\Tr G\wt{B}^{k}G}
		-\omega_{\beta}(\omega_{\beta}m_{\mu}+1)\expct{F'(Y)G_{aa}\Tr G\wt{B}^{k}G}\\
		&+(\omega_{\beta}m_{\mu}+1)(\frb_{a}-\omega_{\beta})\expct{F'(Y)\bsh_{a}\adj G \bse_{a}\Tr G\wt{B}^{k}G}	\\
		&+\frac{1}{\gamma(\fra_{a}-\omega_{\alpha})}\expct{F'(Y)((\omega_{\beta}+m_{\mu}^{-1})\frd_{2}-\frd_{3})\Tr G\wt{B}^{k}G}	
		+O(N^{5/6+C\epsilon}).
	\end{split}
	\eeq
	
	\subsubsection*{Step (ii): Expansion of $\E[F'(Y)\bsh_{a}\adj G\bse_{a}\Tr G\wt{B}^{k}G]$} 
	We follow the computations as in Step (i) almost verbatim, except that we use \eqref{eq:dgdiag_hG} and \eqref{eq:dg_cr_hG} instead of \eqref{eq:dgdiag_BG} and \eqref{eq:dg_cr_BG}, respectively. As a result, we obtain
	\beq\label{eq:YhG}
	\begin{split}
		&\expct{F'(Y)\bsh_{a}\adj G\bse_{a}\Tr G\wt{B}^{k}G}\\
		=&-m_{\mu}\expct{F'(Y)(\wt{B}G)_{aa}\Tr G\wt{B}^{k}G}
		+(\omega_{\beta}m_{\mu}+1)\expct{F'(Y)G_{aa}\Tr G\wt{B}^{k}G}\\
		&+m_{\mu}(\omega_{\beta}+m_{\mu}^{-1}-\frb_{a})\expct{F'(Y)\bsh_{a}\adj G\bse_{a}\Tr G\wt{B}^{k}G}\\
		&+\frac{1}{\gamma(\fra_{a}-\omega_{\alpha})}\expct{F'(Y)(-(\omega_{\beta}+m_{\mu}^{-1})\frd_{1}+\frd_{2})\Tr G\wt{B}^{k}G}+O(N^{5/6+C\epsilon}).
	\end{split}
	\eeq
	
	\subsubsection*{Step (iii): Conclusion}
	Taking
	\beq
	\eqref{eq:YG}+\text{\eqref{eq:YBG}}+(\omega_{\beta}+m_{\mu}^{-1})\cdot\text{\eqref{eq:YhG}}
	\eeq
	implies, after rearrangement, that (recall definitions of $x_{k}$ and $Z$ from \eqref{eq:def_x} and \eqref{eq:def_Z})
	\beq\begin{aligned}\label{eq:GBG}
		&\left(\frac{z}{\gamma}-\fra_{a}+tm_{\mu}-m_{\mu}^{-1}-\omega_{\beta}\right)\E[F'(Y)G_{aa}\Tr G\wt{B}^{k}G]	\\
		=&-x_{k}-\frac{\bfv_{\beta}\tp Z\bse_{k}}{\gamma(\fra_{a}-\omega_{\alpha})} +O(N^{5/6+C\epsilon}).
	\end{aligned}\eeq
	Applying the same argument as in \eqref{eq:Bab2_conseq} -- \eqref{eq:B2_factor} to the left-hand side of \eqref{eq:GBG} proves \eqref{eq:GBGdG_conc}. Similarly, \eqref{eq:GBGdBG_conc} and \eqref{eq:GBGhG_conc} respectively match
	\beqs\begin{aligned}
		(\omega_{\beta}+m_{\mu}^{-1})\cdot\text{\eqref{eq:GBGdG_conc}}&+\text{\eqref{eq:YBG}}+(\omega_{\beta}+m_{\mu}^{-1})\cdot\text{\eqref{eq:YhG}}\qquad\text{and} \\
		&\frac{\text{\eqref{eq:YBG}}+\omega_{\beta}\cdot\text{\eqref{eq:YhG}}}{\frb_{a}-\omega_{\beta}}.
	\end{aligned}\eeqs
	This completes the proof of Lemma \ref{lem:GBG}.
\end{proof}

\subsection{Proof of Lemma \ref{lem:dBGG}}\label{sec:dBGG_pf}
The proof follows a parallel outline to that of Lemma \ref{lem:GBG} with different inputs. Here we again use Lemma \ref{lem:dgbd_1}, but we require much less precision due to the factor $\frd_{k}$ which is small. Namely, we use the following two direct consequences of \eqref{eq:dgBGG} and \eqref{eq:dgGG}, respectively:
\beq\label{eq:dgBGG_d}\begin{aligned}
	&\frd_{k}\frac{1}{N}\sum_{c}^{(a)}\frac{\partial}{\partial g_{ac}}\left[\frac{1}{\norm{\bsg_{a}}}\bse_{c}\adj \wt{B}^{\anga}R_{a}G^{2}\bse_{a}\right]	\\
	=&-(\omega_{\beta}m_{\mu}+1)\frd_{k}(\wt{B}G^{2})_{aa}+\omega_{\beta}(\omega_{\beta}m_{\mu}+1)\frd_{k}(G^{2})_{aa}\\
	&-(\omega_{\beta}m_{\mu}+1)(\frb_{a}-\omega_{\beta})\frd_{k}\bsh_{a}\adj G^{2}\bse_{a}	\\
	&-\frac{\omega_{\beta}+m_{\mu}^{-1}}{\fra_{a}-\omega_{\alpha}}\frd_{k}\tr G\wt{B}G+\frac{1}{\fra_{a}-\omega_{\alpha}}\frd_{k}\tr G\wt{B}^{2}G+O_{\prec}(N^{-1/3+4\epsilon}),
\end{aligned}\eeq
\beq\label{eq:dgGG_d}\begin{aligned}
	&\frd_{k}\frac{1}{N}\sum_{c}^{(a)}\frac{\partial}{\partial g_{ac}}\left[\frac{1}{\norm{\bsg_{a}}}(G^{2})_{ca}\right]	\\
	=&-m_{\mu}\frd_{k}(\wt{B}G^{2})_{aa}+(\omega_{\beta}m_{\mu}+1)\frd_{k}(G^{2})_{aa}	\\
	&-m_{\mu}(\frb_{a}-\omega_{\beta}-m_{\mu}^{-1})\frd_{k}\bsh_{a}\adj G^{2}\bse_{a} \\
	&-\frac{\omega_{\beta}+m_{\mu}^{-1}}{\fra_{a}-\omega_{\alpha}}\frd_{k}\tr G^{2}+\frac{1}{\fra_{a}-\omega_{\alpha}}\frd_{k}\tr G\wt{B}G +O_{\prec}(N^{-1/3+4\epsilon}).
\end{aligned}\eeq
In order to derive \eqref{eq:dgBGG_d} from \eqref{eq:dgBGG}, we apply the following procedure to \eqref{eq:dgBGG}: (i) We sum \eqref{eq:dgBGG} over $b$, multiply both sides by $\frd_{k}$. (ii) Then we replace the prefactors $\tr \wt{B}G$, $\tr \wt{B}^{2}G$, $G_{aa}$ and $(\wt{B}G)_{aa}$ by their deterministic counter parts using \eqref{eq:trace identity},\eqref{eq:ll_etr}, \eqref{eq:x_rough_1}, and \eqref{eq:BG2_rough}. (iii) Finally we absorb the terms with $\bsh_{a}\adj G\bse_{a}$ (recall $\bsh_{a}\adj \wt{B}G\bse_{a}=\frb_{a}\bsh_{a}\adj G\bse_{a}$) into the error using \eqref{eq:ll_etr} and \eqref{eq:x_rough_1}. The same procedure applied to \eqref{eq:dgGG} proves \eqref{eq:dgGG_d}. 

Given \eqref{eq:dgBGG_d} and \eqref{eq:dgGG_d}, the proof of Lemma \ref{lem:dBGG} consists of four steps:
\begin{itemize}
	\item[(0)] Using Stein's lemma, with respect to $W$ in $(WG^{2})_{aa}$, to expand $\E[F'(Y)\frd_{k}(G^{2})_{aa}]$:
	\item[(i)] Using \eqref{eq:dgBGG_d} to expand $\E[F'(Y)\frd_{k}\Tr (\wt{B}G^{2})_{aa}]$:
	\item[(ii)] Using \eqref{eq:dgGG_d} to expand $\E[F'(Y)\frd_{k}\bsh_{a}\adj G^{2}\bse_{a}]$:
	\item[(iii)] Solving the system of three equations from Steps (0) -- (ii).
\end{itemize}

\begin{proof}[Proof of Lemma \ref{lem:dBGG}]\
	\subsubsection*{Step (0): Expansion of $\E[F'(Y)\frd_{k}(G^{2})_{aa}]$}
	We write
	\beq\label{eq:dGG1}
	\begin{split}
		z\expct{F'(Y)\frd_{k}(G^{2})_{aa}}
		=&\gamma\fra_{a}\expct{F'(Y)\frd_{k}(G^{2})_{aa}}
		+\gamma\expct{F'(Y)\frd_{k}(\wt{B}G^{2})_{aa}}\\
		&+\gamma\sqrt{t}\expct{F'(Y)\frd_{k}(WG^{2})_{aa}}
		-\expct{F'(Y)\frd_{k}G_{aa}},
	\end{split}
	\eeq
	Applying Stein's lemma to the third term, we get
	\beq\label{eq:dGG2}
	\begin{split}
		&\sqrt{t}\expct{F'(Y)\frd_{k}(WG^{2})_{aa}}\\
		=&-\gamma t\expct{F'(Y)\frd_{k}\left(\tr G(G^{2})_{aa}+G_{aa}\tr G^{2}\right)}-\frac{\gamma t}{N^{2}}\expct{F'(Y)(G\wt{B}^{k}G^{3})_{aa}}	\\
		&-\frac{\gamma t}{N}\Expct{F''(Y)\frd_{k}\int_{E_{1}}^{E_{2}}(\im[\wt{G}^{2}]G^{2})_{aa}\dd\wt{E}}	\\
		=&-tm_{\mu}\E[F'(Y)\frd_{k}(G^{2})_{aa}]-t\frac{1}{N(\fra_{a}-\omega_{\alpha})}Z_{k1}+O(tN^{-1/3+C\epsilon}),
	\end{split}
	\eeq
	where in the second equality we used \eqref{eq:YBG_1} to $(G\wt{B}^{k}G^{3})_{aa}$ and naive power counting with Proposition \ref{prop:ll} to the rest. This completes Step (0) as
	\beq\label{eq:dGG}
		\left(\frac{z}{\gamma}-\fra_{a}+tm_{\mu}\right)\E[F'(Y)\frd_{k}(G^{2})_{aa}]
		=\E[F'(Y)\frd_{k}(\wt{B}G^{2})_{aa}]+O(N^{-1/6+C\epsilon}),
	\eeq
	where we absorbed the last term of \eqref{eq:dGG1} and the second term of \eqref{eq:dGG2} into the error using $t\ll N^{-1/6}$ and $Z_{k1}=O(N^{1+C\epsilon})$.
	
	\subsubsection*{Step (i): Expansion of $\E[F'(Y)\frd_{k}(\wt{B}G^{2})_{aa}]$} As in \eqref{eq:YBG0}, we write
	\begin{align}\label{eq:dBGG1}
		\begin{split}
			&\expct{F'(Y)\frd_{k}(\wt{B}G^{2})_{aa}}\\
			=&-\frac{1}{N}\sum_{c}^{(a)}\Expct{\frac{\partial}{\partial g_{ac}}\left[F'(Y)\frd_{k}\norm{\bsg_{a}}^{-1}(\wt{B}^{\anga}R_{a} G^{2})_{ca}\right]}+O(N^{-1/2+C\epsilon}).
		\end{split}
	\end{align}
	We again apply Leibniz rule to \eqref{eq:dBGG1}, dividing it into three factors
	\beq\label{eq:Leib_dBGG}
	F'(Y),\qquad \frd_{k},\qquad \frac{1}{N}\sum_{c}^{(a)}\frac{\partial}{\partial g_{ac}}\left[\frac{1}{\norm{\bsg_{a}}}(\wt{B}^{\anga}R_{a}G^{2})_{ca}\right].
	\eeq
	The derivative of $F'(Y)$ has negligible contribution by Lemma \ref{lem:Yder_rough}, that is,
	\beq
	\frac{1}{N}\sum_{c}^{(a)}\Expct{F''(Y)\frac{\partial Y}{\partial g_{ac}}\frd_{k}\norm{\bsg_{a}}^{-1}(\wt{B}^{\anga}R_{a}G^{2})_{ca}}\prec N^{-1/3+C\epsilon},
	\eeq
	where we used $\frd_{k}\prec N^{-1/3+\epsilon}$ and that $(\wt{B}^{\anga}R_{a}G^{2})_{ca}\prec N^{1/3+2\epsilon}$ due to Proposition \ref{prop:ll} and \eqref{eq:BR_1}. Similarly, the contribution of the derivative of $\frd_{k}$ is easily shown to be $O(N^{-1/3+C\epsilon})$ by simply plugging in \eqref{eq:dg_cr_BGG}. Finally that of the last quantity in \eqref{eq:Leib_dBGG} exactly matches the left-hand side of \eqref{eq:dgBGG_d}. To sum up, we have
	\beq\label{eq:dBGG}\begin{aligned}
		&\expct{F'(Y)\frd_{k}(\wt{B}G^{2})_{aa}}	\\
		=&(\omega_{\beta}m_{\mu}+1)\expct{F'(Y)\frd_{k}(\wt{B}G^{2})_{aa}}-\omega_{\beta}(\omega_{\beta}m_{\mu}+1)\expct{F'(Y)\frd_{k}(G^{2})_{aa}}\\
		&+(\omega_{\beta}m_{\mu}+1)(\frb_{a}-\omega_{\beta})\expct{F'(Y)\frd_{k}\bsh_{a}\adj G^{2}\bse_{a}}	\\
		&+\frac{(\omega_{\beta}+m_{\mu}^{-1})Z_{k2}-Z_{k3}}{N(\fra_{a}-\omega_{\alpha})}+O(N^{-1/3+C\epsilon}).
	\end{aligned}\eeq
	
	\subsubsection*{Step (ii): Expansion of $\E[F'(Y)\frd_{k}\bsh_{a}\adj G^{2}\bse_{a}]$} As in the previous subsection, we omit the proof and only record the resulting expansion:
	\beq\label{eq:dhGG}
	\begin{split}
		&\expct{F'(Y) \frd_{k} \bsh_{a}\adj G^{2} \bse_{a}}\\
		=&-m_{\mu}\expct{F'(Y)\frd_{k}(\wt{B}G^{2})_{aa}}
		+(1+\omega_{\beta}m_{\mu})\expct{F'(Y)\frd_{k}(G^{2})_{aa}}\\
		&-m_{\mu}(\frb_{a}-\omega_{\beta}-m_{\mu}^{-1})\expct{F'(Y)\frd_{k}\bsh_{a}\adj G^{2}\bse_{a}}\\
		&-\frac{(\omega_{\beta}+m_{\mu}^{-1})Z_{k1}-Z_{k2}}{N(\fra_{a}-\omega_{\alpha})}+O(N^{-1/3+C\epsilon}),	
	\end{split}
	\eeq
	
	\subsubsection*{Step (iii): Conclusion}
	To prove \eqref{eq:dGG_conc}, we take
	\beq
	\text{\eqref{eq:dGG}}+\text{\eqref{eq:dBGG}}+(\omega_{\beta}+m_{\mu}^{-1})\text{\eqref{eq:dhGG}},
	\eeq
	and use the same asymptotics as in \eqref{eq:B2_factor}. Similarly, suitable linear combinations of \eqref{eq:dGG}, \eqref{eq:dBGG}, and \eqref{eq:dhGG} proves \eqref{eq:dBGG_conc} and \eqref{eq:dhGG_conc}.
\end{proof}

\subsection{Proof of Lemma \ref{lem:veq1}}\label{sec:veq_pf}
\begin{proof}[Proof of Lemma \ref{lem:veq1}]
	Recall the definitions of $\caZ,\wt{Z}$, and $\wt{\caZ}$ from \eqref{eq:def_Z}. We first prove the following analogue of \eqref{eq:dGG_conc}:
	\beq\label{eq:dGG_sym}
	\expct{F'(Y)\frd_{k}(\caG^{2})_{aa}}=\frac{1}{N(\frb_{a}-\omega_{\beta})^{2}}\bse_{k}\tp \wt{Z}\bfv_{\alpha}+O(N^{-1/6+C\epsilon}).
	\eeq
	Interchanging the roles of $A$ and $B$, it suffices to prove
	\beq\label{eq:dGG_sym_1}
	\expct{F'(Y)d_{k}(G^{2})_{aa}}=\frac{1}{N(\fra_{a}-\omega_{\alpha})^{2}}\bse_{k}\tp \wt{\caZ}\bfv_{\beta}+O(N^{-1/6+C\epsilon}).
	\eeq
	It is easy to see that the proof of \eqref{eq:dGG_conc} applies almost verbatim to \eqref{eq:dGG_sym_1}, with the only difference being that we use \eqref{eq:dg_cr_BGG} and \eqref{eq:dg_cr_hGG} with $\tr \wt{B}^{k-1}G$ replaced by $\tr A^{k-1}G$.
	We omit the details to avoid repetition. This proves \eqref{eq:dGG_sym_1}, hence \eqref{eq:dGG_sym}.
	
	Next, we take the sum of \eqref{eq:dGG_sym} over $a$ with weights $\frb_{a}^{\ell-1}$ so that
	\beq\begin{aligned}
		Z_{k\ell}=&\expct{F'(Y)\frd_{k}\Tr \wt{B}^{\ell-1}G^{2}}=\expct{F'(Y)\frd_{k}\Tr B^{\ell-1}\caG^{2}}	\\
		=&(u_{\beta \ell}+O(\bsd))\bse_{k}\tp \wt{Z}\bfv_{\alpha}+O(N^{5/6+C\epsilon})
		=u_{\beta\ell}\bse_{k}\tp\wt{Z}\bfv_{\alpha}+O(N^{5/6+C\epsilon}).\label{eq:tdZa}
	\end{aligned}\eeq
	On the other hand, taking the average over $a$ of \eqref{eq:GBGdG_conc} with weights $\fra_{a}^{\ell-1}$ and using the same estimates as in \eqref{eq:tdZa}, we have
	\beqs
	\wt{\caZ}_{\ell k}=u_{\alpha \ell}\bfv_{\beta}\tp Z\bse_{k}+O(N^{5/6+C\epsilon}).
	\eeqs
	Then we again interchange the roles of $A$ and $B$ to obtain
	\beq\label{eq:veq1}
	\wt{Z}_{\ell k}=u_{\beta \ell}\bfv_{\alpha}\tp \caZ \bse_{k}+O(N^{5/6+C\epsilon}).
	\eeq
	Combining \eqref{eq:tdZa} and \eqref{eq:veq1}, we get
	\beq\label{eq:Zeq}
	Z_{k\ell}=u_{\beta k}u_{\beta\ell}\bfv_{\alpha}\tp\caZ\bfv_{\alpha}+O(N^{5/6+C\epsilon}).
	\eeq
	We may substitute $\bfv_{\alpha}\tp\caZ\bfv_{\alpha}$ by $u_{\beta1}^{-2}Z_{11}$ by taking $k=1=\ell$ in \eqref{eq:Zeq}, completing the proof of Lemma \ref{lem:veq1}.
\end{proof}

\subsection{Proof of Lemma \ref{lem:Yder}}\label{sec:Yder_pf}
In this section, another spectral parameter $w$, in addition to the usual $z$ in Sections \ref{sec:GBG_pf} -- \ref{sec:veq_pf}, serves equally important role. We always take $\re w\equiv \wt{E}\in L_{+}+[E_{1},E_{2}]$ and $\im w=\pm\eta_{0}$, and define $\wt{G}\deq G(w)$. While $\wt{G}$ may denote $G(\wt{E}+L_{+}+\ii\eta_{0})$ in some equations and $G(\wt{E}+L_{+}-\ii\eta_{0})$ in others, the spectral parameter of $\wt{G}$ remains the same within each equation.

\begin{proof}[Proof of Lemma \ref{lem:Yder}]
	First of all, we simplify the left-hand side of \eqref{eq:Y_der_conc} using the following lemma, whose proof is postponed to Section \ref{sec:der_tech}.
	\begin{lem}\label{lem:Y_1}
		The following holds uniformly over $E,\wt{E}\in[E_{1},E_{2}]$ and $k\in\{0,1\}$.
		\beq\label{eq:Y_1}
		\sum_{c}^{(a)}\frac{\partial \Tr\wt{G}}{\partial g_{ac}}(\wt{B}^{k}G^{2})_{ca}=-\gamma(\bse_{a}+\bsh_{a})\adj[\wt{B},\wt{G}^{2}]\wt{B}^{k}G^{2}\bse_{a}+O(N^{7/6+4\epsilon}).
		\eeq
	\end{lem}
	From the definition of $Y$, we have
	\beqs
	\frac{\partial Y}{\partial g_{ac}}=\frac{1}{2\ii}\int_{E_{1}}^{E_{2}} \left( \frac{\partial \Tr \wt{G}}{\partial g_{ac}}-{\frac{\Tr \wt{G}\adj}{\partial g_{ac}}}\right)\dd \wt{E},
	\eeqs
	where $\wt{G}$ denotes $G(\wt{E}+L_{+}+\ii\eta_{0})$. By Lemma \ref{lem:Y_1} we may write
	\begin{align}
		&\sum_{c}^{(a)}\expct{F''(Y)\frac{\partial Y}{\partial g_{ac}}(\wt{B}^{k}G^{2})_{ca}}	\nonumber\\
		=&-\frac{\gamma}{2\ii}\int_{E_{1}}^{E_{2}}\expct{F''(Y)(\bse_{a}+\bsh_{a})\adj[\wt{B},\wt{G}^{2}-(\wt{G}\adj)^{2}]\wt{B}^{k}G^{2}\bse_{a}}\dd \wt{E}+O(N^{1/2+C\epsilon}),\label{eq:Y_der0}
	\end{align}
where the extra factor of $N^{-2/3+\epsilon}$ in the error is due to the integral over $[E_{1},E_{2}]$.
	
	The rest of the proof mostly involves repeating the arguments in previous sections. Namely, we need to decouple the index $a$ from the two quantities \begin{align*}
		&\expct{F''(Y)\bse_{a}\adj[\wt{B},\wt{G}^{2}]\wt{B}^{k}G^{2}\bse_{a}},&
		&\expct{F''(Y)\bsh_{a}\adj[\wt{B},\wt{G}^{2}]\wt{B}^{k}G^{2}\bse_{a}},&
		&k=0,1,
	\end{align*}
	To make a direct analogy, for each $a\in\llbra 1,N\rrbra$ we define $(3\times 3)$ matrices $\frZ_{a}$ and $\frZ$ as
	\begin{align*}
		\frZ_{ak\ell}\equiv \frZ_{ak\ell}(z,w)&\deq\expct{F''(Y)(\wt{B}^{k-1}\wt{G}^{2}\wt{B}^{\ell-1}G^{2})_{aa}},\quad \frZ\deq\sum_{a}\frZ. 
	\end{align*}
	These auxiliary matrices are obvious analogues of $Z_{a}$ and $Z$ defined in \eqref{eq:def_Z}. Likewise, local laws provides rough estimate for the sizes of $\frZ$; for example when $k=\ell=3$ we have
	\beq\label{eq:frZ_rough}
	\absv{\frZ_{22}}\leq \sum_{a,b,c,d} \expct{\absv{\wt{B}\wt{G}}_{ab}\absv{\wt{G}\wt{B}}_{bc}\absv{\wt{B}G}_{cd}\absv{G\wt{B}}_{da}}\prec N^{4-4/3+C\epsilon}=N^{8/3+C\epsilon},
	\eeq
	and similarly $\frZ_{k\ell}\prec N^{8/3+C\epsilon}$ for other choices of $k,\ell$. Also $\frZ_{ak\ell}\prec N^{5/3+C\epsilon}$ unless $k=3$ by the same power counting; we will not use $\frZ_{a3\ell}$ in the proofs.
	
	We decouple the index $a$ from $\frZ$'s using the next two lemmas, which are analogues of Lemmas \ref{lem:dBGG} and \ref{lem:veq1}, respectively.
	\begin{lem}\label{lem:Yder_dcp} Under the conditions in Proposition \ref{prop:gfc},
		the following holds true uniformly over $a\in\braN$ and $k=1,2,3$:
		\begin{align}
			\frZ_{a1k}=&\frac{1}{N(\fra_{a}-\omega_{\alpha})^{2}} \bfv_{\beta}\tp\frZ\bse_{k}+O(N^{3/2+C\epsilon}),	\label{eq:Y_GG_conc}\\
			\frZ_{a2k}=&\frac{\omega_{\beta}+m_{\mu}^{-1}}{N(\fra_{a}-\omega_{\alpha})^{2}}\bfv_{\beta}\tp\frZ\bse_{k}-\frac{1}{N(\fra_{a}-\omega_{\alpha})}\bfv_{\beta}\tp\frZ\bse_{k}+O(N^{3/2+C\epsilon}),	\label{eq:Y_BGG_conc}\\
			\E[F''(Y)\bsh_{a}\adj& \wt{G}^{2}\wt{B}^{k-1}G^{2}\bse_{a}]	\nonumber\\
			=&\frac{\left(\omega_{\beta}(\omega_{\beta}+m_{\mu}^{-1})\bse_{1}-(2\omega_{\beta}+m_{\mu}^{-1})\bse_{2}+\bse_{3}\right)\tp\frZ\bse_{k}}{N(\fra_{a}-\omega_{\alpha})(\frb_{a}-\omega_{\beta})}+O(N^{3/2+C\epsilon})	\label{eq:Y_hGG_conc}.
		\end{align}
	\end{lem}
	\begin{lem}\label{lem:veq_Y}
		For each $k,\ell\in\{1.2.3\}$, we have the following:
		\beq\label{eq:veq_Y}
		\frZ_{k\ell}=\frac{u_{\beta k}u_{\beta \ell}}{u_{\beta 1}^{2}}\frZ_{11}+O(N^{5/2+C\epsilon}).
		\eeq
	\end{lem}
	
	We postpone the proofs of Lemmas \ref{lem:Yder_dcp} and \ref{lem:veq_Y} to the end of this section and move on to deduce Lemma \ref{lem:Yder} from them. First of all, applying Lemma \ref{lem:veq_Y} to the right-hand side of \eqref{eq:Y_hGG_conc}, we obtain
	\beqs
	\begin{split}
		&\left(\omega_{\beta}(\omega_{\beta}+m_{\mu}^{-1})\bse_{1}-(2\omega_{\beta}+m_{\mu}^{-1})\bse_{2}+\bse_{3}\right)\tp\frZ\bse_{k}\\
		=&\left(\omega_{\beta}(\omega_{\beta}+m_{\mu}^{-1})\bse_{1}-(2\omega_{\beta}+m_{\mu}^{-1})\bse_{2}+\bse_{3}\right)\tp\bfu_{\beta}\bse_{1}\tp\frZ\bse_{k}+O(N^{5/2+2\epsilon})=O(N^{5/2+C\epsilon}),
	\end{split}
	\eeqs
	where we used \eqref{eq:veq_2_cancel}. Thus \eqref{eq:Y_hGG_conc} reduces to $O(N^{3/2+C\epsilon})$, so that
	\beq\label{eq:Y_der_h_conc}
	\expct{F''(Y)\bsh_{a}\adj[\wt{B},\wt{G}^{2}]\wt{B}^{k}G^{2}\bse_{a}}
	=O(N^{3/2+C\epsilon}).
	\eeq
	
	On the other hand, applying \eqref{eq:Y_GG_conc}, \eqref{eq:Y_BGG_conc}, and \eqref{eq:veq_Y} we have
	\beq\label{eq:Y_der1}\begin{split}
		&\E[F''(Y)\bse_{a}\adj[\wt{B},\wt{G}^{2}]\wt{B}^{k-1}G^{2}\bse_{a}]=\frZ_{a2k}-\frZ_{a1(k+1)}	\\
		=&\frac{(\omega_{\beta}+m_{\mu}^{-1})\bfv_{\beta}\tp\frZ\bse_{k}-\bfv_{\beta}\tp\frZ\bse_{k+1}}{N(\fra_{a}-\omega_{\alpha})^{2}}-\frac{\bfv_{\beta}\tp\frZ\bse_{k}}{N(\fra_{a}-\omega_{\alpha})}+O(N^{3/2+2\epsilon})	\\
		=&\frac{\bfv_{\beta}\tp\bfu_{\beta}}{u_{\beta 1}^{2}}\left(\frac{(\omega_{\beta}+m_{\mu}^{-1})u_{\beta k}-u_{\beta(k+1)}}{N(\fra_{a}-\omega_{\alpha})^{2}}-\frac{u_{\beta k}}{N(\fra_{a}-\omega_{\alpha})}\right)\frZ_{11}
		+O(N^{3/2+C\epsilon}).	
	\end{split}\eeq
	Plugging \eqref{eq:Y_der_h_conc} and \eqref{eq:Y_der1} into \eqref{eq:Y_der0}, for $k=1,2$ we have
	\begin{multline}\label{eq:Y_der2}
		\sum_{c}^{(a)}\E\bigg[F''(Y)\frac{\partial Y}{\partial g_{ac}}(\wt{B}^{k-1}G^{2})_{ac}\bigg]
		=\gamma\frac{\bfv_{\beta}\tp\bfu_{\beta}}{u_{\beta 1}^{2}}\left(\int_{E_{1}}^{E_{2}}\frac{\frZ_{11}(z,w)-\frZ_{11}(z,\ol{w})}{2\ii}\dd \wt{E}\right)	\\
		\times\left(-\frac{(\omega_{\beta}+m_{\mu}^{-1})u_{\beta k}-u_{\beta(k+1)}}{N(\fra_{a}-\omega_{\alpha})^{2}}+\frac{u_{\beta k}}{N(\fra_{a}-\omega_{\alpha})}\right)+O(N^{5/6+C\epsilon}),
	\end{multline}
	where we denoted $w=\wt{E}+L_{+}+\ii\eta_{0}$ and used $\absv{E_{2}-E_{1}}\lesssim N^{-2/3+\epsilon}$. Recalling the definition of $\frZ_{11}$, we have
	\beq\label{eq:ImfrZ}
	\frac{\frZ_{11}(z,w)-\frZ_{11}(z,\ol{w})}{2\ii}=\E[F''(Y)\Tr G^{2}\im[\wt{G}^{2}]].
	\eeq
	Combining \eqref{eq:Y_der2} and \eqref{eq:ImfrZ} concludes the proof of Lemma \ref{lem:Yder}.
\end{proof}

\begin{proof}[Proof of Lemma \ref{lem:Yder_dcp}]
	The proof again follows the same outline as those of Lemmas~\ref{lem:GBG}~and~\ref{lem:dBGG}, with obvious analogy between each step. Henceforth we present only the major steps and omit details for estimates. We first state the analogues of \eqref{eq:dgdiag_BG} and \eqref{eq:dgdiag_hG} as follows:
	\beq\begin{aligned}\label{eq:Yder_BG}
		&\frac{1}{N}\sum_{c}^{(a)}\frac{\partial}{\partial g_{ac}} \left[\norm{\bsg_{a}}^{-1}(\wt{B}^{\anga}R_{a}\wt{G}^{2}\wt{B}^{k-1}G^{2})_{ca}\right]	\\
		=&-(\omega_{\beta}m_{\mu}+1)(\wt{B}\wt{G}^{2}\wt{B}^{k-1}G^{2})_{aa}+\omega_{\beta}(\omega_{\beta}m_{\mu}+1)(\wt{G}^{2}\wt{B}^{k-1}G^{2})_{aa}	\\
		&-\frac{(\omega_{\beta}+m_{\mu}^{-1})}{N(\fra_{a}-\omega_{\alpha})}\Tr\wt{B}\wt{G}^{2}\wt{B}^{k-1}G^{2}+\frac{1}{N(\fra_{a}-\omega_{\alpha})}\Tr\wt{B}^{2}\wt{G}^{2}\wt{B}^{k-1}G^{2}\\
		&-(\frb_{a}-\omega_{\beta})(\omega_{\beta}m_{\mu}+1)\bsh_{a}\adj\wt{G}^{2}\wt{B}^{k}G^{2}\bse_{a}+O_{\prec}(N^{4/3+5\epsilon}),
	\end{aligned}\eeq
	\beq\begin{aligned}\label{eq:Yder_hG}
		&\frac{1}{N}\sum_{c}^{(a)}\frac{\partial}{\partial g_{ac}}\left[\norm{\bsg_{a}}^{-1}(\wt{G}^{2}\wt{B}^{k}G^{2})_{ca}\right]	\\
		=&-m_{\mu}(\wt{B}\wt{G}^{2}\wt{B}^{k-1}G^{2})_{aa} +(\omega_{\beta}m_{\mu}+1)(\wt{G}^{2}\wt{B}^{k-1}G^{2})_{aa}	\\
		&-\frac{(\omega_{\beta}+m_{\mu}^{-1})}{N(\fra_{a}-\omega_{\alpha})}\Tr \wt{G}^{2}\wt{B}^{k-1}G^{2}+\frac{1}{N(\fra_{a}-\omega_{\alpha})}\Tr \wt{B}\wt{G}^{2}\wt{B}^{k-1}G^{2}\\
		&-m_{\mu}(\omega_{\beta}+m_{\mu}^{-1}-\frb_{a})\bsh_{a}\adj \wt{G}^{2}\wt{B}^{k}G^{2}\bse_{a}+O_{\prec}(N^{4/3+5\epsilon}).
	\end{aligned}\eeq
	We only present a few remarks on the proofs of \eqref{eq:Yder_BG} and \eqref{eq:Yder_hG} and omit details, since the proof is analogous to Lemma \ref{lem:dgdiag}. The leading terms of \eqref{eq:Yder_BG} and \eqref{eq:Yder_hG} appear when the derivative hits the first ($\wt{G}$) and last ($G$) resolvent factors. Recalling \eqref{eq:dg_heu}, one can easily recover the leading terms of \eqref{eq:Yder_BG} and \eqref{eq:Yder_hG}. On the other hand, the main error arise when the derivative hits other resolvent factors. For example, the derivative hitting the second resolvent factor results in terms of the form
	\beq\label{eq:Yder_err}
	(\tr \wt{G}K_{1}\wt{G}K_{2}) (K_{3}\wt{G}K_{4}GK_{5}GK_{6})_{aa},
	\eeq
	where each $K$ is either $U$, (rank-one perturbations of) $\wt{B}$, or their product. Using Proposition \ref{prop:ll} and naive power counting, we can prove that \eqref{eq:Yder_err} is $O_{\prec}(N^{4/3+5\epsilon})$.
	
	We next move on to the analogue of Step (0) in the proof of Lemma \ref{lem:GBG}, that is, we expand $\E[F''(Y)(\wt{G}^{2}\wt{B}^{k-1}G^{2})_{aa}]=\frZ_{a1k}$:
	\beq\label{eq:Y_GG}\begin{aligned}
		\frac{w}{\gamma}\frZ_{a1k}=&\fra_{a}\frZ_{a1k}+\frZ_{a2k}+\sqrt{t}\expct{F''(Y)\bse_{a}\adj W\wt{G}^{2}\wt{B}^{k-1}G^{2}\bse_{a}}	\\
		&-\frac{1}{\gamma}\expct{F''(Y)\bse_{a}\adj\wt{G}\wt{B}^{k-1}G^{2}\bse_{a}}	\\
		=&(\fra_{a}-tm_{\mu})\frZ_{a1k}+\frZ_{a2k}-\frac{t}{N(\fra_{a}-\omega_{\alpha})}\frZ_{1k}+O(N^{1+C\epsilon}).
	\end{aligned}\eeq
	For the analogue of Step (i), we expand $\E[F''(Y)(\wt{B}\wt{G}^{2}\wt{B}^{k-1}G^{2})_{aa}]=\frZ_{a2k}$:
	\beq\label{eq:Y_BGG}
	\begin{split}
		\frZ_{a2k}=&-\frac{1}{N}\sum_{c}^{(a)}\Expct{\frac{\partial}{\partial g_{ac}} \left[F''(Y)\frac{1}{\norm{\bsg_{a}}}\bse_{c}\adj\wt{B}^{\anga}R_{a}\wt{G}^{2}\wt{B}^{k-1}G^{2}\bse_{a}\right]}
		+O(N^{7/6+C\epsilon})	\\
		=&(\omega_{\beta}m_{\mu}+1)\frZ_{a2k}-\omega_{\beta}(\omega_{\beta}m_{\mu}+1)\frZ_{a1k}	
		+\frac{(\omega_{\beta}+m_{\mu}^{-1})\frZ_{2k}-\frZ_{3k}}{N(\fra_{a}-\omega_{\alpha})}\\
		&+(\frb_{a}-\omega_{\beta})(\omega_{\beta}m_{\mu}+1)\E[F''(Y)\bsh_{a}\adj\wt{G}^{2}\wt{B}^{k}G^{2}\bse_{a}]+O_{\prec}(N^{4/3+C\epsilon}),
	\end{split}
	\eeq
	where we dropped the summand for $c=a$ in the first equality and used \eqref{eq:Yder_BG} and Lemma \ref{lem:Yder_rough} in the second. Step (ii) corresponds to the following estimate:
	\beq\begin{aligned}\label{eq:Y_hGG}
		&\E[F''(Y)\bsh_{a}\adj\wt{G}^{2}\wt{B}^{k}G^{2}\bse_{a}]	\\
		=&\frac{1}{N}\sum_{c}^{(a)}\Expct{\frac{\partial}{\partial g_{ac}}\left[F''(Y)\frac{1}{\norm{\bsg_{a}}}\bse_{c}\adj \wt{G}^{2}\wt{B}^{k-1}G^{2}\bse_{a}\right]}+O(N^{7/6+C\epsilon})\\
		=&-m_{\mu}\frZ_{a2k}+(\omega_{\beta}m_{\mu}+1)\frZ_{a1k}
		+\frac{-(\omega_{\beta}+m_{\mu}^{-1})\frZ_{1k}+\frZ_{2k}}{N(\fra_{a}-\omega_{\alpha})}\\
		&+m_{\mu}(\omega_{\beta}+m_{\mu}^{-1}-\frb_{a})\E[F''(Y)\bsh_{a}\adj \wt{G}^{2}\wt{B}^{k-1}G^{2}\bse_{a}]+O(N^{4/3+C\epsilon}).
	\end{aligned}\eeq
	Finally, taking the linear combination 
	\beqs
	\text{\eqref{eq:Y_GG}}+\text{\eqref{eq:Y_BGG}}+(\omega_{\beta}+m_{\mu}^{-1})\text{\eqref{eq:Y_hGG}}
	\eeqs
	leads to
	\beq\label{eq:Y_GG_1}\begin{aligned}
		&\left(\frac{w}{\gamma}+tm_{\mu}-m_{\mu}^{-1}-\omega_{\beta}-\fra_{a}\right)\frZ_{a1k}	\\
		=&-\frac{\bfv_{\beta}\tp\frZ\bse_{k}}{N(\fra_{a}-\omega_{\alpha})}-\frac{t}{N(\fra_{a}-\omega_{\alpha})}\frZ_{1k}+O(N^{4/3+C\epsilon}).
	\end{aligned}\eeq
	Using \eqref{eq:Bab2_conseq} and \eqref{eq:B2_factor} we conclude
	\beq
	\frZ_{a1k}=\frac{\bfv_{\beta}\tp\frZ\bse_{k}}{N(\fra_{a}-\omega_{\alpha})^{2}}+O(N^{3/2+C\epsilon}),
	\eeq
	where we absorbed the second term of \eqref{eq:Y_GG_1} into the error and used $t\ll N^{-1/6}$. This proves \eqref{eq:Y_GG_conc}, and the rest can be proved by taking suitable linear combinations of \eqref{eq:Y_GG_conc}, \eqref{eq:Y_BGG}, and \eqref{eq:Y_hGG}.
\end{proof}

\begin{proof}[Proof of Lemma \ref{lem:veq_Y}]
	The proof uses a symmetry argument as in Lemma \ref{lem:veq1}. To this end, we define a $(3\times 3)$ matrix
	\beq\label{eq:wt_frz_def}
	\wt{\frZ}_{k\ell}\equiv \wt{\frZ}_{k\ell}(z,w)\deq\expct{F''(Y)\Tr \wt{A}^{k-1}\wt{\caG}^{2}B^{\ell-1}\caG^{2}}
	=\expct{F''(Y)\Tr A^{k-1}\wt{G}^{2}\wt{B}^{\ell-1}G^{2}},
	\eeq
	where $\wt{\caG}=\caG(w)$. Note that $\wt{\frZ}_{k\ell}$ also admits the same rough estimate $N^{8/3+C\epsilon}$ as in \eqref{eq:frZ_rough}. Now we prove that the following holds true:
	\beq\label{eq:frz_eq1}
	\frZ_{k\ell}=u_{\beta k}\bfv_{\alpha}^{\intercal}\wt{\frZ}\bse_{\ell}+O_{\prec}(N^{5/2+C\epsilon}).
	\eeq
	To prove \eqref{eq:frz_eq1}, we first note that the proof of \eqref{eq:Y_GG_conc} remains intact if we replace the factor of $\wt{B}^{\ell-1}$ by $A^{\ell-1}$ in the definition of $\frZ_{ak\ell}$, so that for each $a\in\llbra 1,N\rrbra$ we have
	\beq\label{eq:frz_1}\begin{aligned}
		&\expct{F''(Y) (\wt{G}^{2}A^{\ell-1}G^{2})_{aa}}	\\
		=&\frac{1}{N(\fra_{a}-\omega_{\alpha})^{2}}\sum_{i=1}^{3}v_{\beta i}\expct{F''(Y)\Tr \wt{B}^{i-1}\wt{G}^{2}A^{\ell-1}G^{2}}+O(N^{3/2+C\epsilon}).
	\end{aligned}\eeq
	The only difference between the proofs of \eqref{eq:Y_GG_conc} and \eqref{eq:frz_1} is in the possible choices of $K$'s in \eqref{eq:Yder_err}, which does no harm. By symmetry, interchanging the roles of $A$ and $B$ in \eqref{eq:frz_1} gives that
	\beq\label{eq:frz_eq2}\begin{aligned}
		&\expct{F''(Y) (\wt{\caG}^{2}B^{\ell-1}\caG^{2})_{aa}}	\\
		=&\frac{1}{N(\frb_{a}-\omega_{\beta})^{2}}\sum_{i=1}^{3}v_{\alpha i}\expct{F''(Y)\Tr \wt{A}^{i-1}\wt{\caG}^{2}B^{\ell-1}\caG^{2}}+O(N^{3/2+C\epsilon})	\\
		=&\frac{1}{N(\frb_{a}-\omega_{\alpha})^{2}}\bfv_{\alpha}\tp \wt{\frZ}\bse_{\ell}+O(N^{3/2+C\epsilon})
	\end{aligned}\eeq
	Taking the sum over $a$ of \eqref{eq:frz_eq2} with weights $\frb_{a}^{k-1}$ and using \eqref{eq:trace identity} and \eqref{eq:frZ_rough}, we get
	\beq
	\frZ_{k\ell}=u_{\beta k}\bfv_{\alpha}\tp\wt{\frZ}_{\bse_{\ell}}+O(\bsd N^{8/3+C\epsilon}+ N^{5/2+C\epsilon})
	\eeq
	which proves \eqref{eq:frz_eq1}.
	
	We next derive the conclusion from \eqref{eq:frz_eq1}. An immediate consequence of \eqref{eq:frz_eq1} is 
	\beq\label{eq:frz_2}
	\frZ_{k\ell}=\frac{u_{\beta k}}{u_{\beta 1}}\frZ_{1\ell}+O(N^{5/2+C\epsilon}).
	\eeq
	Note from the definition of $\frZ$ that $\frZ(z,w)=\frZ(w,z)\tp$. Hence \eqref{eq:frz_2} implies 
	\beq\label{eq:frz_3}
	\frZ_{1\ell}=\frZ(w,z)_{\ell 1}=\frac{u_{\beta \ell}}{u_{\beta1}}\frZ(w,z)_{11}+O(N^{5/2+C\epsilon}).
	\eeq
	Combining \eqref{eq:frz_2} and \eqref{eq:frz_3} completes the proof of Lemma \ref{lem:veq_Y}.
\end{proof}
\section{Estimates for derivatives}\label{sec:der_tech}

Before moving on to the proofs, we first present an estimate of the sizes of entries of $\wt{B}GU$ and $\wt{B}G\wt{B}$. Its proof is presented in Appendix \ref{sec:rigidity}.
\begin{lem}\label{lem:BGU}
	The following holds uniformly over $a,b\in\llbra 1,N\rrbra$ and $E\in[E_{1},E_{2}]$:
	\beq
	\absv{(\wt{B}GU)_{ab}}\prec N^{-1/3+\epsilon},\quad (\wt{B}G\wt{B})_{ab}\prec\delta_{ab}+N^{-1/3+\epsilon},\quad \im(\wt{B}G\wt{B})_{aa}\prec N^{-1/3+\epsilon}.
	\eeq
\end{lem}

\begin{proof}[Proof of Lemma \ref{lem:dgbd_1}]
	We fix indices $a,b\in\llbra 1,N\rrbra$: Once we prove the result for fixed $a$ and $b$, it automatically holds uniformly over $a$ and $b$ by the definition of stochastic dominance and a union bound. We prove only \eqref{eq:dgBGG}, and \eqref{eq:dgGG} can be proved analogously. 
	
	Since $U^{\anga}$ and $\bsv_{a}$ are independent by Lemma \ref{lem:prd}, so are $\wt{B}^{\anga}$ and $\bsg_{a}$. Thus we have
	\begin{multline}\label{eq:dgGG_expa}
		\frac{\partial}{\partial g_{ac}}\left[\frac{1}{\norm{\bsg}^{-1}}\bse_{c}\adj \wt{B}^{\anga}R_{a}G\bse_{b}G_{ba}\right]
		=\frac{1}{\norm{\bsg_{a}}}\bse_{c}\adj \wt{B}^{\anga}R_{a} \frac{\partial(G\bse_{b}\bse_{b}\adj G)}{\partial g_{ac}}\bse_{a} 	\\
		+\frac{\partial\norm{\bsg_{a}}^{-1}}{\partial g_{ac}}\bse_{c}\adj \wt{B}^{\anga}R_{a}G\bse_{b}G_{ba}+\frac{1}{\norm{\bsg_{a}}}\bse_{c}\adj\wt{B}^{\anga}\frac{\partial R_{a}}{\partial g_{ac}}G\bse_{b}G_{ba}.
	\end{multline}
	First of all, we show that the contributions of derivatives of $\norm{\bsg_{a}}^{-1}$ and $R_{a}$ are negligible. For the derivative of $\norm{\bsg_{a}}^{-1}$ we have
	\beq\label{eq:dgGG_g_1}
	\frac{\partial\norm{\bsg_{a}}^{-1}}{\partial g_{ac}} =-\frac{1}{2\norm{\bsg_{a}}^{3}}\frac{\partial (\norm{\bsg_{a}}^{2})}{\partial g_{ac}} =-\frac{1}{2\norm{\bsg_{a}}^{2}}\ol{h}_{ac},
	\eeq
	so that
	\beq
	\frac{1}{N}\sum_{c}^{(a)}\frac{\partial\norm{\bsg_{a}}^{-1}}{\partial g_{ac}}\bse_{c}\adj \wt{B}^{\anga}R_{a}G\bse_{b}G_{ba}
	=-\frac{1}{N}\norm{\bsg_{a}}^{-2}\bsh_{a}\adj I^{(a)}\wt{B}^{\anga}R_{a}G\bse_{b}G_{ba},
	\eeq
	where we defined $I^{(a)}\deq I-\bse_{a}\bse_{a}\adj$. Using \eqref{eq:Rh=-e}, we find that
	\beq\label{eq:BR_minor1}\begin{aligned}
		\bsh_{a}\adj I^{(a)}\wt{B}^{\anga}R_{a}=(\bsh_{a}-h_{aa}\bse_{a})\adj  \wt{B}^{\anga}R_{a}=-\bse_{a}\adj\wt{B}+h_{aa}\frb_{a}\bsh_{a}\adj.
	\end{aligned}\eeq
	Since $\absv{\frb_{a}}\leq\norm{B}\lesssim 1$, we get
	\beqs
	\absv{\bsh_{a}\adj I^{(a)}\wt{B}^{\anga}R_{a}G\bse_{b}G_{ba}}
	\lesssim (1+h_{aa})(\absv{(U\adj G)_{ab}}+\absv{(\wt{B}G)_{ab}})\absv{G_{ba}}\lesssim N^{-2/3+2\epsilon}+\delta_{ab},
	\eeqs
	where we used $h_{aa}\prec N^{-1/2}$ and Proposition \ref{prop:ll}. As $\norm{\bsg_{a}}^{2}=1+O_{\prec}(N^{-1/2})$, we conclude
	\beq\label{eq:dgGG_g}
	\frac{1}{N}\sum_{c}^{(a)}\frac{\partial\norm{\bsg_{a}}^{-1}}{\partial g_{ac}}\bse_{c}\adj \wt{B}^{\anga}R_{a}G\bse_{b}G_{ba}
	\prec N^{-1}(N^{-2/3+2\epsilon}+\delta_{ab}).
	\eeq
	
	Next, for the derivative of $R_{a}$, by \eqref{eq:Rder} we have 
	\beqs
	\frac{\partial R_{a}}{\partial g_{ac}}=-\frac{\ell_{a}^{2}}{2\norm{\bsg_{a}}}h_{aa}\ol{h}_{ac}\bsr_{a}\bsr_{a}\adj 
	-\frac{\ell_{a}^{2}}{\norm{\bsg_{a}}}\left(\bse_{c}(\bse_{a}+\bsh_{a})\adj-\frac{\ol{h}_{ac}}{2}(\bsh_{a}\bse_{a}\adj+\bse_{a}\bsh_{a}\adj+2\bsh_{a}\bsh_{a}\adj)\right),
	\eeqs
	so that
	\beq\label{eq:Rder_expa}\begin{aligned}
		\frac{1}{N}\sum_{c}^{(a)}\bse_{c}\adj\wt{B}^{\anga}\frac{\partial R_{a}}{\partial g_{ac}}
		=&-\frac{\ell_{a}^{2}}{2\norm{\bsg_{a}}}\frac{h_{aa}}{N}\bsh_{a}\adj I^{(a)}\wt{B}^{\anga}\bsr_{a}\bsr_{a}\adj 
		-\frac{\ell_{a}^{2}}{\norm{\bsg_{a}}}\tr I^{(a)}\wt{B}^{\anga} (\bse_{a}+\bsh_{a})\adj \\
		&+\frac{\ell_{a}^{2}}{2\norm{\bsg_{a}}}\frac{1}{N}\bsh_{a}\adj I^{(a)}\wt{B}^{\anga}(\bsh_{a}\bse_{a}\adj+\bse_{a}\bsh_{a}\adj+2\bsh_{a}\bsh_{a}\adj).
	\end{aligned}\eeq
	Note that 
	\beq\begin{aligned}\label{eq:Rder_expa_1}
		&\absv{\bsh_{a}\adj I^{(a)}\wt{B}^{\anga}\bse_{a}}+\absv{\bsh_{a}\adj I^{(a)}\wt{B}^{\anga}\bsh_{a}}\leq 2\norm{B},\\
		&\absv{\tr I^{(a)}\wt{B}^{(a)}}=\Absv{\tr B-\frac{\frb_{a}}{N}}\leq \frac{\norm{B}}{N}.
	\end{aligned}\eeq
	Plugging \eqref{eq:Rder_expa} into the last term of \eqref{eq:dgGG_expa} and then using \eqref{eq:Rder_expa_1}, we have
	\beq\begin{aligned}\label{eq:dgGG_R}
		\Absv{\frac{1}{N}\sum_{c}^{(a)}\frac{1}{\norm{\bsg_{a}}}\bse_{c}\adj\wt{B}^{\anga}\frac{\partial R_{a}}{\partial g_{ac}}G\bse_{b}G_{ba}}
		\lesssim& \frac{\ell_{a}^{2}}{\norm{\bsg_{a}}^{2}}\frac{(1+h_{aa})}{N}(\absv{(U\adj G)_{ab}}+\absv{G_{ab}})\absv{G_{ba}}	\\
		\prec& N^{-1}(N^{-2/3+2\epsilon}+\delta_{ab}),
	\end{aligned}\eeq
	where in the second line we used Proposition \ref{prop:ll} and 
	\beq\label{eq:gl_1}
	\norm{\bsg_{a}}=1+O_{\prec}(N^{-1/2}),\qquad  \ell_{a}^{2}=\frac{1}{1+h_{aa}}=\frac{1}{1+O_{\prec}(N^{-1/2})}.
	\eeq
	By \eqref{eq:dgGG_expa}, \eqref{eq:dgGG_g}, and \eqref{eq:dgGG_R}, we conclude
	\beq\begin{aligned}\label{eq:dgGG_gR}
		\frac{1}{N}\sum_{c}^{(a)}\frac{\partial}{\partial g_{ac}}\left[\frac{1}{\norm{\bsg_{a}}}\bse_{c}\adj \wt{B}^{\anga}R_{a}G\bse_{b}\adj G_{ba}\right]
		=&\frac{1}{N}\sum_{c}^{(a)}\frac{1}{\norm{\bsg_{a}}}\bse_{c}\adj\wt{B}^{\anga}R_{a}\frac{\partial (G\bse_{b}\bse_{b}\adj G)}{\partial g_{ac}}\bse_{a}	\\
		&+O_{\prec}(N^{-1}(N^{-2/3+2\epsilon}+\delta_{ab})).
	\end{aligned}\eeq
	
	Using Lemma \ref{lem:dgbd}, we extract the leading term from the derivative $\partial G/(\partial g_{ac})$ as follows:
	\beq
	\frac{1}{\norm{\bsg_{a}}}\frac{\partial G}{\partial g_{ac}} =-\gamma G[\bse_{c}(\bse_{a}+\bsh_{a})\adj,\wt{B}]G +G\Delta_{c}G,
	\eeq
	where $\Delta_{c}\equiv\Delta_{c}(z)\in\C^{N\times N}$ is defined by
	\beq\label{eq:dg_Delta}\begin{aligned}
		\Delta_{c}&\deq\Delta_{c1}+\Delta_{c2}+\Delta_{c3},	\\
		\Delta_{c1}&\deq\gamma\left(1-\frac{\ell_{a}^{2}}{\norm{\bsg_{a}}^{2}}\right)[\bse_{c}(\bse_{a}+\bsh_{a})\adj,\wt{B}],\\
		\Delta_{c2}&\deq \gamma\frac{\ell_{a}^{2}}{2\norm{\bsg_{a}^{2}}}\ol{h}_{ac}[(\bse_{a}+2\bsh_{a})\bse_{a}\adj,\wt{B}],\\
		\Delta_{c3}&\deq -\gamma\frac{\ell_{a}^{4}}{2\norm{\bsg_{a}}^{2}}h_{aa}\ol{h}_{ac}[\bse_{a}\bse_{a}\adj+\bse_{a}\bsh_{a}\adj+\bsh_{a}\bse_{a}\adj,\wt{B}]. 
	\end{aligned}\eeq
	Then, for any matrix $K$, we have
	\begin{align}
		&\frac{1}{N}\sum_{c}^{(a)}\frac{1}{\norm{\bsg_{a}}}\bse_{c}\adj K\frac{\partial(G\bse_{b}\bse_{b}\adj G)}{\partial g_{ac}}\bse_{a}	\label{eq:dgKGG}\\
		=&-\gamma(\tr I^{(a)}KG)(\bse_{a}+\bsh_{a})\adj\wt{B}G\bse_{b}G_{ba}	
		+\gamma(\tr I^{(a)}KG\wt{B})(\bse_{a}+\bsh_{a})\adj G\bse_{b}G_{ba}	\label{eq:dgKGG_2}\\
		&-\gamma \frac{(GI^{(a)}KG)_{bb}}{N}(\bse_{a}+\bsh_{a})\adj\wt{B}G\bse_{a}
		+\gamma \frac{(G\wt{B}I^{(a)}KG)_{bb}}{N}(\bse_{a}+\bsh_{a})\adj G\bse_{a}\label{eq:dgKGG_3}\\
		&+\frac{1}{N}\sum_{c}^{(a)}\bse_{c}\adj KG \left(\Delta_{c}G\bse_{b}\bse_{b}\adj+\bse_{b}\bse_{b}\adj G\Delta_{c}\right)G\bse_{a}\label{eq:dgKGG_4},
	\end{align}
	where we recall $I^{(a)}= I-\bse_{a}\bse_{a}\adj$. From now on we always take $K=\wt{B}^{\anga}R_{a}$.
	
	Next, we simplify \eqref{eq:dgKGG_2} -- \eqref{eq:dgKGG_3}. Note that $\wt{B}^{\anga}R_{a}=R_{a}\wt{B}$ and \eqref{eq:Rh=-e} imply
	\beq\label{eq:BR_minor}\begin{aligned}
		I^{(a)}\wt{B}^{\anga}R_{a}-\wt{B}=(I^{(a)}-I)\wt{B}^{\anga}R_{a}+(R_{a}-I)\wt{B}
		=-\frb_{a}\bse_{a}\bsh_{a}\adj+\bsr_{a}\bsr_{a}\adj\wt{B}.
	\end{aligned}\eeq
	Since $\norm{B}\lesssim 1$, plugging \eqref{eq:BR_minor} into tracial prefactors in \eqref{eq:dgKGG_2} gives
	\beq\label{eq:dgKGG_K=B}\begin{aligned}
		&\absv{\tr (I^{(a)}\wt{B}^{\anga}R_{a}-\wt{B})G}	
		\lesssim\frac{\absv{(\wt{B}G)_{aa}}+\absv{(\wt{B}GU)_{aa}}+\absv{\bsh_{a}\adj G\bse_{a}}+\absv{\caG_{aa}}}{N},\\
		&\absv{\tr (I^{(a)}\wt{B}^{\anga}R_{a}-\wt{B})G\wt{B})}
		\lesssim\frac{\absv{(\wt{B}G\wt{B})_{aa}}+\absv{(\wt{B}GU)_{aa}}+\absv{(U\adj G\wt{B})_{aa}}+\absv{\caG_{aa}}}{N}.
	\end{aligned}\eeq
	All matrix entries appearing in the numerators of \eqref{eq:dgKGG_K=B} are $O_{\prec}(1)$: We dealt with $(\wt{B}GU)_{aa}$ and $(\wt{B}G\wt{B})_{aa}$ in Lemma \ref{lem:BGU}, and the rest are direct consequences of Proposition \ref{prop:ll} together with the identity $G\adj=G(\ol{z})$. Thus we have
	\beq\label{eq:pre_tr}
	\absv{\tr (I^{(a)}\wt{B}^{\anga}R_{a}-\wt{B})G}	+\absv{\tr (I^{(a)}\wt{B}^{\anga}R_{a}-\wt{B})G\wt{B})}\prec \frac{1}{N}.
	\eeq
	For prefactors in \eqref{eq:dgKGG_3} that are $(b,b)$-th entries, we have
	\beq\begin{aligned}
		&\frac{\absv{(G(I^{(a)}\wt{B}^{\anga}R_{a}-\wt{B})G)_{bb}}}{N}
		\lesssim \frac{(\absv{G_{ba}}+\absv{(GU)_{ba}})(\absv{(U\adj G)_{ab}}+\absv{(\wt{B}G)_{ab}})}{N},\\
		&\frac{\absv{(G(\wt{B}I^{(a)}\wt{B}^{\anga}R_{a}-\wt{B})\wt{B}G)_{bb}}}{N}	
		\lesssim\frac{(\absv{(G\wt{B})_{ba}}+\absv{(GU)_{ba}})(\absv{(U\adj G)_{ab}}+\absv{(\wt{B}G)_{ab}})}{N}.
	\end{aligned}\eeq
	Hence a direct application of Proposition \ref{prop:ll} proves
	\beq\label{eq:pre_bb}
	\frac{\absv{(G(I^{(a)}\wt{B}^{\anga}R_{a}-\wt{B})G)_{bb}}}{N}+\frac{\absv{(G(\wt{B}I^{(a)}\wt{B}^{\anga}R_{a}-\wt{B})\wt{B}G)_{bb}}}{N}\prec N^{-1}(N^{-2/3+2\epsilon}+\delta_{ab}).
	\eeq
	Also recall from Proposition \ref{prop:ll} that
	\beq\label{eq:dgKGG_main}\begin{aligned}
		\absv{(\wt{B}G)_{ab}G_{ba}}+\absv{G_{ab}G_{ba}}\prec &\, \delta_{ab}+O_{\prec}(N^{-2/3+2\epsilon}),	\\
		\absv{(\bse_{a}+\bsh_{a})\adj\wt{B}G\bse_{a}}+\absv{(\bse_{a}+\bsh_{a})\adj G\bse_{a}}\prec &\,1.
	\end{aligned}\eeq
	Combining \eqref{eq:pre_tr}, \eqref{eq:pre_bb}, and \eqref{eq:dgKGG_main}, we conclude
	\beq\label{eq:dgBGG_main}\begin{aligned}
		\text{\eqref{eq:dgKGG_2}+\eqref{eq:dgKGG_3}}
		=&-\gamma\tr \wt{B}G(\bse_{a}+\bsh_{a})\adj\wt{B}G\bse_{b}G_{ba}	
		+\gamma\tr \wt{B}^{2}G(\bse_{a}+\bsh_{a})\adj G\bse_{b}G_{ba}	\\
		&-\gamma \frac{(G\wt{B}G)_{bb}}{N}(\bse_{a}+\bsh_{a})\adj\wt{B}G\bse_{a}
		+\gamma \frac{(G\wt{B}^{2}G)_{bb}}{N}(\bse_{a}+\bsh_{a})\adj G\bse_{a}\\
		&+O_{\prec}(N^{-1}(\delta_{ab}+N^{-2/3+2\epsilon})).
	\end{aligned}\eeq
	
	We next estimate the contribution of $\Delta_{c1}$ in \eqref{eq:dgKGG_4}, that is, we aim at proving
	\beq\label{eq:Delta1_goal}
	\frac{1}{N}\sum_{c}^{(a)}\bse_{c}\adj \wt{B}^{\anga}R_{a}G(\Delta_{c1}G\bse_{b}\bse_{b}\adj+\bse_{b}\bse_{b}\adj G\Delta_{c1})G\bse_{a}\prec N^{-1/2}(N^{-2/3+2\epsilon}+\delta_{ab}).
	\eeq
	By the definition of $\Delta_{c1}$ in \eqref{eq:dg_Delta}, we have
	\beq\label{eq:Delta1}
	\frac{1}{N}\sum_{c}^{(a)}\bse_{c}\adj \wt{B}^{\anga}R_{a}G(\Delta_{c1}G\bse_{b}\bse_{b}\adj+\bse_{b}\bse_{b}\adj G\Delta_{c1})G\bse_{a}
	=\left(1-\frac{\ell_{a}^{2}}{\norm{\bsg_{a}}}\right)\cdot(\text{\eqref{eq:dgKGG_2}+\eqref{eq:dgKGG_3}}).
	\eeq
	Recall from \eqref{eq:gl_1} that
	\beq\label{eq:Delta1_pre}
	1-\frac{\ell_{a}^{2}}{\norm{\bsg_{a}}}=O_{\prec}(N^{-1/2}).
	\eeq
	On the other hand, recall from \eqref{eq:x_rough_1} and \eqref{eq:trace identity} that
	\beq\label{eq:Delta1_11}
	\absv{\tr\wt{B}^{k}G}\prec 1,\qquad \frac{(G\wt{B}^{k}G)_{bb}}{N}\prec N^{-2/3+2\epsilon},\qquad\quad k=1,2.
	\eeq
	Plugging in \eqref{eq:Delta1_11} and \eqref{eq:dgKGG_main} to \eqref{eq:dgBGG_main} gives
	\beq\label{eq:Delta1_main}
	\text{\eqref{eq:dgKGG_2}+\eqref{eq:dgKGG_3}}\prec \delta_{ab}+N^{-2/3+2\epsilon}.
	\eeq
	Combining \eqref{eq:Delta1}, \eqref{eq:Delta1_pre}, and \eqref{eq:Delta1_main} proves \eqref{eq:Delta1_goal}.
	
	We now move on to the contribution of $\Delta_{c2}$ in \eqref{eq:dgKGG_4}, and we claim that (recall $K=\wt{B}^{\anga}R_{a}$)
	\beq\label{eq:Delta2_goal}
	\frac{1}{N}\sum_{c}^{(a)}\bse_{c}\adj KG(\Delta_{c2}G\bse_{b}\bse_{b}\adj+\bse_{b}\bse_{b}\adj G\Delta_{c2}) G\bse_{a}\prec N^{-1}(N^{-2/3+2\epsilon}+\delta_{ab}).
	\eeq
	We first focus on the first term on the left-hand side of \eqref{eq:Delta2_goal}. Here we use the prefactor $\ol{h}_{ac}$ in the definition of $\Delta_{c2}$ to write
	\beq\label{eq:Delta2_prf_1}
	\frac{1}{N}\sum_{c}^{(a)}\bse_{c}\adj KG\Delta_{c2}G\bse_{b}G_{ba}
	=\frac{\gamma\ell_{a}^{2}}{2\norm{\bsg_{a}}}\frac{1}{N}\bsh_{a}\adj I^{(a)}KG[(\bse_{a}+2\bsh_{a})\bse_{a}\adj,\wt{B}]G\bse_{b}G_{ba}.
	\eeq	
	Recalling \eqref{eq:BR_minor}, we have
	\beq\begin{aligned}
		&\absv{\bsh_{a}\adj I^{(a)}KG[(\bse_{a}+2\bsh_{a})\bse_{a}\adj,\wt{B}]\bse_{b}}
		=\absv{\bsh_{a}\adj(\bsr_{a}\bsr_{a}\adj -\bse_{a}\bsh_{a}\adj)\wt{B}G[(\bse_{a}+2\bsh_{a})\bse_{a}\adj,\wt{B}]\bse_{b}}	\\
		\lesssim& (1+h_{aa})\left(\sum_{K_{1},K_{2}\in\{I,U,\wt{B}\}}\absv{(K_{1}\adj GK_{2})_{aa}}\right)(\absv{G_{ab}}+\absv{(\wt{B}G)_{ab}}).
	\end{aligned}\eeq
	As in \eqref{eq:pre_tr}, Proposition \ref{prop:ll} and Lemma \ref{lem:BGU} implies that the middle factor consisting of $(a,a)$-th entries is $O_{\prec}(1)$. Thus we conclude
	\beq
	\frac{1}{N}\sum_{c}^{(a)}\bse_{c}\adj KG\Delta_{c}G\bse_{b}G_{ba}\prec \frac{1}{N}\absv{G_{ab}}+\absv{(\wt{B}G)_{ab}})\absv{G_{ba}}\prec \frac{1}{N}(N^{-2/3+2\epsilon}+\delta_{ab}),
	\eeq
	where we used \eqref{eq:gl_1} and $h_{aa}\prec N^{-1/2}$ in the first inequality.
	Similarly, for the second term of \eqref{eq:Delta2_goal} we have
	\beq\label{eq:Delta2_prf_2}\begin{aligned}
		&\frac{1}{N}\sum_{c}^{(a)}(KG)_{cb}(G\Delta_{c}G)_{ba}=\frac{\gamma\ell_{a}^{2}}{2\norm{\bsg_{a}}}\frac{1}{N}\bsh_{a}\adj I^{(a)}KG\bse_{b}\bse_{b}\adj G[(\bse_{a}+2\bsh_{a})\bse_{a}\adj,\wt{B}]G\bse_{a}	\\
		\prec& N^{-1}(\absv{(\wt{B}G)_{ab}}+\absv{(U\adj G)_{ab}})(\absv{(G\wt{B})}_{ba}+\absv{G}_{ba}+\absv{GU}_{ba})(\absv{(\wt{B}G)_{aa}}+\absv{G_{aa}})	\\
		\prec& N^{-1}(N^{-2/3+2\epsilon}+\delta_{ab}).
	\end{aligned}\eeq
	Thus we have proved \eqref{eq:Delta2_goal}.
	
	Finally, since $\Delta_{c3}$ also has the same $\ol{h}_{ac}$ factor, we can easily prove \eqref{eq:Delta2_goal} with $\Delta_{c2}$ replaced by $\Delta_{c3}$ following \eqref{eq:Delta2_prf_1} -- \eqref{eq:Delta2_prf_2}; we omit further details. Combining \eqref{eq:dgGG_gR}, \eqref{eq:dgKGG}, \eqref{eq:dgBGG_main}, \eqref{eq:Delta1_goal}, and \eqref{eq:Delta2_goal} completes the proof of \eqref{eq:dgBGG}.
\end{proof}

\begin{proof}[Proof of Lemma \ref{lem:Yder_rough}]
	By the definition of $Y$ we have
	\beq
	\frac{\partial Y}{\partial g_{ac}}=\frac{1}{2\ii}\int_{E_{1}}^{E_{2}}\Tr\frac{\partial}{\partial g_{ac}} (\wt{G}-\wt{G}\adj)\dd\wt{E}.
	\eeq
	Then we use \eqref{eq:dgbd} and the fact that $\bsh_{a}\bsh_{a}\adj$ commutes with $\wt{B}$ to obtain
	\beq\begin{aligned}\label{eq:Yder_rough_0}
		\Tr \frac{\partial \wt{G}}{\partial g_{ac}}=&-\gamma\frac{\ell_{a}^{2}}{\norm{\bsg_{a}}}(\bse_{a}+\bsh_{a})\adj[\wt{B},\wt{G}^{2}]\bse_{c}+\gamma\ol{h}_{ac}\frac{\ell_{a}^{2}}{2\norm{\bsg_{a}}}\bse_{a}\adj[\wt{B},G^{2}](\bse_{a}+2\bsh_{a}) \\
		&-\gamma h_{aa}\ol{h}_{ac}\frac{\ell_{a}^{4}}{2\norm{\bsg_{a}}}(\bse_{a}+\bsh_{a})\adj [\wt{B},G^{2}](\bse_{a}+\bsh_{a}).
	\end{aligned}\eeq
	Now using Proposition \ref{prop:ll}, we see that
	\beq\begin{aligned}\label{eq:Yder_rough_1}
		\absv{\bse_{a}\adj[\wt{B},G^{2}]\bse_{c}} \leq& \absv{(\wt{B}\wt{G}^{2})_{ac}}+\absv{(\wt{G}^{2}\wt{B})_{ac}}	\\
		\leq&\sum_{b}(\absv{(\wt{B}\wt{G})_{ab}}+\absv{\wt{G}_{ab}})(\absv{(\wt{G}\wt{B})_{bc}}+\absv{\wt{G}_{ac}})\prec N^{1/3+2\epsilon}.
	\end{aligned}\eeq
	Similarly we have
	\beq\label{eq:Yder_rough_2}
	\max_{K_{1},K_{2}\in \{I,U\}}\absv{\bse_{a}\adj K_{1}[\wt{B},\wt{G}^{2}]K_{2}\bse_{a}}\prec N^{1/3+2\epsilon}.
	\eeq
	Recalling \eqref{eq:gl_1} and $h_{aa},\absv{h_{ac}}\prec N^{-1/2}$, we plug in \eqref{eq:Yder_rough_1} and \eqref{eq:Yder_rough_2} to \eqref{eq:Yder_rough_0} so that
	\beq
	\Tr\frac{\partial \wt{G}}{\partial g_{ac}}\prec N^{1/3+2\epsilon}.
	\eeq
	The exact same argument gives
	\beq
	\Absv{\Tr\frac{\partial \wt{G}\adj}{\partial g_{ac}}}=\Absv{\Tr\frac{\partial\wt{G}}{\partial\ol{g}_{ac}}}\prec N^{1/3+2\epsilon},
	\eeq
	except that we apply \eqref{eq:dgbd_bar} instead of \eqref{eq:dgbd}. As a result we have
	\beq
	\frac{\partial Y}{\partial g_{ac}}\prec (E_{2}-E_{1})N^{1/3+2\epsilon}\leq N^{-1/3+3\epsilon}.
	\eeq
	This completes the proof of Lemma \ref{lem:Yder_rough}.
\end{proof}

\begin{proof}[Proof of Lemma \ref{lem:dgdiag}]
	We only present the proof of \eqref{eq:dgdiag_BG}, and leave that of \eqref{eq:dgdiag_hG} to interested readers. As in the proof of Lemma \ref{lem:dgbd_1}, we again start with derivatives of $\norm{\bsg_{a}}^{-1}$ and $R_{a}$. To this end, we simply take $b=a$ and drop the last factor $G_{ba}$ respectively in \eqref{eq:dgGG_g_1} -- \eqref{eq:dgGG_g} and \eqref{eq:dgGG_R}. As a result, we get
	\beq\label{eq:dgdiag_g}\begin{aligned}
		\frac{1}{N}\sum_{c}^{(a)}\frac{\partial\norm{\bsg_{a}}^{-1}}{\partial g_{ac}}\bse_{c}\adj\wt{B}^{\anga}R_{a}G\bse_{a}&\prec N^{-1}, \\
		\frac{1}{N}\sum_{c}^{(a)}\frac{1}{\norm{\bsg_{a}}}\bse_{c}\adj\wt{B}^{\anga}\frac{\partial R_{a}}{\partial g_{ac}}G\bse_{a}&\prec N^{-1}.
	\end{aligned}\eeq
	
	Recall the definitions of $\Delta_{c}$ from \eqref{eq:dg_Delta}. Using \eqref{eq:dgdiag_g} and then \eqref{eq:pre_tr} with \eqref{eq:ll_etr}, we have
	\begin{multline}\label{eq:dgdiag_main}
		\frac{1}{N}\sum_{c}^{(a)}\frac{\partial}{\partial g_{ac}}\left[\norm{\bsg_{a}}^{-1}(\wt{B}^{\anga}R_{a}G)_{ca}\right]	\\
		=-\gamma(\tr\wt{B}G)(\bse_{a}+\bsh_{a})\adj \wt{B}G\bse_{a}+\gamma (\tr\wt{B}^{2}G)(\bse_{a}+\bsh_{a})\adj G\bse_{a}	\\
		+\frac{1}{N}\sum_{c}^{(a)}\bse_{c}\adj \wt{B}^{\anga}G\Delta_{c}G\bse_{a}+O_{\prec}(N^{-1}).
	\end{multline}
	We next show that the third term of \eqref{eq:dgdiag_main} is small. As in \eqref{eq:dgdiag_g}, we only need to take $a=b$ and remove the factor $G_{ba}$ in the first terms of \eqref{eq:Delta1_goal} and \eqref{eq:Delta2_goal}. Consequently we obtain
	\beq
	\frac{1}{N}\sum_{c}^{(a)}\bse_{c}\adj\wt{B}^{\anga}G\Delta_{c}G\bse_{a}\prec N^{-1/2}.
	\eeq
	
	Finally, using $\wt{B}\bsh_{a}=\frb_{a}\bsh_{a}$, Proposition \ref{prop:ll}, $\frd_{k}\prec N^{-1/3+\epsilon}$, and \eqref{eq:trace identity}, we may further write the leading term in \eqref{eq:dgdiag_main} as
	\beq\begin{aligned}
		&-\gamma(\tr\wt{B}G)(\bse_{a}+\bsh_{a})\adj \wt{B}G\bse_{a}+\gamma (\tr\wt{B}^{2}G)(\bse_{a}+\bsh_{a})\adj G\bse_{a}	\\
		=&-(\omega_{\beta}m_{\mu}+1)(\wt{B}G)_{aa}+(\omega_{\beta}(\omega_{\beta}m_{\mu}+1)+\frd_{3})G_{aa}	
		+\frac{-(\omega_{\beta}+m_{\mu}^{-1})\frd_{2}+\frd_{3}}{\fra_{a}-\omega_{\alpha}}	\\
		&+(\omega_{\beta}m_{\mu}+1)(\omega_{\beta}-\frb_{a})\bsh_{a}\adj G\bse_{a}+O_{\prec}(N^{-2/3+2\epsilon}).
	\end{aligned}\eeq
	This completes the proof of Lemma \ref{lem:dgdiag}.
\end{proof}

\begin{proof}[Proof of Lemma \ref{lem:dg_cr}]
	We only prove \eqref{eq:dg_cr_BG} and \eqref{eq:dg_cr_BGG} to avoid repetition. We start with \eqref{eq:dg_cr_BG}, and first claim that it suffices to prove the following: For any (possibly random) Hermitian matrix $K$ with $\norm{K}\prec 1$ independent of $U$, we have
	\begin{align}\label{eq:dgdiag_tr_KGG}
		\frac{1}{N}\sum_{c}^{(a)}\frac{\partial \Tr GKG}{\partial g_{ac}}\bse_{c}\adj\wt{B}^{\anga}R_{a}G\bse_{a}\prec N^{2/3+C\epsilon},	\\
		\frac{1}{N}\sum_{c}^{(a)}\frac{\partial\Tr KG}{\partial g_{ac}}\bse_{c}\adj\wt{B}^{\anga}R_{a}G\bse_{a} \prec N^{2/3+C\epsilon}.\label{eq:dgdiag_tr_KG}
	\end{align}
	Indeed, we may express $\Tr G\wt{B}^{k-1}G$ as a linear combination of $\Tr GKG$ and $\Tr KG$ using the identity $\gamma HG=zG+I=\gamma GH$. For example when $k=3$, we have
	\beq\begin{aligned}
		\Tr G\wt{B}^{2}G=\Tr GK^{2}G+2\Tr KG+1,\\
		K=\frac{z-H}{\gamma}+\wt{B}=\frac{z}{\gamma}-A-\sqrt{t}W,
	\end{aligned}\eeq
	and obviously the matrix $K$ above is independent of $U$ with norm $O_{\prec}(1)$. Hence \eqref{eq:dg_cr_BG} follows immediately from \eqref{eq:dgdiag_tr_KGG} and \eqref{eq:dgdiag_tr_KG}.
	
	To prove \eqref{eq:dgdiag_tr_KGG}, we notice from Lemma \ref{lem:dgbd} that the derivative $\partial\Tr (GKG)/(\partial g_{ac})$ is a finite linear combination of
	\beq\label{eq:dgdiag_tr_KGG_main}
	\begin{cases}
		\Tr KG[\bse_{c}\bse_{a}\adj K_{1},\wt{B}]G^{2}, \\
		\Tr KG^{2}[\bse_{c}\bse_{a}\adj K_{1},\wt{B}]G,\\
		\ol{h}_{ac}\Tr KG[K_{1}\bse_{a}\bse_{a}\adj K_{2},\wt{B}]G^{2},	\\
		\ol{h}_{ac}\Tr KG^{2}[K_{1}\bse_{a}\bse_{a}\adj K_{2},\wt{B}]G,
	\end{cases}\qquad K_{1},K_{2}\in \{I,U\},
	\eeq
	with $O_{\prec}(1)$ weights. By the identity $\Tr X[Y,Z]=\Tr Y[Z,X]$, \eqref{eq:dgdiag_tr_KGG_main} implies that the left-hand side of \eqref{eq:dgdiag_tr_KGG} is given by a linear combination of
	\beq\label{eq:dgcr_tr_KGG}
	\begin{cases}
		N^{-1}\bse_{a}\adj K_{1}[\wt{B},G^{2}KG] I^{(a)}\wt{B}^{\anga}R_{a}G\bse_{a},	\\
		N^{-1}\bse_{a}\adj K_{1}[\wt{B}, GKG^{2}]I^{(a)}\wt{B}^{\anga}R_{a}G\bse_{a},\\
		N^{-1}\bsh_{a}\adj I^{(a)}\wt{B}^{\anga}R_{a}G\bse_{a} \bse_{a}\adj K_{2}[G^{2}KG,\wt{B}]K_{1}\bse_{a},	\\
		N^{-1}\bsh_{a}\adj I^{(a)}\wt{B}^{\anga}R_{a}G\bse_{a} \bse_{a}\adj K_{2}[GKG^{2},\wt{B}]K_{1}\bse_{a},
	\end{cases}\qquad K_{1},K_{2}\in \{I,U\}.
	\eeq
	
	Notice from Ward identity, Proposition \ref{prop:ll}, and Lemma \ref{lem:BGU} that 
	\beq\label{eq:ward}
	\norm{GK_{1}\bse_{a}}+\norm{G\wt{B}K_{1}\bse_{a}}\leq \sqrt{\frac{\im (K_{1}GK_{1})_{aa}+(K_{1}\wt{B}G\wt{B}K_{1})_{aa}}{\eta_{0}}}\prec N^{1/6+\epsilon}.
	\eeq
	Then we apply \eqref{eq:ward} and Cauchy-Schwarz, for example to a term in the first quantity of \eqref{eq:dgcr_tr_KGG}, to get
	\beq\label{eq:dg_cr_1}\begin{aligned}
		&N^{-1}\absv{\bse_{a}\adj K_{1}\wt{B}G^{2}KG I^{(a)}\wt{B}^{\anga}R_{a}G\bse_{a}}	\\
		\leq& N^{-1}\norm{G\wt{B}K_{1}\bse_{a}}\norm{GKGI^{(a)}\wt{B}^{\anga}}\norm{G\bse_{a}}\prec N^{-2/3+2\epsilon}\cdot \eta_{0}^{-2}
		\prec N^{2/3+4\epsilon},
	\end{aligned}\eeq
	where we applied $\norm{G}\leq \eta_{0}^{-1}$. For the third quantity of \eqref{eq:dgcr_tr_KGG}, we use \eqref{eq:BR_minor1} so that
	\beq
	\bsh_{a}\adj I^{(a)}\wt{B}^{\anga}R_{a}G\bse_{a}=-(\wt{B}G)_{aa}+\bsh_{aa}\frb_{a}\bsh_{a}\adj G\bse_{a}\prec 1.
	\eeq
	Therefore, for a term in the third quantity of \eqref{eq:dgcr_tr_KGG}, we have
	\beq\label{eq:dg_cr_2}\begin{aligned}
		&N^{-1}(\bsh_{a}\adj I^{(a)}\wt{B}^{\anga}R_{a}G\bse_{a})(\bse_{a}K_{2}G^{2}KG\wt{B}K_{1}\bse_{a})	\\
		\prec &N^{-1}\eta_{0}^{-1}\norm{GK_{2}\bse_{a}}\norm{G\wt{B}K_{1}\bse_{a}}\prec N^{4\epsilon}.
	\end{aligned}\eeq 
	Similarly all quantities in \eqref{eq:dgcr_tr_KGG}, as well as \eqref{eq:dgdiag_tr_KG}, can be estimated as $O_{\prec}(N^{2/3+4\epsilon})$. This concludes the proof of \eqref{eq:dg_cr_BG}.
	
	Now for \eqref{eq:dg_cr_BGG}, by the same reasoning as in \eqref{eq:dgdiag_tr_KGG}, it suffices to prove
	\beq\label{eq:dg_cr_KG_main}
	\sum_{c}^{(a)}\frac{\partial \Tr KG}{\partial g_{ac}}\bse_{c}\adj\wt{B}^{\anga}R_{a}G^{2}\bse_{a}\prec N^{5/3+C\epsilon}.
	\eeq
	Then, as in \eqref{eq:dgcr_tr_KGG}, after some algebra we find that the left-hand side of \eqref{eq:dg_cr_KG_main} is a linear combination of 
	\beq\label{eq:dg_cr_KG}
	\begin{cases}
		\bse_{a}\adj K_{1}[\wt{B},GKG]I^{(a)}\wt{B}^{\anga}R_{a}G^{2}\bse_{a},	\\
		\bsh_{a}\adj I^{(a)}\wt{B}^{\anga}R_{a}G\bse_{a}\bse_{a}\adj K_{2}[GKG,\wt{B}]K_{1}\bse_{a},
	\end{cases}\qquad K_{1},K_{2}\in\{I,U\}.
	\eeq
	All quantities in \eqref{eq:dg_cr_KG} can be estimated with exactly the same argument as in \eqref{eq:ward} -- \eqref{eq:dg_cr_2}. This completes the proof of Lemma \ref{lem:dg_cr}.
\end{proof}

\begin{proof}[Proof of Lemma \ref{lem:Y_1}]
	Using the identity $\Tr X[Y,Z]= \Tr Y[Z,X]$, we have
	\beq\begin{aligned}
		\sum_{c}^{(a)}\frac{\partial\Tr \wt{G}}{\partial g_{ac}}(\wt{B}^{k}G^{2})_{ca}
		=&-\gamma(\bse_{a}+\bsh_{a})\adj[\wt{B},\wt{G}^{2}]I^{(a)}\wt{B}^{k}G^{2}\bse_{a}	\\
		&+\sum_{c}^{(a)}(\Tr G^{2}\Delta_{c})(\wt{B}^{k}G^{2})_{ca}.
	\end{aligned}\eeq
	Recalling $k\in\{0,1\}$ and $\absv{(G^{2}\wt{B})_{aa}}+\absv{(\wt{B}G^{2})_{aa}}+\absv{(G^{2})_{aa}}\prec N^{1/3+2\epsilon}$ due to \eqref{eq:dgKGG_main},
	\beq\begin{aligned}\label{eq:Y_1_main}
		&(\bse_{a}+\bsh_{a})\adj [\wt{B},\wt{G}^{2}](I-I^{(a)})\wt{B}^{k}G^{2}\bse_{a} 	\\
		=&(\bse_{a}+\bsh_{a})\adj [\wt{B},\wt{G}^{2}]\bse_{a} \bse_{a}\adj(\wt{B}^{k}G^{2})\bse_{a}	
		\prec N^{2/3+4\epsilon}.
	\end{aligned}\eeq
	Thus it only remains to prove
	\beq
	\sum_{c}^{(a)}(\Tr \wt{G}^{2}\Delta_{c})(\wt{B}^{k}G^{2})_{ca}\prec N^{7/6+4\epsilon}.
	\eeq
	
	As in the proof of Lemma \ref{lem:dgbd_1}, the contribution of $\Delta_{c1}$ is bounded by
	\beq
	O_{\prec}(N^{-1/2})\absv{(\bse_{a}+\bsh_{a})\adj[\wt{B},\wt{G}^{2}]I^{(a)}\wt{B}^{k}G^{2}\bse_{a}}\prec  N^{7/6+4\epsilon},
	\eeq
	where we expanded $(\bse_{a}+\bsh_{a})\adj$ and the commutator and estimated each term following \eqref{eq:dg_cr_1}; for example
	\beq\begin{aligned}
		&\absv{\bsh_{a}\adj\wt{G}^{2}\wt{B}I^{(a)}\wt{B}^{k}G^{2}\bse_{a}}
		\leq \norm{B}^{k+1}\sum_{b,c}\absv{\bsh_{a}\adj \wt{G}\bse_{b}} \norm{\wt{G}\bse_{b}}\norm{G\bse_{c}}\absv{G_{ca}}	\\
		\leq& \sum_{b,c}\absv{\bsh_{a}\adj \wt{G}\bse_{b}}\frac{\im \wt{G}_{bb}+\im G_{cc}}{\eta}\absv{G_{ca}}\prec N^{5/3+4\epsilon}.
	\end{aligned}\eeq
	
	For $\Delta_{c2}$ and $\Delta_{c3}$, we again use the factor $\ol{h}_{ac}$ so that their contributions are all of the form
	\beq
	\absv{\bse_{a}\adj K_{1}[\wt{B},\wt{G}^{2}] K_{2}\bse_{a} \bsh_{a}\adj I^{(a)}\wt{B}^{k}G^{2}\bse_{a}}\prec N^{2/3+4\epsilon},\qquad K_{1},K_{2}\in \{I,U\},
	\eeq
	where we used minor variants of \eqref{eq:Yder_rough_2} twice. This completes the proof of Lemma \ref{lem:Y_1}.
\end{proof}

\section{Orthogonal case}\label{sec:ort}
In this section, we show how to modify the proof for Haar-distributed $U$ on the orthogonal group, instead of the unitary group.

We explain the change in each section, starting from Section \ref{sec:prelim}. We first need to modify the partial randomness decomposition. Lemma \ref{lem:prd} remains intact, except that the vector $\bsv_{i}$ and the matrix $U^{i}$ in \eqref{eq:prd} and \eqref{eq:Uangi} are uniformly distributed respectively on $\bbS^{N}_{\R}\deq\{\bsv\in\R^{N}:\norm{\bsv}=1\}$ and the orthogonal group of order $(N-1)$. Consequently, $\wt{\bsg}_{i}\sim\caN_{\R}(0,N^{-1}I_{N})$ and $\e{\ii\theta_{i}}=\mathrm{sign}(\wt{g}_{ii})$, so that $g_{ij}\sim\caN_{\R}(0,1/N)$ for $i\neq j$. Also we always use the real Stein's lemma, that is, for a standard real Gaussian $X$ and a suitable function $F:\R\to\C$
\beq
\E[XF(X)]=\E[F'(X)].
\eeq
Finally, for the resolvent $G$ of a real symmetric matrix $H$, we have
\beq
\frac{\partial G}{\partial H_{ab}}=-G\frac{(\bse_{a}\bse_{b}\adj +\bse_{b}\bse_{a}\adj)}{1+\delta_{ab}}G.
\eeq

We move on to Section \ref{sec:outline of proof of gfc}. First of all, we take $W$ to be a GOE instead of GUE. Taking the real symmetric symmetry class does no harm to local laws, so that Proposition \ref{prop:ll} and Lemma \ref{lem:rigidity} remain intact; see \cite[Remark 2.10]{Bao-Erdos-Schnelli2020} and \cite[Appendix C]{Bao-Erdos-Schnelli2019} for details. The same is true for Proposition \ref{prop:t0 tracy widom} if we take $\mu_{1}\geq\cdots\geq\mu_{N}$ to be the eigenvalues of a GOE, as \cite[Theorem 2.2]{Landon-Yau2017} is valid for DBM with $\beta=1$. Also the statement of Proposition \ref{prop:gfc} need no change.

Next, we modify the proof of Proposition \ref{prop:gfc}. Instead of \eqref{eq:Stein_W} -- \eqref{eq:Yder} we have
\beq\begin{aligned}
	\E[F'(Y)\Tr WG^{2}]=&-2\gamma\sqrt{t}\E[F'(Y)(\tr G\Tr G^{2}+\tr G^{3})]	\\
	&-2\gamma\sqrt{t}\int_{E_{1}}^{E_{2}}\sum_{a}\E[F''(Y)\Tr G^{2}\im[\wt{G}^{2}]]\dd\wt{E},
\end{aligned}\eeq 
where we used the fact that
\beq
\E[W_{ab}^{2}]\frac{\partial G}{\partial W_{ab}}=-\frac{\sqrt{t}}{N}G(\bse_{a}\bse_{b}\adj+\bse_{b}\bse_{a}\adj)G.
\eeq
This leads to
\begin{align}
	\frac{\dd\E[F(Y)]}{\dd t}=&\Expct{F'(Y)\int_{E_{1}}^{E_{2}}\left(\dot{L}_{+}\Tr G^{2} -\dot{\gamma}\Tr GHG+\gamma^{2}(\tr G)\Tr G^{2}+\gamma^{2}\tr G^{3}\right)\dd E}\nonumber\\
	&+\frac{\gamma^{2}}{N}\Expct{F''(Y)\int_{E_{1}}^{E_{2}}\int_{E_{1}}^{E_{2}}\Tr \im[\wt{G}^{2}]\im[G^{2}]\dd\wt{E}\dd E}.
\end{align}
Noting that \eqref{eq:dt_gfc_1} -- \eqref{eq:dt_gfc_3} all remain intact, we have
\beq
\frac{\dd}{\dd t}\E[F(Y)]=\frac{\gamma}{2}\im\int_{E_{1}}^{E_{2}}\frX+O(N^{C\epsilon})
\eeq
where we modified the definition of $\frX$ as
\beq\begin{aligned}\label{eq:ort_frX}
	\frX\equiv\frX(E)
	\deq2\E\bigg[F'(Y)(\gamma\tr G-m_{\mu_{t}})\Tr G^{2}&+\gamma F'(Y)\tr G^{3}	\\
	&+\gamma F''(Y)\int_{E_{1}}^{E_{2}}\tr(\im[\wt{G}^{2}]G^{2})\dd\wt{E}\bigg].
\end{aligned}\eeq
Following the proof of Proposition \ref{prop:gfc}, in order to prove the main theorem for the orthogonal case, it suffices to prove
\beq\label{eq:ort_frX_est}
\im\frX=O(N^{5/6+C\epsilon})
\eeq
with $\frX$ defined in \eqref{eq:ort_frX}.

All contents of Section \ref{sec:frX_conc} remain the same. More precisely, the statement of Proposition \ref{prop:decoup} remain true if we replace the definition of $\frX$ by \eqref{eq:ort_frX}, and the proof of \eqref{eq:ort_frX_est} is identical to that of Proposition \ref{prop:main}. Hence it only remains to prove Proposition \ref{prop:decoup} with the new definition of $\frX$, which will be done at the end of this section.

We next modify Section \ref{sec:decouple}. We first present the counter part of Lemma \ref{lem:dgbd};
\beq\label{eq:ort_dgbd}\begin{aligned}
	\frac{\partial G}{\partial g_{ac}}=&-\gamma\frac{\ell_{a}^{2}}{\norm{\bsg_{a}}}G[\bse_{c}(\bse_{a}+\bsh_{a})\tp-(\bse_{a}+\bsh_{a})\bse_{c}\tp,\wt{B}]G	\\
	&+\gamma\frac{\ell_{a}^{2}}{2\norm{\bsg_{a}}}h_{ac}G[(\bse_{a}+2\bsh_{a})\bse_{a}\tp+\bse_{a}(\bse_{a}+2\bsh_{a})\tp,\wt{B}]G.
\end{aligned}\eeq
The easiest proof of \eqref{eq:ort_dgbd} is adding \eqref{eq:dgbd} and \eqref{eq:dgbd_bar}; this is indeed rigorously justified by smuggling in an `imaginary' vector $\bsg_{1}$ to define $\wh{\bsg}=\bsg+\ii\bsg_{1}$, so that
\beq
\frac{\partial}{\partial g_{ac}}=\left[\frac{\partial}{\partial \re \wh{g}_{ac}}\right]_{\bsg_{1}=0}=\left[\frac{\partial}{\partial \wh{g}_{ac}}+\frac{\partial}{\ol{\partial}\wh{g}_{ac}}\right]_{\bsg_{1}=0}.
\eeq
From \eqref{eq:ort_dgbd}, one can easily see that `remainders' in the heuristics 
\beq
\frac{\partial G}{\partial g_{ac}}=-\gamma G[\bse_{c}(\bse_{a}+\bsh_{a})\adj-(\bse_{a}+\bsh_{a})\bse_{c}\adj,\wt{B}]G+\text{(remainders)}
\eeq
can be estimated with exactly the same calculations as in Section \ref{sec:der_tech} (with the same upper bounds) whenever we consider derivatives in Lemmas \ref{lem:dgbd_1}, \ref{lem:dgdiag}, \ref{lem:dg_cr}, \eqref{eq:Yder_BG}, and \eqref{eq:Yder_hG}. Hence we refer to Section \ref{sec:der_tech} for detailed estimates of the remainders.

Now we show how to modify Lemmas \ref{lem:BG2_conc} -- \ref{lem:Yder} and briefly discuss their proofs. We start with Lemma \ref{lem:BG2_conc}. In light of \eqref{eq:ort_dgbd}, the counterpart of Lemma \ref{lem:dgbd_1} is given by
\beq\begin{aligned}\label{eq:ort_dgbd_1}
	\frac{1}{N}\sum_{c}^{(a)}&\frac{\partial}{\partial g_{ac}}\left[\frac{1}{\norm{\bsg_{a}}}(\wt{B}^{a}R_{a}G^{2})_{ca}\right]	
	= \sum_{b}\text{(right-hand side of \eqref{eq:dgBGG})}	\\
	&+\frac{\gamma}{N}\Expct{F'(Y)(\bse_{a}+\bsh_{a})\tp \left([G,\wt{B}]\wt{B}G+[G^{2},\wt{B}]\wt{B}\right)G\bse_{a}},	\\
	\frac{1}{N}\sum_{c}^{(a)}&\frac{\partial}{\partial g_{ac}}\left[\frac{1}{\norm{\bsg_{a}}}G_{ca}\right]
	=\sum_{b}\text{(right-hand side of \eqref{eq:dgGG})}\\
	&+\frac{\gamma}{N}\Expct{F'(Y)(\bse_{a}+\bsh_{a})\tp\left([G,\wt{B}]G+[G^{2},\wt{B}]\right)G\bse_{a}}.
\end{aligned}\eeq
Then we follow the same procedure, that is, expand $\E[F'(Y)(\wt{B}G^{2})_{aa}]$ and $\E[F'(Y)\bsh_{a}\adj \wt{B}G^{2}\bse_{a}]$ and then solve the linear equation. As a result, we obtain
\beq\label{eq:ort_BG2}\begin{aligned}
	&\E[F'(Y)(\wt{B}G^{2})_{aa}]=\text{(right-hand side of \eqref{eq:BG2_1})}	\\
	&+\frac{\gamma}{N}(\bse_{a}+\bsh_{a})\tp \left([G,\wt{B}](\omega_{\beta}+m_{\mu}^{-1}-\wt{B})G+[G^{2},\wt{B}](\omega_{\beta}+m_{\mu}^{-1}-\wt{B})\right)G\bse_{a}.
\end{aligned}\eeq

Since the additional terms in \eqref{eq:ort_BG2} are not decoupled, we need new decoupling lemmas for these quantities. More precisely, they have one of the following forms, up to $N^{-1}\E[F'(Y)\cdot]$ and deterministic weights;
\beq\begin{aligned}
	&\bse_{a}\adj KG\wt{B}^{k-1}G^{2}\bse_{a}, &&\bse_{a}\adj KG^{2}\wt{B}^{k-1}G\bse_{a},
\end{aligned}\eeq
for $K\in\{ I, U,\wt{B}\}$ and $k=1,2,3$. Note that neither type covers the other but they intersect; for example $(\wt{B}G\wt{B}G^{2})_{aa}$ is of the first type but not the second, and $(G\wt{B}G^{2})_{aa}$ falls into both types by $G\tp=G$. We introduce the following counterpart of $Z_{k\ell}$:
\beq
\caY_{k\ell}\deq \E[F'(Y)\tr G\wt{B}^{k-1}G^{2}\wt{B}^{\ell-1}],\qquad k,\ell\in\{1,2,3\}.
\eeq
Note that rough estimates from local laws give $\caY_{k\ell}=O(N^{1+C\epsilon})$, and we are aiming for a decoupling lemma with precision $O(N^{5/6+C\epsilon})$. For these quantities, we have the following analogue of Lemmas \ref{lem:GBG} and \ref{lem:dBGG}:
\beq\label{eq:GGG}\begin{aligned}
	\E[F'(Y)&(G\wt{B}^{k-1}G^{2})_{aa}]=\frac{1}{(\fra_{a}-\omega_{\alpha})^{2}}\bse_{k}\tp\caY\bfv_{\beta}+O(N^{5/6+C\epsilon}),\\
	\E[F'(Y)&(\wt{B}G\wt{B}^{k-1}G^{2})_{aa}]	\\
	=&\left(\frac{(\omega_{\beta}+m_{\mu}^{-1})}{(\fra_{a}-\omega_{\alpha})^{2}}-\frac{1}{(\fra_{a}-\omega_{\alpha})}\right)\bse_{k}\tp\caY\bfv_{\beta}+O(N^{5/6+C\epsilon}),	\\
	\E[F'(Y)&\bsh_{a}\tp G\wt{B}^{k-1}G^{2}\bse_{a}]	\\
	=&\frac{\omega_{\beta}(\omega_{\beta}+m_{\mu}^{-1})\caY_{k1}-(2\omega_{\beta}+m_{\mu}^{-1})\caY_{k2}+\caY_{k3}}{(\fra_{a}-\omega_{\alpha})(\frb_{a}-\omega_{\beta})}\bse_{k}\tp\caY\bfv_{\beta}+O(N^{5/6+C\epsilon}),
\end{aligned}\eeq 
and the same holds for quantities with $G^{2}$ and $G$ interchanged if we replace $\bse_{k}\tp\caY\bfv_{\beta}$ with $\bfv_{\beta}\tp\caY\bse_{k}$. The proof of \eqref{eq:GGG} follows the same four-step strategy as in Section \ref{sec:GBG_pf}, that is,
\begin{itemize}
	\item[(0)] Expand $\caY_{ak\ell}$ applying Stein's lemma to $(WG\wt{B}^{k-1}G\wt{B}^{\ell-1}G)_{aa}$:
	\item[(i)] Expand $(\wt{B}G\wt{B}^{k-1}G\wt{B}^{\ell-1}G)_{aa}$:
	\item[(ii)] Expand $\bsh_{a}\adj G\wt{B}^{k-1}G\wt{B}^{\ell-1}G\bse_{a}$:
	\item[(iii)] Solve the system of three linear equations from Steps (0) -- (iii).
\end{itemize}
We omit further details to avoid repetition. Likewise, following the proof of Lemma \ref{lem:veq1} we can prove
\beq\label{eq:GGG_eq}
\caY_{k\ell}=\frac{u_{\beta k}u_{\beta\ell}}{u_{\beta1}^{2}}\caY_{11}+O(N^{5/6+C\epsilon}).
\eeq
Plugging in \eqref{eq:GGG} and \eqref{eq:GGG_eq} to \eqref{eq:ort_BG2} and then following the same algebra as in \eqref{eq:BG2_2_conc} -- \eqref{eq:BG2_234_conc}, we have
\beq\begin{aligned}
	&\E[F'(Y)(\wt{B}G^{2})_{aa}]=\text{(right-hand side of \eqref{eq:BG2_1})}	\\
	&+\frac{\bfv_{\beta}\tp\bfu_{\beta}}{Nu_{\beta1}^{2}}\left(\frac{\omega_{\beta}+m_{\mu}^{-1}}{(\fra_{a}-\omega_{\alpha})^{2}}-\frac{1}{\fra_{a}-\omega_{\alpha}}\right)(u_{\beta2}-(\omega_{\beta}+m_{\mu}^{-1})u_{\beta1})\caY_{11}
\end{aligned}\eeq

The conclusions of Lemmas \ref{lem:GBG} -- \ref{lem:veq1} remain intact. For the proof of Lemma \ref{lem:GBG}, the only difference is that we have a few additional terms in Lemmas \ref{lem:dgdiag} and \ref{lem:dg_cr} that are absorbed into the error. For example in \eqref{eq:dgdiag_hG}, in light of \eqref{eq:ort_dgbd}, we have
\beq\begin{aligned}\label{eq:ort_dgdiag_hG}
	&\frac{1}{N}\sum_{c}^{(a)}\frac{\partial}{\partial g_{ac}}\left[\frac{1}{\norm{\bsg_{a}}}G_{ca}\right]	\\
	=&\text{(right-hand side of \eqref{eq:dgdiag_hG})}+\gamma \frac{1}{N}\sum_{c}^{(a)}\bse_{c}\adj G[(\bse_{a}+\bsh_{a})\bse_{c}\adj,\wt{B}]G\bse_{a}.
\end{aligned}\eeq
Then, using $G\tp=G$, the sum on the right-hand side of \eqref{eq:ort_dgdiag_hG} is equal to
\beq\begin{aligned}
	&\frac{1}{N}\sum_{c}^{(a)}\bse_{c}\adj[\wt{B},G\bse_{a}\bse_{c}\adj G](\bse_{a}+\bsh_{a})	\\
	=& \frac{1}{N}\bse_{a}\adj G\wt{B}I^{(a)}G(\bse_{a}+\bsh_{a})-\frac{1}{N}\bse_{a}\adj GI^{(a)}G\wt{B}(\bse_{a}+\bsh_{a})\prec N^{-1/3+2\epsilon}
\end{aligned}\eeq
where the last estimate follows immediately from Ward identity. Similarly, in the proof of Lemma \ref{lem:dg_cr}, we use \eqref{eq:ort_dgbd_1} instead of Lemma \ref{lem:dgbd_1}. Notice that the additional terms in \eqref{eq:ort_dgbd_1} are all of the form $N^{-1}(K_{1}GK_{2}GK_{3}G)_{aa}$, hence $O_{\prec}(N^{3\epsilon})$. Multiplying by $\frd_{k}$, the contribution of these terms in \eqref{eq:dGG_conc} -- \eqref{eq:dhGG_conc} is $O(N^{-1/3+C\epsilon})$. Also Lemma \ref{lem:veq1} only requires minor modification to the proof.

Finally, the counterpart of  Lemma \ref{lem:Yder} has an overall factor of two compared to the original conclusion. The factor is due to Lemma \ref{lem:Y_1}; by \eqref{eq:ort_dgbd} and $G\tp=G$, we have
\beq\begin{aligned}
	\frac{1}{N}\sum_{c}^{(a)}\frac{\partial \Tr \wt{G}}{\partial g_{ac}}=&\text{(right-hand side of \eqref{eq:Y_1})}+\gamma \sum_{c}^{(a)}\bse_{c}\adj [\wt{B},\wt{G}^{2}](\bse_{a}+\bsh_{a})(\wt{B}^{k}G^{2})_{ca}	\\
	=&\text{(right-hand side of \eqref{eq:Y_1})}-\gamma (\bse_{a}+\bsh_{a})\tp [\wt{B},\wt{G}^{2}]I^{(a)}\wt{B}^{k}G^{2}\bse_{a}	\\
	=&2\cdot\text{(right-hand side of \eqref{eq:Y_1})}.
\end{aligned}\eeq
The rest of the proof can be modified in a similar fashion to Lemmas \ref{lem:GBG} -- \ref{lem:veq1}.

Collecting all the result and following the same algebra as in the proof of Proposition \ref{prop:decoup}, we have \eqref{eq:Bab2_conc} with the new definition of $\frX$. Then the only other difference in the proof of Proposition \ref{prop:decoup} is in \eqref{eq:x_expa}, where we applied Stein's lemma to $(WG^{2})_{aa}$. The additional terms in \eqref{eq:x_expa} are 
\beq
-2\frac{\gamma}{N} t\E[F'(Y)(G^{3})_{aa}]-\gamma \frac{t}{N}\int_{E_{1}}^{E_{2}}(\im[\wt{G}^{2}]G^{2})_{aa}\dd\wt{E},
\eeq
which can be easily shown to be $O(tN^{C\epsilon})$. Hence \eqref{eq:x_expa_result} remains valid, so that combining with \eqref{eq:Bab2_conc} concludes the proof of Proposition \ref{prop:decoup}. 

\subsection*{Acknowledgements}
The authors would like to thank Ji Oon Lee for helpful discussions. Also the authors are deeply grateful for the anonymous referee for providing helpful comments and suggestions. The work of J. Park was partially supported by National Research Foundation of Korea under grant number NRF-2019R1A5A1028324. The work of H. C. Ji was partially supported by ERC Advanced Grant "RMTBeyond" No.~101020331.

\appendix
	\section{Stability of $\mu_{A}\boxplus\mu_{B}\boxplus\mu_{\SC}^{(t)}$ and $\mu_{\alpha}\boxplus\mu_{\beta}\boxplus\mu_{\SC}^{(t)}$}\label{sec:stab}
	
	Recall that we omitted the subscript $t$ to denote, say, $\omega_{A,t}$ by $\omega_{A}$. The two goals of this section are to prove that the system of equations $\Phi_{\alpha\beta}=0$ defined in \eqref{eq:Phi_def} is stable around $z=E_{+}$ and to extend the same result to $\Phi_{AB}=0$ by comparison. First, we introduce notations for quantities that are widely used throughout the paper. 
	\beqs
	\begin{split}
		\caS_{\alpha\beta}(z)&\deq (F_{\alpha}'(\omega_{\alpha}(z))-1)(F_{\beta}'(\omega_{\beta}(z))-1)-1,\\
		\caT_{\alpha}(z)&\deq \frac{1}{2}\left(F_{\alpha}''(\omega_{\alpha}(z))(F_{\beta}'(\omega_{\beta}(z))-1)^{2}+F_{\beta}''(\omega_{\beta}(z))(F_{\alpha}'(\omega_{\alpha}(z))-1)\right),\\
		\caT_{\beta}(z)&\deq  \frac{1}{2}\left(F_{\beta}''(\omega_{\beta}(z))(F_{\alpha}'(\omega_{\alpha}(z))-1)^{2}+F_{\alpha}''(\omega_{\alpha}(z))(F_{\beta}'(\omega_{\beta}(z))-1)\right).
	\end{split}
	\eeqs
	Similarly we define $\caS_{AB},\caT_{A},$ and $\caT_{B}$ to be the same quantities with $(\alpha,\beta)$ replaced by $(A,B)$. From \eqref{eq:S_at_E}, we see that the edge $E_{+}$ satisfies $\caS_{\alpha\beta}(E_{+})=0$. The main result of this section is the following proposition, whose proof is postponed to the end of this section:
	\begin{prop}\label{prop:stab}
		Let $\sigma>0$ be fixed. Then there exist constants (small) $\tau>0$ and (large) $N_{0}\in\N$ such that each of the following holds uniformly over $z\in\caD_{\tau}(N^{-1+\sigma},1)$ and $t\in[0,1]$ for all $N\geq N_{0}$.
		\begin{itemize}
			\item[(i)] There exist positive constants $k$ and $K$ such that
			\begin{align*}
				&\min_{i}\absv{\fra_{i}-\omega_{A}(z)}\geq k, &&\min_{i}\absv{\frb_{i}-\omega_{B}(z)}\geq k, &
				&\absv{\omega_{A}(z)}\leq K, &&\absv{\omega_{B}(z)}\leq K.
			\end{align*}
			\item[(ii)]
			Recall that $\wh{\mu}_{t}$ denotes the free convolution $\mu_{A}\boxplus\mu_{B}\boxplus\mu_{\SC}^{(t)}$. For its Stieltjes transform $m_{\wh{\mu}_{t}}$, we have
			\beqs
			\im m_{\wh{\mu}_{t}}(z)\sim \begin{cases}
				\sqrt{\kappa+\eta}, & \text{if }E\in \supp \wh{\mu}_{t},\\
				\frac{\eta}{\sqrt{\kappa+\eta}}, & \text{if }E\notin \supp\wh{\mu}_{t},
			\end{cases}
			\eeqs
			where we denoted $z=E+\ii\eta$ and $\kappa=\absv{E_{+}-E}$.
			
			\item[(iii)] There exists a constant $C>0$ such that
			\begin{align*}
				&\caS_{AB}(z)\sim \sqrt{\kappa+\eta}, &&\absv{\caT_{A}(z)}\leq C, &&\absv{\caT_{B}(z)}\leq C.
			\end{align*}
			Moreover, there exist positive constants $\delta$ and $c$ such that, whenever $\absv{z-E_{+}}\leq\delta$,
			\begin{align*}
				&\absv{\caT_{A}(z)}\geq c, &&\absv{\caT_{A}(z)}\geq c.
			\end{align*}
			\item[(iv)] There exists a constant $C>0$ such that
			\begin{align*}
				\absv{\omega_{A}'(z)}\leq C\frac{1}{\sqrt{\kappa+\eta}}, &&\absv{\omega_{B}'(z)}\leq C\frac{1}{\sqrt{\kappa+\eta}}, &&\absv{\caS_{AB}'(z)}\leq C\frac{1}{\sqrt{\kappa+\eta}}.
			\end{align*}
		\end{itemize}
	\end{prop}
	
	\subsection{Stability of $\mu_{\alpha}\boxplus\mu_{\beta}\boxplus\mu_{\SC}^{(t)}$}
	In this subsection, we study regularity properties of $\mu_{\alpha}\boxplus\mu_{\beta}\boxplus\mu_{\SC}^{(t)}$. Also, we present the proof of Lemma \ref{lem:stab} at the end of this subsection.
	
	\begin{lem}\label{lem:Pick_F}
		Let $\mu_{\alpha}$ and $\mu_{\beta}$ be probability measures in Definition \ref{defn:LSD}. Then for each $t\in[0,1]$ there exist unique Borel measures $\wh{\mu}_{\alpha,t}$ and $\wh{\mu}_{\beta,t}$ on $\R$ such that
		\begin{align}\label{eq:Pick_general}
			&F_{\alpha,t}(z)-z=m_{\wh{\mu}_{\alpha,t}}(z), &
			&F_{\beta,t}(z)-z=m_{\wh{\mu}_{\beta,t}}(z)
		\end{align}
		for all $z\in\C_{+}$. Furthermore, we have that
		\begin{align}\label{eq:hatmu_def}
			&\wh{\mu}_{\alpha,0}(\R)=\int_{\R}x^{2}\dd\mu_{\alpha}(x)-\left(\int_{\R}x\dd\mu_{\alpha}(x)\right)^{2},&
			&[E_{\alpha}^{+}-\tau,E_{\alpha}^{+}]\subset\supp\wh{\mu}_{\alpha,0}\subset[E_{\alpha}^{-},E_{\alpha}^{+}],\\
			&\wh{\mu}_{\beta,0}(\R)=\int_{\R}x^{2}\dd\mu_{\beta}(x)-\left(\int_{\R}x\dd\mu_{\beta}(x)\right)^{2},&
			&[E_{\beta}^{+}-\tau,E_{\beta}^{+}]\subset\supp\wh{\mu}_{\beta,0}\subset [E_{\beta}^{-},E_{\beta}^{+}],\nonumber
		\end{align}
		and for all $t\in[0,1]$ that
		\begin{align}\label{eq:hatmu_t-0}
			&\wh{\mu}_{\alpha,t}=\wh{\mu}_{\alpha,0}+t\mu_{\alpha},&
			&\wh{\mu}_{\beta,t}=\wh{\mu}_{\beta}+t\mu_{\beta}.
		\end{align}
	\end{lem}
	\begin{proof}
		The proof is a minor modification of that of \cite[Lemma 3.5]{Bao-Erdos-Schnelli2020}, and we sketch its proof here for readers' convenience. We prove the result only for $\mu_{\alpha}$ and that for $\mu_{\beta}$ is exactly the same. 
		
		The existence and uniqueness of $\wh{\mu}_{\alpha}$ follow from Nevanlinna-Pick representation theorem, and the formula for $\wh{\mu}_{\alpha}(\R)$ is a direct consequence of \eqref{eq:Pick_general} and the definition of $F_{\alpha}$ in \eqref{eq:mF_def}. Given the uniqueness, we see from \eqref{eq:mF_def} that $\wh{\mu}_{\alpha,t}=\wh{\mu}_{\alpha,0}+t\mu_{\alpha}$. 
		
		In order to prove $\supp\wh{\mu}_{\alpha,0}\subset[E_{\alpha}^{-},E_{+}^{\alpha}]$, we observe for each $x\in [E_{\alpha}^{-},E_{\alpha}^{+}]^{c}$ that
		\beq
		\lim_{y\searrow0}\im m_{\wh{\mu}_{\alpha,0}}(x+\ii y)=\lim_{y\searrow 0}\im F_{\alpha,0}(x+\ii y)=\absv{m_{\mu_{\alpha}}(x)}^{-2}\lim_{y\searrow 0}\im m_{\mu_{\alpha}}(x+\ii y)=0.   
		\eeq
		Then Stieltjes inversion directly implies $[E_{\alpha}^{-},E_{\alpha}^{+}]^{c}\subset (\supp\wh{\mu}_{\alpha,0})^{c}$ as desired.
		
		Finally we prove the inclusion $[E_{\alpha}^{+}-\tau,E_{\alpha}^{+}]\subset\supp\wh{\mu}_{\alpha,0}$. Suppose on the contrary that there exists a nonempty open interval $I\subset(E_{\alpha}^{+}-\tau,E_{\alpha}^{+})\setminus\supp\wh{\mu}_{\alpha,0}$. Since $I\subset\supp\wh{\mu}_{\alpha}^{c}$, the function $z\mapsto F_{\alpha,0}(z)-z$ extends analytically through $I$ via Schwarz reflection which satisfies $F_{\alpha,0}(x)-x\in \R$ for each $x\in I$. Then this leads to a meromorphic extension of $m_{\mu_{\alpha}}$ since
		\beq
		m_{\mu_{\alpha}}=-\frac{1}{(F_{\alpha,0}(z)-z)+z}.
		\eeq
		This extension must satisfy $\im m_{\mu_{\alpha}}(x)=0$ for almost all $x\in I$, which contradicts Definition \ref{defn:LSD}. 
	\end{proof}
	
	We introduce the following result from \cite{Belinschi2014} which gives a necessary condition that a free additive convolution has unbounded Stieltjes transform:
	\begin{lem}[Theorem 7 of \cite{Belinschi2014}]\label{lem:m_bdd}
		Let $\mu$ and $\nu$ be compactly supported Borel probability measures on $\R$. If the image $m_{\mu\boxplus\nu}(\C_{+})$ is unbounded, then there exist real numbers $u$ and $v$ such that $\mu(\{u\})+\nu(\{v\})\geq1$.
	\end{lem}
	Clearly our measures $\mu_{\alpha}$ and $\mu_{\beta}$ from Definition \ref{defn:LSD} compactly supported and since they are absolutely continuous it also follows that $\mu_{\alpha}(\{u\})+\mu_{\beta}(\{v\})=0$ for all choices of $u$ and $v$. Therefore it follows that $m_{\mu_{0}}$ is bounded on $\C_{+}$ by Lemma \ref{lem:m_bdd}, and the same bound applies to $m_{\mu_{t}}(z)=m_{\mu_{0}}(z+tm_{\mu_{t}})$. One advantage of applying this result is that we can bypass the assumption in \cite{Bao-Erdos-Schnelli2020} that $\sup_{z\in\C_{+}}\absv{m_{\mu_{\alpha}}(z)}\leq C$. 
	
	\begin{lem}\label{lem:subor_cont}
		The maps $\omega_{\alpha,t}(z), \omega_{\beta,t}(z),$ and $m_{\mu_{t}}(z)$ are continuous in $(t,z)\in[0,\infty)\times(\C_{+}\cup\R)$.
	\end{lem}
	\begin{proof}
		Recall from \eqref{eq:omega(z+tm)} that $\omega_{\alpha,t}(z)=\omega_{\alpha,0}(z+tm_{\mu_{t}}(z))$ and $\omega_{\beta,t}(z)=\omega_{\beta,0}(z+tm_{\mu_{t}}(z))$. Since $\omega_{\alpha,0}$ and $\omega_{\beta,0}$ continuously extend to $\C_{+}\cup\R$ by \cite[Theorem 3.3]{Belinschi2008}, it suffices to consider only $m_{\mu_{t}}(z)$. Thus, our goal here is to prove the following statement; for all fixed $\epsilon>0$, $z\in\ol{\C}_{+}$, and $t\in[0,\infty)$, there exists $\delta>0$ such that $\absv{z-w}<\delta$ and $\absv{t-s}<\delta$ imply
		\beq\label{eq:subor_cont}
		\absv{m_{\mu_{t}}(z)-m_{\mu_{s}}(w)}<\epsilon.
		\eeq
		
		First, we prove \eqref{eq:subor_cont} when $t=0$. Take $\delta_{1}>0$ so that $\absv{w-z}<3\delta_{1}$ implies $\absv{m_{\mu_{0}}(z)-m_{\mu_{0}}(w)}<\epsilon/100$, and take $\delta_{2}>0$ to satisfy $\delta_{2}\sup_{z'\in\C_{+}}\absv{m_{\mu_{0}}(z')}<\delta_{1}$. Then, for all $s\in (0,\delta_{2})$ and $w\in\ol{\C}_{+}$ with $\absv{z-w}<\delta_{1}$, we have
		\beqs
		\absv{z-w-sm_{\mu_{s}(w)}}\leq \absv{z-w}+s\absv{m_{\mu_{0}}(w+sm_{\mu_{s}}(w))}\leq 3\delta_{1},
		\eeqs
		so that
		\beqs
		\absv{m_{\mu_{0}}(z)-m_{\mu_{s}}(w)}=\absv{m_{\mu_{0}}(z)-m_{\mu_{0}}(w+sm_{\mu_{s}}(w))}<\epsilon.
		\eeqs
		
		Next, we prove \eqref{eq:subor_cont} at $(t,z)\in(0,\infty)\times\ol{\C}_{+}$. We first claim that the result follows from the following assertion; there exists a constant $C>0$ such that the following holds whenever $w\in\ol{\C}_{+}$ and $t,s>0$;
		\beq\label{eq:subor_cont1}
		\absv{tm_{\mu_{t}}(w)-sm_{\mu_{s}}(w)}\leq C\min(t,s)^{-1/3}((\max(t,s)+\absv{w}))^{4}\absv{t-s}^{1/3}.
		\eeq
		We deduce \eqref{eq:subor_cont} for $(t,z)\in(0,\infty)\times\ol{\C}_{+}$ assuming the validity of \eqref{eq:subor_cont1}. As an immediate consequence of \eqref{eq:subor_cont1}, whenever $s$ is close enough to $t$, we have
		\beq\label{eq:subor_conti_11}\begin{aligned}
			m_{\mu_{s}}(w)=m_{\mu_{0}}(w+sm_{\mu_{s}}(w))
			=m_{\mu_{0}}(w+tm_{\mu_{t}}(w)+O(\absv{t-s}^{1/3})),
		\end{aligned}\eeq
		uniformly over $w$ in bounded subsets of $\C_{+}\cup\R$. Since $m_{\mu_{0}}$ is continuous and bounded in $\C_{+}$, it is uniformly continuous on each compact subset of $\C_{+}\cup\R$, so that \eqref{eq:subor_conti_11} implies
		\beq
		\absv{m_{\mu_{s}}(w)-m_{\mu_{t}}(w)}=\absv{m_{\mu_{s}}(w)-m_{\mu_{0}}(w+tm_{\mu_{t}}(w))}\to 0,\qquad \text{as} \quad s\to t,
		\eeq
		uniformly over $\absv{w}\leq R$. Then, recalling that $t\in(0,\infty)$ is fixed, the final result \eqref{eq:subor_cont} follows from the continuity of $m_{\mu_{t}}$.
		
		Finally, we prove \eqref{eq:subor_cont1}. We suppose $s>t$ without loss of generality and write
		\begin{align*}
			&m_{t}=m_{\mu_{t}}(w), &&m_{s}=m_{\mu_{s}}(w),  &&\xi_{t}=w+tm_{\mu_{t}}(w), &&\xi_{s}=w+sm_{\mu_{s}}(w)
		\end{align*}
		to simplify the presentation. Using the equation $m_{t}=m_{\mu_{0}}(\xi_{t})$, we find that
		\beqs
		m_{t}-m_{s}=m_{\mu_{0}}(\xi_{t})-m_{\mu_{0}}(\xi_{s})=(\xi_{t}-\xi_{s})\int_{\R}\frac{1}{(x-\xi_{t})(x-\xi_{s})}\dd\mu_{0}(x).
		\eeqs
		Using the definition of $\xi_{t}$ and rearranging the equation, we get
		\beq\label{eq:subor_cont2}
		(\xi_{t}-\xi_{s})\left(1-t\int_{\R}\frac{1}{(x-\xi_{t})(x-\xi_{s})}\dd\mu_{0}(x)\right)=(t-s)m_{s}.
		\eeq
		We next derive \eqref{eq:subor_cont1} from a lower bound for the second factor on the left-hand side of \eqref{eq:subor_cont2}. Using the fact that 
		\beqs
		\im \xi_{s}=\im w+s\im m_{\mu_{s}}(w)=\im w+s\im \xi_{s}\int_{\R}\frac{1}{\absv{x-\xi_{s}}^{2}}\dd\mu_{0}(x),
		\eeqs
		we have
		\begin{align}\label{eq:subor_cont_eq}
			&\int_{\R}\frac{1}{\absv{x-\xi_{s}}^{2}}\dd\mu_{0}(x)\leq\frac{1}{s}, &
			&\int_{\R}\frac{1}{\absv{x-\xi_{t}}^{2}}\dd\mu_{0}(x)\leq\frac{1}{t}.
		\end{align}
		Then we find that
		\beq
		\begin{split}
			&\Absv{1-t\int_{\R}\frac{1}{(x-\xi_{t})(x-\xi_{s})}\dd\mu_{0}(x)}\geq 1-t\re\int\frac{1}{(x-\xi_{t})(x-\xi_{s})}\dd\mu_{0}(x)\\
			\geq& 1-t\frac{s^{-1}+t^{-1}}{2}+\frac{t}{2}\int_{\R}\left(\frac{1}{\absv{x-\xi_{t}}^{2}}+\frac{1}{\absv{x-\xi_{s}}^{2}}-2\re\frac{1}{(x-\xi_{t})(x-\xi_{s})}\right)\dd\mu_{0}(x)\\
			\geq& \frac{t}{2}\int_{\R}\Absv{\frac{1}{x-\xi_{t}}-\frac{1}{x-\xi_{s}}}^{2}\dd\mu_{0}(x)=\frac{t}{2}\int_{\R}\frac{\absv{\xi_{s}-\xi_{t}}^{2}}{\absv{x-\xi_{t}}^{2}\absv{x-\xi_{s}}^{2}}\dd\mu_{0}(x)	\\
			\geq& Ct(1+s+\absv{w})^{-4}\absv{\xi_{t}-\xi_{s}}^{2},	\label{eq:subor_cont3}
		\end{split}
		\eeq
		where we used \eqref{eq:subor_cont_eq} in the second line, $s>t$ in the third, and $\absv{x-\xi_{t}},\absv{x-\xi_{s}}\leq \absv{w}+C(1+s)$ in the last. Now plugging \eqref{eq:subor_cont3} into \eqref{eq:subor_cont2}, we get
		\beqs
		Ct(s+\absv{w})^{-4}\absv{\xi_{t}-\xi_{s}}^{3}\leq C\absv{t-s},
		\eeqs
		which implies \eqref{eq:subor_cont1}.
	\end{proof}
	\begin{lem}\label{lem:subor_bdd}
		There exists a positive constant $\tau$ such that for all fixed $s,\eta_{M}>0$ there is $C>0$ with
		\beqs
		\sup_{t\in[0,s]}\sup_{z\in \caD_{\tau}(0,\eta_{M})}\absv{\omega_{\alpha,t}(z)}+\absv{\omega_{\beta,t}(z)}\leq C.
		\eeqs
	\end{lem}
	\begin{proof}
		The proof closely follows that of \cite[Lemma 3.2]{Bao-Erdos-Schnelli2020}. We prove the bound for $\omega_{\alpha,t}$ and the same proof applies to $\omega_{\beta,t}$. First of all, from \eqref{eq:Jac_edge_b} we find that for each fixed $M>0$ there exists a constant $c>0$ such that
		\beq\label{eq:subor_bdd1}
		\inf\{\absv{m_{\beta}(\omega)}:\re \omega\in(E_{\beta}^{+}-\tau_{\beta},E_{\beta}^{+}+M),\,\im \omega\in(0,2\eta_{M})\}\geq c.
		\eeq
		On the other hand by Lemma \ref{lem:Pick_F}, there exist constants $C_{1},C_{2}>0$ such that
		\beqs
		\absv{F_{\alpha,t}(\omega)}\geq\frac{\absv{\omega}}{2} \AND\absv{F_{\alpha,t}(\omega)-\omega}\leq C_{2}\absv{\omega}^{-1}
		\eeqs
		whenever $\absv{\omega}\geq C_{1}$. 
		
		Now we assume on the contrary that $\absv{\omega_{\alpha,t}(z)}\geq K$ for some $z\in\caD_{\tau}(0,\eta_{M})$ with $\tau<\tau_{\beta}$ and $K>0$ to be chosen later. In particular if $K>C_{1}$ we have
		\beqs
		\absv{\omega_{\beta,t}(z)-z}=\absv{F_{\alpha,t}(\omega_{\alpha,t}(z))-\omega_{\alpha,t}(z)}\leq C_{2}K^{-1}.
		\eeqs
		Thus we can take $K$ to be large enough so that
		\begin{align*}
			&\re \omega_{\beta,t}(z)\geq E_{+,t}-\tau-C_{2}K^{-1}>E_{\beta}^{+}-\tau_{\beta},&
			&\im\omega_{\beta,t}(z)<\eta_{M}+C_{2}K^{-1}<2\eta_{M},
		\end{align*}
		where we used the fact that $E_{\beta}^{+}\leq E_{+,t}\leq E_{\beta}^{+}+E_{\alpha}^{+}+2\sqrt{t}$ from \cite[Lemma 3.1]{Voiculescu1986}. In other words, $\omega_{\beta,t}(z)$ lies within the domain in \eqref{eq:subor_bdd1}. Then we obtain
		\beqs
		\absv{F_{\beta,t}(\omega_{\beta,t}(z))}\leq \frac{1}{\absv{m_{\mu_{\beta}}(\omega_{\beta,t}(z))}}+ s\absv{m_{\mu_{0}}(z+tm_{\mu_{\mu_{t}}}(z))}\leq C
		\eeqs
		for some constant $C>0$. After raising $K$ further, we have a contradiction since
		\beqs
		\absv{F_{\beta,t}(\omega_{\beta,t}(z))}=\absv{F_{\alpha,t}(\omega_{\alpha,t}(z))}\geq \frac{\absv{\omega_{\alpha,t}(z)}}{2}.
		\eeqs
		This proves $\absv{\omega_{\alpha,t}(z)}\leq K$, and the bound for $\omega_{\beta,t}(z)$ follows from the same proof.
	\end{proof}
	
	\begin{lem}\label{lem:subor_real}
		Recall that $E_{+,t}=\sup\supp\mu_{t}$. For each $t\geq 0$, the maps $\omega_{\alpha,t}$ and $\omega_{\beta,t}$ are real-valued and monotone increasing on $(E_{+,t},\infty)$, and they map into $(E_{\alpha}^{+},\infty)$ and $(E_{\beta}^{+},\infty)$, respectively.
	\end{lem}
	\begin{proof}
		For any $z\in[E_{+,t},\infty)$ we have from \eqref{eq:Phi_def} that
		\beqs
		\im\omega_{\alpha,t}(z)+\im\omega_{\beta,t}(z)=\im F_{\mu_{t}}(z)+\im z=0,
		\eeqs
		which implies $\im\omega_{\alpha,t}(z)=\im\omega_{\beta,t}(z)=0$ since both of them should be nonnegative. Furthermore, for a large enough positive $z_{0}$, we have $\omega_{\alpha,t}(z_{0})=m_{\mu_{\alpha}}^{-1}(m_{\mu_{t}}(z_{0}))\in (E_{\alpha}^{+},\infty)$.
		
		Now we suppose on the contrary that there exists $z\in[E_{+,t},\infty)$ such that $\omega_{\alpha,t}(z)<E_{\alpha}^{+}$. Then, since $\re\omega_{\alpha,t}(z)$ is a continuous real function, there must be another point $w$ between $z$ and $z_{0}$ such that $\re\omega_{\alpha,t}(w)\in (E_{\alpha}^{+}-\tau_{\alpha},E_{\alpha}^{+})$. Then by \eqref{eq:Jac_edge_a} we have that 
		\beqs
		\lim_{\eta\searrow 0}\im m_{\mu_{t}}(w+\ii\eta)=\lim_{\eta\searrow 0}\im m_{\mu_{\alpha}}(\omega_{\alpha,t}(w+\ii\eta))=\lim_{y\searrow 0}\im m_{\mu_{\alpha}}(\omega_{\alpha,t}(w)+\ii y)>0,
		\eeqs
		which contradicts $w\notin\supp\mu_{t}$. Here we used \cite[Theorem 2.7]{Belinschi2008} in the second equality when $\im \omega_{\alpha,t}(w)=0$. Thus we have proved that $\omega_{\alpha,t}$ maps $[E_{+,t},\infty)$ into $[E_{\alpha}^{+},\infty)$. The fact that $\omega_{\alpha,t}$ is increasing follows directly from chain rule as in \cite[Lemma 3.3]{Bao-Erdos-Schnelli2020}.
	\end{proof}
	
	\begin{lem}\label{lem:stab_lim}
		There exist positive constants $\tau$ and $k_{0}$ such that
		\begin{align*}
			&\inf_{t\in[0,1]}\inf_{z\in\caD_{\tau}(0,\infty)}\inf_{x\in\supp\mu_{\alpha}}\absv{\omega_{\alpha,t}(z)-x}\geq k_{0},&
			&\inf_{t\in[0,1]}\inf_{z\in\caD_{\tau}(0,\infty)}\inf_{x\in\supp\mu_{\beta}}\absv{\omega_{\beta,t}(z)-x}\geq k_{0}.
		\end{align*}
	\end{lem}
	\begin{proof} 
		We only present an outline of the proof since it is a minor modification of \cite[Lemma 3.7]{Bao-Erdos-Schnelli2020}. First of all, we prove that the following statement implies the result; there exists a constant $k>0$ such that
		\begin{align}\label{eq:stab_prf}
			&\inf_{t\in[0,1]}\absv{\omega_{\alpha,t}(E_{+,t})-E_{\alpha}^{+}}>k, &
			&\inf_{t\in[0,1]}\absv{\omega_{\beta,t}(E_{+},t)-E_{\beta}^{+}}>k.
		\end{align}
		Assuming \eqref{eq:stab_prf}, we find from Lemma \ref{lem:subor_cont} that for a sufficiently small $\tau$
		\begin{align*}
			&\inf_{t\in[0,1]}\inf_{z\in\caD_{\tau}(0,\tau)}\inf_{x\in\supp\mu_{\alpha}}\absv{\omega_{\alpha,t}(E_{+,t})-x}>k, &
			&\inf_{t\in[0,1]}\inf_{z\in\caD_{\tau}(0,\tau)}\inf_{x\in\supp\mu_{\beta}}\absv{\omega_{\beta,t}(E_{+},t)-x}>k.
		\end{align*}
		Since $\im \omega_{\beta,t}(z)=\im \omega_{\beta,0}(z+tm_{\mu_{t}}(z))\geq \im z$, the result directly extends to $\caD_{\tau}(0,\infty)$.
		
		In order to prove \eqref{eq:stab_prf}, we recall the following identities from \eqref{eq:Phi_def};
		\begin{align*}
			\im m_{\mu_{t}}(z)=\im \omega_{\alpha,t}(z)\int_{\R}\frac{1}{\absv{x-\omega_{\alpha,t}(z)}^{2}}\dd\mu_{\alpha}(x)=\im\omega_{\beta,t}(z)\int_{\R}\frac{1}{\absv{x-\omega_{\beta,t}(z)}^{2}}\dd\mu_{\beta}(x),\\
			\im \omega_{\alpha,t}(z)+\im\omega_{\beta,t}(z)-\im z=\im F_{\mu_{t}}(z)=\frac{\im m_{\mu_{t}}(z)}{\absv{m_{\mu_{t}}(z)}^{2}}+t\im m_{\mu_{t}}(z).
		\end{align*}
		We then have
		\beq\label{eq:stab_prf1}
		R_{\alpha}(\omega_{\alpha,t}(z))+R_{\beta}(\omega_{\beta,t}(z))-1=\absv{m_{\mu_{t}}(z)}^{2}\frac{\im z}{\im m_{\mu_{t}}(z)}+t\absv{m_{\mu_{t}}(z)}^{2}\geq 0,
		\eeq
		where we defined for $\omega\in\C_{+}$
		\beqs
		R_{\alpha}(\omega)\deq \Absv{\int_{\R}\frac{1}{x-\omega}\dd\mu_{\alpha}(x)}^{2}/\int_{\R}\frac{1}{\absv{x-\omega}^{2}}\dd\mu_{\alpha}(x)
		\eeqs
		and $R_{\beta}(\omega)$ analogously.
		
		We now prove \eqref{eq:stab_prf} for $\omega_{\alpha,t}$, and the result for $\omega_{\beta,t}$ follows by symmetry. Proceeding as in the proof of \cite[Lemma 3.7]{Bao-Erdos-Schnelli2020}, we find that $\omega_{\alpha,t}(E_{+,t})\geq E_{\alpha}^{+}$ and $\omega_{\beta,t}(E_{+,t})\geq E_{\beta}^{+}$ imply
		\begin{align}\label{eq:stab_prf2}
			&R_{\alpha}(\omega_{\alpha,t}(E_{+,t}))\leq 
			C_{0}(\omega_{\alpha,t}(E_{+,t})-E_{\alpha}^{+})^{1-\absv{t_{\alpha}^{+}}},&
			&R_{\beta}(\omega_{\beta,t}(E_{+,t}))\leq 
			C_{0}(\omega_{\beta,t}(E_{+,t})-E_{\beta}^{+})^{1-\absv{t_{\beta}^{+}}},
		\end{align}
		for a constant $C_{0}$ independent of $t$. On the other hand, due to Cauchy-Schwarz inequality we have
		\beq\label{eq:stab_prf3}
		R_{\beta}(\omega)\leq 1-c_{\beta}(\omega) \qquad \omega\in (E_{\beta}^{+},\infty),
		\eeq
		for a positive continuous function $c_{\beta}$ on $(E_{\beta}^{+},\infty)$. Combining Lemma \ref{lem:subor_bdd}, \eqref{eq:stab_prf2}, and \eqref{eq:stab_prf3} implies that
		\beqs
		\sup_{t\in[0,1]}R_{\beta}(\omega_{\beta,t}(E_{+,t}))\leq 1-c
		\eeqs
		for a constant $c>0$. Therefore by \eqref{eq:stab_prf1} we have
		\beqs
		\omega_{\alpha,t}(E_{+,t})-E_{\alpha}^{+}\geq C_{0}c^{1/(1-\absv{t_{\alpha}^{+}})}.
		\eeqs
		This concludes the proof of \eqref{eq:stab_prf} and thus Lemma \ref{lem:stab_lim}.
	\end{proof}
	
	\begin{lem}\label{lem:edgechar}
		For all $z\in\ol{\C}_{+}$ and $t\in[0,1]$ we have
		\beq\label{eq:edgechar_ineq}
		\Absv{\left(F_{\alpha,t}'(\omega_{\alpha,t}(z))-1\right)\left(F_{\beta,t}'(\omega_{\beta,t}(z))-1\right)}\leq 1,
		\eeq
		and the upper edge $z=E_{+,t}$ is the largest real point at which the equality holds. Furthermore, we have
		\beq\label{eq:edgechar_eq}
		\left(F_{\alpha,t}'(\omega_{\alpha,t}(E_{+,t}))-1\right)\left(F_{\beta,t}'(\omega_{\beta,t}(E_{+,t}))-1\right)=1.
		\eeq
	\end{lem}
	\begin{proof}
		The proof of \eqref{eq:edgechar_ineq} and that $E_{+}$ satisfies \eqref{eq:edgechar_eq} is identical to that of Lemma 3.8 in \cite{Bao-Erdos-Schnelli2020}. To prove the remaining part, that $E_{+}$ is the largest such point, we observe from Lemma \ref{lem:Pick_F} that the left-hand side of \eqref{eq:edgechar_ineq} decreases as $\omega_{\alpha,t}>E_{\alpha}^{+}$ and $\omega_{\beta,t}>E_{\beta}^{+}$ increase. Since $\omega_{\alpha}\vert_{[E_{+},\infty)}$ and $\omega_{\beta}\vert_{[E_{+},\infty)}$ are increasing real functions mapping into $(E_{\alpha}^{+}+k,\infty)$ and $(E_{\beta}^{+}+k,\infty)$, the result follows.
	\end{proof}
	
	\begin{prop}\label{prop:sqrt_lim}
		For each $t\in[0,1]$ there exist positive $\gamma_{\alpha,t}$ and $\gamma_{\beta,t}$ such that the following hold uniformly over $t\in[0,1]$ and $z\in\caD_{\tau}(0,\eta_{M})$;
		\beqs
		\begin{split}
			\omega_{\alpha,t}(z)-\omega_{\alpha,t}(E_{+,t})&=\gamma_{\alpha,t} \sqrt{z-E_{+,t}}+O(\absv{z-E_{+,t}}),\\
			\omega_{\beta,t}(z)-\omega_{\beta,t}(E_{+,t})&=\gamma_{\beta,t}\sqrt{z-E_{+,t}}+O(\absv{z-E_{+,t}}),\\
			\gamma_{\alpha,t}&\sim 1\sim \gamma_{\beta,t}.
		\end{split}
		\eeqs
	\end{prop}
	\begin{proof}
		Given Lemmas \ref{lem:subor_cont}, \ref{lem:subor_real}--\ref{lem:edgechar}, the proof is almost identical to that of \cite[Lemma 3.8]{Bao-Erdos-Schnelli2020} except some minor changes to make the result uniform over $t$. We present below how we modify their proof.
		
		Note the map $F_{\alpha,t}$ has an inverse in a neighborhood of $F_{\alpha,t}(\omega_{\alpha,t}(E_{+,t}))=F_{\beta,t}(\omega_{\beta,t}(E_{+,t}))$ that maps into a neighborhood of $\omega_{\alpha,t}(E_{+,t})$. The first modification is to show that both the domain and image of the inverse $F_{\alpha,t}^{-1}$ can have size of $O(1)$. Note that
		\beqs
		F_{\alpha,t}(\omega_{1})-F_{\alpha,t}(\omega_{2})=(\omega_{1}-\omega_{2})\left(1+\int_{\R}\frac{1}{(x-\omega_{1})(x-\omega_{2})}\dd\wh{\mu}_{\alpha,t}(x)\right).
		\eeqs
		When $\re \omega_{i}\in(E_{\alpha}^{+}+k,\infty), \im\omega_{i}\in(0,k)$ for $i=1,2$ and a positive constant $k$, we have
		\beqs
		\re\frac{1}{(x-\omega_{1})(x-\omega_{2})}=\frac{(x-\re\omega_{1})(x-\re\omega_{2})-\im\omega_{1}\im\omega_{2}}{\absv{x-\omega_{1}}^{2}\absv{x-\omega_{2}}^{2}}>0,\quad x\in\supp\wh{\mu}_{\alpha,t}
		\eeqs
		for all $x\in\supp\wh{\mu}_{\alpha,t}$, so that 
		\beq\label{eq:sqrt_lim_prf}
		\Absv{\frac{F_{\alpha,t}(\omega_{1})-F_{\alpha,t}(\omega_{2})}{\omega_{1}-\omega_{2}}}>1.
		\eeq
		Thus $F_{\alpha,t}$ restricted to the domain $D_{\alpha}\deq\{z:\re z>E_{\alpha}^{+}+k_{0}/2,\im z\in(0,k_{0}/2)\}$ has an analytic inverse by the open mapping theorem, where $k_{0}$ is from Lemma \ref{lem:stab_lim}.
		
		We next prove that the image $F_{\alpha,t}(D_{\alpha})$ contains a disk around $F_{\alpha,t}(\omega_{\alpha,t}(E_{+,t}))$ whose radius admits a uniform lower bound over $t\in[0,1]$. First we observe that we have an upper bound for \eqref{eq:sqrt_lim_prf} since Lemma \ref{lem:Pick_F} implies
		\beqs
		1+\int_{\R}\frac{1}{\absv{x-\omega_{1}}\absv{x-\omega_{2}}}\dd\wh{\mu}_{\alpha,t}(x)\leq 1+4\frac{\wh{\mu}_{\alpha,0}(\R)+t}{k_{0}^{2}}=:C_{\alpha},\quad \omega_{1},\omega_{2}\in D_{\alpha},\,t\in[0,1].
		\eeqs
		We define $C_{\beta}$ similarly. We then see from \eqref{eq:sqrt_lim_prf} that $\absv{\omega-\omega_{\alpha,t}(E_{+,t})}=k_{0}/3$ implies
		\beqs
		\absv{F_{\alpha,t}(\omega)-F_{\alpha,t}(\omega_{\alpha,t}(E_{+,t}))}>k_{0}/3.
		\eeqs
		Now taking $x$ so that $\absv{x-F_{\alpha,t}(\omega_{\alpha,t}(E_{+,t}))}<k_{0}/6$, we have
		\beqs
		\frac{1}{\absv{x-F_{\alpha,t}(\omega)}}<6/k_{0}, \quad \absv{\omega-\omega_{\alpha}(E_{+,t})}=k_{0}/3.
		\eeqs
		If $x-F_{\alpha,t}(\omega)$ has no zero in the domain $\absv{\omega-\omega_{\alpha,t}(E_{+,t})}<k_{0}/3$, we have a contradiction from the maximum modulus principle. This proves $x=F_{\alpha,t}(\omega)$, so that $x\in F_{\alpha,t}(D_{\alpha})$.
		
		As in \cite[Lemma 3.8]{Bao-Erdos-Schnelli2020}, we define
		\beq\label{eq:wtz_def}
		\wt{z}(\omega)\deq -F_{\beta,t}(\omega)+\omega+F_{\alpha,t}^{-1}\circ F_{\beta,t}(\omega).
		\eeq
		By the argument above, we find that $\wt{z}$ is indeed an analytic function in domain $\absv{\omega-\omega_{\beta,t}(E_{+,t})}<k_{0}/6C_{\beta}$. Now we follow the lines of \cite[Lemma 3.8]{Bao-Erdos-Schnelli2020} to get $\wt{z}'(\omega_{\beta,t}(E_{+,t}))=0$ from Lemma \ref{lem:edgechar}, so that
		\beq\label{eq:sqrt_lim_prf1}
		\Absv{\wt{z}(\omega)-E_{+,t}-\frac{1}{2}\wt{z}''(\omega_{\beta,t}(E_{+,t}))(\omega-\omega_{\beta,t}(E_{+,t}))^{2}}\leq C\absv{\omega-\omega_{\beta,t}(E_{+,t})}^{3},
		\eeq
		where the constant $C$ is chosen uniformly over $t\in[0,1]$ and $\absv{\omega-\omega_{\beta,t}(E_{+,t})}<k_{0}/6$.
		
		We apply Lemma \ref{lem:subor_cont} to $\omega_{\beta,t}$, so that there exists $\tau>0$ such that $\absv{z-E_{+,t}}<\tau$ implies $\absv{\omega_{\beta,t}(z)-\omega_{\beta,t}(E_{+,t})}<k_{0}/6$. Then \eqref{eq:sqrt_lim_prf1} reads
		\beq\label{eq:sqrt_lim_prf2}
		\Absv{z-E_{+,t}-\gamma_{\beta,t}^{-2}(\omega-\omega_{\beta,t}(E_{+,t}))^{2}}\leq C\absv{\omega_{\beta,t}(z)-\omega_{\beta,t}(E_{+,t})}^{3},\quad \absv{z-E_{+,t}}<\tau,
		\eeq
		where we defined 
		\beq\label{eq:gamma=2der}
		\gamma_{\beta,t}=\sqrt{\frac{2}{\wt{z}''(\omega_{\beta,t}(E_{+,t}))}}.
		\eeq
		We again follow \cite[Lemma 3.8]{Bao-Erdos-Schnelli2020} to get $\gamma_{\beta,t}\sim1 $ uniformly over $t$. Inverting the expansion \eqref{eq:sqrt_lim_prf2} (taking smaller $\tau$ if necessary) concludes the proof of Proposition \ref{prop:sqrt_lim}. 
	\end{proof}
	
	Now that we have established Proposition \ref{prop:sqrt_lim}, the following result can be easily proved:\begin{cor}\label{cor:sqrt_lim}
		The following hold uniformly over $t\in[0,1]$ and $z\in\C_{+}\cap\{z:\absv{z-E_{+,t}}\leq\tau\}$;
		\begin{gather}\label{eq:m_der}
			\begin{aligned}
				&m'_{\mu_{t}}(z)\sim\absv{z-E_{+,t}}^{-1/2},\qquad \qquad
				&&m''_{\mu_{t}}(z)\sim\absv{z-E_{+,t}}^{-3/2}, \\
				&\omega_{\alpha,t}'(z)\sim\absv{z-E_{+,t}}^{-1/2},\qquad\qquad
				&&\omega_{\alpha,t}''(z)\sim\absv{z-E_{+,t}}^{-3/2},
			\end{aligned}\\
			\begin{aligned}
				&F'_{\alpha,t}(\omega_{\alpha,t}(z))\sim 1,\qquad
				&&F''_{\alpha,t}(\omega_{\alpha,t}(z))\sim 1,\qquad
				&&F''_{\alpha,t}(\omega_{\alpha,t}(z))\sim 1,\qquad 
				&&\caT_{\alpha}(z)\sim 1.
			\end{aligned}
		\end{gather}
		Furthermore, we have the following in the larger domain $\caD_{\tau}(0,\eta_{M})$;
		\beq\begin{aligned}
			\im m_{\mu_{t}}(z)\sim \im \omega_{\alpha}(z)\sim&\im\omega_{\beta}(z)\sim
			\begin{cases}
				\sqrt{\kappa+\eta}, & \text{if }E\leq E_{+,t},\\
				\frac{\eta}{\sqrt{\kappa+\eta}}, & \text{if }E>E_{+,t},
			\end{cases}\\
			\caS_{\alpha\beta}(z)\sim& \sqrt{\kappa+\eta}, \label{eq:S_lim}\\
			\caT_{\alpha}(z)\sim &1\sim \caT_{\beta}(z),
		\end{aligned}\eeq
		where we denoted $\kappa=\absv{z-E_{+,t}}$ and $\eta=\im z$.
	\end{cor}
	\begin{proof}
		The proof of Corollary \ref{cor:sqrt_lim} is the same as \cite[Corollaries 3.10 and 3.11]{Bao-Erdos-Schnelli2020} except some minor modifications. For example, when proving \eqref{eq:S_lim} for the regime $\kappa\sim 1$, we used the following fact;
		\beqs
		\inf\{\omega_{\alpha,t}(E)-\omega_{\alpha,t}(E_{+,t}):E>E_{+,t}+\kappa,\,t\in[0,1]\}>c(\kappa)
		\eeqs
		for some constant $c(\kappa)>0$ depending only on $\kappa>0$. This inequality is a direct consequence of Lemmas \ref{lem:subor_cont} and \ref{lem:subor_real}. We omit further details.
	\end{proof}
	\begin{proof}[Proof of Lemma \ref{lem:stab}]
		The first, second, third, and fourth parts of the lemma are proved in Lemmas \ref{lem:subor_cont}, \ref{lem:subor_bdd}, \ref{lem:stab_lim}, and \ref{lem:edgechar}. Also, \eqref{eq:gamma_def} is a direct consequence of \eqref{eq:subor} and Proposition \ref{prop:sqrt_lim}. In particular, we have
		\beq
		\rho(E)=\im\omega_{\alpha}(E)\int_{\R}\frac{1}{\absv{x-\omega_{\alpha}(E)}^{2}}\dd\mu_{\alpha}(x),\qquad  E\in[E_{+}-\tau,E_{+}],
		\eeq
		so that
		\beq\label{eq:gamma_gamma}
		\gamma_{t}^{3/2}=\gamma_{\alpha,t}\int_{\R}\frac{1}{\absv{x-\omega_{\alpha}(E_{+,t})}^{2}}\dd \mu_{\alpha}(x)\sim 1.
		\eeq
		
		Thus it only remains to prove the last assertion that $\frac{\dd}{\dd t}\gamma_{t}\sim1$. Since $\gamma_{t}\sim 1$, it suffices to prove $\frac{\dd}{\dd t}(\gamma_{t})^{2/3}\sim 1$, which in turn is implied by
		\beq
		\frac{\dd}{\dd t}\omega_{\alpha,t}(E_{+,t})\sim 1,\qquad 	\frac{\dd}{\dd t}\gamma_{\alpha,t}\sim 1
		\eeq
		in light of \eqref{eq:gamma_gamma}. 
		
		We now consider the derivative of $\omega_{\alpha,t}(E_{+,t})$. Recalling $\omega_{\alpha,t}(z)=\omega_{\alpha,0}(z+tm_{\mu_{t}}(z))$ and using \eqref{eq:m_der} with $t=0$, we have
		\beq\label{eq:omega_der_1}
		\frac{\dd}{\dd t}\omega_{\alpha,t}(z)=\omega_{\alpha,0}'(\xi_{t})\frac{\dd\xi_{t}}{\dd t},\qquad \xi_{t}\deq E_{+,t}+tm_{\mu_{t}}(E_{+,t}).
		\eeq
		Thus it suffices to analyze $\xi_{t}$. Since both of $\mu_{0}$ and $\mu_{\SC}^{(t)}$ are Jacobi-type measures for each fixed $t>0$, a direct application of \cite[Proposition 4.7]{Bao-Erdos-Schnelli2020JAM} gives that the upper edge $E_{+,t}$ is the rightmost solution $z=E_{+,t}$ of the equation
		\beq\label{eq:fc_sc}
		\int\frac{1}{\absv{x-z-tm_{\mu_{t}}(z)}^{2}}\dd\mu_{0}(x)=\frac{1}{t}.
		\eeq
		Hence $\xi_{t}$ satisfies
		\beq\label{eq:fc_sc_edge}
		\int\frac{1}{\absv{x-\xi_{t}}^{2}}\dd\mu_{0}(x)=\frac{1}{t},\qquad E_{+,t}=E_{+,t}-tm_{\mu_{t}}(E_{+,t})+tm_{\mu_{t}}(E_{+,t})=\xi_{t}-tm_{\mu_{0}}(\xi_{t}).
		\eeq
		Using the fact that $\mu_{0}$ has square-root decay, we have for each $p\geq 2$ that (see e.g. \cite[Lemma~C.1]{Landon-Yau2017} for a proof)
		\beq\label{eq:sqr_asymp}
		\int\frac{1}{(y-x)^{p}}\dd\mu_{0}(x)\sim (y-E_{+,0})^{3/2-p},\qquad E_{+,0}\leq y\leq E_{+,0}+1.
		\eeq
		Combining \eqref{eq:fc_sc_edge} and \eqref{eq:sqr_asymp}, we have $\absv{\xi_{t}-E_{+,0}}\sim t^{2}$. Hence we have
		\beq\label{eq:xi_der}
		\frac{\dd}{\dd t}\xi_{t}=-\frac{1}{t^{2}}\left(\int_{\R}\frac{1}{(x-\xi_{t})^{3}}\dd\mu_{0}(x)\right)^{-1}\sim t^{-2}\absv{\xi_{t}-E_{+,t}}^{3/2}\sim t,
		\eeq
		where the first equality follows from differentiating \eqref{eq:fc_sc_edge} and the second from \eqref{eq:sqr_asymp} with $p=3$. Plugging in \eqref{eq:xi_der} to \eqref{eq:omega_der_1}, we have
		\beq\label{eq:omega_der}
		\frac{\dd}{\dd t}\omega_{\alpha,t}(z)\sim \absv{\xi_{t}-E_{+,t}}^{-1/2}\cdot t=\absv{m_{\mu_{t}}(E_{+,t})}^{-1}\sqrt{t}\lesssim \sqrt{t}\lesssim 1.
		\eeq
		
		We next move on to the derivative of $\gamma_{\alpha,t}$. As above, we instead prove $\frac{\dd}{\dd t}(\gamma_{\alpha,t})^{-2}\lesssim 1$. Recall from \eqref{eq:gamma=2der} that
		\beq
		\gamma_{\alpha,t}^{-2}=\frac{\wt{z}_{\alpha}''(\omega_{\beta,t}(E_{+,t})}{2},
		\eeq
		where $\wt{z}_{\alpha}$ is the symmetric analogue of $\wt{z}$ in \eqref{eq:wtz_def}. By a direct calculation, we have
		\beq\label{eq:wtz_alp}
		\wt{z}_{\alpha}''(\omega_{\beta,t})=-\frac{F_{\alpha,t}''(\omega_{\beta,t})}{F_{\beta,t}'(\omega_{\alpha,t})}(F_{\beta,t}'(\omega_{\alpha,t})-1)-\frac{F_{\beta,t}''(\omega_{\alpha,t})}{(F_{\beta,t}'(\omega_{\alpha,t}))^{2}}(F_{\alpha,t}'(\omega_{\beta,t}))^{2},
		\eeq
		where we abbreviated $\omega_{\alpha,t}\equiv\omega_{\alpha,t}(E_{+,t})$ and $\omega_{\beta,t}\equiv\omega_{\beta,t}(E_{+,t})$; see e.g. \cite[Eq. (4.34)]{Bao-Erdos-Schnelli2020JAM} for a proof. By the definition of $F_{\alpha,t},F_{\beta,t}$, Lemmas \ref{lem:Pick_F}, \ref{lem:stab_lim}, and \eqref{eq:omega_der}, we easily find that \eqref{eq:wtz_alp} is Lipschitz in $t$, so that
		\beq
		\Absv{\frac{\dd}{\dd t}\wt{z}_{\alpha}''(\omega_{\beta,t}(E_{+,t})}\lesssim 1.
		\eeq
		This completes the proof of $\frac{\dd}{\dd t}\gamma_{t}\lesssim 1$.
	\end{proof}

	\subsection{Stability of $\mu_{A}\boxplus\mu_{B}\boxplus\mu_{\SC}^{(t)}$}
	In this section, we prove Proposition \ref{prop:stab} following the proof of \cite[Proposition 3.1]{Bao-Erdos-Schnelli2020}. More specifically, we prove an upper bound for the distance between $(\omega_{A,t}(z),\omega_{B,t}(z))$ and $(\omega_{\alpha,t}(z),\omega_{\beta,t}(z))$, and Proposition \ref{prop:stab} will follow from the exact same proof as that of \cite[Proposition 3.1]{Bao-Erdos-Schnelli2020}. 
	
	As in \cite{Bao-Erdos-Schnelli2020}, we define the $t$-dependent domain $\caD'\equiv \caD'(\sigma,\tau,\eta_{0})\deq \caD_{\mathrm{in}}\cup\caD_{\mathrm{out}}$ as
	\beqs
	\begin{split}
		\caD_{\mathrm{in}}\equiv\caD_{\mathrm{in}}(\tau,\eta_{0})&\deq \{E+\ii\eta\in\C_{+}:E-E_{+,t}\in[-\tau,N^{-1+\sigma}],\,\im z\in[N^{-1+\sigma},\eta_{0}]\},\\
		\caD_{\mathrm{out}}\equiv\caD_{\mathrm{out}}(\delta,\eta_{0})&\deq \{E+\ii\eta\in\C_{+}:E\in[E_{+,t}+N^{-1+\sigma},\tau^{-1}],\,\im z\in (0,\eta_{0}]\}.
	\end{split}
	\eeqs
	The upper bound for the distance between subordination functions is established in the following lemma;
	\begin{lem}\label{lem:subor_diff}
		Let $\epsilon>0$ be fixed and $\tau$ be as in Lemma \ref{lem:stab_lim}. Then for $N$ sufficiently large, the following hold uniformly over $t\in[0,1]$ and $z\in\caD'$ for some $\eta_{0}$;
		\begin{align}\label{eq:subor_diff}
			\absv{\omega_{A,t}(z)-\omega_{\alpha,t}(z)}+\absv{\omega_{B,t}(z)-\omega_{\beta,t}(z)}
			=O\left(\frac{N^{-1+\epsilon}}{\sqrt{\absv{z-E_{+,t}}}}\right),\\
			|\caS_{AB}|\sim\sqrt{\absv{z-E_{+,t}}},\qquad \im m_{\wh{\mu}}(z)=O(\sqrt{\absv{z-E_{+,t}}})	\nonumber,
		\end{align}
		and $\im m_{\wh{\mu}}(z)\sim\sqrt{z-|E_{+,t}|}$ when $z\in\caD_{\mathrm{in}}$. 
		Furthermore, the following stronger bound hold for $z\in\caD_{\mathrm{out}}\cap\{z:\im z\leq N^{-1}\}$;
		\beq\label{eq:subor_im_diff}
		\absv{\im\omega_{A,t}(z)-\im\omega_{\alpha,t}(z)}
		+\absv{\im\omega_{B,t}(z)-\im\omega_{\beta,t}(z)}
		=O\left(\frac{(\im\omega_{\alpha,t}+\im\omega_{\beta,t})N^{-1+\epsilon}+\im z}{\sqrt{\absv{z-E_{+,t}}}}\right).
		\eeq
	\end{lem}
	The proof of Lemma \ref{lem:subor_diff} can be further divided into two steps; firstly we prove \eqref{eq:subor_diff} in the regime where $\im z=\eta$ is large enough, and secondly we prove a stability result that strengthens a prior upper bound for the left-hand side of \eqref{eq:subor_diff} into the right-hand side of \eqref{eq:subor_diff}.
	
	The first step of the proof of Lemma \ref{lem:subor_diff} is dealt with the following lemma, which corresponds to \cite[Lemma A.2]{Bao-Erdos-Schnelli2020}. To summarize, it enables us to bound the left-hand side of \eqref{eq:subor_diff} in terms of $\Phi_{\alpha\beta}$, defined in \eqref{eq:Phi_def}. For later uses we included the corresponding result for $\Phi_{AB}$.
	\begin{lem}\label{lem:stab_macro}
		Let $(\iota_{1},\iota_{2})$ be either $(\alpha,\beta)$ or $(A,B)$ and let $\wt{\eta}_{0}>0$. For each $t\in[0,1]$, let $\wt{\omega}_{\iota_{1},t},\wt{\omega}_{\iota_{2},t}:\C_{\wt{\eta}_{0}}\to\C_{+}$ be analytic functions where $\C_{\wt{\eta}_{0}}=\{z\in\C_{+}:\im z\geq\wt{\eta}_{0}\}$. Assume that there is a constant $C>0$ such that the following hold for all $t\in[0,1]$ and $z\in\C_{\wt{\eta}_{0}}$;
		\begin{align}
			&\absv{\im\wt{\omega}_{\iota_{1},t}(z)-\im z}\leq C, &
			&\absv{\im\wt{\omega}_{\iota_{2},t}(z)-\im z}\leq C,\\
			&\absv{\wt{r}_{\iota_{1}}(z)}\leq C, &
			&\absv{\wt{r}_{\iota_{2}}(z)}\leq C,\\
			&\wt{r}_{\iota_{1}}(z)\deq \Phi_{\iota_{1}}(\wt{\omega}_{\iota_{1},t}(z),\wt{\omega}_{\iota_{2},t}(z),z), &
			&\wt{r}_{\iota_{2}}(z)\deq \Phi_{\iota_{2}}(\wt{\omega}_{\iota_{1},t}(z),\wt{\omega}_{\iota_{2},t}(z),z).
		\end{align}
		Then there exists a constant $\eta_{0}$ with $\eta_{0}\geq\wt{\eta}_{0}$ such that 
		\begin{align}
			&\absv{\wt{\omega}_{\iota_{1}}(z)-\omega_{\iota_{1}}(z)}\leq 2\norm{r_{\iota}(z)}, &
			&\absv{\wt{\omega}_{\iota_{2}}(z)-\omega_{\iota_{2}}(z)}\leq 2\norm{r_{\iota}(z)},
		\end{align}
		where $r_{\iota}(z)$ denotes the two-dimensional vector $(r_{\iota_{1}}(z),r_{\iota_{2}}(z))$.
	\end{lem}
	Since $\wh{\mu}_{t}$ has the same properties as $\wh{\mu}_{0}$ except for an additional mass of size $t$, we can easily see that the proof of \cite[Lemma A.2]{Bao-Erdos-Schnelli2020} applies to Lemma \ref{lem:stab_macro}. Details of the proof is left to interested readers.
	
	Next, we establish the local stability result which is used to extend the bound to smaller $\eta$.
	\begin{lem}\label{lem:stab_iter}
		Let $\tau,k_{0}>0$ be as in Lemma \ref{lem:stab_lim}. There exist some constants $c,C>0$ depending only on $\mu_{\alpha},\mu_{\beta}$, and $k_{0}$ such that the following holds for all $t\in[0,1]$ and $z_{0}\in\caD_{\tau}(0,\infty)$; if
		\beq\label{eq:stab_iter_assu}
		\absv{\omega_{A,t}(z_{0})-\omega_{\alpha,t}(z_{0})}+\absv{\omega_{B,t}(z)-\omega_{\beta,t}(z)}\leq c\min (k_{0}, \absv{\caS_{\alpha\beta}(z_{0})})
		\eeq
		for the constant $k_{0}$ from Lemma \ref{lem:stab_lim}, then
		\beq\label{eq:stab_iter}
		\absv{\omega_{A,t}(z)-\omega_{\alpha,t}(z)}+\absv{\omega_{B,t}(z)-\omega_{\beta,t}(z)}\leq C\frac{\norm{r(z_{0})}}{\absv{\caS_{\alpha\beta}(z_{0})}}
		\eeq
		where $r(z)\equiv(r_{1}(z),r_{2}(z))\deq\Phi_{\alpha\beta}(\omega_{A}(z),\omega_{B}(z),z)$.
		Furthermore, there exists a constant $C'>0$ depending only on $\mu_{\alpha},\mu_{\beta},k_{0},$ and $\eta_{0}$ such that for all $t\in[0,1]$ and $z_{0}\in\caD_{\tau}(0,\eta_{0})$, \eqref{eq:stab_iter_assu} implies
		\beq\label{eq:stab_iter_r<d}
		\norm{r(z_{0})}\leq C'\bsd.
		\eeq
	\end{lem}
	Specifically, we take $c$ and $C$ so that $c<1/3$, $c^{-1}>4c_{1}^{-3}c_{2}(c_{1}^{-2}c_{2}+1)$, and $C>4(c_{1}^{-2}c_{2}+1)$, where
	\beqs
	c_{1}\deq k_{0}/3, \AND c_{2}\deq \wh{\mu}_{\alpha,0}(\R)+\wh{\mu}_{\beta,0}(\R)+1.
	\eeqs
	\begin{proof}[Proof of Lemma \ref{lem:stab_iter}]
		For simplicity, we abbreviate
		\begin{align*}
			&\omega_{\alpha,t}(z_{0})=\omega_{\alpha}, &
			&\omega_{A,t}(z_{0})=\omega_{A}, &
			&\omega_{A}-\omega_{\alpha}=\Delta\omega_{1},\\
			&\omega_{\beta,t}(z_{0})=\omega_{\beta},	&
			&\omega_{B,t}(z_{0})=\omega_{B},	&
			&\omega_{B}-\omega_{\beta}=\Delta\omega_{2}.
		\end{align*}
		As in \cite[Proposition 4.1]{Bao-Erdos-Schnelli2016}, we first consider the Taylor expansion of $F_{\alpha,t}(\omega_{A})$ around $\omega_{\alpha}$;
		\beq\label{eq:stab_iter_prf1}
		\absv{F_{\alpha,t}(\omega_{A})-F_{\alpha,t}(\omega_{\alpha})-F_{\alpha,t}'(\omega_{\alpha})\Delta\omega_{1}}\leq \absv{\Delta\omega_{1}}^{2}\sup_{\absv{\omega-\omega_{\alpha}}\leq k_{0}/3}\absv{F_{\alpha,t}''(\omega)}\leq c_{1}^{-3}c_{2}\absv{\Delta\omega_{1}}^{2},
		\eeq
		where we used Lemmas \ref{lem:Pick_F} and \ref{lem:stab_lim} and \eqref{eq:stab_iter_assu} in the last inequality. We have the same bound with $\alpha,A$ replaced by $\beta,B$.
		
		On the other hand by the definition of $r_{1}(z)$, we have 
		\beq\label{eq:stab_iter_prf2}
		F_{\alpha,t}(\omega_{A})-F_{\alpha,t}(\omega_{\alpha})=\Delta\omega_{1}+\Delta\omega_{2}+r_{1}(z_{0}).
		\eeq
		Combining \eqref{eq:stab_iter_prf1} and \eqref{eq:stab_iter_prf2}, we find that 
		\beq\label{eq:stab_iter_prf3}
		\absv{(F_{\alpha,t}'(\omega_{\alpha})-1)\Delta\omega_{1}-\Delta\omega_{2}}\leq c_{1}^{-3}c_{2}\norm{\Delta\omega}^{2}+\norm{r(z_{0})},
		\eeq
		where $\Delta\omega=(\Delta\omega_{1},\Delta\omega_{2})$. Now we take a linear combination of \eqref{eq:stab_iter_prf3} and its counterpart with switched indices, so that
		\beq\label{eq:stab_iter_prf4}
		\absv{\caS_{\alpha\beta}(z_{0})}\absv{\Delta\omega_{1}}\leq (\absv{F_{\beta,t}'(\omega_{\beta})-1}+1)\left(c_{1}^{-3}c_{2}\norm{\Delta\omega}^{2}+\norm{r(z_{0})}\right).
		\eeq
		Due to \eqref{eq:stab_iter_assu}, we can solve \eqref{eq:stab_iter_prf4} as a quadratic inequality for $\norm{\Delta\omega}$, so that
		\beqs
		\norm{\Delta\omega}\leq 4\absv{\caS_{\alpha\beta}(z_{0})}^{-1}(c_{1}^{-2}c_{2}+1)\norm{r(z_{0})},
		\eeqs
		where we used $\absv{F'_{\beta,t}(\omega_{\beta})-1}\leq c_{1}^{-2}c_{2}$ from Lemmas \ref{lem:Pick_F} and \ref{lem:stab_lim}.
		
		Finally, we prove \eqref{eq:stab_iter_r<d}. We first note that
		\beqs
		r_{1}(z_{0})=r_{1}(z_{0})-\Phi_{A}(\omega_{A},\omega_{B},z_{0})=F_{\alpha,t}(\omega_{A})-F_{A,t}(\omega_{A}).
		\eeqs
		Then, following \cite[(3.10)]{Bao-Erdos-Schnelli2020} we can easily see that $\absv{m_{A}(\omega_{A})-m_{\alpha}(\omega_{A})}\leq Cd$ from Lemma \ref{lem:stab_lim}. Thus by the definition of $F_{\mu,t}$, the result follows once we have
		\beqs
		\absv{m_{A}(\omega_{A})}\geq c \quad\AND\quad
		\absv{m_{\alpha}(\omega_{A})}\geq c
		\eeqs
		for some constant $c>0$. To see this, we observe from the proof of Lemma \ref{lem:stab_lim} that $\re\omega_{\alpha}-E_{\alpha}^{+}>k_{0}/2$ or $\im \omega_{\alpha}>k_{0}/2$ should hold. This implies either one of the following is true;
		\beq\label{eq:stab_iter_prf5}
		\re \omega_{A}-E_{\alpha}^{+}>k_{0}/6 \qquad\text{or}\qquad \im\omega_{A}>k_{0}/6.
		\eeq
		Combining \eqref{eq:stab_iter_prf5} with Lemma \ref{lem:subor_bdd}, we get $\absv{m_{\alpha}(\omega_{A})}\geq c$ and thus $\absv{m_{A}(\omega_{A})}\geq c$. This concludes the proof of Lemma \ref{lem:stab_iter}.
	\end{proof}
	
	\begin{proof}[Proof of Lemma \ref{lem:subor_diff}]
		Large portion of the proof is identical that of \cite[Lemma 3.12]{Bao-Erdos-Schnelli2020}, and we focus on highlighting the difference rather than explaining the details.
		
		We first prove \eqref{eq:subor_diff}. By \cite[Lemma 4.4]{Bao-Erdos-Schnelli2020} and references therein, we find that the subordination functions $\omega_{A,t}$ and $\omega_{B,t}$ are also Pick functions whose representations satisfy the following;
		\beq\label{eq:Pick_omega}
		\begin{aligned}
			&\omega_{\alpha,t}(z)-z=m_{\wt{\mu}_{\alpha,t}}(z),&
			&\omega_{A,t}(z)-z=m_{\wt{\mu}_{A,t}}(z),\\
			&\wt{\mu}_{\alpha,t}(\R)=\wh{\mu}_{\beta,t}(\R), &
			&\wt{\mu}_{A,t}(\R)=\wh{\mu}_{B,t}(\R).
		\end{aligned}
		\eeq
		In particular $\omega_{\alpha,t}(z)-z$ and $\omega_{A,t}(z)-z$ are both $O(\im z^{-1})$ uniformly over $t\in[0,1]$ and $\im z>\wt{\eta}_{0}$ for some $\eta_{0}>0$. Thus, taking large enough $\eta_{0}$ and applying Lemma \ref{lem:stab_macro} with the choices $\wt{\omega}_{\iota_{1},t}=\omega_{A,t}$ and $\wt{\omega}_{\iota_{2},t}=\omega_{B,t}$, we obtain 
		\beqs
		\absv{\omega_{A,t}(z)-\omega_{\alpha,t}(z)}\leq 2\norm{r(z)},\quad \re z\in [E_{+,t}-\tau,\tau^{-1}],\,\im z=\eta_{0}.
		\eeqs
		Furthermore, replicating the proof of \eqref{eq:stab_iter_r<d} yields $\norm{r(z)}\leq C_{\eta_{M}}\bsd$ for some constant $C_{\eta_{0}}$ depending on $\eta_{0}$.
		
		Now we take $N$ to be sufficiently large so that
		\beqs
		4C_{\eta_{0}}\bsd \leq c\min (k_{0},\inf\{\absv{\caS_{\alpha\beta}(E+\ii\eta_{0})}:E\in[E_{+,t}-\tau,\tau^{-1}]\}),
		\eeqs
		where $c$ is the constant in Lemma \ref{lem:stab_iter}. Here we used the fact that
		\beqs
		\absv{\caS_{\alpha\beta}(E+\ii\eta_{0})}\geq 1-\eta_{0}^{-4}\wh{\mu}_{\alpha,1}(\R)\wh{\mu}_{\beta,1}(\R),
		\eeqs
		which follows from \eqref{eq:Pick_general} and the definition of $\caS_{\alpha\beta}$. Thus by Lemma \ref{lem:stab_iter} we have
		\beq\label{eq:subor_diff_prf1}
		\absv{\omega_{A,t}(z)-\omega_{\alpha,t}(z)}+\absv{\omega_{B,t}(z)-\omega_{\beta,t}(z)}\leq C_{*}\frac{\bsd}{\absv{\caS_{\alpha\beta}(z)}},\quad z=E+\ii\eta_{0}\in\caD_{\tau}(0,\infty).
		\eeq
		where $C_{*}=CC'$ and $C,C'$ are the constants in Lemma \ref{lem:stab_iter} applied to the domain $\caD_{\tau}(0,\eta_{0})$.
		
		Following \cite[Lemma 3.12]{Bao-Erdos-Schnelli2020}, we take $\eta_{m}$ to be the smallest number for which \eqref{eq:subor_diff_prf1} holds for all $z=E+\ii\eta$ with $E\in[E_{+,t}-\tau,\tau^{-1}]$ and $\eta\in[\eta_{m},\eta_{0}]$. By above, such $\eta_{m}$ must exist and we have $\eta_{m}\leq\eta_{0}$. Suppose on the contrary that $\eta_{m}>N^{-1+\sigma}$, and take $\eta_{m}'=\eta_{m}-N^{-2}$ and $E\in[E_{+,t}-\tau,\tau^{-1}]$. By \eqref{eq:Pick_omega}, we find that
		\beqs
		\absv{\omega_{\alpha,t}(E+\ii\eta_{m})-\omega_{\alpha,t}(E+\ii\eta_{m}')}\leq C\eta_{m}^{-2}(\eta_{m}-\eta_{m}')^{2}=CN^{2\sigma-2}
		\eeqs
		for some numeric constant $C>0$, and that the same bound holds for $\omega_{\beta,t}$, $\omega_{A,t}$, and $\omega_{B,t}$. Then we have
		\beqs
		\absv{\omega_{\alpha,t}(E+\ii\eta_{m}')-\omega_{A,t}(E+\ii\eta_{m}')}\leq C_{*}\frac{\bsd}{\absv{\caS_{\alpha\beta}(E+\ii\eta_{m})}}+CN^{2\sigma-2}\leq C(N^{-\gamma/2}\bsd+N^{2\gamma-2}),
		\eeqs
		where we used the fact that $\absv{\caS_{\alpha\beta}(z)}\sim\sqrt{\kappa+\eta}$. Again using the asymptotics for $\absv{\caS_{\alpha\beta}}$, we have
		\beqs
		N^{-\sigma/2}\bsd +N^{2\gamma-2}\lesssim N^{(-1+\gamma)/2}\lesssim\absv{\caS_{\alpha\beta}(E+\ii\eta_{m}')},
		\eeqs
		where we used Assumption \ref{assump:ABconv}. Taking large enough $N$, we see that the point $E+\ii\eta_{m}'$ satisfies the assumptions of Lemma \ref{lem:stab_iter}, so that \eqref{eq:subor_diff_prf1} holds true at the point. Since $E$ was arbitrary chosen in $[E_{+,t}-\tau,\tau^{-1}]$, we obtain a contradiction to $\eta_{m}>N^{-1+\sigma}$. Therefore we conclude that \eqref{eq:subor_diff_prf1} holds for all $z\in\caD_{\tau}(N^{-1+\sigma},\eta_{0})$, and \eqref{eq:subor_diff} for all $z\in\caD_{\mathrm{in}}$. The proof for $\caD_{\mathrm{out}}$ is exactly the same, except we take $E\in[E_{+,t}+N^{-1+\sigma}],\eta_{m}\in(0,N^{-1+\sigma}],$ and $\eta_{m}'=\eta_{m}^{2}$. 
		
		The rest of the proof is exactly the same as \cite[Lemma 3.13]{Bao-Erdos-Schnelli2020} and we omit the details for simplicity.
	\end{proof}

	\begin{proof}[Proof of Proposition \ref{prop:stab}]
		The proof is almost identical to that of \cite[Proposition 3.1]{Bao-Erdos-Schnelli2020}, and the only difference is that the $F$-transform therein should be replaced by $F_{t}$ defined in \eqref{eq:mF_def}. In particular, using Lemmas \ref{lem:stab_lim} and \ref{lem:subor_diff} in place of Lemmas 3.7 and 3.12 of \cite{Bao-Erdos-Schnelli2020}, one can easily check that the proof of \cite[Proposition 3.1]{Bao-Erdos-Schnelli2020} applies verbatim. We omit further details.
	\end{proof}

	\section{Local laws for $H_{t}$}\label{sec:local law}
	
	In this section, we prove local laws for $H_{t}$ following the same strategy as in \cite{Bao-Erdos-Schnelli2020}. The proofs in \cite{Bao-Erdos-Schnelli2020} have to be carefully modified in order to deal with the effect of DBM. Due to similarity, we mainly focus on explaining such modification and refer to \cite{Bao-Erdos-Schnelli2020} whenever the same calculation therein applies.
	
	Recall the definition of $\caD_{\tau}(a,b)$ from \eqref{eq:D_def}. The precise statement of the local law is as follows;
	\begin{thm}[Local laws for $H_{t}$]\label{thm:ll} Suppose that Assumption \ref{assump:ABconv} holds. Let $\tau>0$ be a sufficiently small constant, $\sigma>0$, and $d_{1},\dots, d_{N}\in \bbC$ be deterministic complex numbers satisfying $\max_{i\in\braN} \absv{d_{i}}\leq 1$. Then we have
		\beq\label{eq:averll}
		\Absv{\frac{1}{N}\sum_{i}d_{i}\left(G_{ii}-\frac{1}{\fra_{i}-\omega_{A}(z)}\right)}\prec \frac{1}{N\eta}
		\eeq
		uniformly over $t\in[0,1]$ and $z\in \caD_{\tau}(N^{-1+\sigma},1)$. Furthermore, we have
		\beq\label{eq:etrll}
		\max_{i,j\in\braN}\Absv{G_{ij}(z)-\delta_{ij}\frac{1}{\fra_{i}-\omega_{A}(z)}}
		+|(U\adj G)_{ij}|\prec \sqrt{\frac{\im m_{\wh{\mu}_{t}}(z)}{N\eta}}+\frac{1}{N\eta}
		\eeq
		uniformly over the same domain for $t$ and $z$. The same results hold true if we replace $(G,U,\fra_{i},\omega_{A})$ by $(\caG,U\adj,\frb_{i},\omega_{B})$, respectively.
	\end{thm}
	We postpone the proof of Theorem \ref{thm:ll} to Section \ref{sec:LL}.
	
	As in the statement of Theorem \ref{thm:ll}, we always work with the free convolution $\wh{\mu}_{t}=\mu_{A}\boxplus\mu_{B}\boxplus\mu_{\SC}^{(t)}$ rather than its limit $\mu_{t}$. Thus we use the shorthand notation $\wh{m}(z)=m_{\wh{\mu}_{t}}(z)$ without any confusion. In the following sections, we set $\eta_{m}=N^{-1+\sigma}$, $\eta_{M}=1$ and denote $\caD_{\tau}(\eta_{m},\eta_{M})$ by $\caD_{\sigma}$.
	
	\subsection{Outline of the proof of local law}
	
	To simplify the presentation, we introduce the following control parameters depending on $N$, $t$, and $z$.
	\begin{align*}
		\Psi\equiv\Psi(z)&\deq \sqrt{\frac{1}{N\eta}}, & 
		\Pi\equiv \Pi(z)&\deq \sqrt{\frac{\im m_{H_{t}}}{N\eta}},\\
		\Pi_{i}\equiv\Pi_{i}(z)&\deq \sqrt{\frac{\im(G_{ii}+\caG_{ii})}{N\eta}},&
		\Pi_{i}^{W}\equiv\Pi_{i}^{W}(z)&\deq \sqrt{\frac{\im (WGW)_{ii}}{N\eta}}.
	\end{align*}
	
	Before proceeding to the actual proof, we first present the outline of the proof of Theorem \ref{thm:ll}. We first define random functions $\omega_{A}^{c}$ and $\omega_{B}^{c}$ in $z$ as follows;
	\begin{align}\label{eq:apxsubor_def}
		\omega_{A}^{c}&\deq z-\frac{\tr\wt{B}G}{\tr(G)}+t\tr G,&
		\omega_{B}^{c}&\deq z-\frac{\tr AG }{\tr G}+t\tr G.
	\end{align}
	These functions will serve as random approximates of the genuine subordination functions $\omega_{A}$ and $\omega_{B}$ associated to the free convolution $\mu_{A}\boxplus\mu_{B}\boxplus\mu_{\SC}^{(t)}$. Note that these functions have an additional term $t\tr G$ compared to those used in \cite{Bao-Erdos-Schnelli2020}. With these choices, we follow the same three step strategy as in \cite{Bao-Erdos-Schnelli2020}.
	
	The first step is to prove the entrywise subordination, that is, estimates of the form
	\beq\label{eq:outline_etr}
	\Absv{G_{ij}-\delta_{ij}\frac{1}{\fra_{i}-\omega_{A}^{c}}}\prec	\Psi.
	\eeq
	More specifically, we write
	\beq\label{eq:zG}
	zG_{ij}+\delta_{ij}-\fra_{i}G_{ij}=(\wt{B}G)_{ij}+\sqrt{t}(WG)_{ij}
	\eeq
	using $zG+I=HG$ and apply Gaussian integration by parts to the rightmost side of \eqref{eq:zG}. Following calculations in \cite{Bao-Erdos-Schnelli2020} for the first term and \cite{Lee-Schnelli2018} for the second, we see that estimating the right-hand side of \eqref{eq:zG} is equivalent to $\Psi$ upper bounds for $Q_{ij}$ and $L_{ij}$ defined as
	\beqs
	Q_{ij}\deq (\wt{B}G)_{ij}\tr G -G_{ij}\tr(\wt{B}G)	\AND
	L_{ij}\deq (WG)_{ij} +\sqrt{t}\tr(G)G_{ij}.
	\eeqs
	In fact, after some algebraic calculation, we arrive at 
	\beq\label{eq:Lambdac<Q+L}
	G_{ij}-\delta_{ij}\frac{1}{\fra_{i}-\omega_{A}^{c}}
	=-\frac{Q_{ij}}{(\fra_{i}-\omega_{A}^{c})\tr G}	-\frac{\sqrt{t}L_{ij}}{\fra_{i}-\omega_{A}^{c}}.
	\eeq
	
	In order to apply Stein's lemma to $Q_{ij}$, we use partial randomness decomposition in Lemma \ref{lem:prd} to extract Gaussian random variables from the Haar unitary matrix $U$ as in \cite{Bao-Erdos-Schnelli2020}. Consequently, the quantity $Q_{ij}$ can be controlled by the two quantities in (4.10) of \cite{Bao-Erdos-Schnelli2020}, namely
	\begin{align}\label{eq:ST_def}
		&S_{ij}\deq \bsh_{i}\adj \wt{B}^{\angi}G\bse_{j}, &
		&T_{ij}\deq \bsh_{i}\adj G\bse_{j},
	\end{align}
	where $\bsh_{i}$ and $\wt{B}^{\angi}$ are defined in Lemma \ref{lem:prd}. Indeed, we can see from discussions below (4.10) of \cite{Bao-Erdos-Schnelli2020} that quantities in \eqref{eq:ST_def} are directly connected to $(\wt{B}G)_{ij}$. Following \cite{Bao-Erdos-Schnelli2020}, we find that it easier to work with auxiliary quantities $P_{ij}$ and $K_{ij}$ instead of $Q_{ij}$, defined by
	\beqs
	\begin{split}
		P_{ij}&\deq (\wt{B}G)_{ij}\tr G-G_{ij}\tr(\wt{B}G)+(G_{ij}+T_{ij})\Upsilon, \\
		K_{ij}&\deq T_{ij}+(\frb_{i}T_{ij}+(\wt{B}G)_{ij})\tr G-(G_{ij}+T_{ij})\tr(\wt{B}G),
	\end{split}
	\eeqs
	where we also defined
	\beqs
	\Upsilon\deq \tr \wt{B}G-(\tr \wt{B}G)^{2}+\tr G\tr \wt{B}G\wt{B}.
	\eeqs
	On the other hand, calculations for $L_{ij}$ involve another quantity $J_{i}$ defined as
	\beqs
	J_{i}\deq (WGW)_{ii}-\tr G +\sqrt{t}\tr G (WG)_{ii}.
	\eeqs
	In summary, the first step mainly concerns estimates for the four quantities $J_{i},L_{ij},P_{ij}$, and $K_{ij}$. The corresponding result is Proposition \ref{prop:local law 1}, which is proved in the next subsection.
	
	In the second step, we estimate the distance between $\omega_{A}^{c}$ and $\omega_{A}$. The bulk of the proof is devoted to rough and optimal fluctuation averaging results for $Q_{ii}$ and $L_{ii}$. In the rough fluctuation averaging, Propositions \ref{prop:RFA of Laa} and \ref{prop:FA-1}, we prove that 
	\beq\label{eq:outline_FA}
	\Absv{\frac{1}{N}\sum_{i}d_{i}Q_{ii}}\prec\Psi\Pi \AND
	\Absv{\frac{1}{N}\sum_{i}d_{i}L_{ii}}\prec\Pi^{2},
	\eeq
	for generic bounded weights $d_{1},\cdots,d_{N}$. Note that the bounds in \eqref{eq:outline_FA} are much smaller than the bound $\Psi$ in \eqref{eq:outline_etr} due to an averaging effect of fluctuations. And then in the optimal fluctuation averaging, Proposition \ref{prop:FA-ZAZB}, we take a specific weights for $Q_{ii}$ and $L_{ii}$ and consider a specific combination of two averages. Our choice leads to an improved bound $\Pi^{2}$ for the first term in \eqref{eq:outline_FA} and to an estimate of the form
	\beq\label{eq:outline_OFA}
	\absv{\caS_{AB}(z)\Lambda_{A}(z)^{2}+\caT_{A}(z)\Lambda_{A}(z)+O(\Lambda_{A}(z)^{3})}\prec \Pi^{2}
	\eeq
	where $\Lambda_{A}(z)=\omega_{A}^{c}(z)-\omega_{A}(z)$. The bound \eqref{eq:outline_OFA} eventually results in the bound $\absv{\Lambda_{A}(z)}\prec \Psi^{2}$. Details for the second step can be found in Section \ref{sec:FA}.
	
	Throughout both the first and second steps, we fix a spectral parameter $z$, assume that a weak, probabilistic bound holds at the point $z$, and use this assumption as an input. In the third and final step, we prove a weak local law to ensure that this a priori bound is in fact true in the whole domain. More specifically, we invoke the proofs in previous steps to prove weaker but quantitative versions of entrywise subordination and fluctuation averaging, in the sense that they do not depend on a probabilistic input. Then we use a bootstrapping argument to conclude a weak local law, Theorem \ref{thm:weak local law}. Feeding the weak law back to the first and second steps and using another bootstrapping argument lead to the final result. This step is presented in Section \ref{sec:LL}.

	\subsection{Entrywise subordination}
	
	We introduce the following notations for the errors we need to control;
	\begin{gather}
		\begin{aligned}\label{eq:lambda_def}
			\Lambda_{ij}&\deq\Absv{G_{ij}-\frac{\delta_{ij}}{\fra_{i}-\omega_{A}}}, &
			\Lambda_{ij}^{c}&\deq\Absv{G_{ij}-\frac{\delta_{ij}}{\fra_{i}-\omega_{A}^{c}}},&
			\Lambda_{L}&\deq\max_{i,j}\absv{L_{ij}}\\
			\Lambda_{\rme}&\deq\max_{i,j} \Lambda_{ij}, &
			\Lambda_{\rme}^{c}&\deq\max_{i,j}\Lambda_{ij}^{c},&
			\Lambda_{T}&\deq \max_{i,j}\absv{T_{ij}}.
		\end{aligned}
	\end{gather}
	We also write $\wt{\Lambda}_{ij},\wt{\Lambda}_{ij}^{c},\wt{\Lambda}_{T},\wt{\Lambda}_{\rme},\wt{\Lambda}_{\rme}^{c}$ to represent their analogues obtained by switching the roles of $(A,B)$, $(U,U\adj)$, and $(W,\caW)$. For example, we write
	\beqs
	\wt{\Lambda}_{ij}\deq\Absv{\caG_{ij}-\delta_{ij}\frac{1}{\frb_{i}-\omega_{B}}}.
	\eeqs
	Finally, we take a collection of smooth cut-off function $\varphi_{\caL}\equiv\varphi:\R\to\R$ indexed by $\caL>0$ such that $\varphi(x)$ is non-increasing in $\absv{x}$ and that
	\beq\label{eq:cutoff_def}
	\begin{aligned}
		&\varphi(x)=\begin{cases}
			1 & \absv{x}\leq \caL,\\
			0 & \absv{x}\geq 2\caL,
		\end{cases}&
		&\sup_{x\in\R}\absv{\varphi'(x)}\leq C\caL^{-1}.
	\end{aligned}
	\eeq
	
	We introduce another notation that strengthens the notion of stochastic dominance. For an $N$-dependent random variable $X$ that may also depend on $t$ and $z$, we write $X=O_{\prec}(Y)$ for a positive deterministic function $Y$ of $N,t,$ and $z$ when the following holds: For any fixed $\epsilon>0$ and $p\in\N$, there exists an $N_{0}\in \N$ depending only on $\epsilon$ and $p$ such that
	\beq\label{eq:stocha_dom}
	\expct{\absv{X}^{p}}\leq N^{\epsilon}Y^{p}
	\eeq
	whenever $N\geq N_{0}$. In this case, we often write $O_{\prec}(Y)$ in place of $X$: Recall that the same notation stands for the usual stochastic dominance in the main manuscript. Indeed, we can easily see that \eqref{eq:stocha_dom} is stronger than the usual stochastic dominance by Markov's inequality.
	
	\begin{lem}\label{lem:etrll_J}
		Let $\caL>0$ and $p\in\bbN$ be fixed and define
		\beqs
		\Gamma_{i0}\deq \absv{G_{ii}}^{2}+\absv{(GW)_{ii}}^{2}+\absv{(WG)_{ii}}^{2}+\absv{\tr G}^{2}
		\eeqs
		Then there exist $N_{0}\in\bbN$ and $C_{p,\caL}>0$ depending only on $\caL$ and $p$ such that
		\beqs
		\expct{\absv{J_{i}\varphi(\Gamma_{i0})}^{2p}}\leq C_{p,\caL}\Psi^{2p}
		\eeqs
		for all $N\geq N_{0}$, $i\in\braN$, $t\in[0,1]$ and $z\in\caD$.
	\end{lem}
	\begin{proof}
		We use the following shorthand notations throughout the proof;
		\beq\label{eq:def_frj}
		\frj_{i}^{(p,q)}= J_{i}^{p}\ol{J_{i}}^{q}\varphi(\Gamma_{i0})^{p+q} \qquad p,q\in\N,
		\eeq
		with conventions $\frj_{i}^{(0,0)}=1$ and $\frj_{i}^{(-1,1)}=0$. Using this notation, we write
		\beq\label{eq:rec_WGW}
		\expct{\frj_{i}^{(p,p)}}=\expct{(WGW)_{ii}\varphi(\Gamma_{i0})\frj_{i}^{(p-1,p)}}-\expct{\tr G \varphi(\Gamma_{i0})\frj_{i}^{(p-1,p)}}
		+\sqrt{t}\expct{\tr G(WG)_{ii}\varphi(\Gamma_{i0})\frj_{i}^{(p-1,p)}}.
		\eeq
		Applying Stein's lemma to the first term yields
		\beq\label{eq:rec_WGW1}
		\begin{split}
			&\expct{(WGW)_{ii}\varphi(\Gamma_{i0})\frj_{i}^{(p-1,p)}}
			=\sum_{k}\expct{(WG)_{ik}W_{ki}\varphi(\Gamma_{i0})\frj_{i}^{(p-1,p)}}	\\
			=&\expct{\tr G\varphi(\Gamma_{i0})\frj_{i}^{(p-1,p)}}
			-\sqrt{t}\expct{\tr G (WG)_{ii}\varphi(\Gamma_{i0})\frj_{i}^{(p-1,p)}}\\
			&+\frac{2p-1}{N}\sum_{k}\Expct{(WG)_{ik}\frac{\partial \varphi(\Gamma_{i0})}{\partial W_{ik}}\frj_{i}^{(p-1,p)}}
			+\frac{p-1}{N}\sum_{k}\Expct{(WG)_{ik} \left(\frac{\partial J_{i}}{\partial W_{ik}}\right)\varphi(\Gamma_{i0})\frj_{i}^{(p-2,p)}}\\
			&+\frac{p}{N}\sum_{k}\Expct{(WG)_{ik} \left(\frac{\partial \ol{J_{i}}}{\partial W_{ik}}\right)\varphi(\Gamma_{i0})\frj_{i}^{(p-1,p-1)}}.
		\end{split}
		\eeq
		The first two terms of \eqref{eq:rec_WGW1} are canceled with the last two terms of \eqref{eq:rec_WGW}. For the remaining terms, we use the following lemma;
		\begin{lem}\label{lem:rec_W_err}
			Let $X_{i}$ be one of $G_{ii}$, $(GW)_{ii}$, $(WG)_{ii}$ or $(WGW)_{ii}$. Then there exists a constant $C_{\caL}$ depending only on $\caL$ such that the following hold for all $i\in\braN$, $t\in[0,1]$, and $z\in\caD$;
			\begin{align}
				&\frac{1}{N}\Absv{\sum_{k}(WG)_{ik}\frac{\partial X_{i}}{\partial W_{ik}}}\varphi(\Gamma_{i0})
				\leq C_{\caL}(N^{-1}+\Pi_{i}^{2}+(\Pi_{i}^{W})^{2})\varphi(\Gamma_{i0}),	\label{eq:rec_W_err}\\
				&\frac{1}{N}\Absv{\sum_{k}(WG)_{ik}\frac{\partial \tr G}{\partial W_{ik}}}\varphi(\Gamma_{i0})
				\leq C_{\caL}\Psi^{2}(N^{-1}+\Pi_{i}^{2}+(\Pi_{i}^{W})^{2})\varphi(\Gamma_{i0})	\label{eq:rec_W_err1}.
			\end{align}
			Furthermore, the same bounds hold true if we replace $X_{i}$ or $\tr G$ with its complex conjugate or if $\Gamma_{i0}$ is replaced by a larger quantity than $\Gamma_{i0}$.
		\end{lem}
		\begin{proof}[Proof of Lemma \ref{lem:rec_W_err}]
			Due to similarity, we only consider the first term of \eqref{eq:rec_W_err} with the choice $X_{i}=(WGW)_{ii}$. Computing the derivative explicitly, we find that
			\beqs
			\sum_{k}(WG)_{ik}\frac{\partial(WGW)_{ii}}{\partial W_{ik}}=(WG^{2}W)_{ii}-\sqrt{t}(WG)_{ii}(WG^{2}W)_{ii}+(WG)_{ii}^{2}.
			\eeqs
			Applying Cauchy-Schwarz inequality to the entry $(WGW)_{ii}$ gives
			\beqs
			\absv{(WG^{2}W)_{ii}}\leq \norm{GW\bse_{i}}\norm{G\adj W\bse_{i}}\leq \frac{\im (WGW)_{ii}}{\eta}=N(\Pi_{i}^{W})^{2},
			\eeqs
			and the definition of $\Gamma_{i0}$ yields $\absv{(WG)_{ii}}\varphi(\Gamma_{i0})\leq \sqrt{2\caL}\varphi(\Gamma_{i0})$. Thus we conclude
			\beqs
			\frac{1}{N}\Absv{\sum_{k}(WG)_{ik}\frac{\partial (WGW)_{ii}}{\partial W_{ik}}}\varphi(\Gamma_{i0})\leq C_{\caL}(N^{-1}+(\Pi_{i}^{W})^{2})\varphi(\Gamma_{i0})
			\eeqs
			as desired.
		\end{proof}
		After several applications of Leibniz and chain rules, we see that the coefficient of $\frj_{i}$ in each of the third, fourth, and fifth terms of \eqref{eq:rec_WGW1} can be further decomposed into quantities in Lemma \ref{lem:rec_W_err}, up to a factor of $C_{\caL}$. To sum up, we have proved that
		\beq\label{eq:rec_WGW_1}
		\E[\frj_{i}^{(p,p)}]\leq C_{\caL}\E[(N^{-1}+\Pi_{i}^{2}+\Pi_{i}^{W})^{2}\frj_{i}^{(p-1,p-1)}]
		\eeq
		
		On the other hand, we also have
		\beqs
		\Pi_{i}^{2}\varphi(\Gamma_{i0})\leq C_{\caL}\Psi^{2}\varphi(\Gamma_{i0})	\quad\AND\quad
		(\Pi_{i}^{W})^{2}\varphi(\Gamma_{i0})\leq \Psi^{2}(\absv{J_{i}}+C_{\caL})\varphi(\Gamma_{i0}),
		\eeqs
		where we used the definition of $J_{i}$ in the second inequality. Plugging these inequalities into \eqref{eq:rec_WGW}, we find that
		\beq\label{eq:rec_WGW2}
		\expct{\frj_{i}^{(p,p)}}\leq pC_{\caL}\expct{\Psi^{2}\absv{J_{i}\varphi(\Gamma_{i0})}^{2p-1}}+pC_{\caL}\expct{\Psi^{2}\absv{J_{i}\varphi(\Gamma_{i0})}^{2p-2}}+pC_{\caL}\Psi^{2}\expct{\frj_{i}^{(p,p)}}.
		\eeq
		Applying Jensen's inequality to the first two terms of \eqref{eq:rec_WGW2}, we obtain that $x=\Psi^{-1}\expct{\frj_{i}^{(p,p)}}^{1/2p}$ satisfies the quadratic inequality
		\beqs
		x^{2}\leq C_{p,\caL}(\Psi x+1),
		\eeqs
		so that $x\leq C_{p,\caL}$ for a constant $C_{p,\caL}$ depending only on $p$ and $\caL$. This concludes the proof of Lemma \ref{lem:etrll_J}.
	\end{proof}
	
	\begin{lem}\label{lem:etrll_L}
		Let $\caL>0$ and $p\in\N$ be fixed and define
		\beq\label{eq:Gamma0_def}
		\Gamma_{ij0}\deq \absv{\tr G}^{2}+\absv{G_{ii}}^{2}+\absv{G_{jj}}^{2}+\absv{G_{ij}}^{2}+\absv{G_{ji}}^{2}+\absv{(WG)_{ii}}^{2}+\absv{(GW)_{ii}}^{2},
		\eeq
		Then there exist $N_{0}\in\N$ and $C_{p,\caL}>0$ depending on on $\caL$ and $p$ such that 
		\beq\label{eq:etrll_WG}
		\expct{\absv{L_{ij}\varphi(\Gamma_{ij0})}^{2p}}
		\leq C_{p,\caL}\Psi^{2p}
		\eeq
		for all $N\geq N_{0}$, $i,j\in\braN$, $t\in[0,1]$, and $z\in\caD$.
	\end{lem}
	\begin{proof}
		We take $\frl_{ij}^{(p,q)}=L_{ij}^{p}\ol{L_{ij}}^{q}\varphi(\Gamma_{ij0})^{p+q}$ for $p,q\in\N$, so that the left-hand side of \eqref{eq:etrll_WG} is equal to $\E[\frl_{ij}^{(p,p)}]$. Following lines of proof of Lemma \ref{lem:etrll_J}, we can prove that there are random variables $\rmk_{1},\rmk_{2}$, and $\rmk_{3}$ with
		\beq\label{eq:rec_WG}
		\expct{\frl_{ij}^{(p,p)}}=\expct{\rmk_{1}\frl_{ij}^{(p-1,p)}}+\expct{\rmk_{2}\frl_{ij}^{(p-2,p)}}+\expct{\rmk_{3}\frl_{ij}^{(p-1,p-1)}}
		\eeq
		such that 
		\beq\label{eq:rec_WG_1}
		\absv{\rmk_{i}}\varphi(\Gamma_{ij0})\leq C_{\caL}(N^{-1}+\Pi_{i}^{2}+\Pi_{j}^{2}+(\Pi_{i}^{W})^{2})\varphi(\Gamma_{ij0}), \qquad i=1,2,3.
		\eeq
		The proof of this fact follows from the following off-diagonal variant of Lemma \ref{lem:rec_W_err}. We omit its proof since it is identical to that of Lemma \ref{lem:rec_W_err} except we use $\norm{G\bse_{j}}^{2}\leq N\Pi_{j}^{2}$ as an additional input. 
		\begin{lem}\label{lem:etrll_L_err}
			Let $X_{ij}$ be one of $G_{ii}$, $(WG)_{ii}$, $(GW)_{ii}$, $G_{jj}$, $G_{ij}$, $G_{ji}$. Then there exists a constant $C_{\caL}$ depending only on $\caL$ such that the following hold for all $i,j\in\braN$, $t\in[0,1]$, and $z\in\caD$;
			\begin{align*}
				&\frac{1}{N}\Absv{\sum_{k}G_{kj}\frac{\partial X_{ij}}{\partial W_{ki}}}\varphi(\Gamma_{ij0})\leq C_{\caL}(N^{-1}+\Pi_{i}^{2}+\Pi_{j}^{2}+(\Pi_{i}^{W})^{2})\varphi(\Gamma_{ij0}),\\
				&\frac{1}{N}\Absv{\sum_{k}G_{kj}\frac{\partial \tr G}{\partial W_{ki}}}\varphi(\Gamma_{ij0})\leq C_{\caL}\Psi^{2}(N^{-1}+\Pi_{i}^{2}+\Pi_{j}^{2}+(\Pi_{i}^{W})^{2})\varphi(\Gamma_{ij0}).
			\end{align*}
			The same set of inequalities holds true if $X_{ij}$ and $\tr G$ are replaced by their complex conjugates.
		\end{lem}
		
		We now deduce \eqref{eq:etrll_WG} from \eqref{eq:rec_WG}. As in Lemma \ref{lem:etrll_J}, we further bound the control parameters using
		\beq\label{eq:etr_WG1}
		(\Pi_{i}^{2}+\Pi_{j}^{2})\varphi(\Gamma_{ij0})\leq C_{\caL}\Psi^{2} \AND
		(\Pi_{i}^{W})^{2}\varphi(\Gamma_{ij0})\leq C_{\caL}\Psi^{2}(\absv{J_{i}}+\caL)\varphi(\Gamma_{ij0}).
		\eeq
		Plugging \eqref{eq:etr_WG1} into \eqref{eq:rec_WG} and using Jensen and H\H{o}lder inequalities, we obtain
		\beqs
		\expct{\frl_{ij}^{(p,p)}}\leq C_{\caL}\Psi^{2}\left((1+\expct{\frj_{i}^{p,p}}^{1/2p})\expct{\frl_{ij}^{(p,p)}}^{\frac{2p-1}{2p}}+(1+\expct{\frj_{i}^{p/2,p/2}}^{1/p})\expct{\frl_{ij}^{(p,p)}}^{\frac{p-1}{p}}\right),
		\eeqs
		where $\frj_{i}$ is defined in \eqref{eq:def_frj}. Since $\varphi(\Gamma_{ij0})\leq\varphi(\Gamma_{i0})$, Lemma \ref{lem:etrll_J} implies that
		\beqs
		\expct{\frj_{i}^{k,k}}^{1/2k}\leq C_{k,\caL}\Psi,\quad k\in\N.
		\eeqs
		Then we follow the exact same argument as in Lemma \ref{lem:etrll_J} to conclude \eqref{eq:etrll_WG}.
	\end{proof}
	
	Furthermore, we have the similar estimates for $P_{ij}$ and $K_{ij}$.
	\begin{lem}\label{lem:etrll_PK}
		Let $\epsilon,\caL>0 $ and $p\in \bbN$ be fixed and define
		\beq\label{eq:Gamma_def}
		\Gamma_{ij}\deq \Gamma_{ij0}+|\caG_{ii}|^{2}+|\caG_{jj}|^{2}+|T_{ij}|^{2}+|\wt{T}_{ij}|^{2}+|\tr G|^{2}+|\tr \wt{B}G|^{2}+|\tr \wt{B}G\wt{B}|^{2}.
		\eeq
		Then there exists $N_{0}\equiv N_{0}(\epsilon,\caL,p)\in \bbN$ such that
		\begin{align}
			\expct{\absv{P_{ij}\varphi(\Gamma_{ij})}^{2p}}\leq \Psi^{2p}N^{\epsilon},\\
			\expct{\absv{K_{ij}\varphi(\Gamma_{ij})}^{2p}}\leq \Psi^{2p}N^{\epsilon},
		\end{align}
		for all $N\geq N_{0}$, $i,j\in\braN$, $t\in[0,1]$, and $z\in\caD$.
	\end{lem}
	\begin{proof}
		The proof is similar to that of Lemma 8.3 in \cite{Bao-Erdos-Schnelli2020}, and all the differences originate from the identity 
		\beq\label{eq:rotate identity}
		\wt{B}G=zG-AG-\sqrt{t}WG+I.
		\eeq
		Due to the additional term $\sqrt{t}WG$ in the above identity, several new terms arise that do not appear in \cite{Bao-Erdos-Schnelli2020}. We notice that from the definition of $\Gamma_{ij}$, $\varphi(\Gamma_{i0})$ and $\varphi(\Gamma_{ij0})$ are larger than $\varphi(\Gamma_{ij})$. Hence we can apply the estimates from Lemma \ref{lem:etrll_J} and \ref{lem:etrll_L}. 
		
		We first introduce the counterpart of Lemma 5.3 in \cite{Bao-Erdos-Schnelli2020} that handles errors arising along the proof.
		\begin{lem}[Lemma 5.3 in \cite{Bao-Erdos-Schnelli2020}]\label{lem:A.6}
			Suppose the assumptions of Proposition \ref{lem:etrll_PK} hold. Let $Q\in M_{N}(\C)$ be a generic matrix and set $X_{i}=I$ or $\wt{B}^{\angi}$ and $X=I$ or $A$. Then the following hold true for all $N\in\N$, $t\in[0,1]$, and $z\in\caD$:
			\begin{align*}
				\frac{1}{N}\Absv{\sum_{k}^{(i)}\frac{\partial \norm{\bsg_{i}}^{-1}}{\partial g_{ik}}\bse_{k}\adj X_{i}G\bse_{j}\varphi(\Gamma_{ij})}
				&=O_{\prec}(N^{-1})\norm{Q} \varphi(\Gamma_{ij}), \\
				\frac{1}{N}\Absv{\sum_{k}^{(i)}\bse_{i}\adj X\frac{\partial G}{\partial g_{ik}}\bse_{j}\bse_{k}\adj X_{i}G\bse_{j}\varphi(\Gamma_{ij})}
				&\leq C_{\caL}\norm{Q}(\Pi_{i}^{2}+\Pi_{j}^{2})\varphi(\Gamma_{ij}),\\
				\frac{1}{N}\Absv{\sum_{k}^{(i)}\frac{\partial T_{ij}}{\partial g_{ik}}\bse_{k}\adj X_{i}G\bse_{j}\varphi(\Gamma_{ij})}&
				\leq C_{\caL}\norm{Q}(\Pi_{i}^{2}+\Pi_{j}^{2})\varphi(\Gamma_{ij}),\\
				\frac{1}{N}\Absv{\sum_{k}^{(i)}\tr\left(Q\frac{\partial G}{\partial g_{ik}}\right)\bse_{k}\adj X_{i}G\bse_{j}\varphi(\Gamma_{ij})}
				&\leq C_{\caL}\norm{Q}\Psi^{2}(\Pi_{i}^{2}+\Pi_{j}^{2}+t(\Pi_{i}^{W})^{2})\varphi(\Gamma_{ij}),\\
				\frac{1}{N}\Absv{\sum_{k}^{(i)}\tr\left(Q\frac{\partial G}{\partial g_{ik}}\right)\bse_{k}\adj X_{i}\mr{\bsg}_{j}\varphi(\Gamma_{ij})}
				&\leq C_{\caL}\norm{Q}\Psi^{2}(\Pi_{i}^{2}+\Pi_{j}^{2}+t(\Pi_{i}^{W})^{2})\varphi(\Gamma_{ij}),\\
				\frac{\sqrt{t}}{N}\Absv{\sum_{k}^{(i)}\tr\left(WQ\frac{\partial G}{\partial g_{ik}}\right)\bse_{k}\adj X_{i}G\bse_{j}\varphi(\Gamma_{ij})}
				&\leq C_{\caL}\norm{Q}\Psi^{2}(\Pi_{i}^{2}+\Pi_{j}^{2}+(\Pi_{i}^{W})^{2})\varphi(\Gamma_{ij}),
			\end{align*} 
			for some constant $C_{\caL}$ depending only on $\caL$.
			In addition, the same estimates hold if we replace $\frac{\partial G}{\partial g_{ad}}$ and $\frac{\partial T_{a}}{\partial g_{ad}}$ by their complex conjugates $\frac{\partial \ol{G}}{\partial g_{ad}}$ and $\frac{\partial \ol{T_{a}}}{\partial g_{ad}}$.
		\end{lem}
		\begin{proof}
			The proof is a straightforward modification of Lemma 5.3 in \cite{Bao-Erdos-Schnelli2020}. The only difference is we use
			\beq\label{eq:rec_PK_err}
			\norm{G\wt{B}\bse_{i}}\leq 1+\absv{z}\norm{G\bse_{i}}+\absv{\fra_{i}}\norm{G\bse_{i}}+\sqrt{t}\norm{GW\bse_{i}}
			\leq 1+C\sqrt{\frac{\im G_{ii}}{\eta}}+\sqrt{\frac{t\im (WGW)_{ii}}{\eta}}.
			\eeq
			Note that the last term on the right-hand side of \eqref{eq:rec_PK_err} does not appear in \cite{Bao-Erdos-Schnelli2020}. We omit further details.
		\end{proof}
		With this lemma, we can follow the proof of Lemma 5.3 in \cite{Bao-Erdos-Schnelli2020} verbatim to obtain the recursive moment estimates for $P_{ij}$ and $K_{ij}$.Define for $p,q\in\N$
		\begin{align}
			\frm_{ij}^{(p,q)}\deq P_{ij}^{p}\ol{P_{ij}}^{q}\varphi(\Gamma_{ij})^{k+l}, &&
			\frn_{ij}^{(p,q)}\deq K_{ij}^{p}\ol{K_{ij}}^{q}\varphi(\Gamma_{ij})^{k+l}.
		\end{align}
		Then we have
		\beq\begin{aligned}\label{eq:recursive}
			\expct{\frm_{ij}^{(p,p)}}&
			=\expct{\rmk_{1}\frm_{ij}^{(p-1,p)}}
			+\expct{\rmk_{2}\frm_{ij}^{(p-2,p)}}
			+\expct{\rmk_{3}\frm_{ij}^{(p-1,p-1)}},\\
			\expct{\frn_{ij}^{(p,p)}}&
			=\expct{\rmk_{1}'\frn_{ij}^{(p-1,p)}}
			+\expct{\rmk_{2}'\frn_{ij}^{(p-2,p)}}
			+\expct{\rmk_{3}'\frn_{ij}^{(p-1,p-1)}}
		\end{aligned}\eeq
		where $\rmk_{j}$ and $\rmk_{j}'$ are some random variables satisfying
		\beq\begin{aligned}\label{eq:rec_PK_err_coeff}
			&\absv{\rmk_{1}}+\absv{\rmk_{1}'}\prec N^{-1/2}, \\
			&\absv{\rmk_{2}}+\absv{\rmk_{2}'}\prec\Pi_{a}^{2}+(\Pi_{a}^{W})^{2},\\ 
			&\absv{\rmk_{3}}+\absv{\rmk_{3}'}\prec\Pi_{a}^{2}+(\Pi_{a}^{W})^{2}.
		\end{aligned}\eeq
		As in Lemma \ref{lem:etrll_J}, applying Young's inequality to \eqref{eq:recursive} gives the result.
	\end{proof}
	
	Now we are ready to prove the entrywise subordination;
	\begin{prop}\label{prop:local law 1}
		Fix $z\in \caD$ and assume that
		\begin{align}\label{eq:ansatz}
			&\Lambda_{\rme}(z)\prec N^{-\sigma/4}, &
			&\Lambda_{L}(z)\prec N^{-\sigma/4},&
			&\Lambda_{T}(z)\prec 1,
		\end{align}
		and the same set of bounds hold true for $\wt{\Lambda}_{\rme},\wt{\Lambda}_{L},$ and $\wt{\Lambda}_{T}$. Then we have for all $i,j\in\llbra 1,N \rrbra$ that
		\beq\label{eq:etrll_1}
		|L_{ij}|\prec \Psi(z),\quad |P_{ij}|\prec \Psi(z), \quad |K_{ij}|\prec \Psi(z), \AND	|\Upsilon(z)|\prec \Psi(z).
		\eeq
		Furthermore, we have
		\beq\label{eq:etrll_2}
		\Lambda_{\rme}^{c}\prec \Psi(z) \AND
		\Lambda_{T}\prec \Psi(z).
		\eeq
		The same statements remain true if we switch the roles of $(A,B)$, $(U,U\adj)$, and $(W,\caW)$.
	\end{prop}
	\begin{proof}
		First of all, we remark that the assumption \eqref{eq:ansatz} implies 
		\beq\label{eq:ansatz_conseq}
		\begin{split}
			\tr(G)&=\wh{m}+O_{\prec}(N^{-\sigma/4}),\\
			\tr(\wt{B}G)&=1+\omega_{B}(z)\wh{m}+O_{\prec}(N^{-\sigma/4}), \\
			\tr(\wt{B}G\wt{B})&=\omega_{B}(z)(1+\omega_{B}(z)\wh{m})+O_{\prec}(N^{-\sigma/4}).
		\end{split}
		\eeq
		They can be proved in exactly the same way as in (5.20) of \cite{Bao-Erdos-Schnelli2020}. Also the first two estimates in \eqref{eq:ansatz_conseq} combined with the definition of $\omega_{A}^{c}$ imply
		\beq\label{eq:etrll_prior_omegac}
		\omega_{A}^{c}(z)=F_{\wh{\mu}_{t}}(z)+z-\omega_{B}(z)+O_{\prec}(N^{-\sigma/4})=\omega_{A}(z)+O_{\prec}(N^{-\sigma/4}).
		\eeq
		
		The first three estimates in \eqref{eq:etrll_1} are direct consequences of Lemmas \ref{lem:etrll_L} and \ref{lem:etrll_PK} since $\varphi(\Gamma_{ij})=1=\varphi(\Gamma_{ij0})$ under the assumption \eqref{eq:ansatz}. The remaining estimates can be proved using the same argument as in Proposition 5.1 in \cite{Bao-Erdos-Schnelli2020}.
	\end{proof}
	
	\subsection{Fluctuation averaging estimates}\label{sec:FA}
	\subsubsection{Rough fluctuation averaging for general linear combinations}
	This section is a counterpart of Section 6 in \cite{Bao-Erdos-Schnelli2020}, in the sense that we prove a rough fluctuation averaging estimate for $Q_{aa}$; see Proposition \ref{prop:FA-1} below. To deal with the contribution of $W$, we prove a fluctuation averaging estimates for $L_{aa}$ and $J_{a}$. Then we follow the same method as Proposition 6.1 of \cite{Bao-Erdos-Schnelli2020}, where the results for $L_{aa}$ and $J_{a}$ are used as additional inputs.
	
	Before proceeding to the proof, we observe that the average of $\Pi^{W}_{a}$ is dominated by $\Pi$. To be precise, since $\tr WKW\leq \norm{W}\tr K$ for any positive matrix $K$ and $\norm{W}\prec 1$, we have 
	\beq\label{eq:PiaW estimate}
	\frac{t}{N}\sum_{a}(\Pi_{a}^{W})^{2}=t\frac{\tr W(\im G)W}{N\eta}\leq t\norm{W}\frac{\tr \im G}{N\eta}\prec t\Pi^{2}.
	\eeq
	Next, we show the fluctuation averaging estimates for $L_{aa}$ and $J_{a}$.
	\begin{prop}\label{prop:RFA of Laa}
		Fix a $z\in\caD$. Suppose that the assumptions of Proposition~\ref{prop:local law 1} hold.  Let $d_{1},\dots, d_{N}\in\bbC$ be possibly $H$-dependent quantities satisfying $\max|d_{i}|\prec 1$. Assume that for all $i,j\in\llbra 1,N \rrbra$,
		\beq\label{eq:weak depned-2}
		\frac{\sqrt{t}}{N}\sum_{k}\frac{\partial d_{j}}{\partial W_{ki}}\bse_{k}\adj G\bse_{i}=\caO(\Psi^{2}\Pi_{i}^{2}),
		\eeq
		and the same bounds hold when the $d_{j}$'s are replaced by their complex conjugates $\ol{d_{j}}$. Suppose that $\Pi_{a}\prec \wh{\Pi},\, \forall a\in \llbra 1, N\rrbra$ for some deterministic and positive function $\wh{\Pi}(z)$ that satisfies $\frac{1}{\sqrt{N\sqrt{\eta}}}+\Psi^{2}\prec \wh{\Pi}\prec \Psi$. Then
		\beqs
		\left|\frac{\sqrt{t}}{N}\sum_{a}d_{a}L_{aa}\right|\prec \Psi\wh{\Pi}.
		\eeqs
	\end{prop}
	As in Lemma \ref{lem:etrll_L}, it is suffices to show the following recursive moment estimate.
	\begin{lem}\label{lem:recursive moment for linear comb of Laa}
		Fix a $z\in\caD$. Suppose that the assumptions of Proposition \ref{prop:local law 1} hold. Then, for any fixed integer $p\geq 1$, we have
		\beqs
		\expct{\rmn^{(p,p)}}=\expct{\caO(\wh{\Pi}^{2})\rmn^{(p-1,p)}}+\expct{(\Psi^{2}\wh{\Pi}^{2})\rmn^{(p-2,p)}}+\expct{\caO(\Psi^{2}\wh{\Pi}^{2})\rmn^{(p-1,p-1)}},
		\eeqs
		where we defined
		\beqs
		\rmn^{(k,l)}\deq \left(\frac{\sqrt{t}}{N}\sum_{a}d_{a}L_{aa}\right)^{k}\left(\frac{\sqrt{t}}{N}\sum_{a}\ol{d_{a}L_{aa}}\right)^{l}.
		\eeqs
	\end{lem}
	\begin{proof}
		From the definition of $\rmn^{(p,p)}$ and the Stein's lemma, we have that
		\beq\label{eq:averaged Laa-2}
		\begin{split}
			\expct{\rmn^{(p,p)}}&=\frac{\sqrt{t}}{N}\expct{\sum_{a}\sum_{d}d_{a}W_{ad}G_{da}\rmn^{(p-1,p)}}
			+\frac{t}{N}\sum_{a}\expct{\tr(G)G_{aa}\rmn^{(p-1,p)}}	\\
			&=\frac{\sqrt{t}}{N^{2}}\expct{\sum_{a}\sum_{d}\frac{\partial d_{a}}{\partial W_{da}}G_{da}\rmn^{(p-1,p)}}
			+\frac{(p-1)\sqrt{t}}{N^{2}}\sum_{a}\sum_{d}\expct{d_{a}G_{da}\frac{\partial \rmn^{(1,0)}}{\partial W_{da}}\rmn^{(p-2,p)}}	\\
			&+\frac{p\sqrt{t}}{N^{2}}\sum_{a}\sum_{d}\expct{d_{a}G_{da}\frac{\partial \rmn^{(0,1)}}{\partial W_{da}}\rmn^{(p-1,p-1)}}.
		\end{split}
		\eeq
		The first term can be handed by the assumption in Proposition \ref{prop:RFA of Laa}. Remaining terms can be dealt with Lemma \ref{lem:rec_err_aver_L} below.
	\end{proof}
	\begin{lem}\label{lem:rec_err_aver_L} Fix a $z\in\caD$. Suppose that the assumptions of Proposition \ref{prop:RFA of Laa} hold and let $Q$ be an $(N\times N)$ matrix. Then we have
		\beq\label{eq:recursive moment of averaged Laa}
		\begin{split}
			\frac{\sqrt{t}}{N^{2}}\sum_{i}\sum_{k}d_{i}G_{ki}\tr\left(Q\frac{\partial WG}{\partial W_{ki}}\right)=O_{\prec}(\norm{Q}\Psi^{2}\Pi^{2}),\\
			\frac{\sqrt{t}}{N^{2}}\sum_{i}\sum_{k}d_{i}G_{ki}\tr G \tr\left(Q\frac{\partial G}{\partial W_{ki}}\right)=O_{\prec}(\norm{Q}\Psi^{2}\Pi^{2}),
		\end{split}
		\eeq
		and the same estimates hold if we replace the $\partial W_{da}$ by $\partial W_{ad}$.
	\end{lem}
	\begin{proof}
		We only prove the first estimate, since the second can be obtained in a similar way. Note that 
		\beqs\begin{aligned}
			&\frac{\sqrt{t}}{N^{2}}\left|\sum_{i}\sum_{k} d_{i}G_{ki}\tr\left( Q\frac{\partial W }{\partial W_{ki}}G\right)\right|	\\
			=&\frac{\sqrt{t}}{N^{3}}\left|\sum_{i}\sum_{k} d_{i}G_{ki}(GQ)_{ik}\right|
			\prec \frac{\sqrt{t}}{N^{3}}\sum_{i} \norm{Q}\frac{\im G_{ii}}{\eta}= \frac{\norm{Q}\sqrt{t}}{N}\Pi^{2},
		\end{aligned}\eeqs
		where we used the Cauchy-Schwarz inequality and \eqref{eq:PiaW estimate}. On the other hand, we also have
		\beqs
		\begin{split}
			&\left|\frac{\sqrt{t}}{N^{2}}\sum_{i}\sum_{k} d_{i}G_{ki}\tr\left( QW\frac{\partial G}{\partial W_{ki}}\right)\right|
			=\left|-\frac{t}{N^{3}}\sum_{i}\sum_{k} d_{i}\bse_{i}\adj G QWG\bse_{k}G_{ki}\right|\\
			&\leq \left|\frac{\sqrt{t}}{N^{3}} \sum_{i} d_{i}\bse_{i}\adj GQ (zG+I-AG-\wt{B}G)G\bse_{i} \right|
			\leq C\norm{Q}\left(\frac{1}{N\eta}\Pi^{2}+\frac{1}{N}\Pi^{2}\right)
		\end{split}
		\eeqs
		where we used \eqref{eq:rotate identity} and the Cauchy-Schwarz inequality. Adding the two estimates above proves the first estimate in \eqref{eq:recursive moment of averaged Laa}. 
	\end{proof}
	\begin{prop}\label{prop:RFA of Ja}
		Fix a $z\in\caD$. Suppose the assumptions of Proposition~\ref{prop:local law 1} hold.  Let $d_{1},\dots, d_{N}\in\bbC$ be possibly $H$-dependent quantities satisfying $\max|d_{i}|\prec 1$. Assume that for all $i,j\in\llbra 1,N \rrbra$,
		\beq\label{eq:weak depned-3}
		\frac{t}{N}\sum_{k}^{(a)}\frac{\partial d_{j}}{\partial W_{ki}}\bse_{k}\adj WG\bse_{i}=\caO(\Psi^{2}\Pi_{i}^{2}),
		\eeq
		and the same bounds hold when the $d_{j}$'s are replaced by their complex conjugates $\ol{d_{j}}$. Suppose that $\Pi_{a}\prec \wh{\Pi},\, \forall a\in \llbra 1, N\rrbra$ for some deterministic and positive function $\wh{\Pi}(z)$ that satisfies $\frac{1}{\sqrt{N\sqrt{\eta}}}+\Psi^{2}\prec \wh{\Pi}\prec \Psi$. Then,
		\beqs
		\left|\frac{t}{N}\sum_{a}d_{a}J_{a}\right|\prec \Psi\wh{\Pi}.
		\eeqs
	\end{prop}
	The proof is omitted since it is the same as that of Proposition \ref{prop:RFA of Laa} except that we use Lemma \ref{lem:rec_err_aver_J} below as an input, instead of Lemma \ref{lem:rec_err_aver_L}. 
	\begin{lem}\label{lem:rec_err_aver_J}
		Fix a $z\in\caD$. Suppose that the assumptions of Proposition \ref{prop:RFA of Ja} hold. Let $Q$ be a matrix. Then we have
		\begin{align}\label{eq:recursive moment of averaged Ja}
			\frac{t}{N^{2}}\sum_{i}\sum_{k}d_{i}(WG)_{ik}\tr\left(Q\frac{\partial (WGW)}{\partial W_{ik}}\right)=O_{\prec}(\norm{Q}\Psi^{2}\Pi^{2}),\\
			\frac{t}{N^{2}}\sum_{i}\sum_{k}d_{i}(WG)_{ik}\tr(QWG)\tr\left(\frac{\partial G}{\partial W_{ik}}\right)=O_{\prec}(\norm{Q}\Psi^{2}\Pi^{2}),\\
			\frac{t}{N^{2}}\sum_{i}\sum_{k}d_{i}(WG)_{ik}\tr(G)\tr\left(Q\frac{\partial (WG)}{\partial W_{ik}}\right)=O_{\prec}(\norm{Q}\Psi^{2}\Pi^{2}),\\
		\end{align}
		and the same estimates hold if we replace the $\partial W_{da}$ by $\partial W_{ad}$.
	\end{lem}
	\begin{proof}[Proof of Lemma \ref{lem:rec_err_aver_J}]
		We only consider the first equation in \eqref{eq:recursive moment of averaged Ja} since others can be shown in similar way. Computing the derivative, we have that
		\begin{align*}
			&\frac{t^{2}}{N^{2}}\sum_{i,k}d_{i}(WG)_{ik}\tr\left(Q\frac{(WGW)}{\partial W_{ik}}\right)\nonumber\\
			&=\frac{t^{2}}{N^{3}}\sum_{i,k}d_{a}(WG)_{ik}(QGW)_{ki}
			-\frac{t^{2}\sqrt{t}}{N^{3}}\sum_{i,k}d_{a}(WG)_{ik}(GW QWG)_{ki}
			+\frac{t^{2}}{N^{3}}\sum_{a}d_{a}(WGQWG)_{aa}\nonumber\\
			&=O(\frac{t\norm{Q}}{N}\Pi^{2})+O(\frac{t\norm{Q}}{N\eta}(\Pi^{2}))+O(\frac{t\norm{Q}}{N}(\Pi^{2}))
		\end{align*}
		where we have used the Cauchy-Schwarz inequality, \eqref{eq:rotate identity}, and \eqref{eq:PiaW estimate}.
	\end{proof}
	
	Now we turn to general averages of $Q_{aa}$, which is an analogue of \cite[Proposition 6.1]{Bao-Erdos-Schnelli2020};
	\begin{prop}\label{prop:FA-1}
		Fix a $z\in\caD$. Suppose the assumptions of Proposition~\ref{prop:local law 1} hold. Let $d_{1},\dots, d_{N}\in\bbC$ be possibly $H$-dependent quantities satisfying $\max|d_{i}|\prec 1$. Assume that for all $i,j\in\llbra 1,N \rrbra$,
		\begin{align}\label{eq:weak depned}
			\frac{1}{N}\sum_{k}^{(i)}\frac{\partial d_{j}}{\partial g_{ik}}\bse_{k}\adj X_{i}G\bse_{i}=\caO(\Psi^{2}\Pi_{i}^{2}), &&
			\frac{1}{N}\sum_{k}^{(i)}\frac{\partial d_{j}}{\partial g_{ik}}\bse_{k}\adj X_{i}\mr{\bsg}_{i}=\caO(\Psi^{2}\Pi_{i}^{2}),
		\end{align}
		and the same bounds hold when the $d_{j}$'s are replaced by their complex conjugates $\ol{d_{j}}$. Suppose that $\max_{a}\Pi_{a}\prec \wh{\Pi}$ for some deterministic, positive function $\wh{\Pi}(z)$ satisfying $\frac{1}{\sqrt{N\sqrt{\eta}}}+\Psi^{2}\prec \wh{\Pi}\prec \Psi$. Then.
		\beqs
		\left|\frac{1}{N}\sum_{a}d_{a}Q_{aa}\right|\prec \Psi\wh{\Pi}.
		\eeqs
	\end{prop}
	Again, we may easily reduce the proof of Proposition \ref{prop:FA-1} to the corresponding recursive moment estimate.
	\begin{lem}\label{lem:recursive-2}
		Fix $z\in\caD$. Suppose that the assumptions of Proposition \ref{prop:FA-1} hold. Then, for any fixed integer $p\geq 1$, we have
		\begin{align}\label{eq:recursive moment estimate for dQ-2}
			\expct{\rmm^{(p,p)}}=\expct{\caO(\wh{\Pi}^{2})\rmm^{(p-1,p)}}+\expct{\caO(\Psi^{2}\wh{\Pi}^{2})\rmm^{(p-2,p)}}+\expct{\caO(\Psi^{2}\wh{\Pi}^{2})\rmm^{(p-1,p-1)}}
		\end{align}
		where we defined
		\begin{align}\label{eq:dQsum}
			\rmm^{(k,l)}\deq \left(\frac{1}{N}\sum_{a}d_{a}Q_{aa}\right)^{k}\left(\frac{1}{N}\sum_{a}\ol{d_{a}}\ol{Q_{aa}}\right)^{l}.
		\end{align}
	\end{lem}
	\begin{proof}[Proof of Lemma \ref{lem:recursive-2}]
		We first claim that if $|\Upsilon|\prec \wh{\Upsilon}(z)$ for a deterministic, positive function $\wh{\Upsilon}\leq \Psi(z)$, then
		\begin{align}\label{eq:recursive moment estimate for dQ}
			\expct{\rmm^{(p,p)}}
			=\expct{(\caO(\wh{\Pi}^{2})
				+\caO(\Psi\wh{\Upsilon}))\rmm^{(p-1,p)}}
			+\expct{\caO(\Psi^{2}\wh{\Pi}^{2})\rmm^{(p-2,p)}}
			+\expct{\caO(\Psi^{2}\wh{\Pi}^{2})\rmm^{(p-1,p-1)}},
		\end{align}
		with conventions $\rmm^{(0,0)}=1=\rmm^{(-1,1)}$. The estimate \eqref{eq:recursive moment estimate for dQ} can be proved using the methods from \cite[Lemma 6.2]{Bao-Erdos-Schnelli2020}. The major difference from \cite[Lemma 6.2]{Bao-Erdos-Schnelli2020} arises when we differentiate entries or traces of $\wt{B}G$ with respect to $\bsg_{i}$'s, in which case additional terms involving derivatives of $WG$ appear due to \eqref{eq:rotate identity}. Since $U$ and $W$ are independent, these terms can be handled with estimates in Lemma \ref{lem:A.6}. Due to similarity we omit further details. 
		
		Next, we prove Lemma \ref{lem:recursive-2} from \eqref{eq:recursive moment estimate for dQ}. Firstly, Young's and Markov's inequalities provide that
		\beq
		\left|\frac{1}{N}\sum_{a}d_{a}Q_{aa}\right|\prec \wh{\Pi}^{2}+\Psi\wh{\Upsilon}+\Psi\wh{\Pi}\prec \Psi\wh{\Upsilon}+\Psi\wh{\Pi}.
		\eeq
		Now we note that
		\begin{align*}
			\Upsilon
			=\frac{1}{N}\sum_{a}\fra_{a}Q_{aa}-\sqrt{t}\tr(G)\tr(W)
			+\frac{\sqrt{t}}{N}\sum_{a}\tr((I-z+\wt{B})G)L_{aa}
			+\frac{t}{N}\sum_{a}\tr(G)J_{a}.
		\end{align*}
		We can check that the weights for $L_{aa}$ and $J_{a}$ satisfy the assumptions in Propositions \ref{prop:RFA of Laa} and \ref{prop:RFA of Ja} respectively, hence we can
		apply them to obtain
		\beq\label{eq:Upsilon bound}
		|\Upsilon|\prec \Psi\wh{\Upsilon}+\Psi\wh{\Pi}+\frac{1}{N}+\Psi\wh{\Pi}\prec N^{-\frac{\epsilon}{4}}\wh{\Upsilon}+\Psi\wh{\Pi}.
		\eeq
		Updating $\wh{\Upsilon}$ as the right hand side of \eqref{eq:Upsilon bound} and iterating \eqref{eq:recursive moment estimate for dQ} repeatedly give $|\Upsilon|\prec \Psi \wh{\Pi}$. Hence we finally choose $\wh{\Upsilon}=\Psi\wh{\Pi}$ and use the assumption $\frac{1}{\sqrt{N\sqrt{\eta}}}+\Psi^{2}\prec \wh{\Pi}$ to conclude Lemma \ref{lem:recursive-2}.
	\end{proof}

	\subsubsection{Optimal fluctuation averaging}\label{sec:opt_FA}
	In this section, we prove the following improved estimate for a specific linear combination of $Q_{aa}$'s. This result is an analogue of \cite[Proposition 7.1]{Bao-Erdos-Schnelli2020} and is used later to validate the assumption of Lemma \ref{lem:quad_solve}, which leads to an estimate for $\Lambda$.
	\begin{prop}\label{prop:lambda estimate}
		Fix a $z=E+\ii\eta\in\caD_{\sigma}$. Suppose that the assumptions of Proposition \ref{prop:local law 1} hold. Suppose that $\Lambda\prec \wh{\Lambda}$, for some deterministic and positive function $\wh{\Lambda}\prec N^{-\epsilon/4}$, then
		\beq
		|\caS \Lambda_{\iota} +\caT_{\iota}\Lambda_{\iota}^{2}+O(\Lambda_{\iota}^{3})|\prec \frac{\sqrt{(\im \wh{m}+\wh{\Lambda})(|\caS|+\wh{\Lambda})}}{N\eta}+\frac{1}{(N\eta)^{2}}, 	\quad \iota=A,B.\label{eq:opt lambda estimate}
		\eeq
	\end{prop}
	
	We first express the left-hand side of \eqref{eq:opt lambda estimate} in terms of weighted averages of $Q_{aa}$ and $L_{aa}$'s. To obtain such linear combination, we recall that the subordinate system $\Phi_{AB}$ defined in \eqref{eq:Phi_def} vanishes at the point $(\omega_{A}(z),\omega_{B}(z),z)$. As our final goal is to bound the difference between the approximate and genuine subordination functions, we evaluate the system at the point $(\omega_{A}^{c}(z),\omega_{B}^{c}(z),z)$;
	\beqs
	\Phi_{A}^{c}\deq \Phi_{A}(\omega_{A}^{c},\omega_{B}^{c},z) \AND \Phi_{B}^{c}\deq \Phi_{B}(\omega_{A}^{c},\omega_{B}^{c},z).
	\eeqs
	From \eqref{eq:rotate identity} and \eqref{eq:Lambdac<Q+L}, we have that
	\beq\label{eq:Phi_c_rep1}
	\omega_{A}^{c}+\omega_{B}^{c}-z
	=z-\frac{\tr (A+\wt{B})G}{\tr G}+2t\tr G=-\frac{1}{\tr G}+t\tr G+ \frac{\sqrt{t}}{\tr G}\frac{1}{N}\sum_{a}L_{aa},
	\eeq
	and
	\beq\label{eq:Phi_c_rep2}
	F_{A}(\omega_{A}^{c})+\frac{1}{\tr G}-t\tr G
	=\left(\frac{1}{m_{A}(\omega_{A}^{c})\tr G}+t\right) \frac{1}{N}\sum_{a}\frac{Q_{aa}+\sqrt{t}L_{aa}\tr G}{(\fra_{a}-\omega_{A}^{c})\tr G}.
	\eeq
	Combining \eqref{eq:Phi_c_rep1} and \eqref{eq:Phi_c_rep2}, we can write 
	\beq\label{eq:Phi_c_rep3}
	\Phi_{A}^{c}
	=\frac{1}{N}\sum_{a}(\frd_{A,a}Q_{aa}+\sqrt{t}\frd_{A,a}^{W}L_{aa}) \AND 
	\Phi_{B}^{c}=\frac{1}{N}\sum_{a}(\frd_{B,a}\caQ_{aa}+\sqrt{t}\frd_{B,a}^{\caW}\caL_{aa}),
	\eeq
	where we defined
	\beqs
	\frd_{A,a}\deq \left(\frac{1}{m_{A}(\omega_{A}^{c})\tr G}+t\right)\frac{1}{(\fra_{a}-\omega_{A}^{c})\tr G}, \qquad\quad
	\frd_{A,a}^{W}\deq \frd_{A,a}\tr G+\frac{1}{\tr G},
	\eeqs
	and $\frd_{B,a}$, $\frd_{B,a}^{W}$ symmetrically. On the other hand, we expand $\Phi_{A}^{c}$ and $\Phi_{B}^{c}$ around $(\omega_{A},\omega_{B},z)$ to obtain
	\begin{align}\label{eq:approximation of subordination system}
		\begin{split}
			\Phi_{A}^{c}&=-\Lambda_{B}+(F_{A}'(\omega_{A})-1)\Lambda_{A}+\frac{1}{2}F_{A}''(\omega_{A})\Lambda_{A}^{2}+O(\Lambda_{A}^{3}),\\
			\Phi_{B}^{c}&=-\Lambda_{A}+(F_{B}'(\omega_{B})-1)\Lambda_{B}+\frac{1}{2}F_{B}''(\omega_{B})\Lambda_{B}^{2}+O(\Lambda_{B}^{3}).
		\end{split}
	\end{align}
	Combining \eqref{eq:Phi_c_rep3} and \eqref{eq:approximation of subordination system}, we get
	\beq\label{eq:caZ expansion}
	\caZ_{A}\deq (F_{B}'(\omega_{B})-1)\Phi_{A}^{c}+\Phi_{B}^{c}=\caS \Lambda_{A}+\caT_{A}\Lambda_{A}^{2}+O((\Phi_{A}^{c})^{2})+O(\Phi_{A}^{c}\Lambda_{A})+O(\Lambda_{\iota_{A}}^{3}),
	\eeq
	and the same expansion with $A$ and $B$ interchanged. Note that the leading terms on the  right-hand side of \eqref{eq:caZ expansion} matches the left-hand side of \eqref{eq:opt lambda estimate}. Since $\Phi_{A}^{c}$ and $\Phi_{B}^{c}$ are respectively linear combinations of $Q_{aa}$, $L_{aa}$ and $\caQ_{aa}$, $\caL_{aa}$, we have accomplished the first goal.
	
	Along the proof of Proposition \ref{prop:lambda estimate}, we often need to apply Proposition \ref{prop:FA-1} with $\wh{\Pi}$ chosen to be the square root of the right hand side of \eqref{eq:opt lambda estimate}, i.e.,
	\beq\label{eq:hatPi}
	\wh{\Pi}^{2} =\frac{\sqrt{(\im \wh{m}+\wh{\Lambda})(|\caS|+\wh{\Lambda})}}{N\eta}+\frac{1}{(N\eta)^{2}}.
	\eeq
	Thus we need to prove that $\wh{\Pi}$ satisfies the assumptions in Proposition \ref{prop:FA-1}. To this end, we claim that $|m_{H_{t}}-\wh{m}|\prec\wh{\Lambda}+\Psi^{2}$ when $\Lambda \prec\wh{\Lambda}\prec N^{-\epsilon/4}$. First, observe from the definition of approximate subordination functions \eqref{eq:apxsubor_def} that
	\beq
	\absv{F_{H_{t}}(z)-F_{\wh{\mu}_{t}}(z)}=\absv{\omega_{A}^{c}(z)-\omega_{A}(z)+\omega_{B}^{c}(z)-\omega_{B}(z)+t\frac{1}{Nm_{H_{t}}(z)}\sum_{a}L_{aa}}\prec \Lambda+\frac{\Psi^{2}}{\absv{m_{H_{t}}(z)}},
	\eeq
	where we applied Proposition \ref{prop:RFA of Laa}. On the other hand, recall from the definition of $F_{t}$ that
	\beq 
	|F_{H_{t}}(z)- F_{\wh{m}}(z)|
	=|m_{t}-\wh{m}|\left|\frac{1}{m_{t}\wh{m}}-t\right|.
	\eeq
	Due to Lemmas \ref{lem:subor_bdd} and \ref{lem:subor_diff}, there exists $C>0$ such that $|\wh{m}(z)|<C$ for any $z\in\caD_{\sigma}$. On the other hand, Cauchy-Schwarz inequality and compactness of $\caD_{\sigma}$ imply that there exists $c>0$ satisfying
	\beq
	\int_{\R}\frac{1}{|x-z|^{2}} \dd \wh{\mu}(x) -\Absv{\int_{\R} \frac{1}{x-z}\dd \wh{\mu}(x)}^{2}>c
	\eeq
	uniformly on $\caD_{\sigma}$, so that for another constant $c'>0$ we get $c'<1-t\absv{\wh{m}(z)}^{2}$.
	Thus we have
	\beq
	c'\absv{m_{H_{t}}-\wh{m}}-t\absv{m_{t}-\wh{m}}^{2}\leq \absv{\wh{m}m_{H_{t}}}\absv{F_{\mu_{H_{t}}}-F_{\wh{\mu}}}
	\prec \Lambda+\Psi^{2}+\absv{m_{H_{t}}-\wh{m}}\Lambda.
	\eeq
	Due to the assumption $\Lambda_{d}\prec N^{-\sigma/4}$, we can solve this quadratic inequality to obtain $|m_{H_{t}}-\wh{m}|\prec \Lambda+\Psi^{2}$. Along with Proposition \ref{prop:stab}, we can see the validity of $\wh{\Pi}$ for Proposition \ref{prop:FA-1}.
	
	Next, we prove Proposition \ref{prop:lambda estimate} assuming the validity of the following estimate for $\caZ_{\iota}$.
	\begin{prop}\label{prop:FA-ZAZB}
		Fix $z\in\caD_{\sigma}$. Suppose that the assumptions of Proposition \ref{prop:local law 1} hold and that $\Lambda(z)\prec \wh{\Lambda}(z)$ for some deterministic and positive function $\wh\Lambda(z)\leq N^{-\epsilon/4}$. Choose $\hat{\Pi}(z)$ as in \eqref{eq:hatPi} . Then,
		\begin{align}
			|\caZ_{A}|\prec \wh{\Pi}^{2}, && |\caZ_{B}|\prec \wh{\Pi}^{2}.
		\end{align}
	\end{prop}
	\begin{proof}[Proof of Proposition \ref{prop:lambda estimate}] By definition of $\omega_{\iota}^{c}$, chain rule, and \eqref{eq:ansatz_conseq}, it is easy to check that $\frd_{\iota,i}$ and $\frd_{\iota,i}^{W}$, satisfy the assumptions (for weights $d_{i}$'s) in Proposition \ref{prop:FA-1} and \ref{prop:RFA of Laa} respectively for $\iota=A,B$ and $i=1,2,\dots,N$. Hence we have that 
		\beq
		|\Phi_{\iota}^{c}|\prec \Psi\wh{\Pi} \quad \iota=A,B.
		\eeq
		Combining this estimate, equation \eqref{eq:caZ expansion} and Proposition \ref{prop:FA-ZAZB} gives that
		\beq\label{eq:FAFA}
		|\caS\Lambda_{\iota}+\caT_{\iota}\Lambda_{\iota}^{2}+O(\Lambda_{A}^{3})|\prec \wh{\Pi}^{2}+\Psi\wh{\Pi}\wh{\Lambda} \quad \iota=A,B.
		\eeq
		The definition of $\wh{\Pi}$ implies $\Psi\wh{\Lambda}\prec \wh{\Pi}$ so that the second term on the right hand side of \eqref{eq:FAFA} can be absorbed into the first term. Thus we have Proposition \ref{prop:lambda estimate}.
	\end{proof}
	
	As in the previous proofs, the proof of Proposition \ref{prop:FA-ZAZB} reduces to the following recursive moment.
	\begin{lem}\label{lem:recursive-3}
		Fix a $z\in\caD_{\sigma}$. Suppose that the assumptions of Proposition \ref{prop:FA-ZAZB} hold. For any fixed $p\geq 1$, we have
		\begin{align}
			\expct{\frL^{(p,p)}}=\expct{\caO(\wh{\Pi}^{2})\frL^{(p-1,p)}}+\expct{\caO(\wh{\Pi}^{4})\frL^{(p-2,p)}}+\expct{\caO(\wh{\Pi}^{4})\frL^{(p-1,p-1)}},
		\end{align}
		where we denote
		\beq
		\frL^{(k,l)}\equiv \frL^{(k,l)}(z)\deq \caZ_{A}^{k}\ol{\caZ_{A}}^{l}, \quad k,l\in\bbN.
		\eeq
		with conventions $\frL_{i}^{(0,0)}=1$ and $\frL_{i}^{(-1,1)}=0$. 
	\end{lem}
	\begin{proof}[Proof of Lemma \ref{lem:recursive-3}]
		Recall that
		\begin{align}\label{eq:optimal-1}\begin{split}
				&\expct{\frL^{(p,p)}}=\expct{\caZ_{A}\frL^{(p-1,p)}}	\\
				=&\frac{F_{B}'(\omega_{B})-1}{N}\sum_{a}\expct{\frd_{a,A}Q_{aa}\frL^{(p-1,p)}}
				+\frac{1}{N}\sum_{a}\expct{\frd_{a,B}\caQ_{aa}\frL^{(p-1,p)}}\\
				&+\frac{\sqrt{t} (F_{B}'(\omega_{B})-1)}{N}\sum_{a}\expct{\frd_{A,a}^{W}L_{aa}\frL^{(p-1,p)}}
				+\frac{\sqrt{t}}{N}\sum_{a}\expct{\frd_{B,a}^{\caW}\caL_{aa}\frL^{(p-1,p)}}.
			\end{split}
		\end{align}
		We follow the strategy of proof of Lemma 7.3 in \cite{Bao-Erdos-Schnelli2020} so that the main task of the proof is estimating the terms including the derivatives of $\caZ_{A}$ or $\ol{\caZ_{A}}$ (cf. Lemma 7.4 in \cite{Bao-Erdos-Schnelli2020}). Due to similarity, we only mention that the extra gain for the estimate comes from the fact that 
		\beq\label{eq:opt_FA_gain}
		\frac{\partial \caZ_{A}}{\partial g_{ik}}=(\caS_{AB}+O(\Lambda))\frac{\partial\omega_{A}^{c}}{\partial g_{ik}}+O(\Lambda)\frac{\partial\omega_{B}^{c}}{\partial g_{ik}},
		\eeq
		which directly follows from the definition of $\Phi_{AB}$. Replacing derivatives with respect to $g_{ik}$ by $W_{ab}$ in \eqref{eq:opt_FA_gain} proves the estimate for the third and fourth terms.
	\end{proof}
	\subsection{Proof of Theorem \ref{thm:ll}}\label{sec:LL}
	In this section we prove Theorem \ref{thm:ll}. The proof consists of two parts, weak local law and strong local law. The former implies that the assumptions in Proposition \ref{prop:local law 1} hold uniformly true on $\caD$ and the latter proves Theorem \ref{thm:ll}. 
	
	We state the weak law in the following theorem.
	\begin{thm}\label{thm:weak local law} Suppose that Assumption \ref{assump:ABconv} holds. Then, for all $i,j\in \llbra 1,N \rrbra$, we have
		\begin{align*}
			&|P_{ij}(z)|\prec \Pi_{i}(z)+\Pi_{j}(z), & 
			&|K_{ij}(z)|\prec \Pi_{i}(z)+\Pi_{j}(z), &
			&\Lambda_{\rme}^{c}(z)\prec\Psi(z), \\
			&\Lambda_{L}(z)\prec\Psi(z), &
			& \Lambda_{T}\prec \Psi(z), &
			& |\Upsilon(z)|\prec \Pi,  
		\end{align*}
		uniformly in $z\in\caD$. In addition, we have
		\begin{align}
			\Lambda_{\rme}\prec \frac{1}{(N\eta)^{1/3}}, && \Lambda(z)\prec \frac{1}{(N\eta)^{1/3}},
		\end{align} 
		uniformly in $z\in\caD$.
		The same statements hold for analogous quantities with roles of $(A,B)$, $(U,U\adj)$, and $(W,\caW)$ interchanged.
	\end{thm}
	Before we prove the theorem, we collect two major inputs as lemmas. The first one, Lemma \ref{lem:quad_solve}, enables us to convert the bound in Proposition \ref{prop:lambda estimate} to bounds for $\Lambda$. We omit its proof for it is exactly the same as \cite[Lemma 8.2]{Bao-Erdos-Schnelli2020}.
	\begin{lem}\label{lem:quad_solve}
		Fix $z\in\caD$. Let $\epsilon\in(0,\frac{\sigma}{100})$ and $k\in(0,1]$. Let $\wh{\Lambda}\equiv \wh{\Lambda}(z)$ be some deterministic control parameter satisfying $ \wh{\Lambda}\leq N^{-\sigma/4}$. Suppose that $\Lambda(z) \leq \wh{\Lambda}(z)$ and
		\beqs
		|\caS \Lambda_{C}+\caT_{C} \Lambda_{C}^{2}+O(\Lambda_{C}^{3})|\leq N^{\epsilon}\frac{|\caS|+\wh{\Lambda}}{(N\eta)^{k}}, \qquad C=A,B
		\eeqs
		hold on some event $\Omega(z)$. Then there exists a constant $C>0$ such that for sufficiently large $N$, the following hold:
		\begin{itemize}
			\item [(i)] If $\sqrt{\kappa+\eta}>N^{-\epsilon}\wh{\Lambda}$, there is a sufficiently large constant $K_{0}>0$ independent of $z$, such that
			\begin{align}\label{eq:lemma f.2-(i)}
				\lone\left(\Lambda\leq \frac{|\caS|}{K_{0}}\right)|\Lambda_{A}|\leq \left(N^{-2\epsilon}\wh{\Lambda}+\frac{N^{2\epsilon}}{(N\eta)^{k}}\right), &&
				\lone\left(\Lambda\leq \frac{|\caS|}{K_{0}}\right)|\Lambda_{B}|\leq \left(N^{-2\epsilon}\wh{\Lambda}+\frac{N^{2\epsilon}}{(N\eta)^{k}}\right) && \text{on } \Omega(z),
			\end{align}
			where $\lone$ denotes the indicator function.
			\item [(ii)] If $\sqrt{\kappa+\eta}\leq N^{-\epsilon}\wh{\Lambda}$, we have
			\begin{align*}
				|\Lambda_{A}|\leq\left(N^{-\epsilon}\wh{\Lambda}+\frac{N^{2\epsilon}}{(N\eta)^{k}}\right), && 
				|\Lambda_{B}|\leq \left(N^{-\epsilon}\wh{\Lambda}+\frac{N^{2\epsilon}}{(N\eta)^{k}}\right) && \text{on } \Omega(z).
			\end{align*}
		\end{itemize}
	\end{lem}
	
	In order to state the second input, we introduce a few more notations. For $z\in\caD$ and $\delta,\delta'\in [0,1]$, we define the event
	\beqs
	\Theta(z,\delta,\delta')\deq \left\{\Lambda_{d}\leq \delta,\, \wt{\Lambda}_{d}(z)\leq \delta,\, \Lambda(z)\leq \delta,\, \Lambda_{o}(z)\leq \delta',\, \Lambda_{T}(z)\leq \delta',\, \wt{\Lambda}_{T}(z)\leq \delta' \right\}.
	\eeqs
	In addition, we decompose the domain $\caD$ into the following disjoint parts:
	\begin{align*}
		\caD_{>}\deq \left\{ z\in\caD:\sqrt{\kappa+\eta}>\frac{N^{2\epsilon}}{(N\eta)^{1/3}} \right\}, &&
		\caD_{\leq} \deq \left\{ z\in\caD:\sqrt{\kappa+\eta}\leq \frac{N^{2\epsilon}}{(N\eta)^{1/3}} \right\}. 
	\end{align*}
	For $z\in\caD_{>}$, $\delta,\delta'\in[0,1]$ and $\epsilon' \in[0,1]$, we define the event $\Theta_{>}(z,\delta,\delta',\epsilon')\subset \Theta(z,\delta,\delta')$ as
	\begin{align*}
		\Theta_{>}(z,\delta,\delta',\epsilon')\deq
		\left\{\Lambda_{d}(z)\leq \delta,\, \wt{\Lambda}_{d}(z)\leq \delta,\,
		\Lambda(z)\leq \min\{\delta,N^{-\epsilon'}|\caS|\},\, 
		\Lambda_{o}\leq \delta' ,\,
		\Lambda_{T}(z)\leq \delta',\, \wt{\Lambda}_{T}(z)\leq \delta'
		\right\}.
	\end{align*}
	Now we are ready to state the second input, which lets us use a bootstrapping argument;
	\begin{lem}\label{lem:bootstrap}
		For any fixed $z\in\caD$, any $\epsilon\in(0,\frac{\sigma}{100})$ and any $D>0$, there exists a positive integer $N_{1}(D,\epsilon)$ and an event $\Omega(z)\deq \Omega(z,D,\epsilon)$ with
		\beqs
		\bbP(\Omega(z))\geq 1-N^{-D}, \quad \forall N\geq N_{1}(D,\epsilon),
		\eeqs
		such that the following hold:
		\begin{itemize}
			\item [(i)] If $z\in\caD_{>}$, we have
			\beqs
			\Theta_{>}\left(z,\frac{N^{3\epsilon}}{(N\eta)^{1/3}},\frac{N^{3\epsilon}}{\sqrt{N\eta}},\frac{\epsilon}{10}\right)\cap \Omega(z) \subset \Theta_{>}\left(z,\frac{N^{\frac{5}{2}\epsilon}}{(N\eta)^{1/3}},\frac{N^{\frac{5}{2}\epsilon}}{\sqrt{N\eta}},\frac{\epsilon}{2}\right).
			\eeqs
			\item [(ii)] If $z\in\caD_{\leq}$, we have
			\beqs
			\Theta\left(z,\frac{N^{3\epsilon}}{(N\eta)^{1/3}},\frac{N^{3\epsilon}}{\sqrt{N\eta}}\right)\cap \Omega(z) \subset \Theta\left(z,\frac{N^{\frac{5}{2}\epsilon}}{(N\eta)^{1/3}},\frac{N^{\frac{5}{2}\epsilon}}{\sqrt{N\eta}}\right).
			\eeqs
		\end{itemize}
	\end{lem}
	
	\begin{proof}
		As in \cite[Lemma 8.3]{Bao-Erdos-Schnelli2020}, the proof of Lemma \ref{lem:bootstrap} requires quantitative versions of each of moment estimates in previous sections. To be specific, we need their counterparts which do not depend on the probabilistic input \eqref{eq:ansatz}. In this sense the estimates for entrywise local laws, Lemmas \ref{lem:etrll_J}, \ref{lem:etrll_L}, and \ref{lem:etrll_PK}, are already quantitative; recall that for all $\epsilon_{1}>0$ and $p\in\N$, there exists an $N_{0}\equiv N_{0}(\epsilon_{1},p)\in\N$ such that
		\beq\begin{aligned}\label{eq:JLPK_cutoff}
			&\expct{\absv{J_{i}\varphi(\Gamma_{aa0})}^{2p}}\leq N^{\epsilon_{1}}\Psi^{2p},	&
			&\expct{\absv{L_{ij}\varphi(\Gamma_{ij0})}^{2p}}\leq N^{\epsilon_{1}}\Psi^{2p},		\\
			&\expct{\absv{P_{ij}\varphi(\Gamma_{ij})}^{2p}}\leq N^{\epsilon_{1}}\Psi^{2p},	&
			&\expct{\absv{K_{ij}\varphi(\Gamma_{ij})}^{2p}}\leq N^{\epsilon_{1}}\Psi^{2p},
		\end{aligned}\eeq
		whenever $N\geq N_{0}$, $i,j\in\braN$, $z\in\caD$, and $t\in[0,1]$. Before moving on to the analogue of Lemma \ref{lem:recursive-3}, we first derive consequences of \eqref{eq:JLPK_cutoff}. Since
		\beqs
		\varphi(\Gamma_{ij0})=1=\varphi(\Gamma_{ij})\qquad \text{on }\Theta\left(z,\frac{N^{3\epsilon}}{(N\eta)^{1/3}},\frac{N^{3\epsilon}}{\sqrt{N\eta}}\right),
		\eeqs
		we may apply Markov's inequality to \eqref{eq:JLPK_cutoff} to find an event $\Omega_{1}\equiv \Omega_{1}(z,D,\epsilon)$ so that
		\begin{align}\label{eq:JLPK_markov}
			&\absv{J_{i}}\leq N^{\sigma/4}\Psi,&
			&\absv{L_{ij}}\leq N^{\sigma/4}\Psi, &
			&\absv{P_{ij}}\leq N^{\sigma/4}\Psi, &
			&\absv{K_{ij}}\leq N^{\sigma/4}\Psi	
		\end{align}
		are all true on the event $\Theta\left(z,\frac{N^{3\epsilon}}{(N\eta)^{1/3}},\frac{N^{3\epsilon}}{\sqrt{N\eta}}\right)\cap\Omega_{1}$, and that
		\beqs
		\prob{\Omega_{1}^{c}}\leq \frac{1}{10}N^{-D}.
		\eeqs
		Then we follow the proof of Proposition \ref{prop:local law 1} to obtain that
		\begin{align}\label{eq:prior bound cutoff}
			&\Lambda_{d}^{c}(z)\leq \frac{N^{\epsilon/2}}\Psi, &&
			\Lambda_{o} \leq N^{\epsilon/2}\Psi, &&
			\Lambda_{T}(z)\leq N^{\epsilon/2}\Psi, &&
			\wt{\Lambda}_{T}(z)\leq N^{\epsilon/2}\Psi, &&
			|\Upsilon(z)|\leq N^{\epsilon/2}\Psi
		\end{align}
		hold on $\Theta(z,\frac{N^{3\epsilon}}{(N\eta)^{1/3}},\frac{N^{3\epsilon}}{\sqrt{N\eta}})\cap \Omega_{1}(z)$.
		
		Now we present the quantitative versions of estimates in Section \ref{sec:opt_FA} with weaker bounds. Specifically, we consider the following two quantities;
		\begin{align*}
			\wt{\frl}^{(p,q)}&\deq \left(\frac{1}{N}\sum_{i}\frd_{A,i}^{W}L_{ii}\varphi(\Gamma_{ii0})\right)^{p}\left(\frac{1}{N}\sum_{i}\ol{\frd_{A,i}^{W}L_{ii}}\varphi(\Gamma_{ii0})\right)^{q},\\
			\wt{\frm}^{(p,q)}&\deq \left(\frac{1}{N}\sum_{i}\frd_{A,i}Q_{ii}\varphi(\Gamma_{ii})\varphi(\Gamma)\right)^{p}
			\left(\frac{1}{N}\sum_{i}\ol{\frd_{A,i}Q_{ii}}\varphi(\Gamma_{ii})\varphi(\Gamma)\right)^{q},
		\end{align*}
		where $\Gamma_{ij0}$ and $\Gamma_{ij}$ are as in \eqref{eq:Gamma0_def}, \eqref{eq:Gamma_def} and $\Gamma$ is defined as 
		\begin{align*}
			\Gamma\deq (c\im \wh{m}+\wh{\Lambda})^{-2}(|\Lambda_{A}|^{2}+|\Lambda_{B}|^{2})
			+\left(\frac{N^{5\epsilon}}{(N\eta)^{1/3}}\right)^{-2}|\Upsilon|^{2}
			+\left(\frac{N^{5\epsilon}}{\sqrt{N\eta}}\right)^{-1}\frac{1}{N}\sum_{a}(|T_{ij}|^{2}+N^{-1})^{1/2},
		\end{align*}
		for some sufficiently small constant $c>0$. In the rest of the proof, we choose
		\beqs
		\wh{\Lambda}(z)=\frac{N^{3\epsilon}}{(N\eta)^{1/3}},
		\eeqs
		and define $\wh{\Pi}$ as in \eqref{eq:hatPi}.
		
		Following the same calculations as in Proposition \ref{prop:RFA of Laa} but taking Lemma \ref{lem:etrll_L_err} as an additional input, we find that 
		\beqs
		\expct{\wt{\frl}^{(p,p)}}\leq N^{\epsilon_{1}}\Psi^{4p}.
		\eeqs
		In addition, due to our choice of $\wh{\Lambda}$, we have
		\beqs
		\Psi^{2}=\frac{1}{N\eta}\leq \frac{1}{(N\eta)^{-2/3}}\leq \wh{\Pi}.
		\eeqs
		Now for the average of $Q_{ii}$'s, we follow the exact same argument as in Section C of \cite{Bao-Erdos-Schnelli2020} to prove that 
		\beqs
		\expct{\wt{\frm}^{(p,p)}}\leq N^{\epsilon_{1}}\wh{\Pi}^{2p}.
		\eeqs
		Repeating the same arguments with averages of $\caL_{aa}$ and $\caQ_{aa}$ and applying Markov's inequality, since $\varphi(\Gamma)=\varphi(\Gamma_{ii})=\varphi(\Gamma_{ii0})=1$ on the event $\Theta\left(z\frac{N^{3\epsilon}}{(N\eta)^{1/3}},\frac{N^{3\epsilon}}{\sqrt{N\eta}}\right)$, there exists an event $\Omega_{2}\equiv \Omega_{2}(z,D,\epsilon)$ such that
		\beqs
		\Phi_{\iota}^{c}\leq N^{\epsilon/3}\wh{\Pi},\quad \iota =A,B, \qquad \text{on }\Theta\left(z\frac{N^{3\epsilon}}{(N\eta)^{1/3}},\frac{N^{3\epsilon}}{\sqrt{N\eta}}\right)\cap\Omega_{2}(z)
		\eeqs
		and $\prob{\Omega_{2}(z)}\geq 1-\frac{N^{-D}}{10}$. On this intersection of events, we further have from \eqref{eq:opt lambda estimate} that
		\beqs
		\absv{\caS\Lambda_{\iota}+\caT_{\iota}\Lambda_{\iota}^{2}+O(\absv{\Lambda_{\iota}}^{3})}\leq C N^{\epsilon/3}\wh{\Pi}\leq N^{\epsilon}\frac{\absv{\caS}+\wh{\Lambda}}{(N\eta)^{1/3}},\quad \iota=A,B,
		\eeqs
		where we used the definition of $\wh{\Lambda}$ in the last inequality.
		
		Taking $\Omega\equiv \Omega(z,D,\epsilon)=\Omega_{1}\cap\Omega_{2}$, we find that the assumptions of Lemma \ref{lem:quad_solve} are satisfied with $k=1/3$. The rest of the proof of Lemma \ref{lem:bootstrap} are identical to that of Lemma 8.3 in \cite{Bao-Erdos-Schnelli2020}.
	\end{proof}
	
	With Lemma \ref{lem:bootstrap}, we now prove Theorem \ref{thm:weak local law} by using a continuity argument.
	\begin{proof}[Proof of Theorem \ref{thm:weak local law}]
		To prove Theorem \ref{thm:weak local law}, we modify that of Theorem 8.1 in \cite{Bao-Erdos-Schnelli2020}. Specifically, we first prove the result when $\eta=\eta_{M}\sim 1$, and then use Lemma \ref{lem:bootstrap} and a lattice continuity argument to gradually decrease $\eta$ until we reach the optimal regime $\eta=N^{-1+\sigma}$. Since the bootstrapping part of the proof is identical to that of Theorem 8.1 in \cite{Bao-Erdos-Schnelli2020}, we focus on the first part. That is, we prove that there exists a sufficiently large constant $\eta_{M}>0$ so that the event
		\beqs
		\Theta_{>}\left(z,\frac{N^{3\epsilon}}{N^{1/3}},\frac{N^{3\epsilon}}{\sqrt{N}},\frac{\epsilon}{10}\right)
		\eeqs has high-probability uniformly over $z\in\caD$ with $\im z=\eta_{M}$.
		
		Following (8.35) in \cite{Bao-Erdos-Schnelli2020}, we can prove that $Q_{ij}$ as a function of $U$ is Lipschitz continuous with respect to the Hilbert-Schmidt norm $\norm{X}_{2}=\sqrt{\Tr XX\adj}$ with the Lipschitz constant bounded by $C\eta^{-2}$. Since the constant $C$ here can be chosen independent of $W$, applying Gromov-Milman concentration inequality (see Corollary 4.4.28 of \cite{Anderson-Guionnet-Zeitouni2010}) to $Q_{ij}$ yields
		\beq\label{eq:Qab concentration}
		\Absv{Q_{ij}(E+\ii\eta)-\expct{Q_{ij}(E+\ii\eta)|W}}\prec \frac{1}{\sqrt{N\eta^{2}}}
		\eeq
		whenever $\eta\geq \eta_{M}$. Furthermore, using the invariance of the Haar measure we can check the identity
		\beqs
		\expct{\wt{B}G \otimes G- G\otimes \wt{B}G\vert W}=0.
		\eeqs
		Taking the $(i,j)$-th entry for the first component in the tensor product and the normalized trace for the second component, we have
		\beq\label{eq:EQab}
		\expct{Q_{ij}\vert W}=\expct{(\wt{B}G)_{ij}\tr G-G_{ij}\tr \wt{B}G\vert W}=0.
		\eeq
		Following the same proof as that of (8.38) in \cite{Bao-Erdos-Schnelli2020}, we can extend the bound to the whole domain $\eta\geq\eta_{M}$ (enlarging $\eta_{M}$ if necessary); 
		\beq\label{eq:Qbound}
		\sup_{z:\im z \geq \eta_{M}} |Q_{ij}(z)|\prec\frac{1}{\sqrt{N}},
		\quad \forall i,j\in\llbra 1, N\rrbra.
		\eeq
		Similarly, applying Proposition 2.3.3 in \cite{Anderson-Guionnet-Zeitouni2010} to $L_{ab}$ as a function of $W$, we get
		\beqs
		\sup_{z:\im z\geq\eta_{M}}\absv{L_{ij}(z)}\prec \frac{1}{\sqrt{N}},\quad \forall i,j\in\braN.
		\eeqs
		
		In addition, using that $\norm{H_{t}}\leq \norm{A}+\norm{B}+t\norm{W}\leq C$ with high probability and that $\tr\wt{B}=\tr B=0$, we have, for $z=E+\ii\eta$ with fixed $E$ and any $\eta\geq \eta_{M}$, the expansions
		\begin{align*}
			\tr G(z)=-\frac{1}{z}+O(\frac{1}{|z|^{2}})=\frac{\ii}{\wt{\eta}_{M}}+O(\frac{1}{\eta^{2}}),
			&&
			\tr\wt{B}G(z)=-\frac{\tr \wt{B}}{z}+O(\frac{1}{|z|^{2}})=O(\frac{1}{\eta^{2}}),
		\end{align*}
		where we used $\tr B=0$ in the second equality. Hence, by the definition of $\omega_{A}^{c}$, we see that
		\beq\label{eq:omega expansion}
		\omega_{A}^{c}=z-\frac{\tr(\wt{B}G+\sqrt{t}WG)}{\tr(G)}=z+O(\frac{1}{\wt{\eta}_{M}}), \quad z=E+\ii\wt{\eta}_{M}.
		\eeq
		Using the identity $(\wt{B}G)_{ij}=\delta_{ij}-(\fra_{i}-z)G-\sqrt{t}(WG)_{ij}$, we can rewrite \eqref{eq:Qbound} as
		\begin{align*}
			(\delta_{ij}-(\fra_{i}-\omega_{A}^{c})G_{ij})\tr G +\sqrt{t}\tr(WG)G_{ij}-\sqrt{t}(WG)_{ij}\tr(G)=\caO(\frac{1}{\sqrt{N}}),&& z=E+\ii\wt{\eta}_{M}.
		\end{align*}
		Thus we have
		\beq\label{eq:prior estimate for big eta}
		\Lambda_{\rme}^{c}(z)\prec\frac{1}{\sqrt{N}},\quad \Lambda_{L}(z)\prec\frac{1}{\sqrt{N}},\qquad  z=E+\ii\wt{\eta}_{M}.
		\eeq
		Taking the average of diagonal terms in \eqref{eq:prior estimate for big eta} yields
		\begin{align*}
			\sup_{z:\im z\geq \eta_{M}}|m_{H_{t}}(z)-m_{A}(\omega_{A}^{c}(z))|\prec\frac{1}{\sqrt{N}}, &&
			\sup_{z:\im z\geq \eta_{M}}|m_{H_{t}}(z)-m_{B}(\omega_{B}^{c}(z))|\prec\frac{1}{\sqrt{N}}
		\end{align*}
		where in the large $z$ regime these bounds even hold deterministically. This gives the system
		\begin{align}\label{eq:weak bound for subordinate system}
			\sup_{z:\im z\geq \eta_{M}} |\Phi_{A}(\omega_{A}^{c}(z),\omega_{B}^{c}(z),z)|\prec\frac{1}{\sqrt{N}}, &&
			\sup_{z:\im z\geq \eta_{M}} |\Phi_{B}(\omega_{A}^{c}(z),\omega_{B}^{c}(z),z)|\prec\frac{1}{\sqrt{N}}.
		\end{align}
		We regard \eqref{eq:weak bound for subordinate system} as a perturbation of $\Phi_{AB}(\omega_{A}(z),\omega_{B}(z),z)=0$, whose stability in the macroscopic regime is provided in Lemma \ref{lem:stab_macro}. Since \eqref{eq:weak bound for subordinate system} and \eqref{eq:omega expansion} hold for sufficiently large $\wt{\eta}_{M}$, Lemma \ref{lem:stab_macro} implies that
		\beqs
		|\Lambda_{\iota}|=|\omega_{\iota}^{c}(z)-\omega_{\iota}(z)|\prec\frac{1}{\sqrt{N}}, \quad \iota=A,B, \,\, z=E+\ii\eta_{M},
		\eeqs
		taking larger $\eta_{M}>1$ if necessary. Thus we have
		\beq\label{eq:lambda estimate for big eta}
		|\Lambda_{d}(E+\ii\eta_{M})|\leq |\Lambda_{d}^{c}(E+\ii\eta_{M})|+|\Lambda_{A}(E+\ii\eta_{M})|\prec\frac{1}{\sqrt{N}},
		\eeq
		for any fixed $E\in\bbR$. Using the bound $\norm{G}\leq \frac{1}{\eta}$ and the inequality $|\bfx\adj G\bfy|\leq \norm{G}\norm{\bfx}\norm{\bfy}$, we also get
		\beqs
		\Lambda_{T}(E+\ii\eta_{M})\leq \frac{1}{\eta_{M}},
		\eeqs
		for any fixed $E\in\bbR$. Hence we observe that the assumptions in Proposition \ref{prop:local law 1} are satisfied so that we have, for any fixed $E\in\bbR$, that
		\begin{align}\label{eq:T estimate for big eta}
			|\Lambda_{T}(E+\ii\eta_{M})|\prec \frac{1}{\sqrt{N}},&&
			|\wt{\Lambda}_{T}(E+\ii\eta_{M})|\prec \frac{1}{\sqrt{N}}.
		\end{align}
		Also, note that $E+\ii\eta_{M}\in\caD_{>}$ and $|\caS(E+\tt\eta_{M})|\gtrsim 1$ for any fixed $E$. Hence $\Lambda({E+\ii\eta_{M}})\prec N^{-\epsilon}|\caS(E+\ii\eta_{M})|$. From \eqref{eq:lambda estimate for big eta}, we have
		\beq\label{eq:lambda estimate for big eta-2}
		\Lambda(E+\ii\eta_{M})\prec\frac{1}{\sqrt{N}}.
		\eeq
		Combining \eqref{eq:lambda estimate for big eta}, \eqref{eq:prior estimate for big eta}, \eqref{eq:T estimate for big eta} and \eqref{eq:lambda estimate for big eta-2} with the fact $\Lambda\prec N^{-\epsilon}|\caS(E+\ii\eta_{M})|$, we see that
		\beq\label{eq:initial Theta}
		\Theta_{>}\left(E+\ii\eta_{M},\frac{N^{3\epsilon}}{N^{1/3}},\frac{N^{3\epsilon}}{\sqrt{N}},\frac{\epsilon}{10}\right)\geq 1-N^{-D},
		\eeq
		for all $E$ and $N\geq N_{0}(D,\epsilon)$ with some sufficiently large $N_{0}(D,\epsilon)$. This concludes the proof of Theorem \ref{thm:weak local law}.
	\end{proof}
	
	Now we are ready to prove the strong local law, Theorem \ref{thm:ll}.
	\begin{proof}[Proof of Theorem \ref{thm:ll}]
		We first prove the bound 
		\beq\label{eq:lambda_ll}
		\Lambda(z)\prec \frac{1}{N\eta}.
		\eeq
		Now that we have the weak local law Theorem \ref{thm:weak local law}, the probabilistic assumptions in \ref{prop:local law 1} hold uniformly on the domain $\caD$. Thus the conclusion of Proposition \ref{prop:lambda estimate} holds true uniformly on $\caD$. That is, 
		\beq\label{eq:strong bound for lambda}
		\left| \caS\Lambda_{\iota}+\caT_{\iota}\Lambda_{\iota}^{2}+O(\Lambda_{\iota}^{3})\right|\prec
		\frac{\sqrt{(\im \wh{m}+\wh{\Lambda})(|\caS|+\wh{\Lambda})}}{N\eta}+\frac{1}{(N\eta)^{2}}, \quad
		\iota =A,B
		\eeq
		holds uniformly in $\caD$. Furthermore, since $\im \wh{m}\leq C|\caS|$, we find that the assumptions of Lemma \ref{lem:bootstrap} hold true for the choice $k=1$ as long as $\wh{\Lambda}\geq (N\eta)^{-1}$ and $\Lambda(z)\prec \wh{\Lambda}$.
		
		As in the proof of Theorem 2.5 in \cite{Bao-Erdos-Schnelli2020}, we use the bootstrapping argument for $\wh{\Lambda}$ applying Lemma \ref{lem:bootstrap} with $k=1$. The initial choice is $\wh{\Lambda}(z)=N^{3\epsilon}(N\eta)^{-1/3}$, which is guaranteed by Theorem \ref{thm:weak local law}, and we use the same argument as in \cite{Bao-Erdos-Schnelli2020} to iteratively improve the bound until we have \eqref{eq:lambda_ll}.
		
		Next, we prove Theorem \ref{thm:ll}. Firstly, the averaged local law \eqref{eq:averll} is a consequence of Proposition \ref{prop:FA-1} and \eqref{eq:lambda_ll}. Secondly for the entrywise local law \eqref{eq:etrll}, note that \eqref{eq:lambda_ll} implies
		\beq\label{eq:etrll_improv}
		\Pi_{i}(z)\prec \sqrt{\frac{\im \wh{m}(z)+\Lambda_{\rme}(z)}{N\eta}}+\frac{1}{N\eta}.
		\eeq
		Similarly, we have 
		\beq\label{eq:etrll_improv1}
		\Pi_{i}^{W}(z)=\sqrt{\frac{\im (WGW)_{ii}}{N\eta}}
		\prec \sqrt{\frac{\absv{J_{i}}+\absv{L_{ii}}+\im \wh{m}(z)+\Lambda_{\rme}(z)}{N\eta}}+\frac{1}{N\eta}.
		\eeq
		Then we repeat the proof of Lemma \ref{lem:etrll_PK}, but using Theorem \ref{thm:weak local law}, \eqref{eq:etrll_improv} and \eqref{eq:etrll_improv1} as an input in \eqref{eq:rec_PK_err_coeff}. This gives 
		\beq\label{eq:P_opt}
		P_{ij}\prec \sqrt{\frac{\im \wh{m}(z)+\Lambda_{\rme}(z)}{N\eta}}+\frac{1}{N\eta}.
		\eeq
		Using $\Upsilon\prec \Psi^{2}$, we find that $Q_{ij}$ also satisfies the same estimate as in \eqref{eq:P_opt}. Applying the same procedure to the proof of Lemma \ref{lem:etrll_L}, we have
		\beq\label{eq:QL_opt}
		\absv{Q_{ij}}+\absv{L_{ij}}\prec\sqrt{\frac{\im \wh{m}(z)+\Lambda_{\rme}(z)}{N\eta}}+\frac{1}{N\eta}.
		\eeq
		Combining with \eqref{eq:Lambdac<Q+L} and $\absv{\omega_{A}^{c}(z)-\omega_{\alpha}(z)}\prec \Psi^{2}$, we finally arrive at
		\beq\label{eq:etrll_improv2}
		\Lambda_{\rme}(z)\prec \sqrt{\frac{\im \wh{m}(z)+\Lambda_{\rme}(z)}{N\eta}}+\frac{1}{N\eta}.
		\eeq
		This proves $\Lambda_{\rme}(z)\prec \Pi$, which immediately implies \eqref{eq:etrll}.
	\end{proof}
	\section{Proofs of Proposition \ref{prop:ll} and Lemma \ref{lem:BGU}}\label{sec:rigidity}
	Note the discrepancy in definitions of $G$: In this supplementary material we defined $G\deq(H_{t}-z)^{-1}$, whereas in the main manuscript we used the same alphabet $G$ to denote $(\gamma_{t} H_{t}-z)^{-1}$. In this section and the next, we exclusively use the former notation. Hence, for example, \eqref{eq:ll_etr} is equivalent to
	\beq
	\max_{a,b\in\llbra 1,N\rrbra}\Absv{G_{ab}(z)-\frac{\delta_{ab}}{\fra_{a}-\omega_{\alpha,t}(E_{+,t})}}+\absv{(U\adj G(z))_{ab}}\prec N^{-1/3+\epsilon},
	\eeq
	uniformly over $z=E_{+,t}+(E+\ii\eta_{0})/\gamma_{t}$ with $E\in[E_{1},E_{2}]$. Note however that the difference is merely cosmetic since $\gamma_{t}\sim 1$.
	
	\begin{proof}[Proof of Proposition \ref{prop:ll}]
		To prove \eqref{eq:ll_aver}, we combine Theorem \ref{thm:ll}, Corollary~\ref{cor:sqrt_lim}, and Lemma \ref{lem:subor_diff}. To be specific, we write
		\begin{align}\label{eq:whm=m}
			\begin{split}
				&|\tr G-m_{\mu_{t}}(E_{+,t})|\\		
				\leq&  \left|\frac{1}{N}\sum_{i=1}^{N}\left(G_{ii}(z) -\frac{1}{\fra_{a}-\omega_{A,t}(z)}\right)\right|	\\
				&+\left|\frac{1}{N}\sum_{i=1}^{N} \frac{1}{\fra_{a}-\omega_{A,t}(z)} -\int\frac{1}{x-\omega_{\alpha,t}(z)}\dd \mu_{\alpha}(x) \right|\\
				&+|m_{\mu_{t}}(z)-m_{\mu_{t}}(E_{+,t})|.
			\end{split}
		\end{align}
		Now we prove each term on the right-hand side of \eqref{eq:whm=m} is $N^{-1/3+\epsilon}$: For the first term, we use Theorem \ref{thm:ll} so that it is $O(\Psi^{2})$: For the second, we expand $1/(\fra_{a}-\omega_{A,t})$ around $\omega_{\alpha,t}$ and apply Lemma \ref{lem:subor_diff} to prove that it is $O(\eta_{0}^{-1/2}\bsd)$: The last term is $O(\sqrt{z-E_{+,t}})$ by Proposition \ref{prop:sqrt_lim}. 
		
		The estimate \eqref{eq:ll_etr} can be proved in a similar way by using \eqref{eq:etrll} instead of \eqref{eq:averll}. For \eqref{eq:ll_etr_BG}, we use \eqref{eq:QL_opt} so that
		\beq
		(\wt{B}G)_{ij}=\frac{\tr \wt{B}G}{\tr G}G_{ij}-Q_{ij}=\frac{\tr\wt{B}G}{\tr G}+O_{\prec}(N^{-1/3+\epsilon}).
		\eeq
		Then \eqref{eq:ll_etr_BG} immediately follows from \eqref{eq:whm=m} together with its variant
		\beq\begin{aligned}
			\tr \wt{B}G(z)=\tr B\caG(z)=&\int_{\R}\frac{x}{x-\omega_{\beta,t}(E_{+,t})}\dd\mu_{\beta}(x)+O_{\prec}(N^{-1/3+\epsilon})	\\
			=&(\omega_{\beta}(E_{+,t})m_{\mu_{t}}(E_{+,t})+1)+O_{\prec}(N^{-1/3+\epsilon}).
		\end{aligned}\eeq
	\end{proof}
	
	We next present the proof of Lemma \ref{lem:BGU}, a minor technical consequence of the proof of local law.
	\begin{proof}[Proof of Lemma \ref{lem:BGU}]
		We first consider $(\wt{B}GU)_{ab}$. By the identity $GH_{t}=zG+I=H_{t}G$, we have
		\beq
		(\wt{B}GU)_{ab}=-\fra_{i}(GU)_{ab}-\sqrt{t}(WGU)_{ab}+z(GU)_{ab}+U_{ab}.
		\eeq
		As $(GU)_{ab}=\ol{(U\adj G(\ol{z}))}_{ab}\prec N^{-1/3+\epsilon}$ by \eqref{eq:etrll} and $U_{ab}\prec N^{-1/2}$, it suffices to prove $(WGU)_{ab}\prec N^{-1/3+\epsilon}$.
		
		We follow the proof of Lemma \ref{lem:etrll_L}, that is, we take a high-moment of $(WGU)_{ab}$ and apply Stein's lemma to $W$. More precisely,
		\beq\begin{aligned}\label{eq:WGU}
			&\E\absv{(WGU)_{ab}}^{2p}=\E (WGU)_{ab} (WGU)_{ab}^{p-1}\ol{(WGU)}_{ab}^{p}	\\
			=&-\sqrt{t}\E \tr m (GU)_{ab}(WGU)_{ab}^{p-1}\ol{(WGU)}_{ab}^{p}	\\
			&-\frac{(p-1)}{N}\E \left(\sqrt{t}(WG^{2}U)_{ab}(GU)_{ab}\right)(WGU)_{ii}^{p-2}\ol{(WGU)}_{ii}^{p}	\\
			&+\frac{p}{N}\E\left((U\adj G\adj GU)_{bb}-\sqrt{t}(U\adj G\adj GU)_{bb}(G\adj W)_{aa}\right)\absv{(WGU)_{ii}}^{2(p-1)}.
		\end{aligned}\eeq
		By Theorem \ref{thm:ll} and Ward identity, we can estimate all coefficients of powers of $(WGU)_{ab}$ except $(WG^{2}U)_{ab}$, which can be handled as
		\beqs
		\frac{\absv{(WG^{2}U)_{ab}}}{N}\leq \frac{\im (WGW)_{aa}+\im (U\adj GU)_{bb}}{N\eta}\prec 1.
		\eeqs
		This gives
		\beq\begin{aligned}\label{eq:WGU1}
			\E\absv{(WGU)_{ab}}^{2p}=&\sqrt{t}\E O_{\prec}(N^{-1/3+\epsilon})\absv{(WGU)_{ab}}^{2p-1}	\\
			&+\E O_{\prec}(N^{-1})\absv{(WGU)_{ab}}^{2p},
		\end{aligned}\eeq
		and applying Young's inequality to \eqref{eq:WGU1} proves $(WGU)_{ab}\prec N^{-1/3+\epsilon}$ as desired.
		
		The estimates for $(\wt{B}G\wt{B})_{ab}\prec 1$ is even easier, since the identity $GH_{t}=zG+I=H_{t}G$ and the fact that $A$ is diagonal imply that it suffices to prove
		\beq\begin{aligned}\label{eq:BGB}
			\absv{(WGW)_{ab}}+\absv{(WG)_{ab}}+\absv{(GW)_{ab}}\prec& \delta_{ab}+N^{-1/3+\epsilon},	\\
			\qquad \im(WGW)_{aa}+\im (WG)_{aa}+\im (GW)_{aa}\prec& N^{-1/3+\epsilon}.
		\end{aligned}\eeq
		Note that combining \eqref{eq:etrll_improv1} and \eqref{eq:etrll_improv2} proves 
		\beqs
		\Pi_{a}^{W}\prec\Pi.
		\eeqs
		Feeding this into \eqref{eq:rec_WGW_1} and \eqref{eq:rec_WG}--\eqref{eq:rec_WG_1}, respectively, proves $J_{a}\prec\Pi$ and $L_{ab}\prec\Pi$. Then \eqref{eq:BGB} follows immediately from Theorem \ref{thm:ll}.
	\end{proof}
	\section{Derivatives}\label{sec:der}
	We calculate $\partial R_{a}/(\partial g_{ac})$ in the following self-explanatory lemma:
	\begin{lem}\label{lem:type1}
		For $a\neq c\in\llbra 1, N\rrbra$, we have the following;
		\begin{align}
			\frac{\partial \bsh_{a}}{\partial g_{ac}}
			=&\frac{1}{\norm{\bsg_{a}}}\bse_{c}-\frac{\ol{h}_{ac}}{2\norm{\bsg_{a}}}\bsh_{a}, &
			\frac{\partial \bsh_{a}\adj}{\partial g_{ac}}=&-\frac{\ol{h}_{ac}}{2\norm{\bsg_{a}}}\bsh_{a}\adj,\\
			\frac{\partial\ell_{a}^{2}}{\partial g_{ac}}=&\frac{\partial}{\partial g_{ac}}\frac{1}{1+h_{aa}}=-\ell_{a}^{4}g_{aa}\frac{\partial\norm{\bsg_{a}}^{-1}}{\partial g_{ac}}=\frac{\ell_{a}^{4}}{2\norm{\bsg_{a}}}h_{aa}\ol{h}_{ac}.
		\end{align}
		Consequently, we have
		\beq\begin{aligned}\label{eq:Rder}
			\frac{\partial R_{a}}{\partial g_{ac}}=&-\frac{\partial\ell_{a}^{2}}{\partial g_{ac}}(\bse_{a}+\bsh_{a})(\bse_{a}+\bsh_{a})\adj
			-\ell_{a}^{2}\left(\frac{\partial \bsh_{a}}{\partial g_{ac}}(\bse_{a}+\bsh_{a})\adj+(\bse_{a}+\bsh_{c})\frac{\partial \bsh_{a}\adj}{\partial g_{ac}}\right)	\\
			=&-\frac{\ell_{a}^{4}}{2\norm{\bsg_{a}}}h_{aa}\ol{h}_{ac}(\bse_{a}+\bsh_{a})(\bse_{a}+\bsh_{a})\adj	\\
			&-\frac{\ell_{a}^{2}}{\norm{\bsg_{a}}}\left(\bse_{c}(\bse_{a}+\bsh_{a})\adj-\frac{\ol{h}_{ac}}{2}(\bsh_{a}\bse_{a}\adj+\bse_{a}\bsh_{a}\adj+2\bsh_{a}\bsh_{a}\adj)\right).
		\end{aligned}\eeq
		By symmetry, using $R_{a}\adj=R_{a}$, we have
		\beq\begin{aligned}\label{eq:Rder*}
			\frac{\partial R_{a}}{\ol{\partial}g_{ac}}=\left(\frac{\partial R_{a}}{\partial g_{ac}}\right)\adj=&-\frac{\ell_{a}^{4}}{2\norm{\bsg_{a}}}h_{aa}h_{ac}(\bse_{a}+\bsh_{a})(\bse_{a}+\bsh_{a})\adj	\\
			&-\frac{\ell_{a}^{2}}{\norm{\bsg_{a}}}\left((\bse_{a}+\bsh_{a})\bse_{c}\adj-\frac{h_{ac}}{2}(\bsh_{a}\bse_{a}\adj+\bse_{a}\bsh_{a}\adj+2\bsh_{a}\bsh_{a}\adj)\right).
		\end{aligned}\eeq
	\end{lem}
	Using the result above, we can expand the derivative of $G$ in terms of $h_{b}$;
	\begin{lem}\label{lem:dgbd_comput}
		For $a\neq c\in\llbra 1, N\rrbra$ , we have
		\beq\begin{aligned}
			\frac{\partial G}{\partial g_{ac}}=&-\frac{\ell_{a}^{2}}{\norm{\bsg_{a}}}G[\bse_{c}(\bse_{a}+\bsh_{a})\adj,\wt{B}]G
			+\frac{\ell_{a}^{2}}{2\norm{\bsg_{a}}}\ol{h}_{ac}G[(\bse_{a}+2\bsh_{a})\bse_{a}\adj,\wt{B}]G\\
			&-\frac{\ell_{a}^{4}}{2\norm{\bsg_{a}}}h_{aa}\ol{h}_{ac}G[\bse_{a}\bse_{a}\adj+\bse_{a}\bsh_{a}\adj+\bsh_{a}\bse_{a}\adj,\wt{B}]G.
		\end{aligned}\eeq
		Similarly, we have
		\beq\begin{aligned}
			\frac{\partial G}{\ol{\partial}g_{ac}}=&-\frac{\ell_{a}^2}{\norm{\bsg_{a}}}G [\wt{B},(\bse_{a}+\bsh_{a})\bse_{c}\adj]G
			+\frac{\ell_{a}^{2}}{2\norm{\bsg_{a}}}h_{ac}G[\wt{B},\bse_{a}(\bse_{a}+2\bsh_{a})\adj]G\\
			&-\frac{\ell_{a}^{4}}{2\norm{\bsg_{a}}}h_{aa}h_{ac}G[\wt{B},\bse_{a}\bse_{a}\adj+\bse_{a}\bsh_{a}\adj+\bsh_{a}\bse_{a}\adj]G.
		\end{aligned}\eeq 
	\end{lem}
	\begin{proof}
		From the definition of $G$, we have
		\beqs
		\frac{\partial G}{\partial g_{ac}}=-G\frac{\partial \wt{B}}{\partial g_{ac}}G=-\gamma G\frac{\partial (R_{a}\wt{B}^{\anga}R_{a})}{\partial g_{ac}}G
		=- G\left(\frac{\partial R_{a}}{\partial g_{ac}}\wt{B}^{\anga}R_{a}+R_{a}\wt{B}^{\anga}\frac{\partial R_{a}}{\partial g_{ac}}\right)G,
		\eeqs
		where we used the fact that $U^{\anga}$, and hence $\wt{B}^{\anga}$, are independent of $\bsg_{a}$. Since $R_{a}^{2}=I$, we have
		\beqs
		-R_{a}\frac{\partial R_{a}}{\partial g_{ac}}=\frac{\partial R_{a}}{\partial g_{ac}}R_{a},
		\eeqs
		so that
		\beqs
		\frac{\partial G}{\partial g_{ac}}=-G\left(\frac{\partial R_{a}}{\partial g_{ac}}\wt{B}^{\anga}R_{a}+R_{a}\wt{B}^{\anga}\frac{\partial R_{a}}{\partial g_{ac}}\right)G
		=-G\left[\frac{\partial R_{a}}{\partial g_{ac}}R_{a},\wt{B}\right]G,
		\eeqs
		where $[A,B]=AB-BA$ denotes the commutator. Now using \eqref{eq:Rder} we have
		\begin{align*}
			\frac{\partial R_{a}}{\partial g_{ac}}R_{a}
			&=-\frac{\ell_{a}^{2}}{\norm{\bsg_{a}}}\left(\frac{h_{aa}\ol{h}_{ac}}{2}(I-R_{a})+\bse_{c}(\bse_{a}+\bsh_{a})\adj -\frac{\ol{h}_{ac}}{2}(\bsh_{a}(\bse_{a}+\bsh_{a})\adj+(\bse_{a}+\bsh_{a})\bsh_{a}\adj)\right)R_{a}	\\
			&=\frac{\ell_{a}^{2}}{\norm{\bsg_{a}}}\left(\bse_{c}(\bse_{a}+\bsh_{a})\adj -\frac{\ol{h}_{ac}}{2}((\bse_{a}+2\bsh_{a})\bse_{a}\adj+\bsh_{a}\bsh_{a}\adj)+\frac{h_{aa}\ol{h}_{ac}}{2}(I-R_{a})\right),
		\end{align*}
		where we used the fact that $R_{a}\bse_{a}=-\bsh_{a}$ and $R_{a}\bsh_{a}=-\bse_{a}$.
		
		Therefore we have
		\begin{align*}
			\frac{\partial G}{\partial g_{ac}}
			=&-\frac{\ell_{a}^{2}}{\norm{\bsg_{a}}}G[\bse_{c}(\bse_{a}+\bsh_{a})\adj,\wt{B}]G	
			+\frac{\ell_{a}^{2}}{2\norm{\bsg_{a}}}\ol{h}_{ac}G[(\bse_{a}+2\bsh_{a})\bse_{a}\adj+\bsh_{a}\bsh_{a}\adj,\wt{B}]G	\\
			&-\frac{\ell_{a}^{2}}{2\norm{\bsg_{a}}}h_{aa}\ol{h}_{ac}G[(I-R_{a}),\wt{B}]G	\\
			=&- \frac{\ell_{a}^{2}}{\norm{\bsg_{a}}}G[\bse_{c}(\bse_{a}+\bsh_{a})\adj,\wt{B}]G	
			+\frac{\ell_{a}^{2}}{2\norm{\bsg_{a}}}\ol{h}_{ac}G[(\bse_{a}+2\bsh_{a})\bse_{a}\adj,\wt{B}]G	\\
			&-\frac{\ell_{a}^{4}}{2\norm{\bsg_{a}}}h_{aa}\ol{h}_{ac}G[\bse_{a}\bse_{a}\adj+\bse_{a}\bsh_{a}\adj+\bsh_{a}\bse_{a}\adj,\wt{B}]G,
		\end{align*}
		where we used the fact that $\bsh_{a}\bsh_{a}\adj$ commutes with $\wt{B}$.
		
		The second identity follows immediately from $G\adj=G(\ol{z})$, $\wt{B}\adj=\wt{B}$, and
		\beq
		\frac{\partial G}{\ol{\partial}g_{ac}}=\left(\frac{\partial G\adj}{\partial g_{ac}}\right)\adj.
		\eeq
	\end{proof}

\end{document}